\documentclass[11pt]{amsart}

\usepackage{amsthm}
\usepackage{amssymb}
\usepackage{amsmath}
\usepackage{url}                     
\usepackage{fullpage}
\usepackage{tikz-cd}                 
\usepackage{tikz}
\usepackage{enumitem}
\usepackage{mathrsfs}
\usepackage{mathtools}

\makeatletter
\def\l@subsection{\@tocline{2}{0pt}{2.5pc}{5pc}{}}
\makeatother

\usepackage[                         
   backend=biber,
   style=alphabetic,
   sortlocale=de_DE,
   url=false,
   doi=true,
   eprint=false
]{biblatex}

\usepackage{hyperref}

\newtheorem{theorem}{Theorem}[section]
\newtheorem{lemma}[theorem]{Lemma}
\newtheorem{corollary}[theorem]{Corollary}
\newtheorem{proposition}[theorem]{Proposition}

\theoremstyle{definition}
\newtheorem{definition}[theorem]{Definition}
\newtheorem{remark}[theorem]{Remark}
\newtheorem{example}[theorem]{Example}

\newcommand{\C}{\mathbb{C}}
\newcommand{\R}{\mathbb{R}}
\newcommand{\Q}{\mathbb{Q}}
\newcommand{\Z}{\mathbb{Z}}
\newcommand{\N}{\mathbb{N}}
\newcommand{\HH}{\mathbb{H}}

\newcommand{\Har}{\mathscr{H}}

\newcommand{\fg}{\mathfrak{g}}
\newcommand{\ff}{\mathfrak{f}}

\newcommand{\MC}[1]{\mathcal{#1}}

\newcommand{\eps}{\epsilon}
\newcommand{\dd}{\partial}

\newcommand{\inj}{\hookrightarrow}
\newcommand{\ra}{\rightarrow}
\newcommand{\om}[1]{\operatorname{#1}}
\newcommand{\ovl}[1]{\overline{#1}}
\newcommand{\wt}[1]{\widetilde{#1}}
\newcommand{\Mod}[1]{\ (\mathrm{mod}\ #1)}   
\newcommand{\cyc}[1]{\langle #1 \rangle}

\numberwithin{equation}{section}

\begin{document}

\date{\today}
\author{Gard Olav Helle}
\title{Equivariant Instanton Floer Homology and Calculations for the Binary Polyhedral Spaces}

\begin{abstract}
    We calculate the equivariant instanton Floer homology, in the sense of Miller Eismeier, for the trivial $SO(3)$-bundle over the binary polyhedral
spaces with coefficients in a PID $R$ for which $2\in R$ is invertible. Along the way we modify a part of the algebraic construction needed to define
the equivariant instanton Floer groups.
\end{abstract}

\maketitle
\tableofcontents

\emph{The research contained in this paper was
conducted during the authors time as
a PhD fellow at the University of Oslo
and formed a part of the authors thesis.}

\section{Introduction}
Instanton Floer homology has appeared in many forms in the literature.
In this paper we will be concerned with a fairly recent version: the equivariant instanton Floer homology in the sense of Miller Eismeier \cite{Miller19}.
A pair $(Y,E)$ consisting of a closed oriented $3$-manifold $Y$ and
an $\om{SO}(3)$-bundle $E\ra Y$ is said to be weakly admissible if either
$Y$ is a rational homology sphere, $H_*(Y;\Q)\cong H_*(S^3;\Q)$,
or the second Stiefel-Whitney class $w_2(E)$ admits no torsion lift
in $H^2(Y;\Z)$. Associated with such a pair and a commutative ring of coefficients $R$ (required to be a principal ideal domain), there
are four (relatively) $\Z/8$ graded groups
\begin{equation} \label{Intro-IF-Groups}
 I^+(Y,E;R), \;\; I^-(Y,E;R), \;\; I^\infty(Y,E;R) \; \mbox{ and }\;  \wt{I}(Y,E;R),
\end{equation}  
that together constitute the equivariant instanton Floer homology of
the pair $(Y,E)$. The first three groups carry module structure over 
$H^{-*}(\om{BSO}(3);R)$ and fit into an exact triangle
\begin{equation} \label{Intro-Exact-Triangle} 
\begin{tikzcd} I^+(Y,E;R) \arrow{rr}{[3]} & {} &  I^-(Y,E;R) \arrow{ld}{[0]} \\
 {} & I^\infty(Y,E;R) \arrow{lu}{[-4]} & {}, \end{tikzcd}
\end{equation}
where the numbers denote the degrees of the maps,   
while the final group is a module over $H_*(\om{SO}(3);R)$. 
The construction of the equivariant instanton Floer groups is naturally
divided in two steps. The first step consists of applying Morse-theoretic
techniques to the Chern-Simons functional, defined on a framed configuration
space of connections in the bundle, to construct a type of Morse-Bott complex
$\wt{CI}(Y,E;R)$ equipped with an action of $C_*(\om{SO}(3);R)$.
The second step, purely algebraic in nature, extracts the equivariant groups
from this complex using the bar construction from algebraic topology. 

There are two main purposes of this paper. The first
is to rework some of the algebra needed to set up the equivariant Floer groups.
Below we will elaborate further on why we have found this to be necessary.
The second is to calculate the equivariant Floer groups associated with
the trivial $\om{SO}(3)$-bundle over the binary polyhedral spaces, that is, 
the orbit spaces $Y_\Gamma \coloneqq S^3/\Gamma$ for finite subgroups $\Gamma\subset \om{SU}(2)$. For reasons that will become apparent later, our calculations will be restricted to the case where $2$ is invertible in the ring of coefficients
$R$. In fact, we will focus primarily on the calculations for $\overline{Y}_\Gamma$, the manifold $Y_\Gamma$ equipped with the opposite orientation,
and only briefly treat the calculations for $Y_\Gamma$.  

To be able to explain our results we will first elaborate further on
the brief explanation of the equivariant instanton Floer groups given above. 
To simplify the exposition we will assume that $Y$ is a rational homology
sphere and that $E\ra Y$ is the trivial $\om{SO}(3)$-bundle, both of which
are satisfied in our situation of binary polyhedral spaces. In
this case one may just as well replace $E$ by the trivial
$\om{SU}(2)$-bundle, so we will make this assumption as well.  
Let $\MC{A}_E$ denote the space of connections in $E$ and let $\MC{G}_E$
be the group of gauge transformations, that is, the group of bundle automorphisms $u\colon E\ra E$. The gauge group acts by pullback on the space
of connections, but this action is not free. Indeed, the stabilizer of
a connection $A\in \MC{A}_E$ may be identified, through evaluation
at a point $b\in Y$, with the centralizer of the holonomy group
$\om{Hol}_A(b)\subset \om{Aut}(E_b)\cong \om{SU}(2)$. To remedy this 
one may fix a base point $b\in Y$ and define the space of framed connections by
\[ \wt{\MC{A}}_E \coloneqq \MC{A}_E \times E_b . \]
Here the gauge group acts diagonally, through evaluation in the second component. This action is free and in suitable Sobolev completions (see \cite[Section~2.2]{Miller19}) the configuration space 
\[ \wt{\MC{B}}_E \coloneqq \wt{\MC{A}}_E /\MC{G}_E  \]
is a smooth Hilbert manifold. The right $\om{SU}(2)$-action on the fiber
descends to a right $\om{SO}(3)=\om{SU}(2)/\{\pm1\}$-action on the configuration space. 

The Chern-Simons functional $\om{cs}\colon \wt{\MC{B}}_E\ra \R/\Z\cong S^1$
may be defined for $[A,u]\in \wt{\MC{B}}_E$ represented by
$(A,u)\in \wt{\MC{A}}_E$ by 
\begin{equation}  \label{CS-Functional} 
\om{cs}([A,u]) = \frac1{8\pi^2} \int_Y \om{Tr}(a\wedge da + \frac23 a\wedge a\wedge a)  \in \R/\Z 
\end{equation} 
where $a\in \Omega^1(Y,\mathfrak{su}(2))$ is a connection form representing $A$
in a global trivialization of the bundle (see \cite[p.~18]{Donaldson02}).
The fact that this is independent of the gauge equivalence class of $A$,
or equivalently the choice of global trivialization of the bundle,
is shown in \cite[Proposition~1.27]{Freed95}. Clearly, $\om{cs}$ is independent
of the framing coordinate and is therefore $\om{SO}(3)$-invariant.
The equivariant framed Floer complex $\wt{CI}(Y,E)$ is constructed
as a type of Morse-Bott complex for the functional $\om{cs}\colon \wt{\MC{B}}_E\ra S^1$. This means that the complex is, in a certain sense, generated
by the set of critical orbits of $\om{cs}$ and the differentials are defined
via fiber products with the spaces of negative gradient flow lines between the various critical orbits. First, the set of critical points of $\om{cs}$ regarded as a map $\MC{A}_E\ra \R/\Z$ is precisely the set of flat
connections in $E$, so that $\MC{C}$, the set of critical orbits in $\wt{\MC{B}}_E$,
consists precisely of the $\om{SO}(3)$-orbits of the gauge equivalence
classes of the flat connections. Furthermore, the negative gradient
flow equation for $\om{cs}$ is, for a path of connections $A_t\in \MC{A}_E$,
given by, disregarding the normalizing constant,
\[ \frac{dA_t}{dt} = -*F_{A_t}  \]
where $F_{A_t}$ is the curvature of the connection $A_t$ and
$*:\Omega^2(Y,\om{su}(E))\ra \Omega^1(Y,\om{su}(E))$ is the Hodge star
operator. Now, a path of connections $A_t$ may in a natural way be identified
with a connection $\mathbf{A}$ in the bundle $\R\times E\ra \R\times Y$
over the cylinder. In terms of this relation the negative gradient
flow equation takes the form of the familiar anti-self-dual or instanton equation
\[ F_{\mathbf{A}}^+ =0   .\]
In general, as in finite dimensional Morse theory, one is typically forced
to introduce a perturbation of the functional to ensure that
the critical points are non-degenerate and that the spaces of flow lines
are cut out transversally. A flat connection $\alpha$, regarded as a critical point of $\om{cs}$, is non-degenerate if the twisted cohomology group
$H^1_\alpha(Y;\om{su}(E))$ vanishes. Given a pair of non-degenerate flat
connections $\alpha$ and $\beta$ and a relative homotopy class $z\in \pi_1(\wt{\MC{B}}_E,\alpha,\beta)$ one may form a configuration space
$\wt{\MC{B}}_{z,\R\times E}(\alpha,\beta)$ of connections framed at $(0,b)\in \R\times Y$ approaching $\alpha$ at $-\infty$ and $\beta$ at $+\infty$, whose corresponding
path in $\wt{\MC{B}}_E$ belongs to the homotopy class $z$, modulo gauge
transformations that approach $\alpha$-harmonic and $\beta$-harmonic 
gauge transformations at $\pm\infty$. This space also carries
a residual action of $\om{SO}(3)$ and there are equivariant end-point maps
\begin{equation} \label{Intro-Conf-Endpoint-Maps} 
\begin{tikzcd} \alpha & \wt{\MC{B}}_{z,\R\times E}(\alpha,\beta) \arrow{l}[swap]{e_+}
\arrow{r}{e_-} & \beta \end{tikzcd}
\end{equation} 
defined by parallel transport of the framing to $\pm\infty$ along $\R\times \{b\}$. The moduli space of framed instantons $\wt{\MC{M}}_z(\alpha,\beta)=\{ [\mathbf{A},u] \in \wt{\MC{B}}_{\R\times Y,z}(\alpha,\beta) : F_\mathbf{A}^+=0\}$ is then cut out transversally provided the framed analogue of the
anti-self-dual operator
\[ D_\mathbf{A}=(d_\mathbf{A}^*,d_\mathbf{A}^+)\colon 
\Omega^1(\R\times Y,\mathfrak{su}(E)) \ra \Omega^0(\R\times Y,\mathfrak{su}(E))\oplus
\Omega^+(\R\times Y,\mathfrak{su}(E))  \]
is (in appropriate Sobolev completions) surjective for all $[\mathbf{A},u]\in\wt{\MC{M}}_z(\alpha,\beta)$. For the trivial $\om{SU}(2)$-bundle over the binary polyhedral
spaces all of the flat connections are non-degenerate and the moduli
spaces are cut out transversally, so there is no need to introduce a
perturbation (see \cite[Section~4.5]{Austin90}).

In \cite{Miller19} it is shown, more generally, that there exist generic perturbations 
that achieve the above (equivariant) transversality conditions. We will assume that such a perturbation has been chosen, leaving it out from the notation. There
is then a finite set $\MC{C}$ of critical orbits. Given $\alpha\in \MC{C}$
the possible stabilizers in $\om{SO}(3)$ are $\{1\}$, $\om{SO}(2)$ or  
$\om{SO}(3)$, so that
\begin{equation} \label{Orbits} 
\alpha\cong \om{SO}(3), \;\; \alpha \cong \om{SO}(3)/\om{SO}(2)\cong S^2
\; \mbox{ or } \; \alpha \cong * .
\end{equation} 
We say that $\alpha$ is irreducible, reducible or fully reducible respectively. Given a pair of critical orbits $\alpha,\beta\in \MC{C}$
there is then for each homotopy class $z\in \pi_1(\wt{\MC{B}}_E,\alpha,\beta)$ a finite-dimensional moduli space 
$\wt{\MC{M}}_z(\alpha,\beta)$ equipped with
an $\om{SO}(3)$-action and equivariant end-point maps. There is a
natural $\R$-action by translations on the moduli space. Provided
$\alpha\neq \beta$ this action is free and the
quotient is denoted by $\wt{\MC{M}}_z^0(\alpha,\beta)$. The situation is summarized in the following diagram,
\[ \begin{tikzcd} {} & \wt{\MC{M}}_z(\alpha,\beta) \arrow{ld}[swap]{e_-} \arrow{rd}{e_+}
\arrow{d}{/\R} & {} \\
\alpha & \wt{\MC{M}}^0_z(\alpha,\beta) \arrow{l}{e_-} \arrow{r}[swap]{e_+}
& \beta, \end{tikzcd}  \]
where the upper part is obtained by restriction from diagram \eqref{Intro-Conf-Endpoint-Maps}. A relative grading, depending on the relative homotopy
class $z$, may then be defined by
\[ \om{gr}_z(\alpha,\beta) = \om{dim} \wt{\MC{M}}_z(\alpha,\beta)-\om{dim}\alpha \in \Z .\]
This function satisfies a few key properties. 
\begin{enumerate}[label=(\roman*),ref=(\roman*)]
\item For $\alpha,\beta,\gamma\in \MC{C}$ and relative homotopy classes
$z$, $w$ connecting $\alpha$ to $\beta$ and $\beta$ to $\gamma$ respectively
one has 
\[ \om{gr}_{z*w}(\alpha,\gamma)=\om{gr}_z(\alpha,\beta)+\om{gr}_w(\beta,\gamma). \]
\item For $\alpha,\beta\in \MC{C}$ and relative homotopy classes $z,z'$ connecting them one has 
\[ \om{gr}_z(\alpha,\beta)-\om{gr}_{z'}(\alpha,\beta)\in 8\Z  .\]
Moreover, for any integer $n\equiv \om{gr}_z(\alpha,\beta) \Mod{8}$
there exists a homotopy class $w$ such that $\om{gr}_{w}(\alpha,\beta)=n$.
\end{enumerate} 
The relative grading $\om{gr}\colon \MC{C}\times \MC{C}\ra \Z/8$ is then
defined by $\om{gr}(\alpha,\beta)\equiv \om{gr}_z(\alpha,\beta) \Mod{8}$
for any choice of $z$. In our case of a trivial bundle, we may lift
this to an absolute grading $j\colon \MC{C}\ra \Z/8$ by taking the trivial product connection $\theta$ as a reference point, that is, $j(\alpha) \coloneqq \om{gr}(\alpha,\theta)$.

As a graded module, the framed Floer complex $\wt{CI}(Y,E;R)$ may now be
defined to be the totalization of the bigraded module,
\[ \wt{CI}(Y,E;R)_{s,t} = \bigoplus_{j(\alpha)\equiv s} C_t^{gm}(\alpha; R) \]
for all $s,t\in \Z$, where the direct sum is taken over all $\alpha\in \MC{C}$
with $j(\alpha)\equiv s\Mod{8}$. Here, for a smooth manifold $X$,
$C^{gm}_*(X;R)$ is a chain complex of free $R$-modules generated by
(equivalence classes of) maps $\sigma\colon P\ra X$ from a certain set of
stratified topological spaces with quite a lot of additional structure
(see \cite[Section~6]{Miller19}). Importantly, this version of
geometric homology supports a notion of fiber products on the chain level,
and the homology $H_*^{gm}(X;R)$ coincides with the usual singular homology
of $X$.
The differential in $\wt{CI}(Y,E)$
is then defined on a generator $c=[\sigma\colon P\ra \alpha]\in C_t^{gm}(\alpha)\subset \wt{CI}(Y,E)_{s,t}$ by taking the sum of the internal differential
$\dd^{gm} c\in C^{gm}_{t-1}(\alpha)$
and various fiber products with (compactified) moduli spaces
\[ [P\times_{\alpha} \overline{\MC{M}}_z(\alpha,\beta) \ra \beta] \in C^{gm}_{t+\om{gr}_z(\alpha,\beta)-1}(\beta)  .\] 
We will state the key results that make all of this precise in section $2$.

From this point we omit the reference to the ring of coefficients in
the notation. The framed Floer complex $\wt{CI}(Y,E)$ carries two important pieces of additional structure. First, $A:=C^{gm}_*(\om{SO}(3))$ is a differential
graded algebra and the action on $C^{gm}_*(\alpha)$ for $\alpha\in \MC{C}$
gives $\wt{CI}(Y,E)$ the structure of a right differential graded
$A$-module. Second, $\wt{CI}(Y,E)$ is filtered by index:
\[ F_p\wt{CI}(Y,E)_n = \bigoplus_{j(\alpha)\equiv s,\; s\leq p} C^{gm}_{n-s}(\alpha) \]
for $p\in \Z$. This is a filtration by differential graded $A$-submodules
and is naturally referred to as the index filtration.

When $2$ is invertible in the ground ring, Miller Eismeier shows in \cite[Section~7]{Miller19} that it is possible to replace the pair $(\wt{CI}(Y,E),C^{gm}_*(\om{SO}(3))$ with a simpler pair $(DCI(Y,E),H_*(\om{SO}(3))$.
The complex $DCI(Y,E)$, called the Donaldson model, is defined
as a graded module to be the totalization of the bigraded module
\[ DCI(Y,E)_{s,t} =\bigoplus_{j(\alpha)\equiv s} H_t(\alpha)   \]
and the differential has a concrete description. 
Furthermore, as $1/2\in R$, $H_*(\om{SO}(3))\cong \Lambda_R[u]$ is an exterior algebra on a single generator $u$ in degree $3$. Using the results
of \cite{Austin95}, we are able to describe these complexes explicitly
for all the binary polyhedral spaces. This is the key input needed
for our calculations.

The fourth  group, $\wt{I}(Y,E)$, in \eqref{Intro-IF-Groups} is simply 
defined to be the homology of $\wt{CI}(Y,E)$. To explain the other
three groups, we have to delve into some differential graded algebra.
Let $A$ be an augmented differential graded (DG) algebra.
In \cite[Appendix~A]{Miller19} four functors
$C^+_A,C^{+,tw}_A,C^-_A,C^\infty_A$ from the category of right DG $A$-modules
(degree $0$ maps assumed) to the category of left DG $C^-_A(R)$-modules
are constructed. 
The basic tool employed is the two sided bar construction
$B(M,A,N)$ associated with the algebra $A$, a right $A$-module $M$
and a left $A$-module $N$ (see Definition \ref{Def-Bar-Construction}).
Then 
\begin{equation} \label{Intro-Cpm}
 C^+_A(M) \coloneqq B(M,A,R) \; \mbox{ and } \; C^-_A(M) \coloneqq \om{Hom}_A(B(R,A,A),M).
 \end{equation} 
Furthermore, $C^{+,tw}_A(M)\coloneqq B(M,A,D_A)$ where $D_A$ is a dualizing object
associated with $A$ (see Definition \ref{Def-Dualizing-Object}). There
is a natural transformation $N\colon C^{+,tw}_A\ra C^-_A$, called the norm map,
and the Tate functor $C^\infty_A$ is defined to be the mapping cone of $N$.
It is this that eventually gives rise to the exact triangle connecting the
three Floer groups in \eqref{Intro-Exact-Triangle}.  
For a right $A$-module $M$ the homology of $C^+_A(M)$ is denoted by
$H^+_A(M)$ and similarly for the other functors. 
 
Provided the ground ring $R$ is a principal ideal domain and $A$ is degreewise free over $R$, these functors are all exact and preserve quasi-isomorphisms
(isomorphisms upon passage to homology). Aside from the role played
in the construction of $C^\infty_A$ the functor $C^{+,tw}_A$ is mostly inessential. Indeed, provided $A$ satisfies a Poincar\'{e} duality
hypothesis of degree $d\in \Z$, valid in the situations of interest with $d=3$, 
there is a degree $d$ isomorphism $H^+_A(M)\cong H^{+,tw}_A(M)$ for all
right $A$-modules $M$.   

To define the equivariant Floer groups in \eqref{Intro-IF-Groups},
Miller Eismeier also constructs what he calls completed versions, 
$\hat{C}^\bullet_A$ for $\bullet\in \{+,-,(+,tw),\infty\}$, of the above functors. The equivariant Floer groups $I^\bullet(Y,E)$ are then defined by applying $\hat{C}^\bullet_A$ to $\wt{CI}(Y,E)$, with $A=C^{gm}_*(\om{SO}(3))$,
and then passing to homology. To explain these completed versions
it is important to note that the bar construction $B(M,A,N)$ is obtained
as the totalization of a double complex $(B_{*,*}(M,A,N),\dd',\dd'')$.
It therefore carries two filtrations, one by the first simplicial degree, and one by the second internal degree. In \cite[Definition~A.6]{Miller19} the completed bar construction $\hat{B}(M,A,N)$ is defined to be the completion
of $B(M,A,N)$ with respect to the filtration by internal degree,
and the finitely supported cobar construction $c\hat{B}(N,A,M)\subset \om{Hom}_A(B(N,A,A),M)$ is defined to be the submodule of functionals $B(N,A,A)\ra M$
vanishing on $B_{p,q}(N,A,A)$ for all sufficiently large $p$. 
The completed versions $\hat{C}^+_A$ and $\hat{C}^-_A$ are then defined by
replacing $B(M,A,R)$ with $\hat{B}(M,A,R)$ and $\om{Hom}_A(B(R,A,A),M)$
with $c\hat{B}(R,A,M)$ in \eqref{Intro-Cpm}, while $\hat{C}^{+,tw}_A(M)\coloneqq \hat{B}(M,A,\hat{D}_A)$ where $\hat{D}_A$ is a completed version of the dualizing object. Finally, the norm map extends to a map $\hat{N}:\hat{C}^{+,tw}_A(M)\ra \hat{C}^-_A(M)$ and $\hat{C}^\infty_A(M)$ is defined to be the cone of this extended map.    

Let $M=\wt{CI}(Y,E)$ be equipped with the index filtration and let
$A= C^{gm}_*(\om{SO}(3)$. The intention of the above completions is to
make sure that for $\bullet\in \{+,-,(+,tw),\infty\}$ the induced filtration
$F_p\hat{C}^\bullet_A(M):= \hat{C}^\bullet_A(F_pM)$ is well-behaved;
by which we mean exhaustive and complete Hausdorff (see \cite[Definition~2.1]{Boardman99}). This is to ensure that the associated spectral sequence
converges conditionally, in the sense of \cite{Boardman99}, to the
correct target. However, the fact that this induced filtration
is well-behaved in the above sense is not justified in \cite{Miller19}.
In fact, we find this difficult to believe in the case of
$\hat{C}^{+,tw}_A$ and $\hat{C}^\infty_A$. To remedy this potential issue
we suggest a different approach to the completed functors. Rather than defining
completed functors on the whole category of $A$-modules, we simply promote
the existing functors $C^\bullet_A$ to functors between the corresponding
filtered categories and then compose with a full completion functor
to ensure that the resulting filtrations are well-behaved. To be
precise, given a right DG $A$-module $M$ equipped with an increasing
filtration by DG $A$-submodules $ \{F_pM\}_p$ we define
\begin{align*} & \hat{C}^\bullet_A(M)  \coloneqq \om{lim}_q\om{colim}_p C^\bullet_A(F_pM/F_qM)  \;\; \mbox{ filtered by }   \\
& F_p\hat{C}^\bullet_A(M)  \coloneqq \om{lim}_{q<p} C^\bullet_A(F_pM/F_qM)
\;\; \mbox{ for } \;\; p\in \Z
\end{align*}
for each $\bullet \in \{+,-,(+,tw),\infty\}$. This approach has the 
advantage that essentially all the relevant properties
of the functors $C^\bullet_A$ extend in a straightforward way to the
functors $\hat{C}^\bullet_A$.  

For $A=C^{gm}_*(\om{SO}(3))$ and $M = \wt{CI}(Y,E)$ we define
filtered complexes of $C^-_A(R)$-modules
\[ CI^\bullet(Y,E) = \hat{C}^\bullet_A(M)\;\; \mbox{ for } \;\; \bullet\in \{+,-,(+,tw),\infty\}  \]
and define $I^\bullet(Y,E)$ to be its homology. When $\frac12 \in R$ we
also define $DCI^\bullet(Y,E)$ as above with $A=\Lambda_R[u]$ and $M=DCI(Y,E)$.
Our main results concerning this construction are simply that all of
the desired formal properties of the groups mentioned at the beginning
of the section are valid. This is established in section $5$. Furthermore, 
by the nature of our definitions we are with certainty able to establish the existence of the index spectral sequences
\[ E^1_{s,t} = \bigoplus_{j(\alpha)\equiv s} H_{\om{SO}(3)}^\bullet(\alpha)_t
\implies I^\bullet_{s+t}(Y,E) \]
of $H^{-*}(B\om{SO}(3))$-modules converging conditionally to the desired
target. Here, $H_{\om{SO}(3)}^\bullet(\alpha) \coloneqq H^\bullet_A(C^{gm}_*(\alpha))$. All of these spectral sequences carry a periodicity isomorphism $E^r_{s,t}\cong E^r_{s+8,t}$ for all $s,t\in \Z$, compatible with the module structure, the differentials and the target. 

We should note that this modification of the completion procedure has no essential consequences for any of the results in the main body of \cite{Miller19} as far as we can see. Moreover, we explain, see Remark \ref{Remark-Cpm-Coincides-Miller}, that our construction coincides with
the original one for $I^\pm(Y,E)$.  

For $A=\Lambda_R[u]$ with $|u|=3$ it holds true that $C^-_A(R) = H^-_A(R)=R[U]$, a polynomial algebra on a single generator $U$ of degree $-4$. In this
situation we give concrete models for $C^+_A(M)$, $C^-_A(M)$
and $C^\infty_A(M)$. In the first two cases these are easy to extract and are also essentially (up to signs) contained in \cite{Miller19}, but the chain level model for $C^\infty_A$ seems to be new. For $M=\wt{CI}(Y,E)$ or $M=DCI(Y,E)$ equipped with the index filtration we are able extend these models to the completed functors $\hat{C}^\pm_A(M)$ and $\hat{C}^\infty_A(M)$. In all cases these models express the relevant complex as a suitable totalization of a double
complex $(D_{*,*},\dd',\dd'')$ concentrated in the right half plane
in the $+$ case, the left half-plane in the $-$ case and the whole
plane in the $\infty$ case. Each nonzero column $(D_{s,*},\dd'')$ is a shifted
copy of $M$ and the horizontal differential $\dd'$ is given by the action
of $u\in \Lambda_R[u]$ up to a sign. Moreover, the action of $U\in R[U]$ is easy to express in these models. This leads to a new proof of the fact that $U\colon I^\infty(Y,E)\ra I^\infty(Y,E)$ is an isomorphism ($\frac12\in R$ assumed) established in \cite{Miller19}. We find these models very convenient to use at
various places in our calculations.

We now move over to the calculational content of this paper. The finite
subgroups of $\om{SU}(2)$ are naturally divided into two infinite
families consisting of cyclic groups $C_m$ and binary dihedral groups
$D_n^*$, $n\geq 2$, as well as three exceptional ones: the binary tetrahedral
group $T^*$, the binary octahedral group $O^*$ and the binary icosahedral
group $I^*$. The quotients $S^3/C_m$ are lens spaces and $S^3/I^*$ is
the famous Poincar\'{e} homology sphere. Calculations in these
cases are contained in \cite[Section~7]{Miller19} for one of the two
possible orientations.  

Calculations of equivariant instanton Floer groups for binary polyhedral spaces have been considered before. Austin and Braam introduced a type of equivariant instanton Floer groups in \cite{Austin-Braam96}. Their theory, although similar to the one considered here,
relies on using differential forms to express the equivariant complexes
and is therefore restricted to using real coefficients. Austin calculated
these groups for binary polyhedral spaces in \cite{Austin95}. Most of
the necessary background material needed to set up the Donaldson model
$DCI(Y,E)$ is therefore contained in this paper. The key idea
is to exploit a natural bijection between instantons over the cylinder
$\R\times S^3/\Gamma$ and $\Gamma$-invariant instantons on
$\Gamma$-equivariant $\om{SU}(2)$-bundles over $S^4=(\R\times S^3)\cup\{\pm\infty\}$ with the suspended $\Gamma$-action, and then use
an equivariant version of the well-known ADHM correspondence to determine the
relevant moduli spaces. We review and expand on several of the necessary results in section $4$. In particular, to state the 
equivariant ADHM-correspondence clearly we have included a classification of 
$\Gamma$-equivariant $\om{SU}(2)$-bundles over $S^4$, whose proof is given in Appendix $B$.  

The conclusion of the above is that we are able to describe
the complex $DCI(\overline{Y}_\Gamma)$ for all finite subgroups $\Gamma\subset \om{SU}(2)$ explicitly. The description can be summarized as follows. 
Given a finite subgroup $\Gamma\subset \om{SU}(2)$, we construct
a graph $\MC{S}_\Gamma$ whose vertices correspond to the flat connections in the trivial $\om{SU}(2)$-bundle over $Y_\Gamma=S^3/\Gamma$. In all cases 
this graph is a tree. Attached to an edge connecting two vertices
$\alpha$ and $\beta$ there is a symbol encoding the action
of the differential between these. If $\alpha$ and $\beta$ are both
irreducible the symbol is $(n_{\alpha\beta}|n_{\beta\alpha})$ and
if only, say $\alpha$, is irreducible the symbol is $n_{\alpha\beta}$,
where the $n_{\alpha\beta}$ are certain integers. Then the labeled graph 
$\MC{S}_\Gamma$ contains all the information needed to set up 
$DCI(\overline{Y}_\Gamma)$.
The grading is determined by $j(\alpha)=4p(\alpha)$ where $p(\alpha)$
is the length of a minimal path connecting $\alpha$ to the trivial
connection $\theta$ in $\MC{S}_\Gamma$. In particular, $DCI(\overline{Y}_\Gamma)_{s,t}=0$ unless $4|s$. This has the implication that the
only possibly nontrivial part of the differential is 
$\dd:DCI_{4s,0}\ra DCI_{4(s-1),3}$. This differential has components
$\dd_{\alpha\beta}:H_0(\beta)\ra H_3(\alpha)$, which is nonzero precisely 
when $\alpha$ and $\beta$ are adjacent in $\MC{S}_\Gamma$ and $\alpha$ is irreducible. In that case, the differential is specified by $\dd_{\alpha\beta}(b_\beta)=n_{\alpha\beta} t_\alpha$ where $b_\beta\in H_0(\beta)$ and
$t_\alpha\in H_3(\alpha)$ are generators. The labeled graphs $\MC{S}_\Gamma$
for all the finite subgroups $\Gamma\subset \om{SU}(2)$ are contained
in Proposition \ref{Differential-Graphs} and examples of the resulting
Donaldson complexes are given in Example \ref{TOI-Multicomplexes}.

In the final section of the paper we explicitly calculate 
$I^+(\overline{Y}_\Gamma)$, $I^-(\overline{Y}_\Gamma)$ and $I^\infty(\overline{Y}_\Gamma)$ for each finite subgroup $\Gamma\subset \om{SU}(2)$ and observe that the exact triangle relating them splits into a short
exact sequence. The calculations are contained in Theorem \ref{IPluss-1}, Theorem \ref{IMinus-1}, Theorem \ref{IMinus-2} and Theorem \ref{Iinfty-1}. The main tools in the calculations are the explicit complexes $DCI(\overline{Y}_\Gamma)$, the index
spectral sequences of Theorem \ref{Index-SS-Theorem} and the
explicit models for $DCI^+$, $DCI^-$ mentioned earlier. Interestingly,
we obtain essentially the same results as in \cite{Austin95}. Indeed,
the fact that $\frac12 \in R$ ensures that exactly the same type of
degeneration pattern seen in Austin's calculations occur for us as
well.    

In the final part of Appendix $B$ we also include a discussion of
the Chern-Simons invariant of the flat connections in the trivial
$\om{SU}(2)$-bundle over the various binary polyhedral spaces.
The main result we wish to bring out is that this numerical invariant
can easily be calculated with the help of a few results established
in section $4$. Methods for these calculations are certainly known, 
see for instance \cite{KirkKlassen90} or \cite{Auckly94}. 
However, we also show, possibly more interestingly, that this invariant can in a natural way be identified with the second Chern class of the holonomy representation associated with the flat connection.

\subsection{Organization of Paper}
We start in section $2$ by stating the key theorem of \cite{Miller19}
and some additional details on the geometric homology functor $C^{gm}_*$
needed to rigorously define the framed Floer complex $\wt{CI}(Y,E)$.
Moreover, we also introduce the Donaldson model $DCI(Y,E)$ under a simplifying
assumption valid in the case of binary polyhedral spaces. Finally, we
state the precise result that ensures that $DCI(Y,E)$ also may be used
to calculate the various flavours of equivariant instanton Floer homology.

In section $3$ we recall the classification of finite subgroups of
$\om{SU}(2)$ and specify our conventions concerning the binary polyhedral
spaces. We also recall some simple representation theory and the McKay
correspondence. This is needed to efficiently arrange the set of flat
connections in the trivial $\om{SU}(2)$-bundle over the binary polyhedral
spaces. At the end of the section we define the graphs $\MC{S}_\Gamma$.  

In section $4$ we review the key results from \cite{Austin95} that explicitly
determine the low-dimensional instanton moduli spaces needed to set up
the Donaldson model for the binary polyhedral spaces. Here we have, for the
purpose of clarity and precision, expanded on Austin's rather terse exposition in a number of places. In the final part of the section we give an explicit
description of $DCI(\overline{Y}_\Gamma)$ for each finite subgroup
$\Gamma\subset \om{SU}(2)$.

In section $5$ we move into the world of differential graded algebra. 
First, we review the construction of the functors $C^\bullet_A$;
their invariance and functorial properties for $\bullet\in \{+,-,(+,tw),\infty\}$ mainly following \cite[Appendix~A.2,A.3,A.5]{Miller19}.
Detailed calculations are included in the case $A=\Lambda_R[u]$ for
$|u|=3$, but we note that all of these calculations easily extend
to the case of $|u|$ of arbitrary odd degree.
After this we explain how to extend the functors to the category of
filtered modules and establish all the necessary functoriality and
invariance results in detail. We then define the complexes
$CI^\bullet(Y,E)$ for $\bullet\in \{+,-,(+,tw),\infty\}$
calculating the various flavors of equivariant Floer homology,
and similarly for $DCI^\bullet(Y,E)$. In the final part
of the section we recall some theory on spectral sequences and their
convergence properties needed for our later calculations, and then
go on to set up the index spectral sequences.

In the final section we calculate the various flavors of equivariant
instanton Floer homology for the trivial $\om{SU}(2)$-bundle over
the binary polyhedral spaces $\overline{Y}_\Gamma$. In the last
part of the section we explain how $\om{DCI}(Y_\Gamma)$ can be
recovered from $\om{DCI}(\ovl{Y}_\Gamma)$ and then give the
calculations for $Y_\Gamma$, omitting some details. 

In Appendix $A$ we have included various facts concerning the finite
subgroups $\Gamma\subset \om{SU}(2)$. This include explicit realizations
within $\om{Sp}(1)$, information on their complex representation theory,
character tables for $T^*$, $O^*$ and $I^*$, extended Dynkin diagrams
and complete lists over the $1$-dimensional quaternionic representations.  
This is needed to determine all the flat connections in the trivial
$\om{SU}(2)$-bundle over $\overline{Y}_\Gamma$ and for the indexing
and differentials in the complexes $DCI(\overline{Y}_\Gamma)$.

In Appendix $B$ we have included proofs of two results used in section $4$:
the classification of $\Gamma$-equivariant $\om{SU}(2)$-bundles over $S^4$
and an index calculation. As a byproduct, we obtain a simple procedure for the computation of the Chern-Simons invariants for the various flat connections over the binary polyhedral spaces. We also discuss a relation between these invariants and the group cohomology of the binary polyhedral groups.

\section{Equivariant Instanton Floer Homology}
In this section we will expand on, and make more rigorous, the definition of the framed instanton Floer complex $\wt{CI}(Y,E)$. We will assume, as in the introduction, that $Y$ is a rational homology sphere and that $E\ra Y$ is the
trivial $\om{SU}(2)$-bundle as this is sufficient for our purpose.
For the more general case of weakly admissible bundles we refer
to \cite{Miller19}. We will assume as in the introduction that
a perturbation of the Chern-Simons functional has been fixed such that
every critical orbit is non-degenerate and all the relevant moduli spaces are
cut out transversally. As already noted, there is no
need to introduce a perturbation in the case of binary polyhedral spaces
(see \cite[Section~4.5]{Austin95}). The set of critical orbits will always be denoted by $\MC{C}$.   

\subsection{The Framed Instanton Floer Complex}
Recall that we have fixed a basepoint $b\in Y$ and that 
$\wt{\MC{B}}_E$ is a  configuration space of framed connections
in $E$ modulo gauge equivalence. Moreover, for each pair $\alpha,\beta\in \MC{C}$ and a relative homotopy class $z\in \pi_1(\wt{\MC{B}}_E,\alpha,\beta)$
there is a configuration space $\wt{\MC{B}}_{z,\R\times E}(\alpha,\beta)$
of connections in $\R\times E\ra \R\times Y$ framed at $(0,b)$ that approaches
$\alpha$ and $\beta$ respectively as $t$ tends to $\pm\infty$, modulo
gauge. There is a residual $\om{SO}(3)$-action and
equivariant end-point maps $(e_-,e_+)\colon \wt{\MC{B}}_{z,\R\times E}(\alpha,\beta)\ra \alpha\times \beta$. In reality, these spaces are completed with respect
to certain weighted Sobolev norms (see \cite[Section~2.2]{Miller19} for the
specifics). We may now state the main theorem required to set
up the framed instanton Floer complex. Here we have extracted the
relevant part of \cite[Theorem 6.8]{Miller19}.

\begin{theorem} \label{Moduli-Theorem}
Let $E\ra Y$ be the trivial $\om{SU}(2)$-bundle over $Y$ and let $\MC{C}$ be the finite set of critical orbits of the (perturbed) Chern-Simons functional. 
\begin{enumerate}[label=(\alph*). ,ref=(\alph*), wide]
\item For each pair $\alpha,\beta\in \MC{C}$ and
relative homotopy class $z$ between them in $\wt{\MC{B}}_E$ there is a number $\om{gr}_z(\alpha,\beta)\in \Z$. If $\gamma\in \MC{C}$, $w$ is a
relative homotopy class from $\beta$ to $\gamma$
and $z*w$ denotes the concatenation of paths then
\[ \om{gr}_{z*w}(\alpha,\gamma)=\om{gr}_z(\alpha,\beta)+\om{gr}_w(\beta,\gamma). \]
Moreover, if $z'$ is a different homotopy class
from $\alpha$ to $\beta$ it holds true that
$\om{gr}_z(\alpha,\beta)-\om{gr}_{z'}(\alpha,\beta)\in 8\Z$. Hence, $\om{gr}(\alpha,\beta)\in \Z/8$ is
well defined. 
\item For each pair $\alpha,\beta\in \MC{C}$ and
relative homotopy class $z\in \pi_1(\wt{\MC{B}}_E,\alpha,\beta)$ there is a moduli space of framed
instantons $\wt{\MC{M}}_z(\alpha,\beta)$
approaching $\alpha$ as $t\to-\infty$ and $\beta$
as $t\to\infty$. This is a smooth right $\om{SO}(3)$-manifold of dimension $\om{gr}_z(\alpha,\beta)+\om{dim}(\alpha)$.
\item There is a smooth $\R$-action by translations
and parallel transport of the framing on 
$\wt{\MC{M}}_z(\alpha,\beta)$ with smooth
quotient denoted by $\wt{\MC{M}}_z^0(\alpha,\beta)$.  
\item There are smooth $\om{SO}(3)$-equivariant end-point maps as in the diagram
\[ \begin{tikzcd} {} & \wt{\MC{M}}_z(\alpha,\beta)
\arrow{d}{/\R} \arrow{ld}[swap]{e_-} \arrow{rd}{e_+}
& {} \\
\alpha & \wt{\MC{M}}^0_z(\alpha,\beta) \arrow{l}[swap]{e_-} \arrow{r}{e_+} & \beta
\end{tikzcd}  \]
obtained by parallel transport of the framing towards
$\pm\infty$. Here we identify $\alpha\cong E_b/\Gamma_\alpha$ and $\beta\cong E_b/\Gamma_\beta$, where
$\Gamma_\alpha,\Gamma_\beta\subset \om{Aut}(E_b)$
correspond to the stabilizers of $\alpha$ and
$\beta$.
\item For each $\alpha\in \MC{C}$ there is a two-element set $\Lambda(\alpha)$ corresponding to a choice of orientation. For a pair $\alpha$, $\beta$
any choice of orientations in $\Lambda(\alpha)$ and
$\Lambda(\beta)$ induces an orientation in the
fiber of $e_-:\wt{\MC{M}}^0_z(\alpha,\beta)\ra \alpha$. If either of the choices are changed the
resulting fiber orientation is reversed. 
\item If $\om{gr}_z(\alpha,\beta)-\om{dim}(\alpha)\leq 10$ there is a compactification
\[ \wt{\MC{M}}^0_z(\alpha,\beta)\subset\ovl{\MC{M}}_z(\alpha,\beta)  \]
into a topological $\om{SO}(3)$-manifold with
corners and smooth structure on each stratum. The endpoint maps extend over the compactification.
Furthermore, the $\om{SO}(3)$ action is free.
\item Given any choice of elements from $\Gamma(\alpha)$, $\Gamma(\beta)$ and $\Gamma(\gamma)$
there is a decomposition respecting fiber orientations
\[ (-1)^{\om{dim}\alpha}\dd \overline{\MC{M}}_z(\alpha,\beta) = \coprod_{w_1,w_2: \; w_1*w_2=z}
\overline{\MC{M}}_{w_1}(\alpha,\gamma)\times_\gamma
\overline{\MC{M}}_{w_2}(\gamma,\beta) .\]
\end{enumerate}
\end{theorem} 

The function $\om{gr}\colon \MC{C}\times\MC{C}\ra \Z/8$
is the relative grading on the set of critical orbits.
Since the bundle $E\ra Y$ is trivial we may
define an absolute grading $j\colon \MC{C}\ra \Z/8$ by
setting $j(\alpha)\coloneqq \om{gr}(\alpha,\theta)$, where
$\theta$ denotes the trivial product connection.
The relative grading may be recovered from $j$:
\[ j(\alpha)-j(\beta)=\om{gr}(\alpha,\theta)-\om{gr}(\beta,\theta)=\om{gr}(\alpha,\beta)+\om{gr}(\beta,\theta)-\om{gr}(\beta,\theta)=\om{gr}(\alpha,\beta).\]

To construct the framed Floer complex Miller Eismeier uses a type of geometric homology. For each smooth manifold $X$ he constructs a functorial chain complex
$C_*^{gm}(X)$ whose generators are equivalences
classes of maps from a set of spaces called strong $\delta$-chains. For the precise definition see \cite[Def.:6.1]{Miller19}. This class of spaces
contains all topological manifolds with corners and
smooth structure on each stratum, so in particular contains the
compactified instanton moduli spaces in part $(f)$ of the above theorem.    

\begin{proposition} Given a commutative ring of coefficients $R$ there is a functor $C_*^{gm}(-;R)$ from the category of smooth manifolds to the category of homological chain complexes of $R$-modules. For a smooth manifold
$X$ of dimension $d$ the following holds true. 
\begin{enumerate}[label=(\roman*),ref=(\roman*)]
\item There is a natural isomorphism $H_*^{sing}(X;R)\cong H_*^{gm}(X;R)$.
\item $C_n^{gm}(X;R)$ is a free $R$-module for
each $n\in \Z$.
\item $C_n^{gm}(X;R)=0$ for all $n<0$ and $n>d+1$. 
\end{enumerate}
\end{proposition}

There is one additional property of $C^{gm}_*$ that is needed for
the definition of the framed instanton Floer complex.
First, for a pair of smooth manifolds $X$ and $Y$ there is a cross
product $C_i^{gm}(X;R)\times C^{gm}_j(Y;R)\ra C^{gm}_{i+j}(X\times Y)$.
This ensures that $C^{gm}_*(\om{SO}(3);R)$ obtains the structure of a differential graded algebra and for any right $\om{SO}(3)$-manifold $X$, $C^{gm}_*(X;R)$ obtains the structure of a right differential graded $C^{gm}_*(\om{SO}(3);R)$-module.
  
The framed Floer complex $\wt{CI}(Y,E)$ will be defined to be the chain
complex associated with a multicomplex, whose definition we first recall. 

\begin{definition} A (homological) multicomplex is a pair $(M_{*,*},\{\dd^r\}_{r\geq 0})$ consisting
of a bigraded module $M_{*,*}$ and differentials
$\dd^r\colon M\ra M$ of bidegree $(-r,r-1)$ for $r\geq 0$ such that
\begin{enumerate}[label=(\roman*),ref=(\roman*)]
\item $\sum_{i+j=k}\dd^i\circ \dd^j=0:M_{s,t}\ra M_{s-k,t-k+2}$
for all $k\geq 0$, $s,t\in \Z$, and
\item for each $a\in M_{s,t}$ there exists $r_0\geq 0$ such that 
$\dd^r(a)=0$ for all $r\geq r_0$. 
\end{enumerate}
The associated chain complex $(M_*,\dd)$ is
defined by $M_n = \bigoplus_{s+t=n}M_{s,t}$ for all $n\in \Z$ and $\dd=\sum_{r\geq 0}\dd^r$.
\end{definition} 

For a commutative ring of coefficients $R$ define the bigraded
$R$-module $\wt{CI}(Y,E;R)_{*,*}$ by
\[ \wt{CI}(Y,E;R)_{s,t}=\bigoplus_{j(\alpha)\equiv s}C_t^{gm}(\alpha;R) \]
for all $s,t\in \Z$. Observe that since the grading function $j$ takes
values in $\Z/8$ there is a periodicity isomorphism $\wt{CI}(Y,E;R)_{s,t}\cong \wt{CI}(Y,E;R)_{s+8,t}$ for all $s,t\in \Z$.  
We make a preliminary definition before we define
the differentials $\{\dd^r\}_{r\geq 0}$. For each
pair $\alpha,\beta\in \MC{C}$ fix the unique
relative homotopy class $z$ from $\alpha$ to $\beta$
satisfying $-2\leq \om{gr}_z(\alpha,\beta)\leq 5$.
Then according to part $(f)$ of the above theorem
the moduli space $\wt{\MC{M}}^0_z(\alpha,\beta)$
admits a compactification $\ovl{\MC{M}}_z(\alpha,\beta)$ and we have a diagram
\[ \begin{tikzcd} \alpha & \ovl{\MC{M}}_z(\alpha,\beta) \arrow{l}[swap]{e_-} \arrow{r}{e_+} & \beta
\end{tikzcd}.\]
By \cite[Lemma~6.6]{Miller19} there is an induced fiber product map
\[ f_{\alpha\beta}:C_*^{gm}(\alpha;R)\ra C_{*+\om{gr}_z(\alpha,\beta)-1}^{gm}(\beta;R)  \]
defined as follows. Given a basic chain
$\sigma\colon P\ra \alpha$ with $\om{dim}P=t$ representing
an element $[\sigma\colon P\ra \alpha] \in C_t^{gm}(\alpha;R)$ we form
the fiber product $P\times_\alpha \ovl{\MC{M}}_z(\alpha,\beta)$ as in the commutative diagram
\[ \begin{tikzcd} P \arrow{d}{\sigma} & P\times_\alpha \overline{\MC{M}}_z(\alpha,\beta) \arrow{d}{\tau}
\arrow{rd}{e_+\circ \tau} \arrow{l} & {} \\
\alpha & \overline{\MC{M}}_z(\alpha,\beta) \arrow{l}[swap]{e_-}
\arrow{r}{e_+} & \beta , \end{tikzcd} \]
and define
\[ f_{\alpha\beta}([\sigma:P\ra \alpha])
=[e_+\circ \tau\colon P\times_\alpha \overline{\MC{M}}_z(\alpha,\beta)\ra\beta ]\in C^{gm}_{t+\om{gr}_z(\alpha,\beta)-1}(\beta;R)  .\]
In case $\om{gr}_z(\alpha,\beta)\leq 0$ the compactified moduli space $\ovl{\MC{M}}_z(\alpha,\beta)$ is empty and the map $f_{\alpha\beta}$
vanishes. Otherwise, we note that the dimension
formula
\begin{equation}\label{Dim-Formula}
 \om{dim}\overline{\MC{M}}_z(\alpha,\beta)
=\om{gr}_z(\alpha,\beta)+\om{dim}\alpha-1
\end{equation} 
ensures that the chain 
$e_+\circ \tau\colon P\times_\alpha \overline{\MC{M}}_z(\alpha,\beta)\ra \beta$ has dimension $\om{dim}P+\om{gr}_z(\alpha,\beta)-1$ as asserted. Finally, we note that the maps $f_{\alpha\beta}$ commute with the action of 
$C_*^{gm}(\om{SO(3)};R)$.

\begin{definition} \label{FloerComplex}
The framed instanton Floer complex $\wt{CI}(Y,E;R)$ with coefficients in $R$ is defined to be the chain complex associated with the multicomplex 
$(\wt{CI}(Y,E;R)_{*,*},\{\dd^r\}_{r\geq 0})$ where
\[ \wt{CI}(Y,E;R)_{s,t}=\bigoplus_{\alpha\in \MC{C}\;:\; j(\alpha)\equiv s}C_t^{gm}(\alpha;R) \]
and the differentials $\dd^r:\wt{CI}(Y,E;R)_{s,t}\ra \wt{CI}(Y,E;R)_{s-r,t+r-1}$ for $r\geq 0$ are defined as follows:
\begin{enumerate}[label=(\roman*),ref=(\roman*)]
\item $\dd^0$ is the sum of the internal differentials
$\dd^{gm}\colon C_*^{gm}(\alpha;R)\ra C_{*-1}^{gm}(\alpha;R)$.
\item For $1\leq r\leq 5$ the differential $\dd^r$ is the sum
of the maps $f_{\alpha\beta}:C_*^{gm}(\alpha;R)\ra C_{*+\om{gr}_z(\alpha,\beta)-1}^{gm}(\beta;R)$ for pairs
$\alpha$, $\beta$ with $\om{gr}_z(\alpha,\beta)=r$.
\item $\dd^r=0$ for $r>5$.
\end{enumerate}
The homology of $\wt{CI}(Y,E;R)$ is denoted by $\wt{I}(Y,E)$ and
is called the framed instanton Floer homology of the pair $(Y,E)$.
The $C^{gm}_*(\om{SO}(3);R)$-module structure of $C^{gm}_*(\alpha;R)$
for $\alpha\in \MC{C}$ gives $\wt{CI}(Y,E;R)$ the structure of a right
differential graded $C^{gm}_*(\om{SO}(3);R)$-module.  

The filtration by $C^{gm}_*(\om{SO}(3))$-modules given degreewise by 
\[ F_p\wt{CI}(Y,E;R)_n := \bigoplus_{s\leq p} \wt{CI}(Y,E;R)_{s,n-s}  \]
is called the index filtration.     
\end{definition}

\begin{remark} In the above we have omitted discussing orientations.
To get the complex $\wt{CI}(Y,E)$ in the above form one has
to fix an orientation $\mathfrak{o}_\alpha\in \Lambda(\alpha)$ for each critical orbit $\alpha\in \MC{C}$. 
The verification that $\dd^2=0$ and that this complex has the structure
of a right differential graded $C^{gm}_*(\om{SO}(3))$-module are contained
in \cite[Lemma 6.11]{Miller19}.
\end{remark}  

From this point we will omit the ring of coefficients from the notation.
The complex given in the above definition is the unrolled framed instanton Floer complex. Instead of having a $\Z/8$-graded complex, we have a $\Z$-graded complex equipped with an evident periodicity isomorphism 
\[ \wt{CI}(Y,E)_n\cong \wt{CI}(Y,E)_{n+8}   \]
for all $n\in \Z$, compatible with the $C^{gm}_*(\om{SO}(3))$-action and the differential. This is more convenient when we apply the algebraic machinery of section $5$ to extract the equivariant Floer groups. 
It should be noted that the above periodicity
isomorphisms interact with the index filtration in the following way
\[ F_p\wt{CI}(Y,E)_n \cong F_{p+8}\wt{CI}(Y,E)_{n+8} \]
for all $p,n\in \Z$. 

\subsection{The Donaldson Model}
In the previous definition we introduced the framed
Floer complex $\wt{CI}(Y,E)$ as a differential graded
module over the differential graded algebra
$C_*^{gm}(\om{SO}(3))$. Since we are still working
on the chain level, the complex and action are
difficult to handle. It is shown in \cite[Section 7.2]{Miller19} that provided 
$2\in R$ is invertible, one may replace the pair
$(C_*^{gm}(\om{SO}(3)),\wt{CI}(Y,E))$ with a simpler
pair $(H_*(\om{SO}(3)),DCI(Y,E))$. The complex $DCI_*(Y,E)$ is called
the Donaldson model and is associated
with a multicomplex whose bigraded module is given by
\[ DCI(Y,E)_{s,t}=\bigoplus_{\alpha\in\MC{C}\; :\; j(\alpha)\equiv s} H_t(\alpha)  .\]
Note that $H_*(\om{SO}(3);R)=\Lambda_R[u]$ is an exterior algebra on a single
generator $u$ of degree $3$ as $\frac12 \in R$. For $\alpha \in \MC{C}$ we have by \eqref{Orbits}
\begin{equation} \label{Eq-Orbit-Calc}
H_*(\alpha) \cong \left\{ \begin{array}{cl}
R\oplus R[3] & \alpha \mbox{ irreducible} \\
R\oplus R[2] & \alpha \mbox{ reducible } \\
R & \alpha \mbox{ fully reducible} 
\end{array} \right.  
\end{equation}
Here our convention regarding the shift operation is that if $C$ is a graded module then $C[n]_{n+i}=C_i$ for all $i,n\in \Z$. The above isomorphisms are made into identifications by fixing a base point and an orientation for each
orbit $\alpha\in \MC{C}$. The $\Lambda_R[u]=H_*(\om{SO}(3))$-action
is uniquely determined by the map $\cdot u:DCI_{*,*}\ra DCI_{*,*+3}$,
which has components $\om{id}\colon R=H_0(\alpha)\ra H_3(\alpha)=R$ for $\alpha$ irreducible and vanishes otherwise.     

The complete description of the differentials
in this complex is given in \cite[p.:175-176]{Miller19}. In the following
we will only concern ourselves with the components relevant in the situation
for binary polyhedral spaces. We therefore make the following assumption
(see Lemma \ref{Indexing-Lemma}) 
\begin{enumerate}[label=(\roman*),ref=(\roman*)] \label{Mod4Assumption}
\item For each pair $\alpha,\beta\in \MC{C}$ we have $\om{gr}(\alpha,\beta)\equiv 0 \Mod{4}$. 
\end{enumerate}
This implies that $DCI(Y,E)_{s,t}$ can only be nontrivial if
$4|s$ and $0\leq t\leq 3$. The only possibly nontrivial differential is therefore $\dd^4$ of bidegree $(-4,3)$, whose components are 
$R=H_0(\alpha)\ra H_3(\beta)=R$ for
$\alpha,\beta\in \MC{C}$ with $j(\alpha)-j(\beta)=4$ and $\beta$ irreducible.

Recall that for $\alpha,\beta\in \MC{C}$ we defined $z$ to be the
unique homotopy class between $\alpha$ and $\beta$ for which
$-2\leq \om{gr}_z(\alpha,\beta)\leq 5$. Combining this with the above
assumption and the boundary formula in part $(g)$ of Theorem \ref{Moduli-Theorem} we may conclude that $\MC{M}_z^0(\alpha,\beta)=\overline{\MC{M}}_z(\alpha,\beta)$ is compact without boundary. If $\om{gr}_z(\alpha,\beta)=0$ the moduli space is empty and if $\om{gr}_z(\alpha,\beta)=4$
the dimension is $\om{dim}\alpha +3$ (provided it is nonempty). 
For simplicity of notation we write $\ovl{\MC{M}}_{\alpha,\beta}=\ovl{\MC{M}}_z(\alpha,\beta)$ with this choice of homotopy class $z$ implicit. 

To define the differential $\dd^4$ we have to make a little technical detour.
For each orbit $\alpha\in \MC{C}$ let $\wt{\alpha}\ra \alpha$
denote its universal cover. For irreducible orbits $\alpha$ we
have $\wt{\alpha}=\om{SU}(2)$, while for reducible or fully reducible
orbits $\wt{\alpha}=\alpha$. These coverings carry $\om{SU}(2)$-actions
such that the projections $\tilde{\alpha}\ra \alpha$ are equivariant
along $\om{SU}(2)\ra \om{SO}(3)$. For each pair $\alpha,\beta\in \MC{C}$
define $\ovl{\MC{M}}_{\alpha,\beta}^{cov}$ to be the pullback
as in the diagram
\[ \begin{tikzcd} \ovl{\MC{M}}_{\alpha,\beta}^{cov} \arrow{r}{(\wt{e}_-,\wt{e}_+)} \arrow{d} & \tilde{\alpha}\times\tilde{\beta} \arrow{d} \\
\ovl{\MC{M}}_{\alpha,\beta} \arrow{r}{(e_-,e_+)} & \alpha\times \beta.
\end{tikzcd} \]
Thus $\ovl{\MC{M}}_{\alpha,\beta}^{cov}$ is a $1$, $2$ or $4$-sheeted
covering of $\ovl{\MC{M}}_{\alpha,\beta}$ depending on whether
$\alpha$ and/or $\beta$ are irreducible. The free $\om{SO}(3)$-action
on $\ovl{\MC{M}}_{\alpha,\beta}$ lifts to a free $\om{SU}(2)$-action on
$\ovl{\MC{M}}_{\alpha,\beta}^{cov}$ such that $(\wt{e}_-,\wt{e}_+)$ is
equivariant. Define 
\[ \MC{M}_{\alpha,\beta}\coloneqq \ovl{\MC{M}}_{\alpha,\beta}^{cov}/\om{SU}(2) \;\; \mbox{ and } \;\; X_{\alpha,\beta}\coloneqq (\tilde{\alpha}\times \tilde{\beta})/\om{SU}(2)  \]
and let $e\colon \MC{M}_{\alpha,\beta}\ra X_{\alpha,\beta}$ be the map induced by $(\wt{e}_-,\wt{e}_+)$.

The component $\dd^4:R=H_0(\alpha)\ra H_3(\beta)=R$, for $\beta$ irreducible, is then defined to be multiplication by the degree of the map
\[ e:\MC{M}_{\alpha,\beta}\ra X_{\alpha,\beta}   .\]
From the dimension formula we have $\om{dim}\MC{M}_{\alpha,\beta}=\om{dim}\alpha = \om{dim}X_{\alpha,\beta}$, so as these spaces are compact
oriented manifolds this makes sense. Note that $X_{\alpha,\beta}\cong \wt{\alpha}$ as $\beta$ is irreducible. 

For later purposes it will be important to relate the degree of
the map $e\colon M_{\alpha,\beta}\ra X_{\alpha,\beta}$ to the degree
of a different map determined by Austin in \cite{Austin95}.
Recall that if $X$ and $Y$ are closed, oriented manifolds of dimension $n$ and
$Y$ is connected then the degree of a map $f\colon X\ra Y$ is defined by the relation
$f_*([X])=\om{deg}(f)[Y]$ where $f_*\colon H_n(X)\ra H_n(Y)$ is the induced
map and $[X]$, $[Y]$ denote the fundamental classes. 
If $X$ and $Y$ in addition carry smooth structures
and $f$ is a smooth map one may also calculate the degree as
\[ \om{deg}(f)=\sum_{x\in f^{-1}(y)} \om{sgn}(df_x:T_xX\ra T_yY) , \]
for a regular value $y$, where $\om{sgn}(df_x\colon T_xX\ra T_yY)$
is $\pm1$ depending on whether this map preserves or reverses
orientation, respectively. The proof of the following proposition is
a simple exercise in differential topology, so we leave it out. 

\begin{proposition} In the statements below assume that $X$, $Y$ and $Z$
are smooth, closed, oriented manifolds and that $Y$ and $Z$ are connected. 
\begin{enumerate}[label=(\alph*),ref=(\alph*)]
\item Let $G$ be a compact Lie group acting freely on the manifolds
$X$ and $Y$ of the same dimension. Then if $f\colon X\ra Y$ is a smooth
$G$-equivariant map it holds true that
\[ \om{deg}(f\colon X\ra Y)=\om{deg}(f/G\colon X/G\ra Y/G)  .\]
\item Let $f\colon X\ra Y$ be a smooth map between manifolds of equal dimension
and let $p:\wt{Y}\ra Y$ be a finite, connected covering space with
the induced orientation. Let $q:\wt{X}=f^*\wt{Y}\ra X$ be the
pull-back covering and let $\wt{f}:\wt{X}\ra \wt{Y}$ be the
induced map. Then
\[ \om{deg}(f\colon X\ra Y)=\om{deg}(\wt{f}\colon \wt{X}\ra \wt{Y})  .\]
\item Let $(f,g)\colon X\ra Y\times Z$ be a smooth map and assume that
$\om{dim}X=\om{dim}Y+\om{dim}Z$. If $y\in Y$ is a regular value for
$f$, so that $f^{-1}(y)\subset X$ is a closed, oriented submanifold
of dimension $\om{dim}f^{-1}(y)=\om{dim}Z$, then
\[ \om{deg}((f,g)\colon X\ra Y\times Z)=\om{deg}(g|_{f^{-1}(y)}\colon f^{-1}(y)\ra Z) .\]
\end{enumerate}
\end{proposition}

\begin{lemma} \label{Differential-Degree-Lemma}
Let $\alpha,\beta\in \MC{C}$ be a pair with
$\beta$ irreducible and $\om{gr}(\alpha,\beta)=4$. Consider the diagram
\[ \begin{tikzcd} \alpha & \ovl{\MC{M}}_z(\alpha,\beta) \arrow{l}[swap]{e_-} \arrow{r}{e_+} & \beta , \end{tikzcd} \]
where $z$ is the relative homotopy class with $\om{gr}_z(\alpha,\beta)=4$.   
Then if $j(\alpha)\equiv s \Mod{8}$, the component 
$H_0(\alpha)\ra H_3(\beta)$ of $\dd^4:DCI_{s,0}\ra DCI_{s-4,3}$ is given
by multiplication by the degree of the map
\[ e_+|_{e_-^{-1}(*)}:e_-^{-1}(*)\ra \beta  \]
where $*\in \alpha$ is a point.
\end{lemma} 
\begin{proof} Write $\ovl{\MC{M}}_{\alpha,\beta} =\ovl{\MC{M}}_z(\alpha,\beta)$
as earlier. By part $(c)$ of the above proposition we have
\[ \om{deg}(e_+|_{e_-^{-1}(*)}\colon e_-^{-1}(*)\ra \beta)=\om{deg}((e_-,e_+)\colon \overline{\MC{M}}_{\alpha,\beta} \ra \alpha\times \beta) ,\]
and by part $(b)$
\[ \om{deg}((e_-,e_+)\colon \overline{\MC{M}}_{\alpha,\beta}\ra 
(\alpha\times \beta) =
\om{deg}((\wt{e}_-,\wt{e}_+)\colon \overline{\MC{M}}_{\alpha,\beta}^{cov}\ra \wt{\alpha} \times \wt{\beta}). \]
Finally by part $(a)$ the latter integer coincides with the degree of
\[ e\colon \overline{\MC{M}}_{\alpha,\beta}^{cov}/\om{SU}(2) =\MC{M}_{\alpha,\beta}
\ra X_{\alpha,\beta} = (\tilde{\alpha}\times \tilde{\beta})/\om{SU}(2), \]
which by definition is the component of $\dd^4$ from
$R=H_0(\alpha)\ra H_3(\beta)=R$.
\end{proof}   

\begin{definition} \label{Def-DCI}
Assume that $\om{gr}(\alpha,\beta)\equiv 0\mod{4}$ for all $\alpha,\beta\in \MC{C}$. Then the Donaldson model $DCI(Y,E)$ for the Floer complex
is the chain complex associated with the multicomplex
\[ (DCI(Y,E)_{*,*},\{\dd^r\}_{r\geq 0}) \; \mbox{ where } \;
DCI(Y,E)_{s,t}=\bigoplus_{\alpha\in \MC{C}\; :\; j(\alpha)\equiv s} H_t(\alpha), \]
$\dd^4$ is given as in Lemma \ref{Differential-Degree-Lemma} and $\dd^r=0$ for $r\neq 4$. 
\end{definition}

The complex $DCI(Y,E)$ enjoys the same periodicity as $\wt{CI}(Y,E)$. Furthermore, it also carries a natural filtration
\[ F_pDCI(Y,E)_n \coloneqq \bigoplus_{s\leq p} DCI(Y,E)_{s,n-s}  \] 
by $H_*(\om{SO}(3))\cong \Lambda_R[u]$-submodules. This filtration is also referred to as the index filtration.

Finally, we state the result that allows the replacement of the pair
$(C_*^{gm}(\om{SO}(3)),\wt{CI}(Y,E))$ by the pair $(\Lambda_R[u],DCI(Y,E))$.

\begin{proposition}\cite[Corollary~7.6]{Miller19} \label{DCI-CI-Equivalence}
Assume that $2\in R$ is invertible. There is a quasi-isomorphism $i:\Lambda_R[u]\ra C_*^{gm}(\om{SO}(3);R)$ of differential graded algebras and there is a zigzag of $\Lambda_R[u]$-equivariant quasi-isomorphisms
\[ \begin{tikzcd} DCI(Y,E) \arrow{r}{f} & X & \wt{CI}(Y,E) \arrow{l}[swap]{g} \end{tikzcd}. \]
Furthermore, $X$ also carries a filtration, $f$ and $g$ are filtration preserving and the induced maps 
\[ \begin{tikzcd} \frac{F_pDCI(Y,E)}{F_{p-1}DCI(Y,E)} \arrow{r} & 
\frac{F_pX}{F_{p-1}X} & \frac{F_p\wt{CI}(Y,E)}{F_{p-1}\wt{CI}(Y,E)} \arrow{l}  \end{tikzcd} \]
are quasi-isomorphisms for all $p$. 
\end{proposition}

\begin{remark}
The statement in the cited result is slightly weaker. The fact that
the intermediate objects carry a type of index filtration
and that all the quasi-isomorphisms preserve this filtration,
inducing quasi-isomorphisms on the minimal filtration quotients,
is obtained by a close inspection of the quite involved proof. 
\end{remark} 

\section{Binary Polyhedral Spaces and Flat Connections}
The purpose of this section is to fix our conventions concerning the
binary polyhedral spaces and to recall the basic representation theory
needed to effectively classify the gauge equivalence classes of flat
$\om{SU}(2)$-connections. We also recall the McKay correspondence
between the finite subgroups of $\om{SU}(2)$ and the simply laced
extended Dynkin diagrams, that is, the diagrams of type
$\wt{A}_n$, $\wt{D_n}$, $\wt{E}_6$, $\wt{E}_7$ or $\wt{E}_8$. This
surprising relation is vital for our work in the next section and
is essential in the definition of the graphs $\MC{S}_\Gamma$, given in
the end of the section, mentioned in the introduction.

To avoid making this rather basic section too long we have moved a number
of facts concerning the finite subgroups of $\om{SU}(2)$ into Appendix $A$.
This includes their complex representation theory, the corresponding McKay graphs used in the construction of $\MC{S}_\Gamma$ and complete lists of the flat $\om{SU}(2)$-connections over the binary polyhedral spaces.  

\subsection{Binary Polyhedral Spaces} 
We briefly recall the classification of finite subgroups of
$\om{SU}(2)$. Let $C_l$ denote the cyclic
group of order $l$ for $l\geq 1$, let $D_k$ denote
the dihedral group of order $2k$ for $k\geq 2$, and let $T$, $O$ and $I$
denote the subgroups of $\om{SO}(3)$ that leave
a regular tetrahedron, octahedron and icosahedron, respectively,
in $\R^3$ invariant. $T$ is called the tetrahedral
group, $O$ the octahedral group and $I$ the icosahedral
group. Define $D_k^*$, $T^*$, $O^*$ and $I^*$ in $\om{SU}(2)$ by
pulling back $D_k$, $T$, $O$ and $I$ along the standard double covering
homomorphism $\om{SU}(2)\ra \om{SO}(3)$.
Every cyclic group $C_l$ may also be realized as
a subgroup of $\om{SU}(2)$ by embedding them
in a copy of $U(1)\subset \om{SU}(2)$. We have
thus constructed two infinite families $C_l$, $D_k^*$ of subgroups in $\om{SU}(2)$, as well as the three exceptional ones $T^*$, $O^*$ and $I^*$.
It is a classical fact that up to conjugacy these
exhaust all the finite subgroups of $\om{SU}(2)$. A proof may be found
in \cite[p.~83]{Wolf67}.  

\begin{theorem} 
Let $\Gamma\subset \om{SU}(2)$ be 
a finite subgroup. Then $\Gamma$ is isomorphic to precisely
one of the groups $C_l$ for $l\geq 1$, $D_k^*$ for $k\geq 2$, $T^*$, $O^*$ or
$I^*$. Moreover, if two finite subgroups $\Gamma, \Gamma'\subset \om{SU}(2)$
are isomorphic, then they are conjugate in $\om{SU}(2)$. \end{theorem}

Let $\C^2$ carry the canonical complex orientation and orient
$S^3\subset \C^2$ by the outward pointing normal first convention;
that is, an ordered basis $(v_1,v_2,v_3)\in T_xS^3\subset T_x\C^2\cong \C^2$ is positive if and only if $(x,v_1,v_2,v_3)$ is a positive basis for $\C^2$.
The standard left action of $\om{SU}(2)$ on $\C^2$ restricts to a
transitive and free action by orientation preserving isometries
on the unit sphere $S^3\subset \C^2$.

\begin{definition} \label{RightCosets}
For any finite subgroup
$\Gamma\subset \om{SU}(2)$ let $\Gamma$ act on $S^3\subset \C^2$ by restricting
the standard free action of $\om{SU}(2)$. 
Define  $Y_\Gamma$ to be the quotient manifold $S^3/\Gamma$ equipped with the 
Riemannian metric and orientation induced from the standard round metric and
orientation of $S^3$. Define $\overline{Y}_\Gamma$ to be the
same Riemannian manifold equipped with the opposite orientation.  
The spaces $Y_\Gamma$ and $\ovl{Y}_\Gamma$ are called binary polyhedral 
spaces.  
\end{definition}

Note that by the above theorem 
$Y_\Gamma$ is determined up to $\om{SU}(2)$-equivariant
isometry by the isomorphism class of $\Gamma$.

\begin{lemma} \label{BP-homology}
For $\Gamma\subset \om{SU}(2)$ we have
\[ H_i(Y_\Gamma;\Z) \cong \left\{  \begin{array}{cl}
\Z & \mbox{ for } i=0,3 \\
\Gamma^{ab} & \mbox{ for } i=1 \\
0 & \mbox{ otherwise} \end{array} \right.
\;\; \mbox{ and } \;\; H^i(Y_\Gamma ; \Z)\cong \left\{ 
\begin{array}{cl} \Z & \mbox{ for } i=0,3 \\
       \Gamma^{ab} & \mbox{ for } i=2 \\
      0 & \mbox{ otherwise, } \end{array} \right. \] 
where $\Gamma^{ab}=\Gamma/[\Gamma,\Gamma]$ is the
abelianization of $\Gamma$. In particular,
$H_*(Y_\Gamma ;\Q)\cong H_*(S^3;\Q)$, so the
$Y_\Gamma$ are rational homology $3$-spheres.  \end{lemma} 
\begin{proof} As $Y_\Gamma$ is a connected, closed
and orientable the (co)homology is concentrated in degrees $0\leq i\leq 3$ and
$H_i(Y_\Gamma)\cong H^i(Y_\Gamma)\cong \Z$ for 
$i=0,3$. By construction, $S^3$ is a universal cover
of $Y_\Gamma$ so that $\pi_1(Y_\Gamma)\cong \Gamma$ and hence $H_1(Y_\Gamma)\cong \Gamma^{ab}$. Since this group is finite it follows by the universal coefficient theorem that $H^1(Y_\Gamma)=0$.
By Poincar\'{e} duality we conclude that
$H_2(Y_\Gamma)=0$ and $H^2(Y_\Gamma)\cong \Gamma^{ab}$.
The final statement follows from the universal coefficient theorem as $H_1(Y_\Gamma;\Q)\cong \Q\otimes \Gamma^{ab}=0$. 
\end{proof}

\subsection{A Few Results from Representation Theory}
The framed Floer complex $\wt{CI}(Y,E)$ associated with the trivial 
$\om{SU}(2)$-bundle $E\ra Y_\Gamma$ is generated, in the sense of
Definition \ref{FloerComplex}, by the gauge equivalence classes of the flat connections in $E$. By a well-known result (see for instance \cite{TaubesDG} for a proof)
these gauge equivalence classes are in natural bijection with
the equivalence classes of representations
\[ \Gamma = \pi_1(Y_\Gamma)\ra \om{SU}(2)  \]
or in other words, as $\om{SU}(2)\cong \om{Sp}(1)$, the isomorphism classes of
$1$-dimensional quaternionic representations of $\Gamma$. It is convenient to express this set in terms of the complex representation theory of $\Gamma$, since in this setting we have the calculational power of
character theory at our disposal. The following
three results may be extracted from \cite[II.6]{BrockerDieck85}.

\begin{lemma} \label{Irr-Decomp-C}
Let $G$ be a finite group. The
set $\om{Irr}(G,\C)$ of isomorphism classes of
irreducible complex representations admits a decomposition into disjoint subsets
\[ \om{Irr}(G,\C)=\om{Irr}(G,\C)_\R \cup \om{Irr}(G,\C)_\C \cup \om{Irr}(G,\C)_{\HH}. \]
The set $\om{Irr}(G,\C)_\C$ consists of the classes
$[V]$ where $V\not\cong V^*$, while $\om{Irr}(G,\C)_\R$
and $\om{Irr}(G,\C)_\HH$ consists of the classes
$[W]$ admitting a conjugate linear equivariant map
$s:W\ra W$ satisfying $s^2=1$ or $s^2=-1$, respectively.
\end{lemma}

An irreducible representation $V$ is said to be
of real, complex or quaternionic type according to
whether $[V]\in \om{Irr}(G,\C)_K$ for $K=\R,\C,\HH$,
respectively. There is a simple test for determining
the type of an irreducible complex representation.

\begin{lemma}  \label{Type-Criterion}
Let $V$ be an irreducible complex representation of the finite group $G$ with character
$\chi_V$. Then
\[ \frac1{|G|}\sum_{g\in G} \chi_V(g^2) = \left\{
\begin{array}{lr} 1 & \mbox{ if } [V]\in \om{Irr}(G,\C)_\R \\
0 & \mbox{ if } [V]\in \om{Irr}(G,\C)_\C \\
-1 & \mbox{ if } [V] \in \om{Irr}(G,\C)_\HH 
\end{array} \right. \]
\end{lemma}

Finally we recall how the set $\om{Irr}(G,\HH)$ of isomorphism classes of irreducible quaternionic representations may be recovered from $\om{Irr}(G,\C)$.
Let $\om{Rep}_K(G)$ denote the category of finite dimensional $G$-representations over $K=\C,\HH$. There
is a restriction functor $r\colon \om{Rep}_\HH(G)\ra \om{Rep}_\C(G)$ given by pullback along the inclusion $\C\inj \HH$ and an extension functor $e\colon\om{Rep}_\C(G)\ra \om{Rep}_\HH (G)$ given by $e(V)=\HH\otimes_\C V$. 
Then we have the following result. 

\begin{proposition} \label{Irr-Decomp-H}
There is a decomposition into disjoint subsets
\[ \om{Irr}(G,\HH)=\om{Irr}(G,\HH)_\R \cup \om{Irr}(G,\HH)_\C \cup \om{Irr}(G,\HH)_\HH . \]
Furthermore, the following maps are bijections
\begin{align*}
r:&\om{Irr}(G,\HH)_\HH \ra  \om{Irr}(G,\C)_\HH  \\
e:& \frac12 \om{Irr}(G,\C)_\C \ra \om{Irr}(G,\HH)_\C \\
e:& \om{Irr}(G,\C)_\R \ra \om{Irr}(G,\HH)_\R, 
\end{align*}
where $\frac12 \om{Irr}(G,\C)_\C$ denotes the set
of unordered pairs $\{[V],[V^*]\}$ for $[V]\in \om{Irr}(G,\C)_\C$. 
\end{proposition}

Let $\om{Irr}^n(G,\C)_K$ denote the set of irreducible
$G$-representations of dimension $n$ of type $K=\R,\C,\HH$. 

\begin{corollary} The set of isomorphism classes of
$1$-dimensional quaternionic representations of the
finite group $G$ are in bijection with
\[ \om{Irr}^2(G,\C)_\HH \cup \frac12 \om{Irr}^1(G,\C)_\C
\cup \om{Irr}^1(G,\C)_\R   .\]
In particular, if $G=\Gamma$ is a finite subgroup of 
$\om{SU}(2)$ then the above set is in bijection with
the set of gauge equivalence classes of flat
connections in the trivial $\om{SU}(2)$-bundle
over $Y_\Gamma$. \end{corollary}

Note that a flat connection is irreducible, reducible or fully reducible
if and only if the associated representation is of quaternionic, complex
or real type, respectively.  

\subsection{Arrangement of the Flat Connections} \label{Section-McKay}
The McKay correspondence \cite{McKay} sets up a bijection
between the isomorphism classes of the nontrivial finite subgroups of 
$\om{Sp}(1)\cong \om{SU}(2)$ and the simply-laced extended Dynkin diagrams. 
The extended Dynkin graph is constructed from the group
$\Gamma$ as follows. Let $R_0,R_1,\cdots ,R_n$ be
a complete set of representatives for the elements of $\om{Irr}(\Gamma,\C)$, where we take $R_0$ to be the trivial
representation. Let $Q$ denote the two dimensional representation associated with the inclusion $\Gamma\subset \om{SU}(2)$ and define
a matrix $A=(a_{ij})_{ij}$ by
\[  Q\otimes R_i =\bigoplus_{j=0}^n a_{ij}R_j \;\; \mbox{ for } \;\; 0\leq i\leq n  .\]
One may show that $A$ is symmetric with $a_{ii}=0$ and
$a_{ij}\in \{0,1\}$. Define an unoriented graph $\overline{\Delta}_\Gamma$
by taking $I=\{0,1,\cdots,n\}$ as the set of vertices and
an edge connecting $i$ to $j$ if and only if
$a_{ij}=a_{ji}=1$. This graph will then be an extended Dynkin diagram
of type $\wt{A}_n$, $\wt{D}_n$, $\wt{E}_6$, $\wt{E}_7$ or $\wt{E}_8$. Furthermore, if we let $\Delta_\Gamma$ be the graph obtained from $\ovl{\Delta}_\Gamma$
by deleting the vertex $0$ corresponding to the trivial representation
we obtain the underlying Dynkin diagram of type $A_n$, $D_n$, $E_6$, $E_7$ or $E_8$. The precise correspondence is given in the following table where
$l,k \geq 2$.

\[ \begin{array}{|c|c|c|c|c|c|}
\hline
\Gamma & C_l & D^*_k & T^* & O^* & I^* \\ \hline
\Delta_\Gamma & A_{l-1} & D_{k+2} & E_6
& E_7 & E_8 \\ \hline \end{array} \] 

The set of vertices, say $\{1,2,\cdots,n\}$, in a Dynkin diagram corresponds
to a set of simple roots $\{r_i:1\leq i\leq n\}$ in the associated root system. There is then a unique
maximal positive root $r_{\max} = \sum_{i=1}^n d_i r_i$ where $d_i\in \N$
are positive integers. The final fact we wish to bring out is that these integers are determined by $d_i = \om{dim}_\C R_i$.

By mapping an irreducible representation $R_i$ to its dual representation
$R_i^*$ we obtain an involution $\iota\colon \{0,1,\cdots,n\}\ra \{0,1,\cdots,n\}$,
i.e., $R_{\iota(i)}\cong R_i^*$. This map extends to a graph involution of
$\ovl{\Delta}_\Gamma$ as
\[ a_{\iota(i)\iota(j)} = \om{dim}_\C \om{Hom}_\Gamma(Q\otimes R_i^*,R_j^*)
 = \om{dim}_\C \om{Hom}_\Gamma(Q\otimes R_j,R_i) = a_{ji}=a_{ij}  \]
for all $0\leq i,j\leq n$. In view of Lemma \ref{Irr-Decomp-C} we see that
the fixed points of $\iota$ correspond to the vertices of quaternionic or real type, while the nontrivial orbits are pairs $\{i,j\}$ for which $i\neq j$ and $R_i$, $R_j\cong R_i^*$ are of complex type. The set of vertices in the quotient graph
$\ovl{\Delta}_\Gamma / (\iota)$ may therefore by
Proposition \ref{Irr-Decomp-H} be identified with
$\om{Irr}(\Gamma,\HH)$. In particular, we may identify our set
$\MC{C}$ of flat connections with a subset of the vertices in this quotient
graph. The following may be observed from the quotient graphs
given in Appendix $A$.

\begin{lemma} \label{Tree-Lemma}
For each finite subgroup $\Gamma\subset \om{SU}(2)$ the
quotient graph $\overline{\Delta}_\Gamma/(\iota)$ is a connected tree, except
for $\Gamma=C_l$ with $l$ odd in which case it has the form
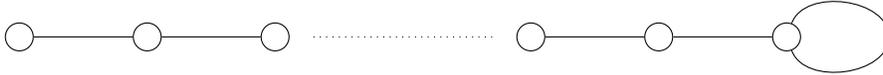
\begin{center}
\begin{tikzpicture}[scale=1.7]
\node (1) at (0,0) [circle,draw] {} ;
\node (2) at (1,0) [circle,draw] {} ;
\node (3) at (2,0) [circle,draw,] {} ;
\node (4) at (4,0) [circle,draw] {} ;
\node (5) at (5,0) [circle,draw] {} ;
\node (6) at (6,0) [circle,draw] {} ;

\path
(1) edge (2)
(2) edge (3)
(4) edge (5)
(5) edge (6);
\draw[dotted] (2.3,0)--(3.7,0);
\draw (6) to [out=290, in = 270] (6.8,0) to [out=90,in = 70] (6);
\end{tikzpicture}
\end{center}
\end{lemma}

From the above lemma it follows that for each pair of vertices there
is a unique minimal edge path connecting them.

\begin{definition} \cite[p.~297]{Austin95} \label{Def-Adjacency}
Let $\Gamma\subset \om{SU}(2)$ be a finite subgroup.
Define a graph $\MC{S}_\Gamma$ by letting the vertices be the set
of $1$-dimensional quaternionic representations of $\Gamma$ and an
(unoriented) edge connecting a pair of distinct vertices $\alpha$ and $\beta$
if and only if the minimal edge path connecting $\alpha$ to $\beta$
in $\overline{\Delta}_\Gamma/(\iota)$ does not pass through a vertex
corresponding to a $1$-dimensional quaternionic representation
different from $\alpha$ and $\beta$. 

Two vertices in $\MC{S}_\Gamma$ are said to be adjacent if there is an edge connecting them. \end{definition}

At the end of the next section we will attach symbols to the edges of the
graph $\MC{S}_\Gamma$ such that the resulting labeled graph contains all
the necessary information needed to set up the complex 
$DCI(\ovl{Y}_\Gamma)$. All of the graphs $\MC{S}_\Gamma$ will be given
in Proposition \ref{Differential-Graphs}. 

\section{Determination of Moduli Spaces and Differentials}
In this section we review the key results that allow us to determine
the Donaldson models $DCI(\ovl{Y}_\Gamma)$, for the finite subgroups
$\Gamma\subset \om{SU}(2)$, explicitly.
The main result of \cite{Austin95} necessary for the calculations, originally due to Kronheimer, is the following.

\begin{theorem} \label{Moduli-Result}
Let $\Gamma\subset \om{SU}(2)$ be a finite subgroup and let 
$\alpha,\beta$ be flat connections in the
trivial $\om{SU}(2)$-bundle over $\overline{Y}_\Gamma$. Let $z$ be the unique
homotopy class from $\alpha$ to $\beta$ with 
$-2\leq \om{gr}_z(\alpha,\beta)\leq 5$. Then if $\alpha$ and $\beta$
are adjacent in $\MC{S}_\Gamma$ (see Def. \ref{Def-Adjacency}) then in
the diagram
\[ \begin{tikzcd} \alpha & \overline{\MC{M}}_z(\alpha,\beta)
\arrow{l}[swap]{e_-} \arrow{r}{e_+} & \beta \end{tikzcd}   \]
there is an identification $e_-^{-1}(*)\cong Y_{\Gamma'}$ for some
other finite subgroup $\Gamma'\subset \om{SU}(2)$.
If $\alpha$ and $\beta$ are not adjacent in $\MC{S}_\Gamma$, then
$\ovl{\MC{M}}_z(\alpha,\beta)$ is empty. 
\end{theorem}

There is also fairly simple graphical procedure to determine the subgroup
$\Gamma'$ from the positions of $\alpha$ and $\beta$ in the extended
Dynkin diagram $\ovl{\Delta}_\Gamma$ associated with $\Gamma$ under 
the McKay correspondence (see Proposition \ref{R(G)-Adjacent-Solution}). 

The main idea of the proof is to translate the problem of determining
instantons over $\R\times Y_\Gamma$ with $L^2$ curvature modulo gauge transformations into the problem of determining
$\Gamma$-invariant instantons in $\Gamma$-equivariant $\om{SU}(2)$-bundles
over $S^4$ modulo $\Gamma$-invariant gauge transformations, and then
to solve the latter problem with the use of the equivariant ADHM correspondence.

We begin by setting up the first correspondence. Recall that
the instantons on the trivial $\om{SU}(2)$-bundle over the cylinder
$\R\times Y_\Gamma$ with $L^2$-curvature are naturally partitioned
into subsets by three pieces of data: the limiting flat connections
$\alpha_\pm \in \MC{C}\cong \om{Irr}^1(\Gamma,\HH)$ and the relative homotopy class $z$ traced out by the connection. For our purpose it is more convenient to replace this latter invariant by the relative Chern number
(see \cite[p.~48]{Donaldson02})
\[ \hat{c}_2(A) \coloneqq \frac{1}{8\pi^2}\int_{\R\times Y_\Gamma} \om{tr}F_A^2 =\frac{1}{8\pi^2}\int_{\R\times Y_\Gamma} |F_A|^2 \in \R  .\]
Thus to any instanton $A$ with $L^2$-curvature over $\R\times Y_\Gamma$ we may associate a unique triple 
\[ (\alpha_+,\alpha_-,\hat{c}_2(A))\in \om{Irr}^1(\Gamma,\HH)\times \om{Irr}^1(\Gamma,\HH)\times \R  .\]

Let $\Gamma$ act on $\R\times S^3$ by $\gamma\cdot (t,x)\mapsto (t,\gamma x)$ such that $q=(1\times p)\colon \R\times S^3\ra \R\times Y_\Gamma$ is an orientation preserving covering map.
Regard $S^4\subset \C^2\oplus \R$ such that the suspended action $\Gamma\times S^4\ra S^4$ takes the form $\gamma\cdot (x,t)=(\gamma x,t)$. By composing
the conformal equivalence 
$\R\times S^3 \cong \C^2-\{0\}$, $(t,x)\mapsto e^tx$, with
the inverse of stereographic projection from the north pole we obtain
an equivariant conformal equivalence $\tau\colon\R\times S^3\cong S^4-\{N,S\}$,
where $N,S=(0,\pm1 )\in S^4$ are the poles. Explicitly, 
\[ \tau(t,x)=(\cosh(t)x,\tanh(t))  .\]
Let $S^4$ carry the orientation that makes this map orientation preserving,
that is, the opposite of the standard orientation.

Recall that a $\Gamma$-equivariant vector-bundle over a $\Gamma$-space
$X$ is a vector bundle $E\ra X$ with a $\Gamma$-action on
the total space making the projection equivariant and such that
$\gamma:E_x\ra E_{\gamma x}$ is linear for each $x\in X$ and $\gamma\in \Gamma$. If $E$
carries an $\om{SU}(2)$-structure, that is, a Hermitian metric and a
fixed unitary trivialization of $\Lambda^2_\C E$, then we call $E$
a $\Gamma$-equivariant $\om{SU}(2)$-bundle provided the action
preserves this additional structure. In that case, there is an induced action
of $\Gamma$ on the space of connections $\MC{A}_E$. A connection is said to
be invariant if it is fixed under this action. The subspace of invariant
connections is denoted by $\MC{A}_E^\Gamma$. The action of the full gauge
group does not preserve this subspace, but the subgroup of equivariant gauge
transformations $\MC{G}^\Gamma_E \subset\MC{G}_E$ does. Therefore, the natural configuration space in this equivariant setting is 
$\MC{B}_{E,\Gamma} \coloneqq \MC{A}^\Gamma_E /\MC{G}_E^\Gamma$. We may also include
a framing coordinate to obtain the framed configuration space
$\wt{\MC{B}}_{E,\Gamma}$. All of the above may, of course, be spelled out
in terms of principal bundles as well. 

Let $F\ra \R\times Y_\Gamma$ be the trivial $\om{SU}(2)$-bundle. Then
$E\coloneqq q^*F\ra  \R\times S^3$ obtains the structure of a $\Gamma$-equivariant vector bundle in a natural way. Moreover, since $q$ is a finite covering map, pullback induces
a bijection $\MC{A}_F\cong \MC{A}_E^\Gamma$ that matches the subsets
of anti-self-dual connections with curvature in $L^2$.
Let $B$ be an ASD connection in $F$ with curvature in $L^2$ and let
$A = q^*(B)$ be the pullback connection in $E$. According to
Uhlenbeck's removable singularities theorem there exists a bundle
$\wt{E}\ra S^4$ with connection $\wt{A}$ and a bundle map
$\rho:E\ra \wt{E}$ covering $\tau\colon \R\times S^3\ra S^4$ such that
$\rho^*(\wt{A})=A$. The following diagram summarizes the situation

\begin{equation} \label{Eq-Correspondence-Diagram}
 \begin{tikzcd} (F,B) \arrow{d} & (E,A) \arrow{l} \arrow{r}{\rho} \arrow{d} &
(\wt{E},\wt{A}) \arrow{d} \\
\R\times Y_\Gamma & \R\times S^3 \arrow{l}[swap]{q} \arrow{r}{\tau} & S^4.
\end{tikzcd}
\end{equation}

According to Austin \cite{Austin95}
the $\Gamma$-action extends over $\wt{E}$ making $\wt{A}$ invariant and
the flat limiting connections of the connection $B$ are encoded in the isomorphism classes of the representations
$\wt{E}_N$ and $\wt{E}_S$ over the fixed points. Furthermore,
the following calculation shows that $c_2(\wt{E})$ is uniquely determined
by $\hat{c}_2(B)$
\begin{align} \label{c_2-relation}
8\pi^2 c_2(\wt{E})[S^4] =& \int_{S^4} \om{tr}(F_{\wt{A}}^2)
=\int_{S^4} |F_{\wt{A}}|^2 = \int_{\R\times S^3} |F_A|^2 \\
=& |\Gamma |\int_{\R\times Y_\Gamma}|F_B|^2 =8\pi^2 |\Gamma | \hat{c}_2(B). \nonumber
\end{align}
Here we have used the fact that the Chern-Weil integrand
$\om{tr}(F_{\wt{A}}^2)$ coincides pointwise with the norm 
$|F_{\wt{A}}|^2$ for ASD connections, the conformal invariance of the integral of $2$-forms in dimension four and that $\R\times S^3\ra \R\times Y_\Gamma$ is a finite Riemannian covering map with fibers of cardinality 
$|\Gamma |$. 

The following theorem classifies the $\Gamma$-equivariant $\om{SU}(2)$-bundles
over $S^4$. The proof is given in Appendix $B$ (see Theorem \ref{Classification of Equivariant Bundles over $S^4$}). 

\begin{theorem} \label{Gamma-Bundles}
A $\Gamma$-equivariant $\om{SU}(2)$-bundle $\wt{E}\ra S^4$
is determined up to isomorphism by the ordered triple
$([\wt{E}_N],[\wt{E}_S],k)$, where $k=c_2(\tilde{E})[S^4]$ is the
second Chern number of the underlying $\om{SU}(2)$-bundle and $[\wt{E}_N], [\wt{E}_S]$ denote the isomorphism classes of the $\Gamma$-representations
over the fixed points. Moreover,
for each pair $\alpha,\beta\in \om{Irr}^1(\Gamma,\HH)$ there is a constant $c\in \Z$ such that for each $k\equiv c \Mod{|\Gamma |}$ there exists
a $\Gamma$-equivariant $\om{SU}(2)$-bundle with invariants $(\alpha,\beta,k)$.
\end{theorem} 

Let $\wt{E}\ra S^4$ be a $\Gamma$-equivariant $\om{SU}(2)$-bundle
with invariants $(\alpha_+,\alpha_-,k)$. Let $E=\tau^*(E)\ra \R\times S^3$
and $F = E/\Gamma \ra \R\times Y_\Gamma$. Then we may reverse the above
procedure to produce a map from the invariant instantons in $\wt{E}$
to the instantons over $\R\times Y_\Gamma$ with limiting flat connection
$\alpha_\pm$ and relative Chern number $k/|\Gamma |$. More is true,
the $\Gamma$-equivariant gauge transformations in $\wt{E}$ correspond
to gauge transformations in $F$ that approach $\alpha_\pm$-harmonic
gauge transformations on the ends. Let $b_0\in Y_\Gamma$ denote a basepoint
and let $\tilde{b}\in S^3$ be a lift. 
Let  
$\wt{\MC{M}}_{\Gamma,\infty}(\wt{E})\subset \wt{\MC{B}}_{\wt{E},\Gamma}$
be the moduli space of invariant instantons in $\wt{E}$ framed at
$(0,\wt{b})$. 
Then the above discussion yields a bijective correspondence
\[ \wt{\MC{M}}_\Gamma(\wt{E})\cong \wt{\MC{M}}_z(\alpha_-,\alpha_+) , \]
where $z$ is the relative homotopy class corresponding to the relative
Chern number $k/|\Gamma |$ and the latter space is framed
at $(0,b)\in \R\times Y_\Gamma$. Recall that the end-point maps
$e_\pm:\wt{\MC{M}}_z(\alpha_-,\alpha_+)\ra \alpha_\pm$ were defined
by parallel transport of the framing to $\pm\infty$. Under the above
bijection this corresponds to parallel transport of the framing
along the great circle through the basepoint $(\tilde{b},0)\in S^4$ and
the poles $N,S=(0,\pm1)$. The equivariant ADHM correspondence, which we turn
to next, describes the instanton moduli space with fixed framing at
$N$ and we then see that this corresponds to the fiber
$e_+^{-1}(*)$. If we change the orientation in $Y_\Gamma$ we may precompose
with the map $(t,x)\mapsto (-t,x)$ in $\R\times Y_\Gamma$ to see that
the ADHM correspondence will describe the fiber $e_-^{-1}(*)$ instead.

\subsection{The ADHM Description of Instantons}
We first recall the non-equivariant ADHM classification of
$\om{SU}(2)$-instantons over $S^4$ \cite[Sec.~3.3]{DK90}.
This construction describes the various moduli spaces of
instantons over $S^4$ in terms of linear algebraic data.
In the following we regard $\C^2\cup \{\infty\}=S^4$ where the orientation
is determined by the complex orientation on $\C^2$. 
Recall that an $\om{SU}(2)$-bundle over $S^4$, or any other closed, oriented
$4$-manifold, is determined up to isomorphim by the second
Chern number $k=c_2(E)[S^4]\in \Z$. The bundle can only support an
ASD connection if $k\geq 0$ and in case $k=0$ every ASD connection
is flat. 

For a fixed integer $k\geq 1$ let $\Har$ be a Hermitian
vector space of dimension $k$ and let $E_\infty$ be a $2$-dimensional complex vector space with symmetry group $\om{SU}(2)$.
Define the complex vector space
\[ M \coloneqq  \om{Hom}(\C^2\otimes \Har,\Har)\oplus \om{Hom}(\Har,E_\infty)\oplus \om{Hom}(E_\infty,\Har). \]
By identifying $\C^2 \otimes \Har =\Har\oplus \Har$ we may write an element $m\in M$ in terms of its components $m=(\tau_1,\tau_2,\pi,\sigma)$ where
$\tau_1,\tau_2\in \om{End}(\Har)$, $\pi\colon \Har\ra E_\infty$ and $\sigma\colon E_\infty\ra \Har$.
Let $U(\Har)$, the group of unitary automorphisms of $\Har$, act on $M$
by 
\[ g\cdot (\tau_1,\tau_2,\pi,\sigma)=(g\tau_1g^{-1},g\tau_2g^{-1},\pi g^{-1},g\sigma). \] 
Set $\mathfrak{u}(\Har)\coloneqq \om{Lie}(U(\Har))$ and define 
$\mu=(\mu_\C,\mu_\R):M\ra \om{End}(\Har)\oplus \mathfrak{u}(\Har)$ by
\begin{align} \label{ADHM-eq}
\mu_\C(\tau_1,\tau_2,\pi,\sigma)=&
[\tau_1,\tau_2]+\sigma \pi \\
\mu_\R(\tau_1,\tau_2,\pi,\sigma)=&[\tau_1,\tau_1^*]+[\tau_2,\tau_2^*]
+\sigma \sigma^*-\pi^*\pi. \nonumber 
\end{align}
These maps are equivariant when $U(\Har)$ acts by conjugation on
$\om{End}(\Har)$ and $\mathfrak{u}(\Har)$. 

Given $(\tau_1,\tau_2,\pi,\sigma)\in M$
we may for each $z=(z_1,z_2)\in \C^2$ form the sequence
\begin{equation} \label{ADHM-sequence}
\begin{tikzcd} \Har \arrow{r}{A_z} & \C^2\otimes \Har\oplus E_\infty \arrow{r}{B_z} & \Har
\end{tikzcd}
\end{equation}
where 
\[ A_z= \left(\begin{array}{c}
\tau_1-z_1 \\ \tau_2-z_2 \\  \pi \end{array} \right)
\; \mbox{ and }\; B_z= \left( \begin{array}{ccc}
-(\tau_2-z_2) & (\tau_1-z_1) & \sigma \end{array} \right)   .\]
The condition $\mu_\C(\tau_1,\tau_2,\pi,\sigma)=0$ is equivalent to
$B_z\circ A_z=0$ for all $z\in \C^2$. 
Define $U_{reg}\subset M$ to be the open set consisting of the 
$(\tau_1,\tau_2,\pi,\sigma)$ for which $A_z$ is injective and $B_z$ is surjective for all $z\in \C^2$.
From the sequence \eqref{ADHM-sequence} we may form a bundle homomorphism
$R$ between the trivial bundles over $\C^2$ with fibers $\C^2\otimes \Har$ and
$\C^2\otimes \Har\oplus E_\infty$, respectively, by $R_z=(A_z,B_z^*)$.
This map extends naturally to a bundle homomorphism 
\[ \begin{tikzcd} \gamma\oplus \gamma \oplus \cdots \oplus\gamma \arrow{rr}{R} \arrow{dr} & {} & \underline{\C^2\otimes \Har\oplus E_\infty} \arrow{dl} \\
{} & \C^2\cup \{\infty\} =S^4, & {} \end{tikzcd} \]
where $\gamma$ is the $\om{SU}(2)$-bundle with $c_2(\gamma)[S^4]=-1$
and we take $k$ copies of $\gamma$ on the left hand side. This fact is
most easily understood using quaternions and the model $\HH P^1\cong S^4$
(see \cite{Atiyah79}). 
If $m=(\tau_1,\tau_2,\pi,\sigma)\in U_{reg}\cap \mu_\C^{-1}(0)$ it follows that $R$ is everywhere injective and we obtain an $\om{SU}(2)$-bundle
\[ E \coloneqq \om{Im}(R)^\perp \cong \om{Coker}(R)  \]
with $c_2(E)[S^4]=k$, and the fiber over the point at $\infty$ is the space
$E_\infty$ we started with. This bundle carries a natural connection $A$ obtained by orthogonal projection from the product connection in the trivial bundle. This connection is ASD provided $\mu_\R(m)=0$ as well.   

The above construction associates a pair $(E,A)$ consisting of
an $\om{SU}(2)$-bundle $E\ra S^4$ with $c_2(E)=k$ and an ASD connection
$A$ in $E$ to each element $m\in U_{reg}\cap \mu^{-1}(0)$. If $m$ is replaced
by $g\cdot m$ for $g\in U(\Har)$ the resulting pair $(E',A')$ is equivalent
to $(E,A)$ by a bundle isomorphism $\rho^g:E\ra E'$ fixing the
fiber at $\infty$. The content of the ADHM classification is that this
construction produces all instantons in the $\om{SU}(2)$-bundle over
$S^4$ with Chern number $k$.  

\begin{theorem}(The ADHM-Correspondence)
The assignment $(\tau_1,\tau_2,\pi,\sigma)\mapsto (E,A)$ descends to an equivalence
\[ (U_{reg}\cap \mu^{-1}(0))/U(\Har) \cong \MC{M}_\infty(k)  \]
where $\MC{M}_\infty(k)$ denotes the moduli space of gauge equivalences
classes of instantons framed at $\infty$ in the bundle of
Chern number $k$ over $S^4=\C^2\cup \{\infty\}$. \end{theorem}

For the proof of the theorem and additional details we refer to \cite[Section~3.3]{DK90}. A part of the proof involves giving an inverse procedure to the construction sketched above. 
We will need one fact from this inverse procedure. Given an
$\om{SU}(2)$-bundle $E\ra S^4$ with $c_2(E)[S^4]=k$ and ASD connection
$A$, the space $\Har$ is recovered by
\begin{equation} \label{Harmonic-ADHM}
\om{Ker}[D_A^*\colon \Gamma(S^-\otimes E)\ra \Gamma(S^+\otimes E)],
\end{equation}
where $S^\pm$ are the complex spinor bundles associated with the
unique spin structure on $S^4$ and $D_A^*$ is the formal adjoint
of the Dirac operator twisted by the connection $A$.   

Let $\Gamma\subset \om{SU}(2)$ be a finite subgroup. Then
according to \cite{Austin95} the ADHM correspondence is natural with respect
to the standard linear action $\Gamma\times \C^2\ra \C^2$ and
the corresponding suspended action on $\C^2\cup\{\infty\}=S^4$. This means
that if $E\ra S^4$ is a $\Gamma$-equivariant bundle with a $\Gamma$-invariant
ASD connection $A$, then the spaces $E_\infty$, $\Har$ are
complex $\Gamma$-representations and $m=(\tau_1,\tau_2,\pi,\sigma)\in U_{reg}\cap \mu^{-1}(0)$ associated with $(E,A)$ belongs to the subspace $M^\Gamma$
of invariants. Conversely, if $\Har$ and $E_\infty$ are
$\Gamma$-representations and $m\in U_{reg}\cap \mu^{-1}(0)\cap M^\Gamma$
then the associated bundle $E$ naturally obtains the structure of
a $\Gamma$-equivariant bundle such that the connection $A$ is invariant.  

As we have seen the $\Gamma$-equivariant $\om{SU}(2)$-bundles over $S^4$
are determined up to isomorphism by the triple of invariants $(\alpha,\beta,k)$
as in Theorem \ref{Gamma-Bundles}. We need to know which representation
$\Har$ that produces invariant instantons in the bundle with data
$(\alpha,\beta,k)$. The relevant formula is obtained by an application
of the equivariant index theorem to the twisted Dirac operator in \eqref{Harmonic-ADHM}. Let $R(\Gamma)$ denote the complex representation ring of $\Gamma$ and let $\eps:R(\Gamma)\ra \Z$ denote the augmentation. Furthermore, let as before $Q$ denote the $2$-dimensional
representation associated with the inclusion $\Gamma\subset\om{SU}(2)$.  

\begin{lemma} \label{Austin-Equation}
Let $E\ra S^4$ be a $\Gamma$-equivariant $\om{SU}(2)$-bundle
with invariants $([E_N],[E_S],c_2(E)[S^4])=(\alpha,\beta,k)$ and set
\[ \Har =\om{Ind}[D_A^*\colon \Gamma(S^-\otimes E)\ra \Gamma(S^+\otimes E)]\in R(\Gamma) \] 
where $A$ is a $\Gamma$-invariant connection in $E$.
Then we have the following relations in the representation ring $R(\Gamma)$:
\begin{align*} 
(2-Q)\Har &= \alpha-\beta  \\
\eps(\Har) &= k.
\end{align*}
\end{lemma}     

These formulas are given in \cite[Sec.~4.3]{Austin95}. We have included
a proof in Appendix \ref{B-IndexCalc}. The following lemma ensures that
the above two equations determine the virtual representation $\Har$
uniquely.

\begin{lemma} \label{R(G)-equation}
Let $\Gamma\subset \om{SU}(2)$ be a finite subgroup and let $Q$ be the
canonical $2$-dimensional representation of $\Gamma$. Then the kernel
of the map $R(\Gamma)\ra R(\Gamma)$ given by $V\mapsto (2-Q)V$ is given by
$\{nR:n\in \Z\}$ where $R\cong \C[\Gamma]$ denotes the regular representation.
\end{lemma}
\begin{proof} Assume that $(2-Q)V=0$ in $R(\Gamma)$. In terms of the
associated characters this means that $(2-\chi_Q(g))\chi_V(g)=0$
for all $g\in \Gamma$. Since $Q$ is a faithful representation of
$\Gamma$ it follows that $\chi_Q(g)\neq \chi_Q(1)=2$ for all $g\neq 1$,
from which we deduce that $\chi_V(g)=0$ for all $g\neq 1$. There exist
representations $U,W$ of $\Gamma$ such that $V=U-W$ in $R(\Gamma)$.
Therefore, 
\[ \chi_V(1)/|\Gamma | = (\chi_V,1_\Gamma)=(\chi_U,1_\Gamma)-(\chi_W,1_\Gamma) = \om{dim}_\C U^\Gamma-\om{dim}_\C W^\Gamma \in \Z , \]
where $1_\Gamma$ is the character of the trivial representation and
$(\cdot,\cdot)$ is the Hermitian inner product on characters.
As the character of the regular representation $R$ is given by
$\chi_R(1)=|\Gamma|$ and $\chi_R(g)=0$ for $g\neq 1$, we conclude
that for $m=\chi_V(1)/|\Gamma|$ we have
$\chi_V = m\chi_R$, and hence $V=mR$ in $R(\Gamma)$.  
\end{proof}

Recall that $R(\Gamma)$ is, as an abelian group, freely generated by the isomorphism classes of the irreducible representations of $\Gamma$. 
Therefore, if we label the irreducible
representations by $R_0,\cdots,R_m$, then every element in $R(\Gamma)$ may
be expressed uniquely as a linear combination $\sum_{i=0}^m n_i[R_i]$ for
$n_i\in \Z$. Such a sum is called a virtual representation and in the
case all $n_i\geq 0$ we call it an actual representation since
it determines the representation $\bigoplus_i n_iR_i$ up to
isomorphism.   

We are now in a position to state the equivariant ADHM correspondence.
This is a more precise version of \cite[Lemma~4.3]{Austin95}.

\begin{theorem} \label{Equivariant-ADHM} (The Equivariant ADHM Correspondence) Let $E$ be a $\Gamma$-equivariant
$\om{SU}(2)$-bundle over $S^4$ with $([E_N],[E_S],c_2(E)[S^4])=(\alpha,\beta,k)$. Let $\Har\in R(\Gamma)$ be the unique element satisfying
$(2-Q)\Har = \alpha-\beta$ and $\eps(\Har)=k$. Then if $\Har$ is
an actual representations the moduli space $\wt{\MC{M}}_{\Gamma,\infty}(E)$ of
$\Gamma$-invariant instantons modulo $\Gamma$-equivariant gauge transformations
framed at $\infty = N$ is given by
\[ (U_{reg}\cap \mu^{-1}(0) \cap M^\Gamma) / U(\Har)^\Gamma , \]
where $U(\Har)^\Gamma$ is the group of unitary $\Gamma$-equivariant
automorphisms of $\Har$ and
\[ M^\Gamma = \om{Hom}_\Gamma(Q\otimes \Har,\Har)\oplus \om{Hom}_\Gamma(\Har,E_\infty)\oplus \om{Hom}_\Gamma(E_\infty,\Har)  .\]
If $\Har$ is not an actual representation, the moduli space is
empty.
\end{theorem}

Note that in the above statement $E_\infty = E_N$. 

\subsection{Determination of the Low-Dimensional Moduli Spaces}
The ADHM equations \eqref{ADHM-eq} are a particularly simple example of the hyper-K\"{a}hler moment map equations associated with an action of a Lie group on a hyper-K\"{a}hler manifold (see \cite{HKL87}). We briefly review the
set up in the linear case relevant for our purpose. 

Let $V$ be a (left) quaternionic vector space equipped with a compatible
real inner product $g\colon V\times V\ra \R$, i.e., multiplication by unit
quaternions are orthogonal. Let $\om{Sp}(V)$ denote the group of
$\HH$-linear orthogonal automorphisms of $V$. 
Define an $\om{Im}\HH\coloneqq \om{Span}_\R\{i,j,k\}$ valued
$2$-form $\omega\colon V\times V\ra \om{Im}\HH$ by the formula
\[ \omega(v,w)=ig(iv,w)+jg(jv,w)+kg(kv,w)   .\]
Note that the unique symplectic inner product $s:V\times V\ra \HH$ determined
by $g$ is given by $s=g-\omega$.

Suppose that $G$ is a compact Lie group acting on $V$ through
a homomorphism $\rho:G\ra \om{Sp}(V)$. The Lie algebra $\fg$ of $G$
acts on $V$ through the derivative $d\rho_1 :\fg\ra \om{sp}(V)=\om{Lie}(\om{Sp}(V))$.
A hyper-K\"{a}hler moment map for the action $G\times V\ra V$ is a $G$-equivariant map $\mu\colon V\ra \om{Hom}(\fg,\om{Im}\HH)$, where $G$ acts on the target through the adjoint representation, satisfying
\[ d\mu_x(v)=[\xi \mapsto \omega(\xi\cdot x,v)]\in \om{Hom}(\fg,\om{Im}\HH)  \]
for all $x\in V$ and $v\in V\cong T_xV$. A moment map is clearly unique up to
a constant and one may check that the unique moment map vanishing at
$0\in V$ is given by
\begin{equation} \label{HK-Moment-Map}
\mu(x) = [\xi\in \fg \mapsto \frac12 \omega(\xi\cdot x,x) \in \om{Im}\HH ] .
\end{equation} 

\begin{definition} For a pair $(V,\rho\colon G\ra \om{Sp}(V,g))$ as above we
define the hyper-K\"{a}hler quotient of $V$ by $G$ to be the space
\[ \mu^{-1}(0)/G   .\]
\end{definition} 

In general neither $\mu^{-1}(0)$ nor $\mu^{-1}(0)/G$ will be manifolds
due to singularities. However, if $U\subset V$ is an open $G$-invariant
subset on which $G$ acts freely, then $\mu^{-1}(0)\cap U$ and the quotient
by $G$ will be manifolds. Moreover, in this case the quotient inherits the structure of a hyper-K\"{a}hler manifold. 

\begin{example}(The ADHM equations) \label{Example-ADHM-eq}
Let $\Har$ and $E_\infty$ be a pair
of Hermitian vector space and define
\[ V = \om{Hom}(\C^2\otimes \Har , \Har)\oplus \om{Hom}(\Har,E_\infty)
\oplus \om{Hom}(E_\infty,\Har )  \]
equipped with the Hermitian metric induced from $\Har$ and $E_\infty$.
Using the standard basis of $\C^2$ an element of $V$ is written
$(\tau_1,\tau_2,\pi,\sigma)$. Let $G=U(\Har)$ act on $V$ by
$g\cdot (\tau_1,\tau_2,\pi,\sigma)=(g\tau_1 g^{-1},g\tau_2 g^{-1},\pi g^{-1},g\sigma)$. Define a quaternionic structure map $J\colon V\ra V$ by
\[ J(\tau_1,\tau_2,\pi,\sigma) = (-\tau_2^*,\tau_1^*,-\sigma^*,\pi^*).  \]
Then the action of $G$ commutes with
$J$ and therefore defines a homomorphism $\rho\colon G\ra \om{Sp}(V)$.  
If we use the trace inner product in $\mathfrak{u}(\Har)$ and the decomposition
$\om{Im}(\HH)\cong \R\oplus \C$ to identify
\[ \om{Hom}(\mathfrak{u}(\Har),\om{Im}\HH) \cong \mathfrak{u}(\Har)\otimes \om{Im}(\HH) \cong \mathfrak{u}(\Har)\oplus \mathfrak{u}(\Har)\otimes \C
\cong \mathfrak{u}(\Har)\oplus \om{End}_\C(\Har),
\]
the resulting moment map $\mu\colon V\ra \mathfrak{u}(\Har)\oplus \om{End}_\C (\Har)$ vanishing at $0$ is given by the ADHM equations in (\ref{ADHM-eq}). 
\end{example} 

Kronheimer \cite{Kronheimer89} used this construction in the following
situation. Let $\Gamma\subset \om{SU}(2)$ be a finite subgroup and
let $Q$ and $R$ denote the canonical representation and the
regular representation, respectively. Assume that $Q$ and $R$ are equipped with
$\Gamma$-invariant Hermitian metrics and fix a $\Gamma$-invariant
quaternionic structure $j:Q\ra Q$ compatible with the metric. Let
\[  P = Q\otimes \om{End}(R)\cong \om{Hom}(Q\otimes R,R) \]
equipped with the induced Hermitian metric. By combining the
quaternionic structure on $Q$ with the real structure $f\mapsto f^*$
on $\om{End}(R)$ one obtains a $\Gamma$-equivariant quaternionic structure
$J\colon P \ra P$ compatible with the Hermitian metric. As this structure
is preserved by the action of $\Gamma$, the subspace
$K= P^\Gamma$ carries the same structure. 

The group $U(R)$ acts on $P$ by conjugation in the second factor. This action preserves the quaternionic structure and the Hermitian metric. The subgroup $U(R)^\Gamma$ of $\Gamma$-equivariant unitary automorphisms preserves the subspace $K$ and therefore defines a homomorphism $U(R)^\Gamma \ra \om{Sp}(K)$. Here, the subgroup $U(1)$ of scalars is contained in the kernel so the homomorphism descends to a homomorphism $F\coloneqq U(R)^\Gamma/U(1)\ra \om{Sp}(K)$.   
Let  $\mu:K \ra \om{Hom}(\ff,\om{Im}\HH)$ be the unique hyper-K\"{a}hler moment vanishing at $0\in K$. In this situation Kronheimer
identified the hyper-K\"{a}hler quotient. 

\begin{lemma} \cite{Kronheimer89} \label{Kronheimer-Quotient}
There is a homeomorphism
$\mu^{-1}(0)/F \cong \C^2/\Gamma$
that restricts to an isometry 
\[ \frac{\mu^{-1}(0)-\{0\}}{F} \cong \frac{\C^2-\{0\}}{\Gamma}  .\]
\end{lemma}

\begin{remark} \label{Free-Action} A consequence of this is that
$F$ acts freely on $\mu^{-1}(0)-\{0\}$. \end{remark}   

The space $K$ and the group $F$ admit simple descriptions in terms of
the McKay graph $\ovl{\Delta}_\Gamma$ associated with $\Gamma$. Let
$R_0,\cdots,R_m$ be a complete set of representatives for 
$\om{Irr}(\Gamma,\C)$, where we take $R_0=\C$ to be the trivial
representation. The regular representation $R=\C[\Gamma]$ decomposes as
\[ R \cong \bigoplus_{i=0}^m n_iR_i,  \]
where $n_i = \om{dim}_\C R_i$. Recall from subsection \ref{Section-McKay} 
that a matrix
$A$ was defined by $Q\otimes R_i = \bigoplus_j a_{ij}R_j$, or equivalently
$a_{ij}=\om{dim}_\C\om{Hom}_\Gamma(Q\otimes R_i,R_j)$. Furthermore,
the McKay graph $\overline{\Delta}_\Gamma$ was defined by taking the
irreducible representations as vertices and an (unoriented) edge connecting
$R_i$ to $R_j$ precisely when $a_{ij}=a_{ji}=1$. Using this
\begin{align*}
K = \om{Hom}_\Gamma(Q\otimes R,R) &= \bigoplus_{i,j} \om{Hom}_\C (\C^{n_i},\C^{n_j})\otimes \om{Hom}_\Gamma (Q\otimes R_i,R_j) \\
&= \bigoplus_{i\to j \mbox{ in } \overline{\Delta}_\Gamma} 
\om{Hom}(\C^{n_i},\C^{n_j})
\end{align*}
where each edge is repeated twice, once with each orientation, 
and 
\[ F=U(R)^\Gamma/U(1)=U\left(\bigoplus_i n_iR_i\right)^\Gamma/U(1)=\left(\prod_i U(n_i)\right) /U(1) . \]
This has a simple graphical interpretation on the Dynkin graph $\overline{\Delta}_\Gamma$ corresponding to $\Gamma$. Attach the vector space
$\C^{n_i}$ to the vertex corresponding to $R_i$. Then an element of $K$ is an assignment of a linear map $f_{ji}:\C^{n_i}\ra \C^{n_j}$ whenever $a_{ij}=1$.
An element  $g=(g_i)_i\in \prod_i U(n_i)$ acts on $f=(f_{ji})_{i,j}\in K$
by $(g\cdot f)_{ji}=g_jf_{ji}g_i^{-1}$.   

\begin{example} For $\Gamma=O^*$ an element of $K=\om{Hom}_{O^*}(Q\otimes R,R)$ is given by assigning a linear map to each arrow in the following diagram
\[ \begin{tikzcd} {} & {} & {} & \C^2 \arrow[shift left]{d} & {} & {} & {} \\
\C \arrow[shift left]{r} & \C^2 \arrow[shift left]{r} \arrow[shift left]{l} & \C^3 \arrow[shift left]{r} \arrow[shift left]{l} &
\C^4 \arrow[shift left]{r} \arrow[shift left]{l} \arrow[shift left]{u} & \C^3 \arrow[shift left]{l} \arrow[shift left]{r} &
\C^2 \arrow[shift left]{r} \arrow[shift left]{l} & \C \arrow[shift left]{l} \end{tikzcd} \]
\end{example}

The next result is the key step in relating Kronheimer's hyper-K\"{a}hler
quotient to the space occuring in the equivariant ADHM correspondence.  
Recall that an $\om{SU}(2)$-representation $\alpha$ of $\Gamma$ is
either 
\begin{enumerate}[label=(\arabic*), ref=(\arabic*)]
\item of quaternionic type and irreducible as a complex
representation,
\item of complex type and splits $\alpha = \lambda\oplus \lambda^*$
with $\lambda \not\cong \lambda^*$ or
\item of real type and splits $\alpha =2\eta$ for some $\eta$ with $\eta\cong \eta^*$.
\end{enumerate} 
By the vertex or vertices corresponding to $\alpha$ in $\overline{\Delta}_\Gamma$ we will mean the vertex corresponding to $\alpha$ in case $(1)$, the vertices corresponding to $\lambda$ and $\lambda^*$ in case $(2)$ and the vertex corresponding to $\eta$ in case $(3)$. The following is a more refined
version of \cite[Lemma~4.6]{Austin95}. Austin gives no explicit proof,
but states that one may verify it by a case by case analysis. 

\begin{proposition} \label{R(G)-Adjacent-Solution}
Let $\Gamma\subset \om{SU}(2)$ be a finite subgroup
and let $\alpha$ and $\beta$ be $\om{SU}(2)$-representations that are adjacent
in the graph $\MC{S}_\Gamma$ (Definition \ref{Def-Adjacency}). 
Then the solution
$\Har$ of the equation $(2-Q)\Har = \alpha-\beta$ in $R(\Gamma)$ with
minimal $\eps(\Har)=k>0$ is given in the following way.
\begin{enumerate}[label=(\roman*),ref=(\roman*)]
\item Delete the vertex or vertices corresponding to $\beta$ in $\overline{\Delta}_\Gamma$. The resulting graph has one or two components.
Let $\Delta_\Har$ be the component containing $\alpha$.
\item Up to isomorphism there is a unique nontrivial finite subgroup $\Gamma'\subset \om{SU}(2)$ and an isomorphism of graphs 
$\phi\colon\Delta_{\Gamma'}\ra \Delta_\Har$. 
\item Let $S_0,S_1,\cdots,S_r$ be the irreducible representations of
$\Gamma'$, where $S_0$ is the trivial representation. Then
\[ \Har \coloneqq \bigoplus_{j=1}^r (\om{dim}_\C S_j) R_{\phi(j)}  .\]
\end{enumerate}
Furthermore, if $Q'=m_1S_{i_1}\oplus m_2S_{i_2}$ is the isotypical decomposition of the canonical representation of $\Gamma'$, then 
$\alpha=m_1R_{\phi(i_1)}\oplus m_2 R_{\phi(i_2)}$
is the isotypical decomposition of $\alpha$. In particular, $\phi$ takes
the vertex or vertices associated with $Q'$ to the vertex or vertices
corresponding to $\alpha$.   
\end{proposition}

\begin{remark} The isotypical decomposition of the canonical representation
$Q$ of a finite subgroup $\Gamma\subset \om{SU}(2)$ is $Q=Q$ for
$\Gamma$ non-cyclic and $Q=\rho\oplus \rho^*$ for some $1$-dimensional
representation $\rho$ for $\Gamma$ cyclic. \end{remark}

\begin{corollary} \label{R(G)-Sol-Corollary} 
Let $\Gamma\subset \om{SU}(2)$ be a finite subgroup
and let $\alpha\neq \beta$ be $\om{SU}(2)$-representations that
are not adjacent in $\MC{S}_\Gamma$. Let $\alpha=\alpha_0,\alpha_1,\cdots,\alpha_s=\beta$ be the vertices of the minimal edge path connecting
$\alpha$ to $\beta$ in $\MC{S}_\Gamma$ and let
$\Har_i\in R(\Gamma)$ be the minimal positive solution of the equation
$(2-Q)\Har_i = \alpha_i-\alpha_{i+1}$ for $0\leq i\leq s-1$. Then
$\Har \coloneqq \bigoplus_{i=0}^{s-1}\Har_i$ is the minimal positive
solution $(2-Q)\Har = \alpha-\beta$ in $R(\Gamma)$. \end{corollary}
\begin{proof} The calculation
\[ (2-Q)\Har = \sum_{i=0}^{s-1} (2-Q)\Har_i = \sum_{i=0}^{s-1} \alpha_i-\alpha_{i+1} = \alpha_0-\alpha_s = \alpha-\beta  \]
shows that $\Har$ solves the relevant equation.
By Lemma \ref{R(G)-equation} every other solution is given by
$\Har + nR$ for some $n\in \Z$, where $R$ is the regular representation.
To see that $\Har$ is the minimal positive solution, it therefore suffices to show that $\Har-R$ is not positive, i.e., not an actual representation. 
Let $\Delta_i = \Delta_{\Har_i}$ be
the subgraph of $\ovl{\Delta}_\Gamma$ described in the above proposition.
By the construction of these graphs and Lemma \ref{Tree-Lemma} we
conclude that $\Delta_0\subset \Delta_1\subset \cdots \subset \Delta_{s-1}$.
This implies that the representation $\Har$ is only supported on the
vertices of $\Delta_{s-1}$, which by construction does not contain the
vertices corresponding to $\beta = \alpha_s$. Hence, in the isotypical
decomposition $\Har = \bigoplus_{i=0}^n k_iR_i$ there exists some
$i$ with $k_i=0$. We therefore conclude that $\Har-R$ is not an actual
representation. \end{proof}     

Let $\Gamma$, $\alpha$, $\beta$, $\Har$, $\Gamma'$ and $\phi\colon \Delta_{\Gamma'}\ra \Delta_\Har$ be defined as in Proposition 
\ref{R(G)-Adjacent-Solution} and set
\[ N = M^\Gamma= \om{Hom}_\Gamma(\Har\otimes Q,\Har) \oplus \om{Hom}_\Gamma(\Har,E_\infty) \oplus \om{Hom}_\Gamma(E_\infty,\Har), \]
as in Theorem \ref{Equivariant-ADHM}. This space is endowed with
a quaternionic structure and compatible metric preserved by the action
of $G\coloneqq U(\Har)^\Gamma$ as in Example \ref{Example-ADHM-eq}. Note that
$E_\infty=\alpha$.
Let $R'$ and $Q'$ denote the regular and canonical representation
of $\Gamma'$ respectively and set
\[ K = \om{Hom}_{\Gamma'}(Q'\otimes R',R') \;\mbox{ and } \;\;
F = U(R')^{\Gamma'}/U(1)  \]
as in Kronheimer's construction. The following is a more precise version
of \cite[Lemma~4.7]{Austin95}.  

\begin{proposition} In the above situation there is an isometric isomorphism
$f:N\ra K$ of quaternionic vector spaces and an isomorphism of groups
$\tau:G\ra F$ such that $f(g\cdot n) = \tau(g)f(n)$ for all
$n\in N$ and $g\in G$. 
\end{proposition} 
\begin{proof} Let $S$ be the regular representation of $\Gamma'$ and
write $S=S_0\oplus S'$, where $S_0$ denotes the trivial representation. Then 
\[ K = \om{Hom}_{\Gamma'}(Q'\otimes S,S)
= \om{Hom}_{\Gamma'}(Q'\otimes S',S')\oplus \om{Hom}_{\Gamma'}(Q',S')
\oplus \om{Hom}_{\Gamma'}(S',Q')  \]
since $\om{Hom}_{\Gamma'}(Q',S_0)=0$ and $\om{Hom}_{\Gamma'}(Q'\otimes S',S_0)
\cong \om{Hom}_{\Gamma'}(S',Q')$ as $Q'$ is self-dual. 
We will now match these components with the components of $N$. Let
$n_i = \om{dim}_\C S_i$. Recall that the graph $\Delta_{\Gamma'}$ is the
graph obtained from $\overline{\Delta}_{\Gamma'}$ by deleting the
vertex corresponding to $S_0$. Using part $(iii)$ of the above proposition
we can write $\Har = \oplus_{i=1}^r n_i R_{\phi(i)}$ where
$\phi:\Delta_{\Gamma'}\ra \Delta_{\Har}$ is the isomorphism of graphs. 
Then
\begin{align} \label{Decomp1}
 &\om{Hom}_\Gamma(Q\otimes \Har,\Har) =
\bigoplus_{i\to j \mbox{ in } \Delta_{\Gamma'}} \om{Hom}_\C(\C^{n_i},\C^{n_j})\otimes \om{Hom}_\Gamma(Q\otimes R_{\phi(i)},R_{\phi(j)})  \\
& \cong \bigoplus_{i\to j \mbox{ in } \Delta_{\Gamma'}}\om{Hom}_\C(\C^{n_i},\C^{n_j})\otimes \om{Hom}_{\Gamma'}(Q'\otimes S_i,S_j) =\om{Hom}_\Gamma(Q'\otimes S',S'). \nonumber 
\end{align}
To make this an $\HH$-linear isometric isomorphism we specify the
isomorphisms
\begin{equation}\label{Hlinear}
\om{Hom}_\Gamma(Q\otimes R_{\phi(i)},R_{\phi(j)})\cong 
\C\cong \om{Hom}_{\Gamma'}(Q'\otimes S_i,S_j)
\end{equation}
for $i$ adjacent to $j$ in $\Delta_{\Gamma'}$ as follows. 
Choose an arbitrary orientation on the edges in $\Delta_{\Gamma'}$. Then
for each positive arrow $i\to j$ in $\Delta_{\Gamma'}$ we pick basis vectors
$u_{ij}\in \om{Hom}_\Gamma(Q\otimes R_{\phi(i)},R_{\phi(j)})$
and $v_{ij}\in \om{Hom}_{\Gamma'}(Q'\otimes S_i,S_j)$ of unit length.
Then $u_{ji}=Ju_{ij}\in \om{Hom}_\Gamma(Q\otimes R_{\phi(j)},R_{\phi(i)})$
and $v_{ji}=Jv_{ij}\in \om{Hom}_{\Gamma'}(Q'\otimes S_j,S_i)$
are also unit basis vectors. We then define the isomorphisms in
(\ref{Hlinear}) by sending $u_{ij}$ to $v_{ij}$
for each $i\to j$ in $\Delta_{\Gamma'}$. 

Next, by the final part of the above proposition we know that if
$Q'=m_1S_{i_1}\oplus m_2S_{i_2}$ is the isotypical decomposition, then
$\alpha=E_\infty = m_1R_{\phi(i_1)}\oplus m_2R_{\phi(i_2)}$ is the
isotypical decomposition of $\alpha$. We may therefore identify
\begin{align} \label{Decomp2}
&\om{Hom}_{\Gamma'}(S',Q')=\bigoplus_{j=1}^2
\om{Hom}_\C(\C^{n_{i_j}},\C^{m_j})\otimes
\om{Hom}_{\Gamma'}(S_{i_j},S_{i_j})   \\
&\cong \bigoplus_{j=1}^2 \om{Hom}_\C(\C^{n_{i_j}},\C^{m_j})\otimes
\om{Hom}_\Gamma(R_{\phi(i_j)},R_{\phi(i_j)})\cong \om{Hom}_\Gamma(\Har,E_\infty) \nonumber
\end{align}
and similarly $\om{Hom}_{\Gamma'}(Q',S')\cong \om{Hom}_\Gamma(E_\infty,\Har)$.
Once again one should match unit vectors carefully to ensure that
the direct sum of these isomorphisms is an $\HH$-linear isometry.
The direct sum of these isomorphisms gives the required $\HH$-linear
isometry $N\cong K$.

The (orthogonal) decomposition $\Har = \bigoplus_{i=1}^r n_iR_{\phi(i)}$ induces an isomorphism $G = U(\Har)^\Gamma \cong \prod_{i=1}^r U(n_i)$
and this group only acts on the matrix components in the decompositions
\eqref{Decomp1} and \eqref{Decomp2}. In the latter case it only acts
on factors coming from $\Har$. Moreover, 
\[ F \cong \left( \prod_{i=0}^r U(n_i) \right)/U(1)
\cong \prod_{i=1}^r U(n_i)  \]
as $n_0=1$, and the action can again be seen to be on the matrix component
in \eqref{Decomp1} and \eqref{Decomp2}, and in the latter case only on
the factors coming from $S'$. Therefore, the isomorphism
\[ F \cong \prod_{i=1}^r U(n_i) \cong G   \]
does the trick and the proof is complete. 
\end{proof} 

\begin{proposition} \label{Moduli-S4-Prop}
Let $\Gamma\subset \om{SU}(2)$ be a finite subgroup
and let $\alpha$ and $\beta$ be $\om{SU}(2)$-representations
adjacent in the graph $\MC{S}_\Gamma$. Let $E\ra S^4$
be the $\Gamma$-equivariant $\om{SU}(2)$-bundle with
$[E_N]=\alpha$, $[E_S]=\beta$ and minimal $k=c_2(E)[S^4]>0$. Then the moduli
space $\wt{\MC{M}}_{\Gamma,\infty}(E)$ of $\Gamma$-invariant instantons in $E$ framed at $N=\infty$ is given by 
\[ (\C^2-\{0\})/\Gamma'\cong \R\times S^3/\Gamma'   \]
where the finite subgroup $\Gamma'\subset \om{SU}(2)$ is determined as in Proposition \ref{R(G)-Adjacent-Solution}.
\end{proposition}
\begin{proof} We continue to use the notation of the two preceding
propositions. According to the equivariant ADHM correspondence the
moduli space $\wt{\MC{M}}_{\Gamma,\infty}(E)$ is given by
\begin{equation} \label{Space-ADHM}
 \frac{\mu^{-1}(0)\cap U_{reg}\cap M^\Gamma}{U(\Har)^\Gamma},
\end{equation}
where in the above notation $N=M^\Gamma$ and $G=U(\Har)^\Gamma$. 
Let $\mu_N\colon N\ra \fg\otimes \om{Im}\HH$ and $\mu_K\colon K\ra \ff\otimes \om{Im}\HH$ be the unique hyper-K\"{a}hler
moment maps vanishing at $0$ associated with the action of
$G$ on $N$ and $F$ on $K$. Due to the uniqueness of the moment map it
follows that the isomorphisms $f\colon N\ra K$ and $\psi\colon G\ra F$ of the previous proposition induce a homeomorphism $\mu_N^{-1}(0)\cong \mu_K^{-1}(0)$
equivariant along the map $\psi$. The group $F$ acts freely on
$\mu_K^{-1}(0)-\{0\}$ (see Remark \ref{Free-Action}), so by Lemma
\ref{Kronheimer-Quotient} we obtain isometries of hyper-K\"{a}hler $4$-manifolds
\[ \frac{\mu_N^{-1}(0)-\{0\}}{G} \cong \frac{\mu_K^{-1}(0)-\{0\}}{F}
\cong \frac{\C^2-\{0\}}{\Gamma'}  .\]
To complete the proof we have to show that the left hand side
coincides with the space in (\ref{Space-ADHM}). This will be achieved by
proving that
\begin{enumerate}[label=(\arabic*),ref=(\arabic*)]
\item $\mu_N$ is the restriction of $\mu$ to $N=M^\Gamma\subset M$,
so that $\mu_N^{-1}(0) = \mu^{-1}(0)\cap N$ and
\item $U_{reg}\cap \mu_N^{-1}(0)=\mu_N^{-1}(0)-\{0\}$.
\end{enumerate}
We may decompose an element $m\in M^\Gamma$ as
\[ m = (\tau_1,\tau_2,\pi,\sigma)\in \om{Hom}_\Gamma(Q\otimes \Har,\Har)\oplus
\om{Hom}_\Gamma(\Har,E_\infty)\oplus \om{Hom}_\Gamma(E_\infty,\Har), \]
where we regard $Q=\C^2$, as earlier. Here $\pi$ and $\sigma$ are $\Gamma$-equivariant, while $\tau_1,\tau_2$ satisfy the properties
(see \cite[Eq.~2.2]{Kronheimer89})
\begin{equation} \label{Tau-Action}
\begin{array}{c} \gamma^{-1}\tau_1 \gamma = a\tau_1+b\tau_2 \\
\gamma^{-1}\tau_2 \gamma = -\overline{b}\tau_1+\overline{a}\tau_2
                    \end{array} 
\;\; \mbox{ for all } \;\;
\gamma = \left( \begin{array}{cc} a & b \\ -\overline{b} & \overline{a}
\end{array} \right) \in \Gamma \subset \om{SU}(2),
\end{equation} 
where on the left hand side $\gamma$ is regarded as an isometry $\Har\ra \Har$.
Using the above description and the moment map equations \eqref{ADHM-eq}
\begin{align*}
\mu_\C(\tau_1,\tau_2,\pi,\sigma)=&
[\tau_1,\tau_2]+\sigma \pi \\
\mu_\R(\tau_1,\tau_2,\pi,\sigma)=&[\tau_1,\tau_1^*]+[\tau_2,\tau_2^*]
+\sigma \sigma^*-\pi^*\pi \end{align*}
one easily verifies that $\mu_\R$ and $\mu_\C$ map $N=M^\Gamma$ into
$\mathfrak{u}(\Har)^\Gamma$ and $\om{End}_\Gamma(\Har)$, respectively. By the
uniqueness of the moment map it follows that $\mu_N$ is the restriction
of $\mu$ to $N=M^\Gamma$. This proves the first point.

Before we proceed to the second point recall that $U_{reg}\subset M$ was defined to be the set of $m=(\tau_1,\tau_2,\pi,\sigma)$ for which 
\begin{align*} A_z &=(\tau_1-z_1,\tau_2-z_2,\pi)^t\colon \Har\oplus \Har\oplus E_\infty \ra \Har \\
B_z&=(-(\tau_2-z_2),\tau_1-z_1,\sigma)\colon \Har\oplus\Har\oplus E_\infty \ra \Har
\end{align*}
were injective and surjective, respectively, for all $z=(z_1,z_2)\in \C^2$
(see \eqref{ADHM-sequence}). Observe that the equation $\mu_\R(m)=0$ may be rewritten as
\[ \tau_1\tau_1^*+\tau_2\tau_2^*+\sigma\sigma^*=\tau_1^*\tau_1+\tau_2^*\tau_2+\pi^*\pi, \]
or equivalently $B_0B_0^*=A_0^*A_0$. From this it follows that
\[ \om{Ker}(A_0)=\om{Ker}(A_0^*A_0)=\om{Ker}(B_0B_0^*)=\om{Ker}(B_0^*). \]
Hence, $A_0$ is injective if and only if $B_0$ is surjective.
In view of the easily verified fact that for any $(z_1,z_2)\in \C^2$ we have
$\mu(\tau_1-z_1,\tau_2-z_2,\pi,\sigma)=\mu(\tau_1,\tau_2,\pi,\sigma)$, the
above argument applies to show that if $m\in \mu^{-1}(0)$, then for
any $z\in \C^2$, $A_z$ is injective if and only if $B_z$ is surjective.   

We may now verify the second claim as follows. As $U_{reg}\cap \mu_N^{-1}(0)\subset \mu_N^{-1}(0)-\{0\}$, it suffices to show that
if $m=(\tau_1,\tau_2,\pi,\sigma) \in \mu_N^{-1}(0)-U_{reg}$ then
$m=0$. To this end assume that $m\in \mu_N^{-1}(0)-U_{reg}$. In view of the
above considerations this means that there exists $z= (z_1,z_2)\in \C^2$
and $h\neq 0\in \Har$ such that $A_z(h)=0$. Then $h$ belongs to $\om{Ker}\pi$ and is a common eigenvector for 
$\tau_1,\tau_2$ with eigenvalues $z_1,z_2$ respectively. Using the
equations in \eqref{Tau-Action} we deduce that $\gamma\cdot h$ is a common eigenvector for $\tau_1,\tau_2$ with eigenvalues
\[ \left( \begin{array}{c} z_1^\gamma \\ z_2^\gamma \end{array} \right)\coloneqq
\left( \begin{array}{cc} a & b \\ -\overline{b} & \overline{a}
\end{array} \right) \left( \begin{array}{c} z_1 \\ z_2 \end{array} \right)
\;\; \mbox{ for all } \;\; 
\gamma = \left( \begin{array}{cc} a & b \\ -\overline{b} & \overline{a}
\end{array} \right) \in \Gamma \subset \om{SU}(2)  .\]
Let $L \coloneqq \om{Span}_\C \{\gamma h:\gamma \in \Gamma\}$. This is a nontrivial $\Gamma$-invariant subspace of $\Har$ satisfying
$\tau_i(L)\subset L$ for $i=1,2$ and $\pi(L)=0$. Moreover, since
$\gamma\cdot h \in \om{Ker}(A_{(z_1^\gamma,z_2^\gamma)})=\om{Ker}(B_{(z_1^\gamma,z_2^\gamma)}^*)$, it follows that
$\tau_i^*(\gamma h)=\overline{z_i^\gamma} h$ for $i=1,2$ and
$\sigma^*(\gamma h) = 0$. Therefore, $\tau_i^*(L)\subset L$ for $i=1,2$
and $\sigma^*(L)=0$ as well. Hence, 
$\tau_i(L^\perp)\subset L^\perp$ for $i=1,2$ and $\sigma(E_\infty)\subset L^\perp$.

To summarize, we have proved that there is an orthogonal splitting of 
$\Gamma$-modules $\Har = L\oplus L^\perp$ such that $\tau_i = \tau_i'\oplus \tau_i''$ preserves the splitting for $i=1,2,$, $\pi(L)=0$ and $\om{Im}(\sigma)\subset L^\perp$. Pick $\lambda\neq 1\in U(1)$ and define
$g\in U(\Har)^\Gamma$ by setting $g = \lambda \om{id}$ on $L$ and
$g = \om{id}$ on $L^\perp$. Then $(g\tau_1g^{-1},g\tau_2g^{-1},\pi g^{-1},g\sigma)=(\tau_1,\tau_2,\pi,\sigma)$, so as $g\neq 1$ and
$U(\Har)^\Gamma$ acts freely on $\mu_N^{-1}(0)-\{0\}$ we conclude that
$m = 0$ as desired.  
\end{proof}

\begin{proof}[Proof of Proposition \ref{Moduli-Result}] Given an
adjacent pair $\alpha,\beta\in \MC{C}$ let $E\ra S^4$ be the unique
$\Gamma$-equivariant $\om{SU}(2)$-bundle with $[E_N]=\alpha$ and $[E_S]=\beta$
and minimal $c_2(E)[S^4]>0$. Then by the above proposition
the moduli space $\wt{\MC{M}}_{\Gamma,\infty}(E)$ of $\Gamma$-invariant instantons framed at $\infty =N$ may be identified with $\R\times Y_{\Gamma'}$ for
some other finite subgroup $\Gamma'\subset \om{SU}(2)$. Using the correspondence between invariant instantons in $E\ra S^4$ and instantons
over $\R\times \overline{Y}_\Gamma$, taking the framing and orientation
reversal into account, one obtains
\[\R\times Y_{\Gamma'} \cong e_{-}^{-1}(*) \subset \wt{\MC{M}}_z(\alpha,\beta), \]
where $e_-\colon \wt{\MC{M}}_z(\alpha,\beta)\ra \alpha$ is the end-point map and $z$ is the homotopy class from $\alpha$ to $\beta$ corresponding to the relative Chern number $\hat{c}_2=c_2(E)/|\Gamma |$. According to 
\cite{Austin95} the natural translation action on the left hand side corresponds to the translation action on the right hand side. Hence,
\[ Y_{\Gamma'}\cong e_{-}^{-1}(*) \subset \wt{\MC{M}}_z^0(\alpha,\beta).  \]
Since the end-point map is $\om{SO}(3)$-equivariant it follows that
$\wt{\MC{M}}_z^0(\alpha,\beta) = e_{-}^{-1}(*)\cdot \om{SO}(3)$ is compact.
Therefore, $\wt{\MC{M}}_z^0(\alpha,\beta)=\ovl{\MC{M}}_z(\alpha,\beta)$,
and as $\om{gr}_z(\alpha,\beta)=\om{dim}(e_{-}^{-1}(*))+1=4$, we conclude
that $z$ is indeed the unique homotopy class with $-2\leq \om{gr}_z(\alpha,\beta)\leq 5$. 

To complete the proof we have to verify that $\ovl{\MC{M}}_z(\alpha,\beta)$
is empty  when $\alpha$ and $\beta$ are not adjacent in $\MC{S}_\Gamma$. Let
$\alpha = \alpha_0,\alpha_1,\cdots,\alpha_s=\beta$ be the vertices of
the minimal edge path connecting $\alpha$ to $\beta$ in $\MC{S}_\Gamma$.
Let $z_i$ denote the unique homotopy class connecting $\alpha_i$ to
$\alpha_{i+1}$ with $\om{gr}_{z_i}(\alpha_i,\alpha_{i+1})=4$ for
each $0\leq i<s$. Let $\Har_i$ be the minimal positive solution of
$(2-Q)\Har_i = \alpha_i-\alpha_{i+1}$ for $0\leq i<s$. By Corollary
\ref{R(G)-Sol-Corollary} the minimal positive solution of
$(2-Q)\Har = \alpha-\beta$ is given by $\Har = \bigoplus_{i=0}^{s-1}\Har_i$.
Put $k = \eps(\Har)$ and let $E^l\ra S^4$ denote the $\Gamma$-equivariant
$\om{SU}(2)$-bundle with invariants $([E^l_N],[E^l_S],c_2(E^l)[S^4])=(\alpha,\beta,k+|\Gamma|l)$ for $l\in \Z$. The representation $\Har_l$ used in 
the description of the moduli space $\wt{\MC{M}}_{\Gamma,\infty}(E^l)$
in the equivariant ADHM correspondence is then given by
$\Har_l = \Har + lR$, where $R$ is the regular representation. The dimension
of the moduli space increases linearly with $l$ so, as $\Har_0 = \Har$ is the
minimal positive solution of $(2-Q)\Har = \alpha-\beta$, we conclude that
$\wt{\MC{M}}_{\Gamma,\infty}(E^l)$ is empty for $l<0$. This space
corresponds to the fiber $e_-^{-1}(*)\subset \wt{\MC{M}}_w(\alpha,\beta)$ where
$w=z_0*z_1*\cdots *z_{s-1}$. By Theorem \ref{Moduli-Theorem} part $(a)$ we
have
\[ \om{gr}_w(\alpha,\beta)=\sum_{i=0}^{s-1} \om{gr}_{z_i}(\alpha_i,\alpha_{i+1}) = 4s >4 ,\]
since by assumption $s>1$. Consequently, the fiber  $e_-^{-1}(*)\subset \wt{\MC{M}}_z(\alpha,\beta)$ where $z$ is the unique homotopy class with $\om{gr}_z(\alpha,\beta)=0,4$, corresponds to $\wt{\MC{M}}_{\Gamma,\infty}(E^l)$
for some $l<0$ and must therefore be empty. This implies that the whole
moduli space must be empty and the proof is complete. 
\end{proof}

\subsection{The Structure of 
\texorpdfstring{$DCI(\ovl{Y}_\Gamma)$}{DCIYG}}
In this section we determine the complexes $DCI(\ovl{Y}_\Gamma)$
explicitly. Let $\Gamma\subset \om{SU}(2)$ be a finite subgroup and
let $\MC{C}$ denote the set of gauge equivalence classes of flat connections in the trivial $\om{SU}(2)$-bundle over $\ovl{Y}_\Gamma$. In this section
we make the convention that the relative and absolute gradings 
\[ \om{gr}\colon \MC{C}\times \MC{C}\ra \Z/8 \;\; \mbox{ and } \;\;
j\colon \MC{C}\ra \Z/8,  \]
introduced after Theorem \ref{Moduli-Theorem}, are always defined with respect to the orientation $\ovl{Y}_\Gamma$ of $S^3/\Gamma$. Furthermore, 
the vertices of the graph $\MC{S}_\Gamma$, introduced in Definition \ref{Def-Adjacency}, i.e., the set of isomorphism classes of $\om{SU}(2)$-representations of $\Gamma$, will be identified freely with the set
$\MC{C}$.  

\begin{lemma} \label{Indexing-Lemma}
For any pair $\alpha,\beta \in \MC{C}$ that are adjacent in $\MC{S}_\Gamma$ 
it holds true that $\om{gr}(\alpha,\beta)=4$. For each $\alpha\in \MC{C}$ let
$p(\alpha)$ denote the number of edges in the minimal edge path connecting the
trivial representation $\theta$ to $\alpha$. Then $j(\alpha)\equiv 4p(\alpha)\Mod{8}$ for all $\alpha\in \MC{C}$.
\end{lemma} 
\begin{proof} Suppose that $\alpha$ and $\beta$ are adjacent in $\MC{S}_\Gamma$ and let $z$ be the unique homotopy class with $-2\leq \om{gr}_z(\alpha,\beta)\leq 5$. Then according to Theorem \ref{Moduli-Result} the fiber of $e_-\colon \ovl{\MC{M}}_z(\alpha,\beta)\ra \alpha$ has dimension $3$. 
Using the dimension formula \eqref{Dim-Formula} we deduce that
\[ \om{gr}(\alpha,\beta) = \om{dim} \overline{\MC{M}}_{z}(\alpha,\beta)
-\om{dim}(\alpha)+1 = \om{dim}e_-^{-1}(*)+1 = 4  \]
proving the first assertion. For the second let $\alpha\in \MC{C}$
be arbitrary and let $\theta = \alpha_0,\alpha_1,\cdots,\alpha_s = \alpha$
denote the vertices in the minimal edge path connecting $\theta$ to
$\alpha$ in $\MC{S}_\Gamma$. Then $p(\alpha)=s$ and by the additivity of the grading
\[ j(\alpha)= \om{gr}(\alpha,\theta) = \sum_{i=0}^{s-1} \om{gr}(\alpha_{i+1},\alpha_{i})=4s \in 4\Z/8 . \]
Consequently, $j(\alpha)\equiv 4p(\alpha) \Mod{8}$ as required. \end{proof}

To determine the differentials in $DCI(\overline{Y}_\Gamma)$ we need the following result, which is contained in the proof of \cite[Lemma~5.1]{Austin95}.

\begin{lemma} \label{Degree-Austin}
Let $\alpha,\beta \in \MC{C}$ be adjacent in $\MC{S}_\Gamma$ with
$\beta$ irreducible, let $\Har\in R(\Gamma)$ be the solution of
$(2-Q)\Har = \alpha-\beta$
given by Proposition \ref{R(G)-Adjacent-Solution} and let $\Gamma'\subset \om{SU}(2)$ be the corresponding finite subgroup.  
Let $z$ be the unique homotopy class with
$\om{gr}_z(\alpha,\beta)=4$. Then the degree of the map
\[ e_{-}^{-1}(*)=Y_{\Gamma'} \ra \beta = \om{SO}(3)  \]
is up to a sign given by $2\dim_{\C}\Har/|\Gamma'|$.
\end{lemma}  

\begin{remark} We are not aware of a procedure to pin down the orientation
of $Y_{\Gamma'}$ in the above statement. For this reason we are only
able to determine the degree up to a sign. As we will see, this inaccuracy
will not affect our calculations in any relevant way. \end{remark} 

For $i=0,1$ let $\MC{C}^i = \{\alpha\in \MC{C}:j(\alpha)\equiv 4i\mod{8} \}$
such that $\MC{C}=\MC{C}^0\cup \MC{C}^i$ is a disjoint union. Let
us fix generators $b_\alpha\in H_0(\alpha)$ for all $\alpha\in \MC{C}$,
and $t_\alpha \in H_2(\alpha)$ for all reducible $\alpha$ and
$t_\alpha:=b_\alpha u \in H_3(\alpha)$ for all irreducible $\alpha$.   

\begin{definition} \label{n-alpha-beta-Def}
Let $\MC{C}^{irr}\subset \MC{C}$ be the subset of
irreducibles. For each adjacent pair $\alpha,\beta\in \MC{C}$ with
$\beta\in \MC{C}^{irr}$ define $n_{\beta\alpha}$ to be the integer
$2\om{dim}_C\Har /|\Gamma '|$ where $\Har\in R(\Gamma)$ is the minimal
positive solution of the equation $(2-Q)\Har = \alpha-\beta$. If
$\alpha$ and $\beta$ are not adjacent in 
$\MC{S}_\Gamma$ we define $n_{\beta\alpha}=0$.
\end{definition} 

\begin{theorem} \label{Structure-DCI}
Let $\Gamma\subset \om{SU}(2)$ be a finite subgroup.
Then the multicomplex $(DCI(\overline{Y}_\Gamma)_{*,*},\{\dd^r\}_{r\geq 0})$
is given by
\[ DCI_{8s,*}=\bigoplus_{\alpha\in \MC{C}^0} H_*(\alpha) \mbox{ and }
   DCI_{8s+4,*}=\bigoplus_{\beta\in \MC{C}^1} H_*(\beta)  \]
for all $s\in \Z$ and $DCI_{s,*}=0$ otherwise. The differentials
are given by $\dd^r=0$ for $r\neq 4$ and 
$\dd^4\colon DCI_{4s,0}\ra DCI_{4(s-1),3}$ is determined on generators by
\[ \dd^4(b_\beta) = \sum_{\alpha\in \MC{C}^{irr}} n_{\alpha\beta} t_{\alpha} \]
where the integers $n_{\alpha\beta}$ are defined above. 
\end{theorem}
\begin{proof} The description of $DCI(Y,E)$ as a graded module is a simple
consequence of the definition and Lemma \ref{Indexing-Lemma}. By the final
part of Theorem \ref{Moduli-Result} the moduli space 
$\ovl{\MC{M}}_z(\alpha,\beta)$, where $z$ is the homotopy class with
$-2\leq \om{gr}_z(\alpha,\beta)\leq 5$, is empty unless $\alpha$ and
$\beta$ are adjacent in $\MC{S}_\Gamma$. The
statement concerning the differentials therefore follows from
Lemma \ref{Differential-Degree-Lemma} and Lemma \ref{Degree-Austin}.  \end{proof} 

We will end this section by giving a few applications of Proposition
\ref{R(G)-Adjacent-Solution} and show how we can easily extract both
the grading and the integers $n_{\alpha\beta}$ provided we know where
the one dimensional quaternionic representations are placed in the
graph $\overline{\Delta}_{\Gamma}$. All the necessary information is
contained in Appendix $A$. 

\begin{example}($\Gamma=O^*$) There are four $1$-dimensional quaternionic
representations; two of real type $\theta,\eta$ and two of quaternionic
type $\alpha$, $\beta$. The graph $\overline{\Delta}_{O^*}$ is shown
below.

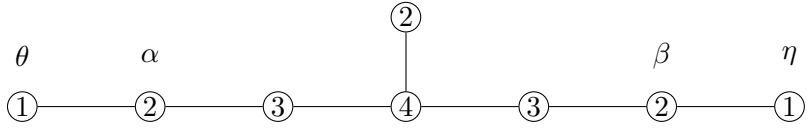
\begin{center}
\begin{tikzpicture}[inner sep=0.8, scale = 1.7]
\node (1) at (0,0) [circle,draw] {$1$};
\node (2) at (1,0) [circle,draw] {$2$};
\node (3) at (2,0) [circle,draw] {$3$};
\node (4) at (3,0) [circle,draw] {$4$};
\node (5) at (4,0) [circle,draw] {$3$};
\node (6) at (5,0) [circle,draw] {$2$};
\node (7) at (6,0) [circle,draw] {$1$};
\node (8) at (3,0.7) [circle,draw] {$2$};

\node at (0,0.4) {$\theta$}; \node at (1,0.4) {$\alpha$};
\node at (5,0.4) {$\beta$}; \node at (6,0.4) {$\eta$}; 
\path
(1) edge (2)
(2) edge (3)
(3) edge (4)
(4) edge (5)
(5) edge (6)
(6) edge (7)
(4) edge (8);
\end{tikzpicture}
\end{center}
The integers are the dimensions of the corresponding irreducible complex representations and we have labeled the vertices corresponding to the $1$-dimensional quaternionic representation. Here there are no representations
of complex type, in particular $\overline{\Delta}_{\Gamma}=\overline{\Delta}_\Gamma/(\iota)$. One may immediately conclude that the
indexing is given by $j(\theta)=j(\beta)=0$ and $j(\alpha)=j(\eta)=4$.
To determine the integer $n_{\beta\eta}$ we need to solve the equation
$(2-Q)\Har = \eta-\beta$ in $R(O^*)$. According to the recipe given
in Proposition \ref{R(G)-Adjacent-Solution} we first delete the vertex
corresponding to $\beta$ and let $\Delta_{\Har}$ be the component containing
$\eta$; that is, $\Delta_\Har$ is just the single vertex corresponding to
$\eta$. We then recognize this graph as $\Delta_{C_2}$, from which we deduce that $\Gamma'=C_2$ and $\Har=\frac12 \eta$. Therefore, $n_{\beta\eta}=2\om{dim}_\C \Har /|\Gamma'| = 2/2=1$. The situation with $\theta$ and $\alpha$ is
completely symmetric so $n_{\alpha\theta}=1$ as well.

The case of $n_{\beta\alpha}$ is more interesting.
If we delete $\alpha$, the component $\Delta_\Har$ containing $\beta$ is given by
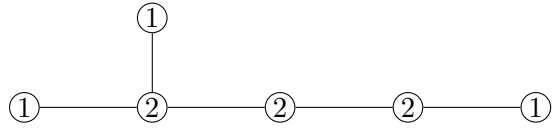
\begin{center}
\begin{tikzpicture}[inner sep=0.8, scale = 1.7]
\node (3) at (2,0) [circle,draw] {$1$};
\node (4) at (3,0) [circle,draw] {$2$};
\node (5) at (4,0) [circle,draw] {$2$};
\node (6) at (5,0) [circle,draw] {$2$};
\node (7) at (6,0) [circle,draw] {$1$};
\node (8) at (3,0.7) [circle,draw] {$1$};

\path

(3) edge (4)
(4) edge (5)
(5) edge (6)
(6) edge (7)
(4) edge (8);
\end{tikzpicture}
\end{center}
We recognize this as the Dynkin graph $D_6$ which corresponds to the group
$D_4^*$ of order $16$. We find the integral weights of $\Har$, as given in the above graph, by knowledge of the dimensions of the irreducible representations of $D_4^*$ (see Appendix $A.2$). From this we calculate  
\[ \om{dim}_\C\Har=3\cdot 1+4\cdot 2+2\cdot 1+3\cdot 2+2\cdot 2+1\cdot 1=24 \]
and hence $n_{\beta\alpha}=2\cdot 24/16 = 3$. By symmetry
$n_{\alpha\beta}=3$ as well. 
\end{example} 

The calculation of $n_{\alpha\theta}$ and $n_{\beta\eta}$ in the above example
generalizes.

\begin{lemma} Let $\Gamma\subset \om{SU}(2)$ be a finite subgroup and let
$Q$ be the canonical representation. Suppose $\rho$ is a $1$-dimensional
complex representation of real type and that $\rho\otimes Q$ is irreducible
so that $\eta=2\rho$ and $\beta = \rho\otimes Q$
correspond to a pair of distinct vertices in $\MC{S}_\Gamma$. Then
$n_{\beta\eta}=1$. \end{lemma}
\begin{proof} We trivially have $(2-Q)\rho = 2\rho -Q\otimes \rho = \eta-\beta$
so that $\Har = \rho$. This means that $\Delta_\Har$ consists of a single
vertex and that the associated subgroup is $\Gamma'=C_2$. It follows that
$n_{\beta\eta}=2\om{dim}_\C \Har/|\Gamma'|=2/2=1$. \end{proof}

We also consider the most involved case $\Gamma=D_n^*$. The $1$-dimensional
quaternionic representations and the graph $\overline{\Delta}_{D_n^*}$
are given in Appendix $A.2$. For $n$ even we label the quaternionic representations by
\[ \theta, \eta_1,\eta_2,\eta_3,\alpha_1,\alpha_2,\cdots,\alpha_{n/2} \]
and for $n$ odd by
\[ \theta,\eta,\alpha_1,\alpha_2,\cdots,\alpha_{(n-1)/2},\lambda . \]
Here $\theta, \eta$ and the $\eta_i$ are fully reducible, the $\alpha_i$
are irreducible and $\lambda$ is reducible.  

\begin{lemma} In the above situation we have $n_{\alpha_i,\alpha_{i+1}}=n_{\alpha_{i+1},\alpha_i}=2$ for all $i$. \end{lemma}
\begin{proof} 
Due to the symmetry in the graph $\overline{\Delta}_{D_n^*}$ it suffices to
show that $n_{\alpha_{i+1},\alpha_i}=2$ for all $i$. We therefore have
to solve $(2-Q)\Har = \alpha_i-\alpha_{i+1}$ for $\Har$. The relevant
portion of the graph $\overline{\Delta}_{D_n^*}$ is shown below
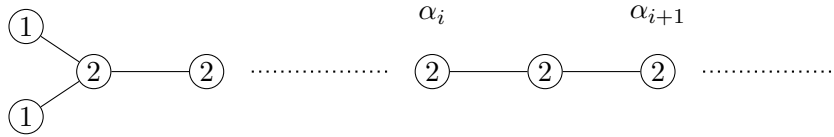
\begin{center}
\begin{tikzpicture}[inner sep=0.5mm, scale=1.5]
\node (1a) at (-0.6,0.4) [circle, draw] {$1$};
\node (1b) at (-0.6,-0.4) [circle, draw] {$1$};
\node (2) at (0,0) [circle, draw] {$2$};
\node (3) at (1,0) [circle, draw] {$2$};
\node (4) at (3,0) [circle, draw] {$2$};
\node (5) at (4,0) [circle, draw] {$2$};
\node (6) at (5,0) [circle, draw] {$2$};
\node at (3,0.5) {$\alpha_{i}$};
\node at (5,0.5) {$\alpha_{i+1}$};

\path
(1a) edge (2)
(1b) edge (2)
(2) edge (3)
(4) edge (5)
(5) edge (6);
\draw[dotted, thick] (1.4,0)--(2.6,0);
\draw[dotted,thick] (5.4,0)--(6.6,0);
\end{tikzpicture}
\end{center}
Note that $\alpha_i$ is the $2i-1$'th vertex labeled with $2$ from left
to right. Following the description in Proposition \ref{R(G)-Adjacent-Solution}
we delete the vertex corresponding to $\alpha_{i+1}$ and let
$\Delta_{\Har}$ be the component containing $\alpha_i$. This graph
has a total of $2(i+1)$ vertices so we recognize $\Delta_{\Har}=D_{2(i+1)}=\Delta_{D^*_{2i}}$. Thus $\Gamma'=D_{2i}^*$ of order $8i$. The weights of $\Har$ are therefore given by attaching the integer $2$ to each internal 
vertex (of degree $\geq 2$) and $1$ to the rest. This yields
\[ n_{\alpha_{i+1}\alpha_i} =2\om{dim}_\C \Har /|D_{2i}^*|
= 2(1+1+4(2i-1)+2)/8i=2  \]
as required.   
\end{proof} 
 
We note that for cyclic groups there are no irreducible 
$\om{SU}(2)$-representations, from which it follows that 
$DCI(\overline{Y}_{C_m})$ has trivial differential for each $m$.
There are now only a finite number of $n_{\alpha\beta}$ left to calculate. We state the calculations
of these below without proof and trust that the
interested reader will verify these for themselves.

For a finite subgroup $\Gamma\subset \om{SU}(2)$ we attach labels
to the edges in the graph $\MC{S}_\Gamma$ to express all the information
needed to set up $DCI(\overline{Y}_\Gamma)$. If $\alpha$ and $\beta$
are adjacent and both irreducible we attach the symbol $(n_{\alpha\beta}|n_{\beta\alpha})$ to the edge connecting $\alpha$ to $\beta$ such that
$n_{\alpha\beta}$ is closest to $\alpha$. If $\alpha$
and $\beta$ are adjacent and only $\alpha$ is irreducible we attach
$n_{\alpha\beta}$ to the edge.  

\begin{proposition} \label{Differential-Graphs}
The graphs $\MC{S}_\Gamma$ with labels for the finite
subgroups $\Gamma\subset \om{SU}(2)$ are given below.
\begin{center}
\begin{tikzpicture}
\node (1) at (0,0) {$\theta$};
\node (2) at (1,0) {$\alpha$};
\node (3) at (2,0) {$\beta$};
\node (4) at (0.5,-0.4) {$1$};
\node (5) at (1.5,-0.4) {$(3|4)$};

\draw (1)--(2)--(3);

\node (6) at (4,0) {$\theta$};
\node (7) at (5,0) {$\alpha$};
\node (8) at (6,0) {$\beta$};
\node (9) at (7,0) {$\eta$};
\node (10) at (4.5,-0.4) {$1$};
\node (11) at (5.5,-0.4) {$(3|3)$};
\node (12) at (6.5,-0.4) {$1$};

\draw (6)--(7)--(8)--(9);

\node (13) at (9,0) {$\theta$};
\node (14) at (10,0) {$\alpha$};
\node (15) at (11,0) {$\lambda$};
\node (16) at (9.5,-0.4) {$1$};
\node (17) at (10.5,-0.4) {$3$};

\draw (13)--(14)--(15);

\node (18) at (1,0.6) {$\MC{S}_{I^*}$};
\node (19) at (5.5,0.6) {$\MC{S}_{O^*}$};
\node (20) at (10,0.6) {$\MC{S}_{T^*}$};
\end{tikzpicture}
\end{center}

\begin{center}
\begin{tikzpicture}[scale = 0.9]
\node (1) at (0,0) {$\theta$};
\node (2) at (1,0) {$\lambda_1$};
\node (3) at (2,0) {$\lambda_2$};
\node (4) at (3.7,0) {$\lambda_{m-1}$};
\node (5) at (5,0) {$\eta$};
\draw (1)--(2)--(3); \draw (4)--(5);
\draw[dotted, thick] (3)--(4);

\node (6) at (7,0) {$\theta$};
\node (7) at (8,0) {$\lambda_1$};
\node (8) at (9,0) {$\lambda_2$};
\node (9) at (10.7,0) {$\lambda_{m-1}$};
\node (10) at (12,0) {$\lambda_m$};
\draw (6)--(7)--(8); \draw (9)--(10);
\draw[dotted, thick] (8)--(9);

\node (11) at (2.5,0.6) {$\MC{S}_{C_{2m}}$};
\node (12) at (9.5,0.6) {$\MC{S}_{C_{2m+1}}$};
\end{tikzpicture}
\end{center} 

\begin{center}
\begin{tikzpicture}
\node (27) at (2,0.6) {$\MC{S}_{D^*_{2m}}$};
\node (1) at (0,0) {$\alpha_1$};
\node (2) at (1,0) {$\alpha_2$};
\node (3) at (2.6,0) {$\alpha_{m-1}$};
\node (4) at (4,0) {$\alpha_m$};
\node (5) at (-1,0.4) {$\theta$};
\node (6) at (-1,-0.4) {$\eta_1$};
\node (7) at (5,0.4) {$\eta_2$};
\node (8) at (5,-0.4) {$\eta_3$};

\node (9) at (0.5,-0.4) {$(2|2)$};
\node (10) at (3.4,-0.4) {$(2|2)$};
\node (11) at (-0.4,0.4) {$1$};
\node (12) at (-0.4,-0.4) {$1$};
\node (13) at (4.5,0.4) {$1$};
\node (14) at (4.5,-0.4) {$1$};

\draw (5)--(1)--(2); \draw (6)--(1);
\draw (3)--(4)--(7); \draw (4)--(8);
\draw[dotted,thick] (2)--(3);
\end{tikzpicture}
\end{center} 

\begin{center}
\begin{tikzpicture}
\node (28) at (10,0.6) {$\MC{S}_{D^*_{2m+1}}$};

\node (15) at (8,0) {$\alpha_1$};
\node (16) at (9,0) {$\alpha_2$};
\node (17) at (10.6,0) {$\alpha_{m-1}$};
\node (18) at (12,0) {$\alpha_m$};
\node (19) at (13,0) {$\lambda$};
\node (20) at (7,0.4) {$\theta$};
\node (21) at (7,-0.4) {$\eta$};

\node (22) at (8.5,-0.4) {$(2|2)$};
\node (23) at (11.4,-0.4) {$(2|2)$};
\node (24) at (12.5,-0.4) {$2$};
\node (25) at (7.6,0.4) {$1$};
\node (26) at (7.6,-0.4) {$1$};

\draw (16)--(15)--(21); \draw (15)--(20);
\draw (17)--(18)--(19);
\draw[dotted, thick] (16)--(17);
\end{tikzpicture}
\end{center}
In all cases $\theta$ denotes the trivial representation. The letters
$\alpha$ and $\beta$ are reserved for representation of quaternionic type,
$\lambda$ for representations of complex type and $\eta$ for representations
of real type.   
\end{proposition}

The notation for the $1$-dimensional quaternionic representations is
compatible with the notation in Appendix $A$.

\begin{example} \label{TOI-Multicomplexes} 
The multicomplexes $(DCI(\overline{Y}_\Gamma)_{*,*},\dd^4)$
for $\Gamma = T^*,O^*,I^*$ are shown below for $(s,t)$ with 
$0\leq s\leq 8$ and $0\leq t\leq 3$. 
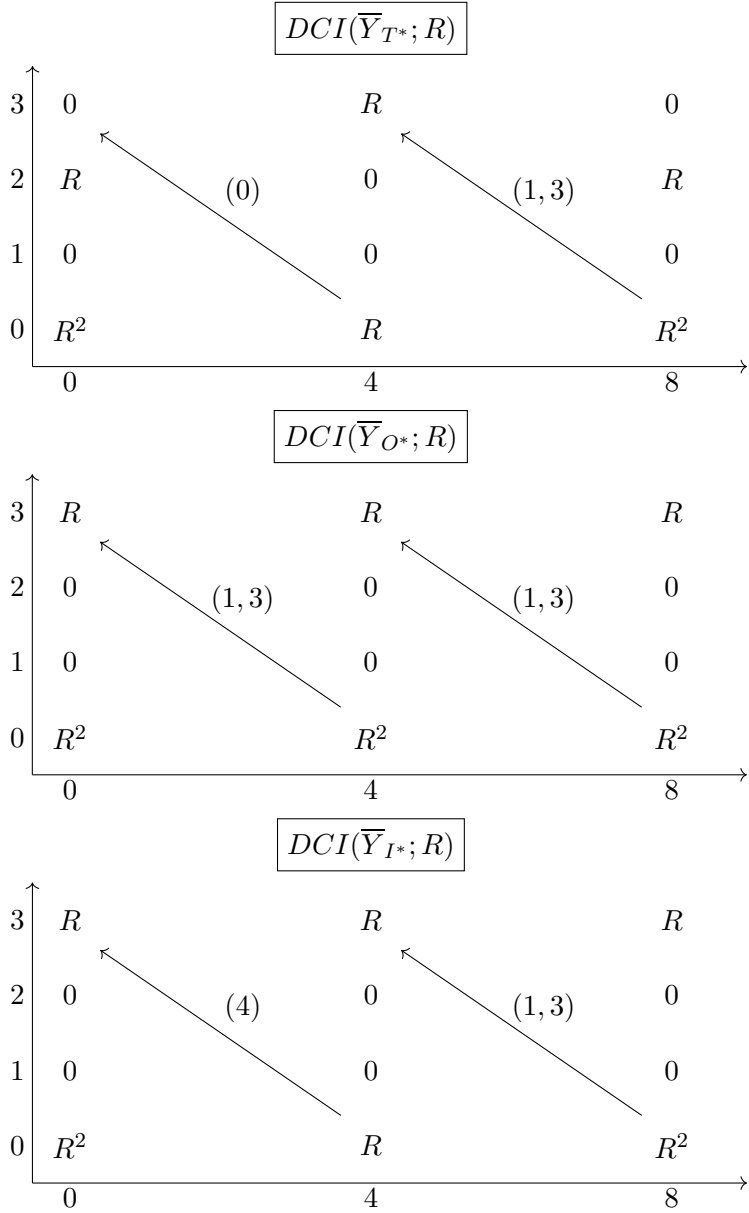
\begin{center}
\begin{tikzpicture}
\draw[->] (-10.5,-0.5)--(-1,-0.5);
\draw[->] (-10.5,-0.5)--(-10.5,3.5);
\node at (-10.7,1) {$1$}; \node at (-10.7,2) {$2$};
\node at (-10.7,3) {$3$}; \node at (-10.7,0) {$0$};
\node at (-10,0) {$R^2$}; \node at (-10,2) {$R$};
\node at (-10,1) {$0$}; \node at (-10,3) {$0$};
\node at (-10,-0.7) {$0$};

\node at (-6,0) {$R$}; \node at (-6,2) {$0$};
\node at (-6,1) {$0$}; \node at (-6,3) {$R$};
\node at (-6,-0.7) {$4$}; 

\node at (-2,0) {$R^2$}; \node at (-2,2) {$R$};
\node at (-2,1) {$0$}; \node at (-2,3) {$0$};
\node at (-2,-0.7) {$8$};

\draw [->] (-2.4,0.4)--(-5.6,2.6);
\node[above] at (-3.7,1.5) {$(1,3)$};
\draw [->] (-6.4,0.4)--(-9.6,2.6);
\node[above] at (-7.7,1.5) {$(0)$};

\node at (-6,4) [rectangle, draw] {$DCI(\overline{Y}_{T^*};R)$};
\end{tikzpicture}
\end{center}

\begin{center}
\begin{tikzpicture}
\draw[->] (0,-0.5)--(9.5,-0.5);
\draw[->] (0,-0.5)--(0,3.5);

\node at (-0.2,1) {$1$}; \node at (-0.2,2) {$2$};
\node at (-0.2,3) {$3$}; \node at (-0.2,0) {$0$};

\node at (0.5,0) {$R^2$}; \node at (0.5,2) {$0$};
\node at (0.5,1) {$0$}; \node at (0.5,3) {$R$};
\node at (0.5,-0.7) {$0$};

\node at (4.5,0) {$R^2$}; \node at (4.5,2) {$0$};
\node at (4.5,1) {$0$}; \node at (4.5,3) {$R$};
\node at (4.5,-0.7) {$4$};

\node at (8.5,0) {$R^2$}; \node at (8.5,2) {$0$};
\node at (8.5,1) {$0$}; \node at (8.5,3) {$R$};
\node at (8.5,-0.7) {$8$}; 

\draw [->] (8.1,0.4)--(4.9,2.6);
\node[above] at (6.8,1.5) {$(1,3)$};
\draw [->] (4.1,0.4)--(0.9,2.6);
\node[above] at (2.8,1.5) {$(1,3)$};

\node at (4.5,4) [rectangle,draw] {$DCI(\overline{Y}_{O^*};R)$};
\end{tikzpicture}
\end{center}

\begin{center}
\begin{tikzpicture}
\draw[->] (10.5,-0.5)--(20,-0.5);
\draw[->] (10.5,-0.5)--(10.5,3.5);

\node at (10.3,1) {$1$}; \node at (10.3,2) {$2$};
\node at (10.3,3) {$3$}; \node at (10.3,0) {$0$};

\node at (11,0) {$R^2$}; \node at (11,2) {$0$};
\node at (11,1) {$0$}; \node at (11,3) {$R$};
\node at (11,-0.7) {$0$};

\node at (15,0) {$R$}; \node at (15,2) {$0$};
\node at (15,1) {$0$}; \node at (15,3) {$R$};
\node at (15,-0.7) {$4$}; 

\node at (19,0) {$R^2$}; \node at (19,2) {$0$};
\node at (19,1) {$0$}; \node at (19,3) {$R$};
\node at (19,-0.7) {$8$}; 

\draw [->] (18.6,0.4)--(15.4,2.6);
\node[above] at (17.3,1.5) {$(1,3)$};
\draw [->] (14.6,0.4)--(11.4,2.6);
\node[above] at (13.3,1.5) {$(4)$};

\node at (15,4) [rectangle, draw] {$DCI(\overline{Y}_{I^*};R)$};
\end{tikzpicture}
\end{center}
\end{example} 

\section{The Algebraic Construction of Equivariant Floer Homology}
The equivariant instanton Floer homology associated with the
trivial $\om{SU}(2)$-bundle over a closed oriented $3$-manifold $Y$
comes in three flavors: $I^+(Y)$, $I^-(Y)$, $I^\infty(Y)$, the positive, negative and Tate homology, respectively. These are all constructed algebraically
from the complex $\wt{CI}(Y)$, the associated action of $C_*^{gm}(\om{SO}(3);R)$ and the index filtration $\{F_p\wt{CI}(Y)\}_{p\in \Z}$. 

The construction proceeds as follows. First, following \cite[Appendix A]{Miller19}, we review the construction of four functors
$C^+_A$, $C^-_A$, $C^{(+,tw)}$ and $C^\infty_A$ from the category of right modules over a differential graded $R$-algebra $A$ to the category of differential
graded $R$-modules. In fact, $C^-_A(R)$ will carry the structure of
a differential graded algebra and the three functors take values in the
subcategory of left $C^-_A(R)$-modules. Under the assumptions that the
ground ring $R$ is a principal ideal domain and the algebra $A$ is
degreewise free over $R$, the four functors will be
exact and preserve quasi-isomorphisms. 

Next, we promote the functors to the category of filtered right $A$-modules.
If $M$ is a right $A$-module equipped with an increasing filtration 
\[ \cdots \subset F_pM\subset F_{p+1}M\subset \cdots \subset M  \]
of $A$-submodules, the exactness of the functors ensure that $C^\bullet_A(M)$  is naturally filtered by
$F_pC^\bullet_A(M) := C^\bullet_A(F_pM)$ for each$\bullet\in \{+,-,(+,tw),\infty\}$. However, this filtration may not be very well behaved. To remedy this issue we pass to the so-called full completion,
thereby obtaining four new functors $\hat{C}^\bullet_A$, $\bullet\in \{(+,-,(+,tw),\infty\}$ from the category of filtered right $A$-modules to the category of filtered left
$C^-_A(R)$-modules; to be specific,
\[ \hat{C}^\bullet_A(M) = \om{lim}_q\om{colim}_p C^\bullet_A(F_pM/F_qM)  .\]
For $A=C^{gm}_*(\om{SO}(3);R)$ the groups $I^\bullet(Y)$ for $\bullet\in \{+,-,\infty\}$ are obtained
by applying $\hat{C}^\bullet_A$ to $M=\wt{CI}(Y)$ equipped with
the index filtration and then passing to homology. These groups will then be modules over $H(C^-_A(R))$, which, provided $\frac12\in R$, is 
isomorphic to the polynomial algebra $R[U]$ with a single generator
$U$ of degree $-4$. The mod $8$ periodicity of $M$ will be preserved so that $I^\bullet_n\cong I^\bullet_{n+8}$ for all $n\in \Z$.

The above constructions are also functorial in a suitable sense depending on
$\bullet\in \{+,-,\infty\}$, in the differential graded algebra $A$. 
Moreover, this functoriality will preserve quasi-isomorphisms in the
situation of interest; that is, we may equally well calculate
$I^\bullet(Y)$ using $A=\Lambda_R[u]$ and $M=DCI(Y)$. When we proceed
to the calculations for binary polyhedral spaces in the next section, we
will indeed work in this context. 

\subsection{Conventions}
Throughout $R$ will be a fixed principal ideal domain.
A module or algebra will always mean an $R$-module or $R$-algebra, and
the same applies in the graded or differential graded setting. We will
work with homological gradings and differentials $\dd$ of degree $-1$.   
We follow the conventions on graded and differential graded (DG) modules given in \cite[Chapter ~VI]{MacLane75}. We briefly recall the most central
concepts. If $M$ and $N$ are DG modules then $M\otimes N$ and $\om{Hom}(M,N)$ are the DG modules given by
\[ (M\otimes N)_n = \bigoplus_{p+q=n}M_n\otimes N_q \; \mbox{ and } \;
\om{Hom}(M,N)_n = \prod_{q-p=n} \om{Hom}(M_p,N_q)  \]
with differentials $\dd_{M\otimes N}=\dd_M\otimes 1+1\otimes \dd_N$
and $\dd_{\om{Hom}(M,N)}=\om{Hom}(1,\dd_N)-\om{Hom}(\dd_M,1)$.
This means that $\dd_{M\otimes N}(a\otimes b) = \dd_Ma\otimes b +(-1)^{|a|}a\otimes \dd_Nb$ and $\dd_{\om{Hom}(M,N)}(f) = \dd_N\circ f-(-1)^{|f|}f\circ \dd_M$, where we write $|x|$ for the degree of an element in a graded module.
There is a natural interchange isomorphism $\tau:M\otimes N\cong N\otimes M$
given by $m\otimes n\mapsto (-1)^{|m||n|}n\otimes m$. 

For each integer $p\in \Z$ let $R[p]$ denote the complex with a single
copy of $R$ concentrated in degree $p$. For a DG module $M$ the shifted
complex $M[p]$ is defined to be $R[p]\otimes M$. Therefore, $M[p]_{n+p}=M_n$
for all $n\in \Z$ with differential $\dd_{M[p]} =(-1)^p \dd_M$. It is also
possible to define $M[p]=M\otimes R[p]$, but this will result in different
sign conventions.  

A differential graded algebra is a graded algebra $A$ equipped with a
differential satisfying the Leibniz rule $\dd_A(ab)=\dd_A(a)b+(-1)^{|a|}a\dd_A b$ for all $a,b\in A$. We will always assume that $A$ is unital, i.e., there
is a unit $1\in A_0$ for the multiplication. A left or right $A$-module is a graded left or right $A$-module equipped with a differential satisfying a corresponding Leibniz rule. This ensures that $H(M)$ is a graded
$H(A)$-module. 
If $M$ is a left or right $A$-module then so is $M[p]$. For $a\in A$ and
$x\in M[p]_n = M_{n-p}$ the induced product
is given by $a\cdot x = (-1)^{|a|p}ax$ in the left module context and by  $x\cdot a=xa$ in the right module context. This is a consequence of our left shifting (suspension) convention. 

We also mention that there is a natural adjunction isomorphism of DG modules
\[ \om{Hom}(M\otimes N,P) \cong \om{Hom}(M,\om{Hom}(N,P))  \]
given by $(f\colon M\otimes N\ra P)\mapsto (g\colon M\ra \om{Hom}(N,P))$ where
$f(m\otimes n)=g(m)(n)$. This isomorphism admits various natural generalizations to the case where $M$, $N$ and $P$ carry additional module
structures over various DG algebras, see \cite[Chapter~VI.8]{MacLane75}.
All the signs occurring in the upcoming theory can be deduced from
these conventions: the interchange morphism, the natural adjunction and
the left shift convention.

Finally, a homomorphism $f\colon M\ra N$ of DG modules will mean a degree $0$
chain map unless otherwise stated. We say
that $f$ is a quasi-isomorphism if the induced map in homology,
$H(f)\colon H(M)\ra H(N)$, is an isomorphism. 

\subsection{The Bar Construction and the Functors \texorpdfstring{$C^\pm_A$}{CpmA}} 
Let $A$ be a unital DG algebra and assume that $A$ is equipped with an augmentation $\eps\colon A\ra R$, that is, $\eps$ is a map of DG algebras
where $R$ is regarded as a DG algebra concentrated in degree $0$. 
The augmentation gives $R$ the structure of an $A$-bimodule.
We define $\bar{A} = A/R\cdot 1$ and note that there is a canonical
identification of $\bar{A}$ with the augmentation ideal $\om{Ker}(\eps)$.
Given a right $A$-module $M$ and a left $A$-module $N$ define
the bigraded module $B_{*,*}(M,A,N)$ by
\[ B_{p,q}(M,A,N) = (M\otimes \bar{A}^{\otimes p } \otimes N)_q . \]
We adopt the standard notation and write
\[ m[a_1|a_2|\cdots |a_p]n \coloneqq m\otimes a_1\otimes a_2 \otimes \cdots \otimes
a_p \otimes n \in B_{p,q}(M,A,N) .\]
Here, $p$ is the simplicial degree, $q=|m|+\left(\sum_{i=1}^p |a_i|\right) + |n|$ is the internal degree and $p+q$ is the total degree. Let
$\dd^s\colon B_{p,q}(M,A,N)\ra B_{p-1,q}(M,A,N)$ and $\dd^i\colon B_{p,q}(M,A,N)\ra B_{p,q-1}(M,A,N)$ act on $x=m[a_1|\cdots |a_p]n$ by
\begin{align*} \label{Bar-Differential}
\dd^s(x) =& \; ma_1[a_2|\cdots |a_p]n +\sum_{i=1}^{p-1}(-1)^i m[a_1|\cdots |a_ia_{i+1}|\cdots |a_p]n \\
\; +&(-1)^p m[a_1|\cdots |a_{p-1}]a_pn  \\
\dd^i(x) =& \; (-1)^p( \dd_M m[a_1|\cdots |a_p]n +\sum_{i=1}^p (-1)^{\eps_{i-1}}
m[a_1|\cdots |\dd a_i | \cdots |a_p]n\\
 +& \;(-1)^{\eps_p}m[a_1|\cdots |a_p]\dd_N n),
\end{align*} 
where $\eps_i= |m|+|a_1|+\cdots +|a_i|$. Note that $\dd^i$ is the usual tensor product differential in $M\otimes \bar{A}^{\otimes p}\otimes N$ twisted by a sign depending on the simplicial degree. We call $\dd^s$ the simplicial differential and $\dd^i$ the internal differential. It holds true that
$(\dd^i)^2=0$, $(\dd^s)^2=0$ and $\dd^i\dd^s+\dd^s\dd^i=0$. This ensures that the triple $(B_{*,*}(M,A,N),\dd^s,\dd^i)$ is a double complex.   

\begin{definition} \label{Def-Bar-Construction}
\cite[Appendix A]{GugenheimMay74}  
The bar construction $B(M,A,N)$ is the total complex associated
with the double complex $(B_{*,*}(M,A,N),\dd^s,\dd^i)$ introduced above, 
that is,
\[ B_k(M,A,N)=\bigoplus_{p+q=n} B_{p,q}(M,A,N) \; \mbox{ and } \;
\dd = \dd^s+\dd^i  .\]
\end{definition} 

The bar construction is functorial in the triple $(M,A,N)$. For a triple
$(f,g,h)\colon (M,A,N)\ra (M',A',N')$ where $g$ is a map of augmented DG algebras
and $f,h$ are chain maps linear along $g$, there is an induced map
$B(f,g,h)\colon B(M,A,N)\ra B(M',A',N')$ of DG modules. 
There is also a natural augmentation $B(M,A,N)\ra M\otimes_A N$ given by
$m[a_1|\cdots |a_p]n\mapsto 0$ for $p>0$ and $m[]n\mapsto m\otimes n$. 

\begin{remark} The DG module $B(M,A,A)$ inherits a right $A$-module
structure from the last factor. The augmentation $B(M,A,A)\ra M$ is a map of
right $A$-modules and should be regarded as an analogue of a projective
resolution of $M$. Moreover, $B(M,A,N)\cong B(M,A,A)\otimes_A N$,
so that $B(M,A,N)$ is a derived tensor product $M\otimes^L_A N$
in a certain sense. These statements are made precise in the language
of model categories in \cite[Section~10.2]{BMR14}.
\end{remark} 

Following \cite{Miller19} we introduce the notation
\[ BA \coloneqq  B(R,A,R) \;\; \mbox{ and } \;\; EA\coloneqq B(R,A,A) . \]
Furthermore, we write 
$[a_1|a_2|\cdots |a_p] \coloneqq 1_R[a_1|a_2|\cdots |a_p]1_R\in BA$ and similar
conventions apply for $B(R,A,N)$ and $B(M,A,R)$. 

\begin{lemma} \label{Co-Alg-Mod-Structure} 
For any right $A$-module $M$ and left $A$-module $N$ there is a natural map
of DG modules $\psi_{M,N}\colon B(M,A,N)\ra B(M,A,R)\otimes B(R,A,N)$ given by
\[ m[a_1|\cdots |a_p]n\mapsto \sum_{i=0}^p (-1)^{(p-i)\eps_i}m[a_1|\cdots |a_i]\otimes [a_{i+1}|\cdots |a_p]n , \]
where $\eps_i = |m|+|a_1|+\cdots +|a_i|$. 
The map $\psi_{R,R}\colon BA\ra BA\otimes BA$ along with the augmentation 
$BA\ra R$ give $BA$ the structure of a DG coalgebra. The map
$\psi_{M,R}$ gives $B(M,A,R)$ the structure of a right DG $BA$-comodule
and the map $\psi_{R,N}$ gives $B(R,A,N)$ the structure of a left DG $BA$-comodule.
\end{lemma}

\begin{definition}\cite[Definition~A.5]{Miller19} Let $A$ be an augmented
DG algebra and let $M$ be a right $A$-module. Define the positive and 
negative $A$-chains of $M$ to be
\[ C^+_A(M) \coloneqq B(M,A,R) \;\; \mbox{ and } \;\; 
C^-_A(M)\coloneqq \om{Hom}_A(EA,M)  .\]
The homology of $C^\pm_A(M)$ is denoted by $H^\pm_A(M)$ and is called
the positive and negative $A$-homology of $M$, respectively. 
\end{definition}

The above constructions are functorial in the following sense. 
A homomorphism $f\colon A\ra B$ of augmented DG algebras induces maps
\[ f_*\colon C^+_A(f^{-1}M)\ra C^+_{B}(M)\;  \mbox{ and } \; 
f^*\colon C^-_{B}(M)\ra C^-_A(f^{-1}M) \]
where $M$ is a $B$-module and $f^{-1}M$ is the $A$-module obtained from $M$ by restriction along $f$. Furthermore, a homomorphism $g\colon M\ra N$ of $A$-modules induces maps
\[ g^+\colon C^+_A(M)\ra C^+_A(N) \; \mbox{ and } \; g^-\colon C^-_A(M)\ra C^-_A(N)  .\]
The following invariance result is proven in \cite[Theorem~A.1,A.2]{Miller19};
it relies on the standing assumption that $R$ is a PID. 

\begin{proposition} \label{Invariance} 
Let $f\colon A'\ra A$ is a homomorphism of augmented
DG algebras and let $g\colon M\ra M'$ be a homomorphism of $A$-modules. 
\begin{enumerate}[label=(\roman*),ref=(\roman*)]
\item If $A$ and $A'$ are degreewise free and $f$ is a quasi-isomorphism,
then $f_*$ and $f^*$ are quasi-isomorphisms.
\item If $A$ is degreewise free and $g$ is a quasi-isomorphism, then
$g^+$ and $g^-$ are quasi-isomorphisms. 
\end{enumerate}
\end{proposition}

The DG modules $C^\pm_A(M)$ carry additional structure. First, we have
the following sequence of isomorphisms of DG modules
\[ C_A^-(R) = \om{Hom}_A(EA,\eps^{-1}(R))\cong
\om{Hom}_R(EA\otimes_A R,R)\cong \om{Hom}(BA,R),  \]
which identifies $C_A^-(R)$ with the $R$-dual of $BA$. The DG coalgebra structure of $BA$ dualizes to give $C_A^-(R)$ the structure of a DG algebra.
Next, for any right $A$-module $M$ the DG modules $C^\pm_A(M)$ obtain
left $C_A^-(R)$-module structures in the following way. Write
$BA^\vee =\om{Hom}_R(BA,R)=C_A^-(R)$, then the structure map
$C_A^-(R)\otimes C^+_A(M)\ra C^+_A(M)$ is given by the following composition
\begin{equation} \label{ModuleStructure1}
\begin{tikzcd} BA^\vee \otimes C^+_A(M) \arrow{r}{1\otimes \psi_M} &
BA^\vee\otimes C^+_A(M)\otimes BA  & {} \\
{} \arrow{r}{\tau\otimes 1} & C^+_A(M)\otimes BA^\vee \otimes BA  
\arrow{r}{1\otimes \om{ev}} & C^+_A(M)  \end{tikzcd} 
\end{equation} 
where $\psi_M=\psi_{M,R}$ is the comodule structure map of 
$C^+_A(M)=B(M,A,R)$ in Lemma \ref{Co-Alg-Mod-Structure}, $\tau$ is the interchange isomorphism $a\otimes b\mapsto (-1)^{|a||b|}b\otimes a$ and $\om{ev}\colon BA^\vee \otimes BA\ra R$ is the evaluation map. In the other case the structure map $C_A^-(R)\otimes C^-_A(M)\ra C_A^-(M)$
is given by the following composition
\begin{equation} \label{ModuleStructure2}
\begin{tikzcd} BA^\vee \otimes \om{Hom}_A(EA,M) \arrow{r}{\zeta} &
\om{Hom}_A(BA\otimes EA,M) \arrow{r}{\psi_{EA}^*} &
\om{Hom}_A(EA,M) \end{tikzcd}
\end{equation}
where $\zeta(f\otimes g)(a\otimes b)=(-1)^{|a||g|}f(a)\cdot g(b)$
and $\psi_{EA}^* = \om{Hom}_A(\psi_{R,A},1)$ is the dual of the
comodule structure map $\psi_{R,A}\colon EA\ra BA\otimes EA$.  

The properties of the functors $C^\pm_A$ that will be important for us
are summarized in the following theorem. Here the morphisms in the 
category of right or left modules over a DG algebra $A$ are taken to
be the $A$-linear chain maps of degree $0$.   

\begin{theorem} \label{Cpm-Structure-Theorem}
Let $A$ be a degreewise $R$-free augmented DG algebra.
Then the assignments 
\[ M\mapsto C^\pm_A(M) \; \mbox{ and } \; 
(g\colon M\ra M')\mapsto (g^\pm\colon C^\pm_A(M)\ra C^\pm_A(M')  \]
define functors $C^\pm_A$ from the category of right $A$-modules
to the category of left $C_A^-(R)$-modules. Moreover, these functors
preserve quasi-isomorphisms and short exact sequences. 
\end{theorem}

We will now treat the case $A=\Lambda_R[u]$ with $|u|=3$ in detail. These calculations will be important for later purposes and will serve as an excellent example of how the above constructions work. We should note
that all of these calculations easily generalize to the case where
the degree of $u$ is any odd number.  

We will write $R[U]$ for the graded polynomial algebra on a single generator
$U$ of even non-zero degree. This is also a coalgebra with counit given
by the identity in degree $0$ and diagonal determined by
\[ U^p\mapsto \sum_{i=0}^p U^i\otimes U^{p-i}   \]
where $U^0=1$. This structure is compatible with the algebra structure and turns $R[U]$ into a Hopf algebra, but this will not be important for us.  

\begin{lemma} \label{Orbit-Calc-Lemma}
Let $A=\Lambda_R[u]$ with $|u|=3$. Then there is an isomorphism
of DG coalgebras $BA\cong R[V]$ where $|V|=4$, and an isomorphism of
DG algebras $C_A^-(R)=BA^\vee\cong R[U]$ where $|U|=-4$. In particular, these
complexes have trivial differentials so that
$H^+_A(R)=R[V]$ and $H^-_A(R)=R[U]$. The first isomorphism is given by
\[ V^p\mapsto (-1)^{p(p+1)/2}[u|\cdots |u] \;\; \mbox{ (p times) }\]
for $p\geq 0$. The second is obtained by dualization of the first
where $U^p$ is determined by the evaluation $U^p(V^p)=1$ for all $p\geq 0$. Furthermore, the left $BA^\vee$-module structure of $BA$ is given by
$U^s\cdot V^p = V^{p-s}$, where we interpret $V^i=0$ for $i<0$. 
\end{lemma}  
\begin{proof} We have $A=\Lambda_R[u]$ so that $A_0=R$, $A_3=Ru$ and
$A_n=0$ otherwise. The augmentation is necessarily given by the identity
in degree $0$. Hence, $\ovl{A}=Ru$ concentrated in degree $3$. From the
definition of $B(R,A,R)$ we find that $BA_{p,3p}=R\cdot Z^p$ for $p\geq 0$, where $Z^p=[u|\cdots |u]$ with $u$ repeated $p$ times, and $B_{p,q}=0$ otherwise. It follows immediately that the differential vanishes identically. 
Using the formula for the coalgebra structure map in Lemma \ref{Co-Alg-Mod-Structure} we find that
\[ \psi_{R,R}(Z^p) = \sum_{i=0}^p (-1)^{(p-i)3i}Z^i\otimes Z^{p-i}  .\]
Define $V^p = (-1)^{p(p+1)/2}Z^p$ for $p\geq 0$. Then using the identity
\[ \frac{p(p+1)}{2} = (p-i)i + \frac{i(i+1)}{2}+\frac{(p-i)(p-i+1)}{2}
 \]
we deduce that the above formula transforms into
$\psi_{R,R}(V^p)= \sum_{i=0}^p V^i\otimes V^{p-i}$.
Hence, $BA\cong R[V]$ as DG coalgebras with trivial differentials. 

Define $U^p\in BA^\vee$ to be the dual of $V^p$, i.e., $U^p\colon BA\ra R$
has degree $-4p$ and acts by $U^p(V^p)=1$ and vanishes otherwise. The differential is still of course trivial. The product $U^s\cdot U^t$ has degree
$-4(s+t)$ and must therefore be an $R$-multiple of $U^{s+t}$. The
calculation
\[ (U^s\cdot U^t)(V^{s+t}) = (U^s\otimes U^t)\left( \sum_{i=0}^{s+t} V^i\otimes V^{s+t-i}\right) = U^s(V^s)U^t(V^t)=1  \]
shows that $U^s\cdot U^t = U^{s+t}$. Hence, $C_A^-(R)\cong BA^\vee \cong R[U]$
as DG algebras with trivial differentials. 
Finally, we need to check that the left module structure $BA^\vee\otimes BA\ra BA$ is given by $U^s\cdot V^p=V^{p-s}$. Using the definition we find
\[ U^s\cdot V^p = \sum_{i=0}^p (-1)^{|U^s||V^i|} V^i \cdot U^s(V^{p-i})
=\left\{ \begin{array}{cl} V^{p-s} & \mbox{ if } p\geq s \\
                           0 & \mbox{ otherwise.} \end{array} \right. \]
as required.   
\end{proof}   

We will also need explicit calculations of $H^\pm_A(M)$ when
$M=H_*(\alpha;R)$ for $\alpha = \om{SO}(3)/\om{SO}(2)$ or $\alpha=\om{SO}(3)$,
under the assumption $\frac12 \in R$. In other words, 
$M=R\oplus R[2]$ with the trivial $A$-module structure or
$M=A=\Lambda_R[u]$ (see \eqref{Eq-Orbit-Calc}). For this purpose we give concrete models for the complexes $C^\pm_A(M)$. Note that a right 
$\Lambda_R[u]$-module $M$ is simply a DG module equipped with a degree
$3$ map $u:M\ra M$ satisfying $u^2=0$ and $u\dd_M=\dd_Mu$ (because $A$ acts from the right).   

\begin{proposition} \label{Explicit-Cpm-Models}
Let $M$ be a right $A=\Lambda_R[u]$-module. Define
bigraded modules $D^\pm_{*,*}$ by
\[ D^+_{p,q} = \left\{ \begin{array}{cl} M_{q-3p} & p\geq 0 \\
                                         0 & p<0 \end{array} \right.
\; \mbox{ and } \; D^-_{p,q} =\left\{
 \begin{array}{cl} M_{q-3p} & p\leq  0 \\
                  0 & p>0 \end{array} \right.    .\]
Define $U\colon D^\pm_{p,q}\ra D^\pm_{p-1,q-3}$ to be the identity $M_{q-3p}\ra M_{q-3p}$ when $p\geq 1$ and $p\leq 0$, respectively, and otherwise to be
zero. Define
\begin{align*}
 \dd'=(-1)^{p+q+1}u\colon D^{\pm}_{p,q}&=M_{q-3p}\ra M_{q-3p+3}=D^\pm_{p-1,q}\\
\dd''=\dd_M\colon D^\pm_{p,q}&=M_{q-3p}\ra M_{q-3p-1}=D^\pm_{p,q-1},
\end{align*} 
in the range it makes sense. Then $(D^+_{*,*},\dd',\dd'')$ and $(D^-_{*,*},\dd',\dd'')$ are double complexes and there are natural isomorphisms of
DG $R[U]$-modules
\[ \om{Tot}^\oplus (D^+_{*,*},\dd',\dd'') \cong C^+_A(M) \;\; \mbox{ and } \;\;
  \om{Tot}^\Pi (D^-_{*,*},\dd',\dd'') \cong C^-_A(M)  .\]
\end{proposition} 
\begin{proof} We have $\bar{A}=R\cdot u$ with $u$ in degree $3$. Therefore,
$B_{p,q}=(M\otimes \bar{A}^{\otimes p})_q \cong M_{q-3p}\otimes \bar{A}^{\otimes p}_{3p} \cong M_{q-3p}$. The isomorphism is given by
$m\mapsto m[u|u|\cdots |u]$ with $u$ repeated $p$ times. For simplicity
of notation write this as $m\otimes Z^p$. Then, using the formulas for
the differentials, we find $\dd^s(m\otimes Z^p) = mu\otimes Z^{p-1}$
and $\dd^i(m\otimes Z^p)=(-1)^p \dd_Mm\otimes Z^p$. The action of
$U$ is given by first applying the comodule structure map to
$m\otimes Z^p$, and then evaluating $U$ on the right factors in the
resulting sum ($U$ has even degree so the interchange introduces no sign).
We have
\[ \psi_{M,R}(m\otimes Z^p) = \sum_{i=0}^p (-1)^{(p-i)(|m|+3i)}(m\otimes Z^i)\otimes Z^{p-i}  \]
so that $U\cdot (m\otimes Z^p)=(-1)^{|m|+p-1}(m\otimes Z^{p-1})\cdot U(Z^1)
= (-1)^{|m|+p}m\otimes Z^{p-1}$ as $U$ was defined to be the dual of
$V^1 = -Z^1$. It is now easy to check that when we introduce the
sign change $m\in D^+_{p,q}=M_{q-3p}\mapsto (-1)^{p(p+1)/2+p|m|}m\otimes Z^p$
the $U$ action and the differentials are given as in the proposition. 

In the second case, $EA=B(R,A,A)$ has generators over $R$ of the form
$Z^p = [u|u|\cdots |u]$ and $Z^p\otimes u = [u|u|\cdots |u]u$
in degrees $4p$ and $4p+3$ respectively. It is convenient as earlier
to introduce the sign twist $V^p = (-1)^{p(p+1)/2}Z^p$, since then the
differential is given by $\dd(V^p)=V^{p-1}\otimes u$ and
$\dd(V^p\otimes u)=0$ and the comodule structure map $\psi_{R,A}\colon EA\ra BA\otimes EA$ is given by
\[ \psi_{R,A}(V^p) = \sum_{i=0}^p V^i\otimes V^{p-i}   .\]
An element $f\in \om{Hom}_A(EA,M)_n$ is by $A$-linearity uniquely determined
by the elements $f(V^p)\in M_{n+4p}$ for $p\geq 0$. The differential is given
by
\[ (\dd f)(V^p) = \dd_M(f(V^p))-(-1)^{|f|}f(\dd V^p)
                = \dd_M(f(V^p))-(-1)^{|f|}f(V^{p-1})u  \]
and the action of $U$ is given by
\[ (U\cdot f)(V^p)=(U\otimes f)(\psi_{R,A}(V^p)) =(U\otimes f)(V\otimes V^{p-1}) = f(V^{p-1})  .\]
It is now easy to check that the complex $C^-_A(M)=\om{Hom}_A(EA,M)$
can be described as the product totalization of the double complex
given in the proposition. \end{proof}

For later reference we record the following facts from the above proof. 
The isomorphism $\om{Tot}^\oplus(D^+,\dd',\dd'')\cong C^+_A(M)$ is given by sending an element $m\in M_{q-3p}=D^+_{p,q}$ to
\begin{equation} \label{Iso-+} 
(-1)^{p(q-p)+p(p+1)/2}m\overbrace{[u|u|\cdots |u]}^{p} \in B(M,A,R)_{p,q}
\end{equation} 
and the isomorphism $C^-_A(M)\cong \om{Tot}^\Pi(D^-,\dd',\dd'')$ is given by sending an element $f\in \om{Hom}_A(EA,M)_n$ to
\begin{equation} \label{Iso--}
 ((-1)^{p(p+1)/2}f(\overbrace{[u|u|\cdots |u]}^p))_{p\geq 0} \in \prod_{p\geq 0} M_{n+4p} =\prod_{p\geq 0} D^-_{-p,n+p} .
\end{equation}  

It should be noted that there are other choices of signs possible. We
have chosen the signs that ensure that the action of $U$ corresponds to
the identity map. 
A part of the double complex $(D^+_{*,*},\dd',\dd'')$ corresponding
to a right $A$-module $M$ is shown below
\[ \begin{tikzcd} 
{} & \vdots \arrow{d} & \vdots \arrow{d} & \vdots \arrow{d} & {} \\
0 & M_2 \arrow[dotted]{dddl} \arrow{l} \arrow{d}{\dd_M} & M_{-1} \arrow[dotted]{dddl} \arrow{l}{u} \arrow{d}{\dd_M}
& M_{-4} \arrow{l}{-u} \arrow[dotted]{dddl} \arrow{d}{\dd_M} & \cdots \arrow{l} \\
0 & M_1 \arrow{l} \arrow{d}{\dd_M} & M_{-2} \arrow{d}{\dd_M} \arrow{l}{-u}
& M_{-5} \arrow{d}{\dd_M} \arrow{l}{u} & \cdots \arrow{l} \\
0 & M_0 \arrow{l} \arrow{d}{\dd_M} & M_{-3} \arrow{l}{u} \arrow{d}{\dd_M}
& M_{-6} \arrow{l}{-u} \arrow{d}{\dd_M} & \cdots \arrow{l} \\
0 & M_{-1} \arrow{l} \arrow{d} & M_{-4} \arrow{l}{-u} \arrow{d} &
M_{-7} \arrow{l}{u} \arrow{d} & \cdots . \arrow{l} \\
{} & \vdots & \vdots & \vdots & {} \end{tikzcd} \]
The dotted arrows represent the action of $U$. An element $x$ of
$C^+_A(M)_n$ can be represented by a sequence $(m_{n-4p})_{p\geq 0}$,
where almost all terms vanish. This element $x$ is a cycle if and only if
$\dd_Mm_{n-4p}=(-1)^n m_{n-4(p-1)}u$ for each $p\geq 0$. Each column
is a shifted copy of the original complex $M$. The double complex
$(D^-_{*,*},\dd',\dd'')$ is obtained by shifting the whole diagram from
the right half plane over to the left half plane. A more concise representation
of this double complex is given in the following diagram
\[ \begin{tikzcd} \cdots & M[-9[ \arrow{l} & M[-6] \arrow{l}[swap]{\pm u} 
& M[-3] \arrow{l}[swap]{\mp u} & M \arrow{l}[swap]{\pm u} 
& 0. \arrow{l} \end{tikzcd} \] 
Note that in this case the internal differential in $M[-3n]$, $n\geq 0$, 
is not twisted by a sign, and that the $\pm u$ notation means that
the sign depends on the internal degree. 

With these concrete models in hand it is easy to calculate the following.
\begin{corollary} \label{Orbit-Calc-Corollary}
We have $H^+_A(A)\cong R$, $H^-_A(A)\cong R[3]$, with
necessarily trivial $R[U]$-module structure, and
$H^+_A(R\oplus R[2])\cong R[W]$ with $|W|=2$, $H^-_A(R\oplus R[2])=R[Z][2]$
with $|Z|=-2$. The $R[U]$-module structures are determined by $U\cdot W^p = W^{p-2}$ and $U\cdot Z^p=Z^{p+2}$. 
\end{corollary}

\begin{remark} The notation $R[W]$ and $R[Z][2]$ in the above result is only used to effectively express the underlying graded modules and the action
of $R[U]$, it should not be interpreted to mean that $H^\pm_A(R\oplus R[2])$
carry any (co)algebra structure. \end{remark}    

\subsection{Tate Homology}
We will now proceed to the construction of $C^\infty_A(M)$.  

\begin{definition} \label{Def-Dualizing-Object} 
For an augmented $DG$ algebra $A$ define the dualizing object by
\[ D_A\coloneqq C^-_A(A)= \om{Hom}_A(EA,A)   .\]
\end{definition}

The dualizing object is a (left-left) $A-C^-_A(R)$-bimodule. The left
$C^-_A(R)$-module structure is the one already given in \eqref{ModuleStructure2},
while the left $A$-module structure is obtained from the second factor
in $D_A = \om{Hom}_A(EA,A)$, that is, 
$(a\cdot f)(x):=af(x)$ for all $a\in A$, $f\in \om{Hom}_A(EA,A)$ and $x\in EA$. This justifies the following definition. 

\begin{definition} For a right $A$-module $M$ define the twisted
positive $A$-chains of $M$ to be the $DG$ module
\[ C^{+,tw}_A(M) \coloneqq B(M,A,D_A)   .\]
The homology of this complex is denoted by $H^{+,tw}_A(M)$ and is called
the positive twisted $A$-homology of $M$.
\end{definition}

The complex $C^{+,tw}_A(M)$ carries a left $C^-_A(R)$-module structure
induced from the left $C^-_A(R)$-module structure of $D_A$.

\begin{definition} \label{Def-Norm-Map1} 
The norm map $N_M\colon C^{+,tw}_A(M)\ra C^-_A(M)$ is
for each right $A$-module $M$ defined to be the composite
\[ \begin{tikzcd} B(M,A,D_A) \arrow{r}{\eps} &
M\otimes_A D_A  \arrow{r}{\xi} & \om{Hom}_A(EA,M), \end{tikzcd} \]
where $\eps$ is the augmentation and $\xi$ is the adjoint of
\[ \begin{tikzcd} M\otimes_A(D_A\otimes EA) \arrow{r}{1\otimes \om{ev}} &
M\otimes_A A=M.  \end{tikzcd}  \]
Explicitly, for $m[a_1|\cdots |a_p]v\in B(M,A,D_A)$ and $x\in EA$,
\begin{equation} \label{Norm-Map-Def}
N_M(m[a_1|\cdots | a_p]v)(x) =\left\{ \begin{array}{cc} 
m\cdot v(x) & \mbox{ for } p=0 \\
0 & \mbox{ for } p>0. \end{array} \right.
\end{equation} 
\end{definition}  

\begin{lemma} The norm map $N_M\colon C^{+,tw}_A(M)\ra C^-_A(M)$ is
$C^-_A(R)$-linear and natural in $M$. \end{lemma}
\begin{proof} The naturality in $M$ is clear from the definition.
To verify the $C^-_A(R)$-linearity let $f\in C^-_A(R)$ and 
$w=m[a_1|\cdots|a_p]u \in C^{+,tw}_A(M)=B(M,A,D_A)$. Then
$f\cdot w = (-1)^{\eta}m[a_1|\cdots |a_p](f\cdot w)$ for
$\eta = p+|m|+|a_1|+\cdots +|a_p|$ and therefore
$N_M(f\cdot w) = 0 = f\cdot N_M(w)$ provided $p>0$. We may therefore
assume that $p=0$ such that $w = m[]u$. Let $v = N_M(w)$ and note that
$|v| = |m|+|u|$. Let $x\in EA$ and write $\psi_{R,A}(x) = \sum_i y_i\otimes z_i \in BA\otimes EA$ where $\psi_{R,A}$ is the map of Lemma \ref{Co-Alg-Mod-Structure}. Then by definition $(f\cdot v)(x)=(f\otimes v)(\psi_{R,A}(x))$,
which equals
\begin{equation}  \label{Norm-Map-Eq}
\sum_i (-1)^{|v||y_i|}f(y_i)v(z_i) = \sum_i (-1)^{|m||y_i|+|u||y_i|}m\cdot f(y_i)u(z_i) .
\end{equation}
On the other hand, $N_M(f\cdot m[]u) = (-1)^{|f||m|}N_M(m[](f\cdot u))$, and
by \eqref{Norm-Map-Def} this coincides with
\[ (-1)^{|f||m|} \sum_i m\cdot (f\otimes u)(y_i\otimes z_i)
= \sum_i (-1)^{|f||m|+|u||y_i|}m\cdot f(y_i)u(z_i) . \]
The final expression is equal to the final expression in \eqref{Norm-Map-Eq} because $f\colon BA\ra R$ so that $f(y_i)$ can only be nonzero if 
$|f|=-|y_i|$. Hence, $f\cdot N_M(w) = N_M(f\cdot w)$ and the proof is complete. 
\end{proof}   

The Tate complex
$C^\infty_A(M)$ will be defined to be the mapping cone complex of the norm map.
At a later stage we will also have to consider the cone of a map
of nonzero degree, so we fix our conventions in the following definition.

\begin{definition} \label{Cone-Definition}
Let $B$ be a $DG$ algebra and suppose that
$f\colon M\ra N$ is a degree $d$ chain map of left $B$-modules. This means that
\begin{enumerate}[label=(\alph*),ref=(\alph*)]
\item $f(M_n)\subset N_{n+d}$ for all $n\in \Z$,
\item $f(bm)=(-1)^{|b|d}bf(m)$ for all $b\in B$ and $m\in M$ and
\item $f\circ \dd_M =(-1)^{d}\dd_N\circ f$.
\end{enumerate} 
The left $B$-module $\om{Cone}(f)$ is defined by $\om{Cone}(f)_k = N_k\oplus M_{k-d-1}$ for all $k\in \Z$, with differential
\[ \dd_{\om{Cone}(f)} \cdot \left( \begin{array}{c} n \\ m \end{array} \right) =
\left( \begin{array}{cc} \dd_N & (-1)^df \\ 0 & (-1)^{d+1} \dd_M \end{array}
\right) \left( \begin{array}{c} n \\ m \end{array} \right) =
\left( \begin{array}{c} \dd_N n +(-1)^dfm \\ (-1)^{d+1} \dd_M m \end{array}
\right) \]
and $B$-module structure $b\cdot (n,m) = (bn,(-1)^{|b|(d+1)}bm)$.
\end{definition}

\begin{remark} If $B=R$ is concentrated in degree $0$ and $f$ has degree
$0$, then this definition agrees with \cite[p.~46]{MacLane75}. The above
definition is what we get when we regard $f$ as a degree $0$ map
$f\colon M[d]\ra N$ and take into account the sign twist in the $B$-action
upon shifting $M$.
\end{remark} 

The cone complex sits in a natural short exact sequence
\[ \begin{tikzcd} 0 \arrow{r} & N \arrow{r} & \om{Cone}(f) \arrow{r} & M[d+1]
\arrow{r} & 0 \end{tikzcd} \]
of $B$-modules, whose connecting homomorphism $\delta\colon H(M)_n\ra H_{n+d}(N)$
recovers $H(f)$ up to a sign. In particular, there is a natural exact
triangle in homology. 

\begin{definition} Let $M$ be a right $A$-module. Then the Tate complex
of $M$ is defined to be the $C^-_A(R)$-module
\[ C^\infty_A(M) \coloneqq \om{Cone}(N_M\colon C^{+,tw}_A(M)\ra C^-_A(M)) .\]
The homology of this complex is denoted by $H^\infty_A(M)$ and is called
the Tate homology of $M$. 
\end{definition} 

Given a map $g\colon M\ra M'$ of a right $A$-modules, there are induced maps
\[ g^{+,tw}\colon C^{+,tw}_A(M)\ra C^{+,tw}_A(M') \;\; \mbox{ and } \;\;
g^\infty\colon C^\infty_A(M)\ra C^\infty_A(M')  \]
of $C^-_A(R)$-modules. The key properties of the functors $C^{+,tw}_A$ and $C^\infty_A$ are summarized in the following theorem.

\begin{theorem} \label{Cptwinfty-Structure-Theorem}
Let $A$ be a degreewise $R$-free augmented DG algebra.
Then the assignments 
\[ (M\mapsto C^\bullet_A(M)) \; \mbox{ and } \; (g\colon M\ra M') \mapsto (g^\bullet\colon C^\bullet_A(M)\ra C^\bullet_A(M'))  \]
for $\bullet\in \{(+,tw),\infty\}$, define functors $C^{+,tw}_A$, $C^\infty_A$
from the category of right $A$-modules to the category of left
$C^-_A(R)$-modules. They are both exact and preserve quasi-isomorphisms.
Moreover, there is a natural exact triangle of $H^-_A(R)$-modules
\[ \begin{tikzcd} H^{+,tw}_A(M) \arrow{rr}{H(N_M)} & {} & H^-_A(M) \arrow{ld} \\
 {} & H^\infty_A(M). \arrow{ul}{[-1]} & {} \end{tikzcd} \] 
\end{theorem}

Suppose $f\colon A\ra B$ is a quasi-isomorphism of augmented DG algebras.
For a given right $B$-module $M$ there is no induced map between
$C^\bullet_A(f^{-1}M)$ and $C^\bullet_B(M)$ for $\bullet\in \{(+,tw),\infty\}$.
Rather one has to introduce an intermediate object to compare them.

\begin{definition} \label{Def-Relative-twisted}
Let $f\colon A\ra B$ be a homomorphism of augmented
DG algebras. Define the relative dualizing object to be the
(left-left) $C^-_A(R)-B$-bimodule 
\[ D_f \coloneqq C^-_A(f^{-1}B)=\om{Hom}_A(EA,f^{-1}B),   \]
where $f^{-1}B$ is regarded as a (right-left) $A-B$-bimodule.

Given a right $B$-module $M$ define the relative positive twisted complex to be
$C^{+,tw}_f(M) \coloneqq B(M,B,D_f)$ equipped with the left $C^-_A(R)$-module structure induced from $D_f$.  
\end{definition}

The relative positive twisted complex is a functor from the
category of right $B$-modules to the category of left $C^-_A(R)$-modules.
Given a map $g\colon M\ra N$ of right $B$-modules the induced map
is given by $B(g,1,1)\colon B(M,B,D_f)\ra B(N,B,D_f)$.  

\begin{lemma} \label{Relative-Twisted-Lemma}
In the situation of the above definition, if $A$ and $B$
are degreewise free, the relative functor $C^{+,tw}_f$ preserves
quasi-isomorphisms and short exact sequences. \end{lemma}
\begin{proof} The assumptions that $A$ and $B$ are degreewise free along
with the fact that $R$ is PID ensure that the relative dualizing object
$D_f$ is degreewise flat. This is verified as in \cite[Lemma~A.10]{Miller19}.
It then follows from the invariance theorem \cite[Theorem~A.1]{Miller19}
that $C^{+,tw}_f=B(-,B,D_f)$ is an exact functor preserving quasi-isomorphisms.
\end{proof}   

\begin{proposition} \label{Refined-Invariance} 
Let $f\colon A\ra B$ be a homomorphism of augmented DG algebras   
and let $M$ be a right $B$-module. There is a relative norm map
$N_M^f\colon C^{+,tw}_f(M)\ra C^-_A(f^{-1}M)$ fitting into the following
commutative diagram, natural in $M$,
\[ \begin{tikzcd} C^{+,tw}_B(M) \arrow{r}{\beta_M} \arrow{d}{N_M^B} & 
C^{+,tw}_f(M) \arrow{d}{N^f_M} & C^{+,tw}_A(f^{-1}M) \arrow{d}{N^A_M} 
\arrow{l}[swap]{\alpha_M} \\
C^-_B(M) \arrow{r}{f^*} & C^-_A(f^{-1}M) & C^-_A(f^{-1}M). \arrow{l}[swap]{=}
\end{tikzcd} \]
All of the maps in the diagram are $C^-_A(R)$-linear, $C^-_B(R)$-linear
or linear along $f^*:C^-_B(R)\ra C^-_A(R)$, as appropriate.
 
Moreover, if $f$ is a quasi-isomorphism and $A$ and $B$ are degreewise
$R$-free, then all the horizontal maps in the diagram are quasi-isomorphisms.
\end{proposition}
\begin{proof} Define $g\colon D_B\ra D_f$ and $h\colon D_A\ra D_f$
to be the maps
\[ g = f^*\colon C^-_B(B)\ra C^-_A(f^{-1}B) \;\; \mbox{ and } \;\;
   h = f^-\colon C^-_A(A)\ra C^-_A(f^{-1}B) .\] 
Note that $D_A$ is a $C^-_A(R)-A$-bimodule, $D_f$ is a $C^-_A(R)-B$-bimodule
and $D_B$ is a $C^-_B(R)-B$-bimodule. Moreover, observe that
$g$ is a homomorphism along $(f^*,1)\colon (C^-_B(R),B)\ra (C^-_A(R),B)$ and
that $h$ is a homomorphism along $(1,f)\colon (C^-_A(R),A)\ra (C^-_A(R),B)$.
We may therefore define the upper horizontal maps in the diagram by
\begin{align*}
\beta_M\coloneqq B(1,1,g) &\colon B(M,B,D_B)\ra B(M,B,D_f)   \\
\alpha_M \coloneqq B(1,f,h) &\colon B(f^{-1}M,A,D_A)\ra B(M,B,D_f).
\end{align*}
It is clear by construction that these maps are natural in $M$. Furthermore, by the above observations $\beta_M$
is linear along $f^*\colon C^-_B(R)\ra C^-_A(R)$ and $\alpha_M$ is
$C^-_A(R)$-linear.
 
We define the relative norm map $N^f_M\colon B(M,B,D_f)\ra C^-_A(f^{-1}M)$ to be the composition of the augmentation
$B(M,B,D_f)\ra M\otimes_B D_f$ and the map $\kappa\colon M\otimes_B D_f \ra C^-_A(M)$  which is given as the adjoint of
\[ \begin{tikzcd} 
M\otimes_B (\om{Hom}_A(EA,f^{-1}B)\otimes EA) \arrow{r}{1\otimes \om{ev}} &
M\otimes_B f^{-1}B \cong f^{-1}M. \end{tikzcd} \]
Then $N^B_M$ is $C^-_B(R)$-linear while $N_M^f$ and $N_M^A$ are
$C^-_A(R)$-linear. Consider the diagram
\[ \begin{tikzcd} B(M,B,D_B) \arrow{d}{\eps} \arrow{r}{B(1,1,g)} & 
B(M,B,D_f) \arrow{d}{\eps} & B(f^{-1}M,A,D_A) \arrow{l}[swap]{B(1,f,h)}
\arrow{d}{\eps} \\
M\otimes_B D_B \arrow{r}{1\otimes_B g} \arrow{d}{\xi} & M\otimes_B D_f \arrow{d}{\kappa} & f^{-1}M\otimes_A D_A \arrow{l}[swap]{1\otimes_f h}
\arrow{d}{\xi} \\
\om{Hom}_B(EB,B) \arrow{r}{Ef^*} & \om{Hom}_A(EA,f^{-1}M) & 
\om{Hom}_A(EA,f^{-1}M), \arrow{l}[swap]{=} \end{tikzcd} \]
where $\eps$ denotes the augmentations and $\xi$ are the natural
maps in the definition of the norm map (see \ref{Def-Norm-Map1}). In particular, the outer compositions $\xi\circ \eps$ are the respective
norm maps and by definition $\kappa\circ \eps = N^f_M$. We wish
to verify that the diagram commutes. The two upper squares
commute by the naturality of the augmentation in the bar construction.
To verify that the lower squares commute, consider the diagram obtained
by passing to the adjoints of the vertical maps $\xi$, $\kappa$ and $\xi$ in the bottom half of the diagram. The resulting diagram is then seen to commute using the fact that each of the maps $\xi$, $\kappa$ and $\xi$ was defined as the adjoint of $1\otimes \om{ev}$ for a suitable evaluation. 
By omitting the middle row of the above diagram, we obtain the commutative diagram in the statement of the proposition.  

Finally, we need to verify that all the horizontal maps in the diagram
are quasi-isomorphisms provided that $f\colon A\ra B$ is a quasi-isomorphism of
degreewise $R$-free algebras. We already know that this is true for
$f^*\colon C^-_B(M)\ra C^-_A(f^{-1}M)$ by Proposition \ref{Invariance}.
For the upper horizontal maps, note first that
$g=f^*$ and $h=f^-$ are quasi-isomorphisms by the same result. Then, since $D_A$, $D_B$ and $D_f$ are degreewise flat, one concludes by 
\cite[Theorem~A.1]{Miller19} that $\beta_M=B(1,1,g)$ and $\alpha_M=B(1,f,g)$ are quasi-isomorphisms as well. This completes the proof. 
\end{proof}   

From this result one quickly obtains the required invariance result.
Nevertheless, we will need the more refined statement of the above proposition in the next section. 

\begin{corollary} 
Let $f\colon A\ra B$ be a quasi-isomorphism of degreewise
$R$-free augmented DG algebras. Then for every right $B$-module
$M$ there are natural isomorphisms of $H^-_A(R)\cong H^-_B(R)$ modules
\[ H^{+,tw}_A(f^{-1}M) \cong H^{+,tw}_B(M)\; \mbox{ and } \;
H^\infty_A(f^{-1}M) \cong H^\infty_B(M)  .\]
Moreover, these isomorphisms are compatible with the exact triangles
of Theorem \ref{Cptwinfty-Structure-Theorem}. 
\end{corollary}

A DG algebra $A$ is said to satisfy Poincar\'{e} duality of degree $d\in \Z$
provided there is a weak equivalence $A\simeq \om{Hom}(A,R[d])$ of
$A-A$-bimodules, where the right and left actions on the target are 
induced by the left action on $A$ and the trivial left
action on $R[d]$, respectively. Here, weak equivalence means that
there exist a finite number $A=X_0,X_1,\cdots,X_r=\om{Hom}(A,R[d])$
of $A$-bimodules and for each $0\leq i<r$ there is a quasi-isomorphism
$X_i\ra X_{i+1}$ or $X_{i+1}\ra X_i$ of $A$-bimodules. 
According to \cite[Theorem~A.19]{Miller19} there is under this hypothesis, 
an isomorphism $H^{+,tw}_A(M)\cong H^+_A(M)[n]$ of $H^-_A(R)$-modules.
We will give a different proof of this result for $A=\Lambda_R[u]$ with
$|u|=3$. This case is sufficient for the application to equivariant instanton Floer homology when $\frac12 \in R$.  

We will now consider these constructions for $A=\Lambda_R[u]$ with $|u|=3$
as we did for $C^\pm_A$. The Tate complex $C^\infty_A(M)$, even in this simple case, is big and difficult to handle explicitly. To remedy this we will
construct a much simpler and very explicit complex that also computes
the Tate homology $H^\infty_A(M)$. Indeed, we will show that if
we take one of the double complexes $(D^\pm_{*,*},\dd',\dd'')$ of Proposition
\ref{Explicit-Cpm-Models} and extend them in the natural way to the
whole plane, then a suitable totalization of this complex will do the
trick.    

\begin{lemma} Let $A=\Lambda_R[u]$ with $|u|=3$. Then there is an isomorphism of left
$A$-modules $A\otimes R[\alpha]\cong D_A$ where $|\alpha|=-4$. Under this
isomorphism the differential on $A\otimes R[\alpha]$ is given by
$\dd(1\otimes \alpha^p)=-u\otimes \alpha^{p+1}$ and $\dd(u\otimes \alpha^p)=0$
for $p\geq 0$, and the action of $R[U]=C^-_A(R)$ is given by
$U\cdot (a\otimes \alpha^p)=a\otimes \alpha^{p+1}$ for $a\in A$ and
$p\geq 0$. 

Furthermore, the map $\rho\colon R\ra A\otimes R[\alpha]$ of degree $3$
given by $1\mapsto u\otimes \alpha^0$ is an $A$-linear quasi-isomorphism.
\end{lemma} 
\begin{proof} Using the explicit model of Theorem \ref{Explicit-Cpm-Models}
we see that 
\[ (D_A)_n = C^-_A(A)_n \cong \left\{ \begin{array}{cl} 
A_0 & \mbox{ for } n=4m, m\leq 0 \\
A_3 & \mbox {for } n=4m+3, m\leq 0 \\
0 & \mbox{ otherwise} \end{array} \right.  \]
and a quick check shows that these identifications are compatible with
the left $A$-module structure. The required isomorphism is therefore obtained
by letting $1\otimes \alpha^p$ and $u\otimes \alpha^p$ correspond to the 
generators $1\in A_0 \cong C^-_A(A)_{-4p}$ and $u\in A_3 =C^-_A(A)_{-4p+3}$, respectively, for $p\geq 0$. The same result shows that the differential
and $R[U]$-action are determined by the formulas $\dd(u\otimes \alpha^p)=0$, $\dd(1\otimes \alpha^p)=-u\otimes\alpha^{p+1}$ and $U\cdot (a\otimes \alpha^p)
=a\otimes \alpha^{p+1}$ for $a\in A$ and $p\geq 0$.

From this description we find that $\dd\colon (D_A)_{-4p}\ra (D_A)_{-4(p+1)-3}$
is an isomorphism for each $p\geq 0$. Therefore, the map
$\rho\colon R\ra A\otimes R[\alpha]$ of degree $3$ given by $1\mapsto u\otimes \alpha^0$ is an $A$-linear quasi-isomorphism. 
\end{proof}

In the following result we establish the isomorphism $H^{+,tw}_A(M)\cong H^+_A(M)[3]$ of $R[U]$-modules for each right $A$-module $M$. It also
contains the key ingredients needed to establish the simplified model for $C^\infty_A(M)$.  

\begin{lemma} \label{f-Homotopy-Lemma}
Let $A=\Lambda_R[u]$ with $|u|=3$ and let $M$ be a right
$A$-module. Then there is a quasi-isomorphism $f_M\colon C^+_A(M)\ra C^{+,tw}_A(M)$ of degree $3$ and a chain homotopy $s_M\colon C^+_A(M)\ra C^{+,tw}_A(M)$ of degree $0$ such that
\[ f_M\circ U-U\circ f_M = \dd \circ s_M-s_M\circ \dd  .\]
In particular, $H(f_M)\colon H^+_A(M)\ra H^{+,tw}_A(M)$ is a degree $3$
isomorphism of $R[U]$-modules. Furthermore, $f_M$ and $s_M$ are natural in the module $M$, and in terms
of the explicit models $C^+_A(M)\cong \om{Tot}^{\oplus}(D^+_{*,*},\dd',\dd'')$
and $C^-_A(M) \cong \om{Tot}^{\Pi}(D^-_{*,*},\dd',\dd'')$ of
Proposition \ref{Explicit-Cpm-Models} the compositions
$N_M\circ f_M$ and $N_M\circ s_M$ are given in degree $n$ by
\[ \begin{tikzcd} \bigoplus_{p\geq 0}M_{n-4p} \arrow{r}{\pi} & M_n \arrow{r}{(-1)^nu} & M_{n+3} \arrow{r}{\iota } & \prod_{q\geq 0} M_{n+3+4q} \\
\bigoplus_{p\geq 0}M_{n-4p} \arrow{r}{\pi} & M_{n} \arrow{r}{\iota}
& \prod_{q\geq 0} M_{n+4q}, & {} \end{tikzcd} \]
respectively.
\end{lemma}  
\begin{proof} Recall that $C^+_A(M)=B(M,A,R)$ and $C^{+,tw}_A(M)=B(M,A,D_A)$.
It is convenient to introduce the notation
\begin{align*}
m\otimes V^p \coloneqq & (-1)^{p(p+1)/2}m[u|u|\cdots |u] \in B(M,A,R)   \\ 
m\otimes V^p\otimes \psi \coloneqq & (-1)^{p(p+1)/2}m[u|u|\cdots |u]\psi \in B(M,A,D_A)
\end{align*}
for $m\in M$ and $\psi\in D_A$, where $u$ is repeated $p$ times in both formulas. In terms of this notation the differentials and actions of $U\in R[U]$ of $B(M,A,R)$ and $B(M,A,D_A)$ are given by the formulas
\begin{align*}
U(m\otimes V^p)  =& \; (-1)^{|m|}m\otimes V^{p-1}   \\   
U(m\otimes V^p\otimes \psi) =& \; m\otimes V^p\otimes (U\psi) \\
\dd(m\otimes V^p) =& \; (-1)^p(\dd_M m +mu \otimes V^{p-1}) \\
\dd(m\otimes V^p\otimes \psi) =& \; (-1)^p(\dd_Mm \otimes V^p+mu\otimes V^{p-1})\otimes \psi  \\
 & \; + m\otimes (V^{p-1}\otimes u\psi +(-1)^{|m|}V^p\otimes \dd_{D_A}\psi).
\end{align*}
By the previous lemma there is an isomorphism $D_A\cong A\otimes R[\alpha]$
with $|\alpha|=-4$. To simplify the notation write
$a\alpha^p\coloneqq a\otimes \alpha^p$ for $a\in A$. There is also an 
$A$-linear quasi-isomorphism $\rho:R\ra D_A$ of degree $3$ given by $\rho(1)=u\alpha^0$. This map induces a quasi-isomorphism 
\[ f_M \coloneqq B(1,1,\rho)\colon B(M,A,R)\ra B(M,A,D_A)  \] 
of degree $3$, which in terms of the above notation
is given by $f_M(m\otimes V^p) = (-1)^{|m|}m\otimes V^p\otimes u\alpha^0$
(the sign is forced upon us by the fact that we regard $\rho$ as a map
of degree $3$). Using the above formulas we find that
\begin{equation}  \label{fMU-equation}
(f_M\circ U-U\circ f_M)(m\otimes V^p) = m\otimes V^{p-1}\otimes u\alpha^0
-(-1)^{|m|}m\otimes V^p\otimes u\alpha^1  .
\end{equation}  
Define $s_M\colon B(M,A,R)\ra B(M,A,D_A)$ of degree $0$ by $s_M(m\otimes V^p)=m\otimes V^p\otimes \alpha^0$. A straightforward calculation using the above formulas and the formula for the differential in $D_A\cong A\otimes R[\alpha]$ given in the previous lemma shows that
\begin{align*}
\dd s_M(m\otimes V^p)  =& \; (-1)^p(\dd_M m\otimes V^p\otimes \alpha^0+mu\otimes V^{p-1}\otimes \alpha^0)  \\
& \; +  m\otimes V^{p-1}\otimes u\alpha^0 -(-1)^{|m|}m\otimes V^p\otimes u\alpha^1 \\
s_M\dd (m\otimes V^p)  =& (-1)^p(\dd_Mm\otimes V^p\otimes \alpha^0+mu\otimes V^{p-1}\otimes \alpha^0).
\end{align*} 
Combining this with \eqref{fMU-equation} we obtain $f_M\circ U-U\circ f_M=\dd s_M-s_M\dd$ as required. 

Finally, we need to consider the compositions of $f_M$ and $s_M$ with the
norm map $N_M\colon C^{+,tw}_A(M)\ra C^-_A(M)$ in terms of the models
given in Proposition \ref{Explicit-Cpm-Models}. By \eqref{Iso-+}
$m\in M_{q-3p}=D^+_{p,q}$ corresponds to $(-1)^{|m|p}m\otimes V^p$,
which is mapped by $f_M$ to $(-1)^{|m|(p+1)}m\otimes V^p\otimes u\alpha^0$.
The norm map kills off all such elements with $p>0$ and sends
$m\otimes V^0\otimes u\alpha^0\mapsto mu\cdot \alpha^0$. This is the
functional that sends $V^0$ to $mu\in M_{q+3}$ and vanishes on $V^q$ for
$q>0$. By \eqref{Iso--} the identification $C^-_A(M)_n\cong \prod_{q\geq 0}M_{n+4q}$ is given by $\psi\mapsto (\psi(V^q))_{q\geq 0}$. We conclude that $N_M\circ f_M$ is given by the composition in the statement of the lemma. The
verification for $N_M\circ s_M$ is analogous.  
\end{proof} 
 
We need two simple but slightly technical results before we can reach
our goal.  

\begin{lemma} \label{Cone-Lemma1}
 Let $X$ and $Y$ be DG $R[U]$-modules and suppose that
$f\colon X\ra Y$ is a degree $d$ map of $R[U]$-modules up to homotopy; that
is, $f$ is a chain map of degree $d$ and there is a chain homotopy
$s\colon X\ra Y$ such that $fU-Uf = \dd_Ys+(-1)^{d}s\dd_X$. Let $\om{Cone}(f)$ be given
as in Definition \ref{Cone-Definition}, except that we redefine the
$R[U]$-module structure by setting
\[ U\cdot \left( \begin{array}{c} y \\ x \end{array} \right) =
\left( \begin{array}{cc} U & (-1)^{|s|}s \\ 0 & U \end{array} \right)
\left(\begin{array}{c} y \\ x \end{array} \right) =
\left( \begin{array}{c} Uy-(-1)^dsx \\ Ux \end{array} \right) \]
for all $(y,x)\in Y_n\oplus X_{n-d-1}=\om{Cone}(f)_n$. 
Then $\om{Cone}(f)$ is a DG $R[U]$-module and the maps in the natural
short exact sequence
\[ \begin{tikzcd} 0 \arrow{r} & Y \arrow{r} & \om{Cone}(f) \arrow{r} &
X[d+1] \arrow{r} & 0 \end{tikzcd}  \]
are (strict) maps of $R[U]$-modules.
\end{lemma}
\begin{proof} It suffices to check that $\dd_{\om{Cone}(f)}U=U\dd_{\om{Cone}(f)}$. This follows easily from the relations $\dd_XU=U\dd_X$, $\dd_YU=U\dd_Y$
and $fU-Uf = \dd_Ys+(-1)^{d}s\dd_X$. Next, the maps in the natural
exact sequence are given degreewise by the inclusion 
$Y_n\ra Y_n\oplus X_{n-d-1}$ and the projection $Y_n\oplus X_{n-d-1}\ra X_{n-d-1}$ and are seen to commute with the action of $U$.
\end{proof}   

\begin{lemma} \label{Cone-Lemma2}
 Let $f\colon X\ra Y$ be a degree $3$ map of DG $R[U]$-modules up
to homotopy, and let $g\colon Y\ra Z$ be a degree $0$
map of DG $R[U]$-modules. Then $g\circ f:X\ra Z$ is a degree $3$ map of 
$R[U]$-modules up to homotopy and there is a commutative diagram
\[ \begin{tikzcd}  0 \arrow{r} & Z \arrow{r} \arrow{d}{=} & \om{Cone}(g\circ f) \arrow{r} \arrow{d}{\om{Cone}(f,1)} &
X[4] \arrow{d}{f[1]} \arrow{r} & 0 \\
0 \arrow{r} & Z \arrow{r} & \om{Cone}(g) \arrow{r} & Y[1] \arrow{r} & 0
\end{tikzcd} \]
with short exact rows. If $\om{Cone}(g\circ f)$ is given the DG $R[U]$-module
structure of the previous lemma, then the maps $\om{Cone}(f,1)$ and $f[1]$ are maps of DG $R[U]$-modules up to homotopy and all the other maps in the diagram
are strict maps of DG $R[U]$-modules. \end{lemma} 
\begin{proof} By assumption there is a chain homotopy $s\colon X\ra Y$ of
degree $0$ such that $fU-Uf = \dd s-s\dd$. Then for $h\coloneqq g\circ f$
and $t\coloneqq g\circ s$ we have $hU-Uh = \dd t-t\dd$ as well. This shows that $h$ is a map of $R[U]$-modules up to homotopy. By functoriality of the cone
construction we obtain a commutative diagram of DG modules with
short exact rows
\[ \begin{tikzcd}  0 \arrow{r} & Z \arrow{r} \arrow{d}{=} & \om{Cone}(g\circ f) \arrow{r} \arrow{d}{\om{Cone}(f,1)} &
X[4] \arrow{d}{f[1]} \arrow{r} & 0 \\
0 \arrow{r} & Z \arrow{r} & \om{Cone}(g) \arrow{r} & Y[1] \arrow{r} & 0.
\end{tikzcd} \]
Here, $\om{Cone}(f,1)\colon \om{Cone}(g\circ f)_n = Z_n\oplus X_{n-4}\ra Z_n\oplus Y_{n-1}=\om{Cone}(g)_n$ is given by $(z,x)\mapsto (z,-fx)$, and
$f[1]\colon X[4]_n = X_{n-4}\ra Y_{n-1}=Y[1]_n$ is given by $x\mapsto -fx$.
Equip $\om{Cone}(g\circ f)$ with the adjusted $R[U]$-module structure
given in the previous lemma. Then all the horizontal maps in the
diagram are strict maps of $R[U]$-modules. The right hand vertical map
is a map of $R[U]$-modules up to homotopy, so the only thing
we have to verify is that $\om{Cone}(f,1)$ is a map of $R[U]$-modules
up to homotopy. 

We claim that $v\colon \om{Cone}(h)\ra \om{Cone}(g)$ given by $v(z,x)=(0,sx)$
is the required homotopy. Indeed, using matrix notation for the various
maps, 
\begin{align*}
C(f,1)U-UC(f,1) &= \left(
\begin{array}{cc} 1& 0 \\ 0 & -f \end{array} \right)\left( \begin{array}{cc}
U & t \\ 0 & U \end{array} \right) -\left( \begin{array}{cc}
U & 0 \\ 0 & U \end{array} \right)\left(
\begin{array}{cc} 1& 0 \\ 0 & -f \end{array} \right) \\
&= \left( \begin{array}{cc} 0 & t \\ 0 & Uf-fU \end{array} \right) \\
\dd_{\om{Cone}(g)}v+v\dd_{\om{Cone}(h)} &=
\left( \begin{array}{cc} \dd_Z & g \\ 0 & -\dd_Y \end{array} \right)
\left( \begin{array}{cc} 0 & 0 \\ 0 & s \end{array} \right) +
\left( \begin{array}{cc} 0 & 0 \\ 0 & s \end{array} \right)
\left( \begin{array}{cc} \dd_Z & -h \\ 0 & \dd_X \end{array} \right) \\
&= \left( \begin{array}{cc} 0 & g\circ s \\ 0 & s\dd_X-\dd_Ys
\end{array} \right) , 
\end{align*} 
where $C(f,1)$ is shorthand for $\om{Cone}(f,1)$. Therefore, 
as $t= g\circ s$ and $\dd_Y s-s\dd_X s = fU-Uf$, we deduce that
$\om{Cone}(f,1)$ is $R[U]$-linear up to homotopy. 
\end{proof}  

Given a sequence of $R$-modules $\{D_s\}_{s\in \Z}$ we write
$\prod_{s\to -\infty} D_s$ for the submodule of $\prod_s D_s$ consisting
of the sequences $(x_s)_{s\in \Z}$ for which there exists $s_0$ such that $x_s=0$ for all $s\geq s_0$. Similarly, $\prod_{s\to\infty}D_s$ is
the submodule of $(x_s)_{s\in \Z}$ with $x_s=0$ for all sufficiently
small $s$. 

\begin{definition} \label{Totalization-Def}
Given a double complex $E=(D_{*,*},\dd',\dd'')$, define
$\om{Tot}^{\Pi,-\infty}(E)$ and $\om{Tot}^{\Pi,+\infty}(E)$ to be
the complexes that in degree $n$ are given by
\[ \prod_{s\to -\infty} D_{s,n-s} \; \mbox{ and } \; \prod_{s\to\infty}D_{s,n-s}  \]
respectively, with differential $\dd=\dd'+\dd''$. 
\end{definition} 

\begin{theorem} \label{Explicit-Cinfty-model}
Let $A=\Lambda_R[u]$ with $|u|=3$ and let $M$ be a right
DG $A$-module. Define $D^\infty_{p,q}=M_{q-3p}$ for all $p,q\in \Z$. Set
\begin{align*}
 \dd'=(-1)^{p+q+1}u\colon D^\infty_{p,q}&=M_{q-3p}\ra M_{q-3p+3}=D^\infty_{p-1,q}  \\
\dd''=\dd_M\colon D^\infty_{p,q}&=M_{q-3p}\ra M_{q-3p-1}=D^\infty_{p,q-1}
\end{align*}
and let $U\colon D^\infty_{p,q}=M_{q-3p}\ra M_{q-3p}=D^\infty_{p-1,q-3}$ be
the identity for all $p,q$. Then there is a quasi-isomorphism, natural in $M$,
\[ \om{Tot}^{\Pi,-\infty}(D_{*,*},\dd',\dd'')\ra C^\infty_A(M),  \]
which is a homomorphism of $R[U]$-modules up to homotopy.
\end{theorem}
\begin{proof} By Lemma \ref{f-Homotopy-Lemma} there is a quasi-isomorphism
$f=f_M\colon C^+_A(M)\ra C^{+,tw}_A(M)$ and a homotopy $s=s_M\colon C^+_A(M)\ra C^{+,tw}_A(M)$ such that $fU-Uf=\dd s-s\dd$. We apply Lemma \ref{Cone-Lemma2} for
$f\colon C^+_A(M)\ra C^{+,tw}_A(M)$ and the norm map 
$N_M\colon C^{+,tw}_A(M)\ra C^-_A(M)$ to obtain the commutative
diagram with exact rows
\begin{equation} \label{Simplifying-Tate-Diagram}
\begin{tikzcd} 0\arrow{r} & C^-_A(M) \arrow{d}{=} \arrow{r} & \om{Cone}(N_M\circ f) \arrow{d}{\om{Cone}(f,1)} \arrow{r} & C^+_A(M)[4] \arrow{d}{f[1]}\arrow{r}  & 0 \\
0 \arrow{r} & C^-_A(M) \arrow{r} & C^\infty_A(M) \arrow{r} & C^{+,tw}_A(M)[1]
\arrow{r} & 0, \end{tikzcd}
\end{equation} 
where we recall that $C^\infty_A(M)=\om{Cone}(N_M)$. By the same result
$\om{Cone}(f,1)$ is a map of $R[U]$-modules up to homotopy. As both the
outer vertical maps are quasi-isomorphisms, it follows by passage to the long exact sequences in homology and the $5$-lemma that $\om{Cone}(f,1)$
is a quasi-isomorphism. Thus, $\om{Cone}(f,1)$ is the required map.
This quasi-isomorphism is natural in $M$, since both $f$ and $N_M$ are
natural transformations in $M$ and the cone construction is functorial.

The remaining task is to establish an isomorphism 
$\om{Cone}(N_M\circ f)\cong \om{Tot}^{\Pi,-\infty}(D^\infty_{*,*},\dd',\dd'')$. 
Using the identifications of Proposition \ref{Explicit-Cpm-Models} we obtain 
\[ \om{Cone}(N_M\circ f)_n = C^-_A(M)_n\oplus C^+_A(M)_{n-4} \cong \left( \prod_{q\geq 0} M_{n+4q} \right) \oplus \left( \bigoplus_{p\geq 0} M_{n-4(p+1)} \right) \]
for each $n\in \Z$. Moreover, by Lemma \ref{f-Homotopy-Lemma} $h\coloneqq N_M\circ f$ and $t\coloneqq N_M\circ s$ are given on $(x_p)_{p\geq 0}\in C^+_A(M)_n$ by
$h((x_p)_p)=((-1)^{|x_0|}x_0u,0,\cdots )$ and $t((x_p)_{p\geq 0})=(x_0,0,0,\cdots)$. Therefore, using the formulas
\[ \dd_{\om{Cone}(h)} = \left( \begin{array}{cc} \dd_{C^-_A(M)} & -h \\ 
0 & \dd_{C^+_A(M)} \end{array} \right) \; \mbox{ and } \; 
U = \left( \begin{array}{cc} U & t \\ 0 & U \end{array} \right), \]
as well as the description of the differentials and the $U$-action in terms
of the models for $C^\pm_A(M)$, we deduce that
$\om{Cone}(h) \cong \om{Tot}^{\Pi,-\infty}(D^\infty_{*,*},\dd',\dd'')$.    
\end{proof}

A part of the double complex $(D^\infty_{*,*},\dd',\dd'')$ is shown
below.
\[ \begin{tikzcd} 
{} & \vdots \arrow{d} & \vdots \arrow{d} & \vdots \arrow{d} & \vdots \arrow{d} & {}  \\
\cdots & M_8 \arrow[dotted]{dddl} \arrow{l} \arrow{d}{\dd_M} & M_{5} \arrow[dotted]{dddl} \arrow{l}{u} \arrow{d}{\dd_M}
& M_{2} \arrow{l}{-u} \arrow[dotted]{dddl} \arrow{d}{\dd_M} & M_{-1} \arrow[dotted]{dddl} \arrow{l}{u} \arrow{d}{\dd_M} & \cdots \arrow{l} \\
\cdots & M_7 \arrow{l} \arrow{d}{\dd_M} & M_{4} \arrow{d}{\dd_M} \arrow{l}{-u}
& M_{1} \arrow{d}{\dd_M} \arrow{l}{u} & M_{-2} \arrow{l}{-u} \arrow{d}{\dd_M} & \cdots \arrow{l} \\
\cdots & M_6 \arrow{l} \arrow{d}{\dd_M} & M_{3} \arrow{l}{u} \arrow{d}{\dd_M}
& M_{0} \arrow{l}{-u} \arrow{d}{\dd_M} & M_{-3} \arrow{l}{u} \arrow{d}{\dd_M} & \cdots \arrow{l} \\
\cdots & M_{5} \arrow{l} \arrow{d} & M_{2} \arrow{l}{-u} \arrow{d} &
M_{-1} \arrow{l}{u} \arrow{d} & M_{-4} \arrow{l}{-u} \arrow{d} & \cdots \arrow{l} \\
{} & \vdots & \vdots & \vdots & \vdots & {} \end{tikzcd}. \]
This is indeed the natural concatenation of the two double complexes
$D^+_{*,*}$ and $D^-_{*,*}$. 

\begin{corollary} Let $M$ be a right $A=\Lambda_R[u]$-module. Then
$U\colon H^\infty_A(M)\ra H^\infty_A(M)$ is an isomorphism.
\end{corollary} 
\begin{proof} $U\colon \om{Tot}^{\Pi,-\infty}(D^\infty_{*,*},\dd',\dd'')\ra\om{Tot}^{\Pi,-\infty}(D^\infty_{*,*},\dd',\dd'')$ is an isomorphism. \end{proof}

As in the cases $C^\pm_A$ we will need calculations for
$M=R$, $M=R\oplus R[2]$ and $M=\Lambda_R[u]$. This is easily achieved with 
the help of the above explicit model.

\begin{corollary}\label{Orbit-Calc-Tate-Corollary}
We have
\[ H^\infty_A(A)=0, \;\;  H^\infty_A(R) = R[T,T^{-1}] \; \mbox{ and } \;
H^\infty_A(R\oplus R[2])=R[S,S^{-1}],   \]
where $|T|=-4$, $|S|=-2$ and the $R[U]$-action is given by $U\cdot T^i = T^{i+1}$ and $U\cdot S^i=S^{i+2}$ for all $i\in \Z$. 
\end{corollary}

We note once again that our notation should not be interpreted to mean that
there is any product structure in $H^\infty_A(R\oplus R[2])$ or $H^\infty_A(R)$.      

\subsection{The Equivariant Instanton Floer Groups}
In this section we will give the definition of the various flavors
of equivariant instanton Floer homology. Before we can get to this we
have to recall a few definitions concerning filtered DG modules.
Let $M$ be a $DG$ module equipped with an increasing filtration
\[ \cdots \subset F_{p-1}M \subset F_pM \subset F_{p+1}M \subset \cdots
\subset M  \]
by DG submodules. The minimal filtration quotients $F_pM/F_{p-1}M$
enter many times in the following discussion, so we introduce the
simpler notation $\overline{F}_pM \coloneqq F_pM/F_{p-1}M$. 
The filtration gives rise to a homological spectral sequence with
\begin{equation} \label{FilteredSpectralSequence}
 E^1_{p,q} = H_{p+q}(\overline{F}_pM) \implies H_{p+q}(M) 
\end{equation}
(we will elaborate on this in the next section). However, without any additional information, this spectral sequence may fail to converge in any sense to the desired target $H(M)$. 

The following properties of a filtration are highly desirable in order to achieve any reasonable convergence properties of the above spectral sequence. 
The filtration $\{F_pM\}_p$ of $M$ is said to be \cite[Definition~2.1]{Boardman99}
\begin{enumerate}[label=(\roman*),ref=(\roman*)]
\item exhaustive if $\om{colim}_p F_pM = \bigcup_p F_pM = M$,
\item Hausdorff if $\om{lim}_p F_pM = \bigcap_p F_pM=0$ and
\item complete if $\om{Rlim}_p F_pM = 0$. 
\end{enumerate} 
Here $\om{Rlim}$ is the first (and only nonzero) right derived functor of 
$\om{lim}$ (see \cite[Section~1]{Boardman99} for additional details).
It is important to note that these limits are calculated degreewise. 
In particular, if the filtration is degreewise bounded below, that is,
for each degree $n$ there exists $p_0(n)$ such that $F_pM_n=0$ for
all $p\leq p_0$, then the filtration is automatically complete Hausdorff.
Similarly, if the filtration is degreewise bounded above then the filtration
is automatically exhaustive. 

\begin{definition}\label{Filtered-Completion} Let $M$ be a DG module
equipped with an increasing filtration $\{F_pM\}_p$ by DG submodules.
Then the full completion of $(M,\{F_pM\}_p)$ is defined to be the complex
\[ \hat{M} \coloneqq \om{lim}_q \om{colim}_p F_pM/F_qM   \]
equipped with the filtration $F_p\hat{M} \coloneqq \om{lim}_q F_pM/F_qM$.
\end{definition}

\begin{remark} The standard completion of a filtered module $M$ is
typically defined by $\om{lim}_q M/F_qM$ (see for instance \cite[2.7]{Boardman99}) and is naturally filtered by $\{\om{lim}_q F_pM/F_qM\}_p$. We have therefore chosen to use the terminology full
completion to distinguish it from this notion. The lemma below demonstrates
that the two notions coincide if the filtration of $M$ is exhaustive.  
\end{remark} 

A map $f\colon M\ra M'$ between filtered DG modules is filtered if
$f(F_pM)\subset F_pM'$ for all $p\in \Z$. There is therefore
a category of filtered DG modules and filtered maps. A sequence
$M'\ra M\ra M''$ of filtered DG modules and filtered maps is short
exact if $F_pM'\ra F_pM\ra F_pM''$ is short exact for each $p\in \Z$.   
The relevant properties of the full completion are summarized in the
following lemma. 

\begin{lemma} \label{Filtration-Lemma}
Let $M$ be a filtered DG module. Then the following holds true.  
\begin{enumerate}[label=(\roman*),ref=(\roman*)]
\item The filtration $F_p\hat{M}=\om{lim}_{q<p} F_pM/F_qM$ of $\hat{M}$ is exhaustive
and complete Hausdorff.
\item For each $q<p$ there is a natural isomorphism $F_pM/F_qM\ra F_p\hat{M}/F_q\hat{M}$. 
\item If the filtration of $M$ is exhaustive, then
$\hat{M} = \om{lim}_q M/F_qM$. 
\item If the filtration of $M$ is complete Hausdorff, then
$\hat{M}=\om{colim}_p F_pM = \bigcup_p F_pM$.
\end{enumerate}
The full completion $M\mapsto \hat{M}$ defines a functor from the
category of filtered modules to itself, which preserves short exact
sequences.  
\end{lemma}
\begin{proof} Let $M$ be a filtered DG module and set $N=\om{colim}_p F_pM=\bigcup_p F_pM$ filtered by $F_pN = F_pM$. Then the full
completion of $M$ coincides with the standard completion of $N$, i.e., 
$\hat{M} = \om{lim}_q N/F_qN$. Therefore part (i) and (ii) follow from
the corresponding properties of the standard completion (see \cite[Proposition~2.8]{Boardman99}). If the filtration of $M$ is exhaustive,
then $\om{colim}_p F_pM/F_qM = M/F_qM$ for each $q$ and part (iii) easily follows. For part (iv) 
\[ \hat{M} = \om{lim}_q\om{colim}_p F_pM/F_qM = \om{lim}_q N/F_qN
=N = \om{colim}_p F_pM .\]

It is clear that $M\mapsto \hat{M}$ is functorial with respect to
filtered maps. The exactness of $\hat{M}$ follows from the exactness
of sequential colimits and exactness of the standard completion
in the category of $R$-modules (see for instance \cite[Corollary~10.3]{Atiyah-Macdonald}).   
\end{proof} 

\begin{remark}
If $B$ is a DG algebra and $M$ is a left or right $B$-module equipped
with a filtration by $B$-submodules the definition of the full completion
and the above lemma apply without change. The reason for this
is that quotients, limits and colimits are formed in the category
of graded modules, with the $B$-action and differentials carried along.
\end{remark} 

Recall from the previous section that for a degreewise $R$-free
augmented DG algebra $A$, four functors $C^\bullet_A$, $\bullet\in \{+,-,(+,tw),\infty\}$, from the category of right $A$-modules 
to the category of left $C^-_A(R)$-modules were defined. By Theorem 
\ref{Cpm-Structure-Theorem} and Theorem
\ref{Cptwinfty-Structure-Theorem} these are all exact and preserve quasi-isomorphisms. Therefore, if $M$ is a right $A$-module equipped with
a filtration $F_pM$ by $A$-submodules, $C^\bullet_A(M)$ carries
a natural filtration $F_pC^\bullet_A(M)\coloneqq  C^\bullet_A(F_pM)$ by
$C^-_A(R)$-submodules. Now, even if the initial filtration of
$M$ is exhaustive and complete Hausdorff, the resulting filtration
of $C^\bullet_A(M)$ may fail to satisfy these properties. This
motivates the following definition.

\begin{definition} Let $A$ be a degreewise $R$-free augmented DG algebra.
For $\bullet\in \{+,-,(+,tw),\infty\}$ define a functor 
$\hat{C}^\bullet_A$ from the category of filtered right $A$-modules
to the category of filtered left $C^-_A(R)$-modules by
\begin{align*}
\hat{C}^\bullet_A(M) \coloneqq & \om{lim}_q \om{colim}_p C^\bullet_A(F_pM/F_qM) 
\;\; \mbox{ filtered by } \; \\
F_p\hat{C}^\bullet_A(M)\coloneqq &\om{lim}_{q<p} C^\bullet_A(F_pM/F_qM) 
\;\;\;\;\;\;\;\;\; \mbox{ for all} \; p\in \Z.
\end{align*}
In other words, $\hat{C}^\bullet_A$ is the composition of the natural
extension of $C^\bullet_A$ to filtered $A$-modules followed by the
full completion functor. We denote the homology of $\hat{C}^\bullet_A(M)$
by $\hat{H}^\bullet_A(M)$. 
\end{definition} 

The various flavors of equivariant instanton Floer homology are
obtained by applying this construction to the framed Floer
complex. 

\begin{definition} \label{Def-Equivariant-Floer}
Let $A=C^{gm}_*(\om{SO}(3))$ and let $M=\wt{CI}(Y,E)$
be the framed Floer complex associated with an $\om{SO}(3)$-bundle $E$ over a
closed oriented three-manifold $Y$, equipped with the index filtration. 
For $\bullet\in \{+,-,(+,tw),\infty\}$ define
\[ CI^{\bullet}(Y,E) \coloneqq \hat{C}^\bullet_A(M) \;\; \mbox{ and } \;\;
I^\bullet(Y,E) \coloneqq H(CI^{\bullet}(Y,E)). \]
The groups $I^\bullet(Y,E)$, for $\bullet\in \{+,-,(+,tw),\infty\}$, are called the positive, negative, positive twisted and Tate equivariant
instanton Floer homologies of $(Y,E)$, respectively. These homology groups are modules over $H^-_A(R)$. 

For $A=\Lambda_R[u]$ and $M=DCI(Y,E)$ we similarly define the complexes
\[ DCI^\bullet(Y,E)= \hat{C}^\bullet_A(M) .\]
\end{definition}

\begin{remark} For $A=C_*(\om{SO}(3);R)$ and $M=R$ it holds true that
$H^-_A(R)\cong H^{-*}(\om{BSO}(3);R)$ (see \cite[Lemma~A.21]{Miller19}).
\end{remark}

The norm map $N_M\colon C^{+,tw}_A(M)\ra C^-_A(M)$ is a natural in $M$ and
therefore extends to a natural map $\hat{N}_M\colon \hat{C}^{+,tw}_A(M)\ra \hat{C}^-_A(M)$ for any filtered module $M$. The following simple lemma
ensures that $\hat{C}^\infty_A(M)$ is naturally isomorphic to
$\om{Cone}(\hat{N}_M)$.

\begin{lemma} \label{Completed-Exact-Triangle-Lemma}
Let $A$ be a degreewise $R$-free augmented DG algebra
and let $M$ be a filtered right $A$-module. Then there is a natural
(filtered) isomorphism of $C^-_A(R)$-modules
\[ \om{Cone}(\hat{N}_M) \cong \hat{C}^\infty_A(M) . \]
In particular, we have the following exact triangle of $H^-_A(R)$-modules. 
\[ \begin{tikzcd} \hat{H}^{+,tw}_A(M) \arrow{rr}{H(\hat{N}_M)} & {} & \hat{H}^-_A(M) \arrow{ld}{[0]} \\
 {} & \hat{H}^\infty_A(M) \arrow{ul}{[-1]} & {} \end{tikzcd} \]
\end{lemma}
\begin{proof} For $q<p$ and $n$ there is a natural identification
\[ \frac{F_pC^\infty_A(M)}{F_qC^\infty_A(M)}_n \cong 
 C^\infty_A \left( \frac{F_pM}{F_qM} \right)_n =
 C^-_A\left( \frac{F_pM}{F_qM} \right)_n \oplus C^{+,tw}_A\left( \frac{F_pM}{F_qM} \right)_{n-1} . \]
In terms of the right hand side the differential is given by
\[ \left( \begin{array}{cc} \dd^-_{p,q} &  N_{p,q} \\
                             0 & -\dd^{+,tw}_{p,q}  \end{array} \right),  \]
where $\dd^-_{p,q}$ and $\dd^{+,tw}_{p,q}$ denote the differential
in the first and second summand respectively and $N_{p,q}$ denotes the
evident norm map. Applying $\om{lim}_q\om{colim}_p$ to this identification,
taking into account that this is calculated degreewise with differentials
and $A$-module structure carried along, we obtain
$\hat{C}^\infty_A(M)_n \cong \hat{C}^-_A(M)_n\oplus \hat{C}^{+,tw}_A(M)_{n-1}$
in each degree $n$. The differential on the right hand side is
given by
\[ \left( \begin{array}{cc} \om{lim}_q\om{colim}_p \dd^-_{p,q} & 
\om{lim}_q\om{colim}_p N_{p,q} \\ 0 & \om{lim}_q\om{colim}_p \dd^{+,tw}_{p,q} \end{array} \right). \]
Therefore, the right hand side is precisely $\om{Cone}(\hat{N}_M)$, since
by definition $\hat{N}_M = \om{lim}_q\om{colim}_p N_{p,q}$.  
\end{proof} 
  
To verify that $DCI^\bullet(Y,E)$ also calculates $I^\bullet(Y,E)$ when
$\frac12\in R$ we will establish the necessary invariance results for the functors $\hat{C}^\bullet_A$. The appropriate tool is the following lemma, which is a simple consequence of the Eilenberg-Moore comparison theorem of \cite[Theorem~5.5.11]{Weibel94}. For a filtered map $f\colon M\ra N$ between
filtered DG modules we use the notation $\overline{f}_p\colon \overline{F}_pM\ra \overline{F}_pN$ for the induced maps between the the minimal
filtration quotients. 

\begin{lemma} \label{Eilenberg-Moore-Comparison-Lemma}
Let $M$ and $N$ be DG modules equipped with
exhaustive and complete Hausdorff filtrations. Then if $f\colon M\ra N$
is a filtered map such that the induced map
$\overline{f}_p\colon \overline{F}_pM\ra \overline{F}_pN$ is a quasi-isomorphism for each $p\in \Z$, then $f$ is a quasi-isomorphism.  
\end{lemma}

\begin{proposition} \label{Completed-Invariance-Module}
Let $A$ be a degreewise $R$-free augmented DG algebra.
Suppose $f\colon M\ra N$ is a filtered map of filtered right 
$A$-modules such that
$\overline{f}_p\colon \overline{F}_pM\ra \overline{F}_pN$ is a quasi-isomorphism for each $p\in \Z$. Then for $\bullet\in \{+,-,(+,tw),\infty\}$ the
induced map
\[ \hat{f}^\bullet\colon \hat{C}^\bullet_A(M)\ra \hat{C}^\bullet_A(N)  \]
is a (filtered) quasi-isomorphism of $C^-_A(R)$-modules. Furthermore, upon
passage to homology, these isomorphisms are compatible with the
exact triangle of Lemma \ref{Completed-Exact-Triangle-Lemma}.
\end{proposition} 
\begin{proof} Let $\bullet\in \{+,-,(+,tw),\infty\}$ and note that
the functors $C^\bullet_A$ and $\hat{C}^\bullet_A$ are covariant
in $M$. For simplicity of notation let $g=\hat{f}^\bullet\colon \hat{C}^\bullet_A(M)\ra \hat{C}^\bullet_A(N)$ denote the induced filtered map of $C^-_A(R)$-modules. Then by Lemma \ref{Filtration-Lemma} part (ii) and the exactness of $C^\bullet_A$, there is for each $p$ a commutative diagram
\[ \begin{tikzcd} C^\bullet_A(\overline{F}_pM) \arrow{r}{(\overline{f}_p)^\bullet} \arrow{d}{\cong} & C^\bullet_A(\overline{F}_pN) \arrow{d}{\cong}\\
\overline{F}_p\hat{C}^\bullet_A(M) \arrow{r}{\overline{g}_p} & 
\overline{F}_p\hat{C}^\bullet_A(N). \end{tikzcd} \]
By assumption $\overline{f}_p$ is a quasi-isomorphism, so since the functor
$C^\bullet_A$ preserves quasi-isomorphisms it follows that $\overline{f}_p^\bullet$ is a quasi-isomorphism. As the vertical maps are isomorphisms,
we conclude that $\overline{g}_p$ is a quasi-isomorphism. By Lemma
\ref{Filtration-Lemma} part (i) the filtrations of $\hat{C}^\bullet_A(M)$
and $\hat{C}^\bullet_A(N)$ are exhaustive and complete Hausdorff. We
conclude that $g=\hat{f}^\bullet$ is a quasi-isomorphism using the comparison result; Lemma \ref{Eilenberg-Moore-Comparison-Lemma}.
The lemma applies equally well to the filtered subcomplexes. Consequently,
the restrictions $g_p\colon F_p\hat{C}^\bullet_A(M)\ra F_p\hat{C}^\bullet(N)$
are also quasi-isomorphisms for all $p$. This shows that
$g$ is in fact a filtered quasi-isomorphism. 
\end{proof} 

Next we have to establish a similar invariance result for variations in
the DG algebra.

\begin{proposition} \label{Completed-Invariance-Algebra}
 Let $f\colon A\ra B$ be a quasi-isomorphism of degreewise
$R$-free augmented algebras and let $M$ be a filtered $B$-module.
Then there is for each $\bullet\in \{+,-,(+,tw),\infty\}$ a natural 
isomorphism of $H^-_A(R)\cong H^-_B(R)$-modules
\[ \hat{H}^\bullet_A(f^{-1}M)\cong \hat{H}^\bullet_B(M),  \]
compatible with the exact triangles of Lemma \ref{Completed-Exact-Triangle-Lemma}.
\end{proposition} 
\begin{proof} We first consider the case $\bullet=+$. For each pair
$q<p$ there is by Proposition \ref{Invariance} a natural quasi-isomorphism 
\[ f_*\colon C^+_A(F_pf^{-1}M/F_qf^{-1}M)\ra C^+_B(F_pM/F_qM)  \]
linear along $f^*\colon C^-_B(R)\ra C^-_A(R)$. By applying $\om{lim}_q\om{colim}_p$
we obtain a filtered map $g=\hat{f}_*\colon \hat{C}^+_A(f^{-1}M)\ra \hat{C}^+_B(M)$
linear along $f^*:C^-_B(R)\ra C^-_A(R)$. 
By construction, the following diagram commutes for each $p$.
\[ \begin{tikzcd} C^+_A(\overline{F}_pf^{-1}M) \arrow{r}{\overline{(f_*)}_p}
\arrow{d}{\cong} & C^+_B(\overline{F}_pM) \arrow{d}{\cong} \\
\overline{F}_p\hat{C}^+_A(f^{-1}M) \arrow{r}{\overline{g}_p} &
\overline{F}_p\hat{C}^+_B(M) \end{tikzcd} \]
As the upper horizontal map is a quasi-isomorphism, so is the lower
horizontal map and we conclude that $g=\hat{f}_*$ is a filtered
quasi-isomorphism by Lemma \ref{Eilenberg-Moore-Comparison-Lemma}.
For $\bullet=-$ there is for each pair $q<p$ a natural quasi-isomorphism
\[ f^*\colon C^-_B(F_pM/F_qM)\ra C^-_A(F_pf^{-1}M/F_qf^{-1}M).  \]
By passing to the colimit and limit we obtain a filtered map
$\hat{f}^*\colon \hat{C}^-_B(M)\ra \hat{C}^-_A(f^{-1}M)$, linear along
$f^*\colon C^-_B(R)\ra C^-_A(R)$. The same argument as above shows that
this is a filtered quasi-isomorphism. 

For the two remaining cases recall that there is a relative functor
$C^{+,tw}_f$, see Definition \ref{Def-Relative-twisted}, from the category of right $B$-modules to the category of left $C^-_A(R)$-modules. By Lemma
\ref{Relative-Twisted-Lemma} this functor is exact and preserves quasi-isomorphism, so we may promote it to a functor $\hat{C}^{+,tw}_f$
between the corresponding filtered categories as we did with the others.
By passing to the colimit and then the limit in the commutative diagram
of Proposition \ref{Refined-Invariance} applied to $F_pM/F_qM$
for all $q<p$, we obtain a commutative diagram 
\[ \begin{tikzcd} \hat{C}^{+,tw}_B(M) \arrow{r}{\hat{\beta}_M} \arrow{d}{\hat{N}_M^B} & 
\hat{C}^{+,tw}_f(M) \arrow{d}{\hat{N}^f_M} & \hat{C}^{+,tw}_A(f^{-1}M) \arrow{d}{\hat{N}^A_M} 
\arrow{l}[swap]{\hat{\alpha}_M} \\
\hat{C}^-_B(M) \arrow{r}{\hat{f}^*} & \hat{C}^-_A(f^{-1}M) & \hat{C}^-_A(f^{-1}M) \arrow{l}[swap]{=}
\end{tikzcd} \]
of filtered DG modules over $C^-_A(R)$ or $C^-_B(R)$ as appropriate.
Furthermore, all the maps are filtered and $C^-_A(R)$-linear, $C^-_B(R)$-linear
or linear along $f^*\colon C^-_B(R)\ra C^-_A(R)$ as appropriate.
By the same proposition, as $f\colon A\ra B$ is a quasi-isomorphism of degreewise free DG algebras, all the horizontal arrows in the diagram applied
to $F_pM/F_qM$ are quasi-isomorphisms. This implies, by the now familiar
argument, that all the horizontal arrows in the above diagram are
filtered quasi-isomorphisms. In particular, there is a natural isomorphism
$\hat{H}^{+,tw}_A(f^{-1}M)\cong \hat{H}^{+,tw}_B(M)$. Moreover, we
obtain a zigzag of quasi-isomorphisms between the cones of the vertical maps
\[ \begin{tikzcd} \om{Cone}(\hat{N}^B_M) \arrow{r} & \om{Cone}(\hat{N}^f_M)
& \om{Cone}(\hat{N}^A_{f^{-1}M}) \arrow{l} \end{tikzcd}. \]
In view of Lemma \ref{Completed-Exact-Triangle-Lemma}, we obtain an
isomorphism $\hat{H}^\infty_A(f^{-1}M)\cong \hat{H}^\infty_B(M)$ linear along $H^-_A(R)\cong H^-_B(R)$ as well. The above
diagram also ensures that these isomorphisms are compatible with the corresponding exact triangles. 
\end{proof} 

The key application of these invariance results is the following.
 
\begin{corollary} \label{DCI-CI-Equivariant-Eq-Corollary}
Suppose that $\frac12\in R$. Then for each $\bullet\in \{+,-,(+,tw),\infty\}$ there is an isomorphism
$H(DCI^\bullet(Y,E))\cong I^\bullet(Y,E)$ of $H^-_A(R)\cong R[U]$-modules.
These isomorphisms are compatible with the corresponding exact triangles. 
\end{corollary}
\begin{proof} As $\frac12\in R$ there are, according to
Proposition \ref{DCI-CI-Equivalence}, a quasi-isomorphism $i\colon \Lambda_R[u]\ra C^{gm}_*(\om{SO}(3))$ of DG algebras and a zigzag
\[ \begin{tikzcd} \wt{CI}(Y,E) \arrow{r}{f} & X & DCI(Y,E) \arrow{l}[swap]{g}
\end{tikzcd} \]
of filtered objects and filtered $\Lambda_R[u]$-homomorphisms
such that $\overline{f}_p$ and $\overline{g}_p$ are quasi-isomorphism for each $p\in \Z$. The required result now follows from Propositions
\ref{Completed-Invariance-Module} and \ref{Completed-Invariance-Algebra}.
\end{proof}

For $A=\Lambda_R[u]$ we can extend a number of the results from
the previous section.   

\begin{proposition} \label{Completed-Twisted-Tate-Simplification}
Let $A=\Lambda_R[u]$ and let $M$ be a filtered
$A$-module. Then there is a filtered quasi-isomorphism
$\hat{f}_M\colon \hat{C}^+_A(M)\ra \hat{C}^{+,tw}_A(M)$ of degree $3$
and a (filtered) homotopy $\hat{s}_M\colon \hat{C}^+_A(M)\ra \hat{C}^{+,tw}_A(M)$
such that $\hat{f}_MU-U\hat{f}_M=\dd \hat{s}_M-\hat{s}_M\dd$ for
$U\in R[U]\cong C^-_A(R)$. Moreover, there is a (filtered) quasi-isomorphism
\[ \om{Cone}(\hat{N}_M\circ \hat{f}) \ra \om{Cone}(\hat{N}_M)\cong \hat{C}^\infty_A(M),  \]
which is $R[U]$-linear up to homotopy, where the source is given the
$R[U]$-module structure determined by the homotopy $\hat{t}_M = \hat{N}_M\circ \hat{s}_M$ as in Lemma \ref{Cone-Lemma1}. Furthermore, the resulting isomorphism in homology is compatible with the associated exact triangles.  \end{proposition} 
\begin{proof} Recall from Lemma \ref{f-Homotopy-Lemma} that there
is for each right $A$-module $N$ a quasi-isomorphism
$f_N\colon C^+_A(N)\ra C^{+,tw}_A(N)$ of degree $3$ 
and a homotopy $s_N\colon C^+_A(N)\ra C^{+,tw}_A(N)$ such that
$f_N  U-Uf_N=\dd s_N-s_N \dd$. Let $f_{p,q}$ and $s_{p,q}$ denote the
maps obtained for $N=F_pM/F_qM$ for $q<p$. By naturality we
may pass to the colimit over $p$ and then to the limit over $q$ to
obtain maps $\hat{f}_M\colon \hat{C}^+_A(M)\ra \hat{C}^{+,tw}_A(M)$
and $\hat{s}_M\colon \hat{C}^+_A(M)\ra \hat{C}^{+,tw}_A(M)$. As $f_{p,q}U-Uf_{p,q}=s_{p,q}\dd - \dd s_{p,q}$ is valid for each pair $q<p$,
it follows that
$\hat{f}_MU-U\hat{f}_M=\dd \hat{s}_M-\hat{s}_M\dd$. The fact that
$\hat{f}_M$ is a quasi-isomorphism is proved as before using
Lemma \ref{Eilenberg-Moore-Comparison-Lemma}.

The final statement follows by applying $\om{lim}_q\om{colim}_p$
to the diagram \eqref{Simplifying-Tate-Diagram} in the proof of
Theorem \ref{Explicit-Cinfty-model} for $F_pM/F_qM$ and another
application of Lemma \ref{Eilenberg-Moore-Comparison-Lemma}.
It is clear from this construction that one obtains an isomorphism
between the corresponding exact triangles in homology. 
\end{proof}

Combining the above proposition with Corollary \ref{DCI-CI-Equivariant-Eq-Corollary} we obtain the following. 

\begin{corollary} \label{Floer-Exact-Triangle-Corollary}
Suppose that $\frac12 \in R$. Then there is a degree $3$ isomorphism $I^+(Y,E)\cong I^{+,tw}(Y,E)$ of $R[U]$-modules and there is an exact triangle of $R[U]$-modules
\[ \begin{tikzcd} I^+(Y,E) \arrow{rr}{[3]} & {} & I^-(Y,E) \arrow{ld}{[0]} \\
         {} & I^\infty(Y,E) \arrow{lu}{[-4]} & {}, \end{tikzcd} \]
where the numbers specify the degrees of the maps. 
\end{corollary} 
 
We will now restrict ourselves to the cases of interest, namely,
\[ (A,M)=(C^{gm}_*(\om{SO}(3)),\wt{CI}(Y,E)) \;\; \mbox{ or } \;\; 
   (A,M) = (\Lambda_R[u],DCI(Y,E)), \]
where $M$ is equipped with the index filtration. Recall from
Definition \ref{Def-DCI} and Definition \ref{FloerComplex} that $M$
was obtained from a multicomplex $(M_{*,*},\{ \dd^r\}_{r= 0}^5)$, and that the index filtration was defined to be the column filtration, i.e., $F_pM_n = \bigoplus_{s\leq p}M_{s,n-s}$. In both cases it holds true that $M_{s,t}=0$ for all $t>4$ and $t<0$ (for $M=DCI(Y,E)$ this is true for $t\geq 4$). This implies that the filtration is degreewise finite and therefore exhaustive and
complete Hausdorff. Explicitly, 
\begin{equation} \label{Finite-Filtration-Eq}
 F_pM_n = \left\{ \begin{array}{cl} M_n & \mbox{ for all } p\geq n \\
                                      0 & \mbox{ for all } p\leq n-5
                                      \end{array} \right.
\end{equation} 
The following lemma indicates why it is desirable to introduce the full completions of the various functors.  

\begin{lemma} \label{Bounded-FloerFiltration}
The filtration of $C^+_A(M)$, $F_pC^+_A(M)=C^+_A(F_pM)$, is degreewise bounded above, exhaustive and Hausdorff.
The filtration of $C^-_A(M)$ is degreewise bounded below and complete
Hausdorff. If $M\neq 0$ the filtration of $C^+_A(M)$ fails to be complete
and the filtration of $C^-_A(M)$ fails to be exhaustive. Therefore,
\[ \hat{C}^+_A(M) = \om{lim}_q C^+_A(M/F_qM)\;\; \mbox{ and } \;\; \hat{C}^-_A(M) = \om{colim}_p C^-_A(F_pM) .\]
\end{lemma}
\begin{proof} We will use the underlying bigrading of the multicomplex
$M_{*,*}$. As a graded module (not DG) we have $C^+_A(M)=B(M,A,R)\cong M\otimes BA$. Therefore, in degree $n$
\begin{equation} \label{DirectSumC+}
 C^+_A(M)_n = \bigoplus_{u+r=n}M_u\otimes BA_r =\bigoplus_{s+t+r=n}M_{s,t}\otimes BA_r  ,
 \end{equation} 
and $F_pC^+_A(M)$ is obtained by imposing the condition $s\leq p$ in the
above right hand direct sum. First, notice that $A$ and hence $BA$ are
supported in nonnegative degrees. Then, as $M_{s,t}=0$ for $t<0$,
it follows $F_pC^+_A(M)_n=C^+_A(M)_n$ for $p\geq n$. This shows
that the filtration is degreewise bounded above and hence exhaustive.
It also follows from the above description that the filtration is Hausdorff.
On the other hand
\begin{equation} \label{Cpluss-Limit}
\lim_p C^+_A(M/F_pM)_n \cong \lim_p \bigoplus_{\substack{s+t+r=n \\ s> p}} M_{s,t}\otimes BA_r \cong \prod_{s+t+r=n} M_{s,t}\otimes BA_r .
\end{equation} 
This does not coincide with the direct sum in equation \eqref{DirectSumC+}
if $M\neq 0$, due to the periodicity $M_{s,t}\cong M_{s+8,t}$ for all $s,t$.
Thus, $C^+_A(M)$ fails to be complete if $M\neq 0$.  

Next, there is an isomorphism $C^-_A(M) = \om{Hom}_A(EA,M)\cong \om{Hom}(BA,M)$ of graded modules. In a fixed degree $n$, we find
\begin{equation} \label{DirectProductC-}
 C^-_A(M)_n = \om{Hom}(BA,M)_n = \prod_{s+t-r=n}\om{Hom}(BA_r,M_{s,t})
\end{equation} 
and as above $F_pC^-_A(M)$ is obtained by imposing the condition
$s\leq p$ in the right hand product. For the factor $\om{Hom}(BA_r,M_{s,t})$
to be nonzero it is necessary that $r\geq 0$ and $0\leq t\leq 4$, and therefore
$s=n+r-t\geq n-4$. Hence, $F_pC^-_A(M)_n=0$ for $p<n-4$ showing
that the filtration is degreewise bounded below, and thus complete
Hausdorff. Moreover, the above
also implies that the right hand product in \eqref{DirectProductC-} subject
to the condition $s\leq p$ is finite and therefore coincides with the
direct sum. Consequently, $\om{colim}_p F_pC^-_A(M)$ is given by
\begin{equation} \label{Cminus-Colimit}
\om{colim}_p \bigoplus_{\substack{s+t-r=n \\ s\leq p}}
\om{Hom}(BA_r,M_{s,t}) = \bigoplus_{s+t-r=n} \om{Hom}(BA_r,M_{s,t}),
\end{equation} 
which does not coincide with the product in \eqref{DirectProductC-} as
long as $M\neq 0$. This shows that the filtration fails to be exhaustive
if $M\neq 0$. 

The description of $\hat{C}^+_A(M)$ and $\hat{C}^-_A(M)$ given in the
statement is now a consequence of Lemma \ref{Filtration-Lemma} part
(iii) and (iv). 
\end{proof}

\begin{remark} \label{Remark-Cpm-Coincides-Miller}
Miller \cite[Appendix~A]{Miller19} defines completed versions
of $C^\pm_A$ using what he names the completed bar construction and
the finitely supported cobar construction. He defines the completed bar construction, $\hat{B}(M,A,N)$, to be the completion of $B(M,A,N)$ with respect to the filtration by internal degree, while the finitely supported cobar
construction $c\hat{B}(N,A,M)\subset \om{Hom}_A(B(N,A,A),M)$ is the
subcomplex consisting of those functionals
that vanish on $B_{p,q}(N,A,A)$ for all sufficiently large $p$. The above lemma, in particular equations \eqref{Cpluss-Limit} and
\eqref{Cminus-Colimit}, verifies that our $\hat{C}^+_A(M)$ and $\hat{C}^-_A(M)$
coincide with Miller's $\hat{B}(M,A,R)$ and $c\hat{B}(R,A,M)$
for $M=\wt{CI}(Y,E)$ and $M=DCI(Y,E)$ with $A$ as appropriate.
\end{remark}

The final part of this section is devoted to extending the concrete
models of Proposition \ref{Explicit-Cpm-Models} and Theorem \ref{Explicit-Cinfty-model} to the complex $DCI^\bullet(Y,E)$ for $\bullet\in \{+,-,\infty\}$. We will therefore assume that $\frac12 \in R$ throughout.  

\begin{theorem} \label{Explicit-DCI-Models}
Let $A=\Lambda_R[u]$ and let $(D^\pm_{*,*},\dd',\dd'')$
be the double complexes associated with $M=DCI(Y,E)$ in Proposition
\ref{Explicit-Cpm-Models}. 
Then there are isomorphisms of DG $R[U]$-modules
\[ DCI^+(Y,E) \cong \om{Tot}^\Pi (D^+_{*,*},\dd',\dd'')\;\;
\mbox{ and } \;\; DCI^-(Y,E) \cong \om{Tot}^\oplus (D^-_{*,*},\dd',\dd'') .\]
In particular, there are identifications 
\[ DCI^+(Y)_n \cong \prod_{p\geq 0} M_{n-4p} \;\; \mbox{ and } \;\;
DCI^-(Y)_n \cong \bigoplus_{p\geq 0} M_{n+4p}   \]
for each $n\in \Z$. An element $x\in DCI^\pm(Y)_n$ can therefore be expressed
by a sequence $x=(m_{n+4p})_p$ where we require $p\leq 0$ in the $+$
case and $p\geq 0$ and almost all $m_{n+4p}=0$ in the $-$ case. The
differential is in both cases given by
\[ \dd x = (\dd_M m_{n+4p}-(-1)^n m_{n+4(p-1)}u)_p  \]
and the action of $R[U]$ is determined by $U\cdot x = (m_{n+4(p+1)})_p$.
\end{theorem}
\begin{proof} The isomorphism
$C^+_A(M)\cong \om{Tot}^\oplus(D^+_{*,*},\dd',\dd'')$
of Proposition \ref{Explicit-Cpm-Models} yields in degree $n$
\[ C^+_A(M)_n \cong \bigoplus_{s\geq 0} M_{n-4s} =\bigoplus_{s\geq 0,t} M_{t,n-4s-t} \]
and $F_pC^+_A(M)$ is obtained by imposing the condition $t\leq p$ in the
right hand sum. Now the calculation of the limit $\om{lim}_p C^+_A(M)_n/F_pC^+_A(M)_n$ goes through as in equation \eqref{Cpluss-Limit} of the
previous proof to show that $\hat{C}^+_A(M)_n \cong \prod_{s\geq 0} M_{n-4s}$.

The other case is analogous and the calculation of the relevant
colimit proceeds as in equation \eqref{Cminus-Colimit} of the previous proof.
\end{proof}  

Note that the only difference between the 
above theorem and Proposition \ref{Explicit-Cpm-Models} is that we have interchanged the type of totalization we apply to the double complexes. 

We also obtain the following explicit model computing $I^\infty(Y,E)$. 

\begin{theorem} \label{Explicit-DCI-infty-model}
Let $A=\Lambda_R[u]$, let
$(D^\pm_{*,*},\dd',\dd'')$ and $(D^\infty_{*,*},\dd',\dd'')$ be the double
complexes associated with $M=DCI(Y,E)$ in the
above theorem and Theorem \ref{Explicit-Cinfty-model}, respectively. Let 
$\nu,\psi\colon \om{Tot}^\Pi(D^+,\dd',\dd'')\ra \om{Tot}^\oplus(D^-,\dd',\dd'')$
of degree $3$ and $0$ respectively, be given degreewise by the compositions
\[ \begin{tikzcd} 
\prod_{s\geq 0} M_{n-4s} \arrow{r}{\pi} & M_n \arrow{r}{(-1)^nu} &
M_{n+3} \arrow{r}{\iota} & \bigoplus_{t\geq 0} M_{n+3+4t} \\
\prod_{s\geq 0} M_{n-4s} \arrow{r}{\pi} & M_n \arrow{r}{\iota} &
\bigoplus_{t\geq 0} M_{n+4t}, & {} \end{tikzcd} \]
where $\pi$ and $\iota$ denote the projection and inclusion, respectively.
Then $\nu$ is a chain map satisfying $\nu U-U\nu = \dd \psi-\psi \dd$. Define
$\om{Cone}(\nu)$ with the $R[U]$-module structure adjusted by the homotopy
$\psi$ as in Lemma \ref{Cone-Lemma1}. Then there is an isomorphism
$R[U]$-modules $\om{Tot}^{\Pi,\infty}(D^\infty_{*,*},\dd',\dd'')\cong \om{Cone}(\nu)$ and a filtered quasi-isomorphism between this complex and 
$DCI^\infty(Y,E)$, which is $R[U]$-linear up to homotopy. Moreover, this quasi-isomorphism is compatible with the corresponding exact triangles
in homology.
\end{theorem}
\begin{proof} By Proposition \ref{Completed-Twisted-Tate-Simplification}
there is a filtered quasi-isomorphism $\hat{f}_M\colon \hat{C}^+_A(M)\ra \hat{C}^{+,tw}_A(M)$ of degree $3$ and a filtered homotopy
$\hat{s}_M:\hat{C}^+_A(M)\ra \hat{C}^{+,tw}_A(M)$ such that
$\hat{f}_MU-U\hat{f}_M=\dd \hat{s}_M -\hat{s}_M\dd$. Furthermore, there is
an induced filtered quasi-isomorphism $\om{Cone}(\hat{N}_M\circ \hat{f}_M)\ra \hat{C}^\infty_A(M)$, which is $R[U]$-linear up to homotopy when the cone
is given the adjusted $R[U]$-module structure determined by the homotopy
$\hat{t}_M=\hat{N}_M\circ \hat{s}_M$. In terms of the explicit models
of the above theorem $\hat{N}_M\circ \hat{f}_M$
and $\hat{N}_M\circ \hat{s}_M$ are precisely $\nu$ and $\psi$ as defined
in the statement. The verification of this is done just as in Lemma
\ref{f-Homotopy-Lemma}. The fact that
$\om{Cone}(\nu )\cong \om{Tot}^{\Pi,\infty}(D^\infty_{*,*},\dd',\dd'')$
is proved just as in Theorem \ref{Explicit-Cinfty-model}.
\end{proof}

\begin{remark} The proofs of Theorem \ref{Explicit-DCI-Models} and
Theorem \ref{Explicit-DCI-infty-model} go through without any essential
change with $M=\wt{CI}(Y,E)$, regarded as a $\Lambda_R[u]$-module, in
place of $M=DCI(Y,E)$. \end{remark} 

We may extract the following noteworthy consequences, showing in particular
that the framed homology groups $\wt{I}(Y,E)\coloneqq H(DCI(Y,E))\cong H(\wt{CI}(Y,E))$ may be recovered from the $R[U]$-action on $I^\pm(Y,E)$
up to extension. 

\begin{corollary} \label{U-action-Corollary} 
There are long exact sequences
\[ \begin{tikzcd} \tilde{I}_n(Y,E) \arrow{r} & I^+(Y,E)_n \arrow{r}{U} &
I^+(Y,E)_{n-4} \arrow{r} & \tilde{I}(Y,E)_{n-1} \\
\tilde{I}(Y,E)_{n-3} \arrow{r} & I^-(Y,E) \arrow{r}{U} & I^-(Y,E)_{n-4}
\arrow{r} & \tilde{I}(Y,E)_{n-4}. \end{tikzcd}  \]
Moreover, $U:I^\infty(Y,E)\ra I^\infty(Y,E)$ is
an isomorphism. \end{corollary} 
\begin{proof} Let $M=DCI(Y,E)$ and $A=\Lambda_R[u]$. In terms of the
explicit models for $DCI^\pm(Y,E)=\hat{C}^\pm_A(M)$ the action of
$U\in R[U]$ are in degree $n$ given by the natural shifts
\[ U\colon \prod_{s\geq 0} M_{n-4s} \ra \prod_{s\geq 0} M_{n-4-4s}  \;\;
\mbox{ and } \;\; 
U\colon \bigoplus_{s\geq 0} M_{n+4s} \ra \bigoplus_{s\geq 0} M_{n-4+4s},  \]
respectively. Explicitly, $(x_s)_{s\geq 0}\mapsto (x_{s+1})_{s\geq 0}$ in the first case and
$(y_s)_{s\geq 0}\mapsto (0,y_0,y_1,\cdots)$ in the second. From this we
deduce that the first map is surjective with kernel $D^+_{0,*}\cong M$
and the second is injective with cokernel isomorphic to $D^-_{0,*}\cong M$.
There are therefore short exact sequences of chain complexes
\[ \begin{tikzcd} 
 0 \arrow{r} & M \arrow{r}  & D^+_* \arrow{r}{U} & D^+_{*}[4] \arrow{r} & 0 \\
 0 \arrow{r} & D^-_*[-4] \arrow{r}{U} & D^-_* \arrow{r} & M \arrow{r} & 0 \end{tikzcd} \]
giving rise to the stated long exact sequences in homology.  
The final statement is a consequence of the fact that the shift
\[ U\colon D^\infty_n = \prod_{s\to\infty} M_{n-4s} \ra \prod_{s\to\infty} M_{n-4-4s}=D^\infty_{n-4}  \]
is an isomorphism in each degree $n$ and therefore also an isomorphism
upon passage to homology. 
\end{proof} 

\begin{remark} The fact that $U\colon I^\infty(Y,E)\ra I^\infty(Y,E)$ is an isomorphism when $\frac12 \in R$ was first proved using a spectral sequence argument in \cite{Miller19}. 
\end{remark}

\subsection{The Index Spectral Sequences}
For $\bullet\in \{+,-,(+,tw),\infty\}$ the complexes $CI^\bullet(Y,E)$ and $DCI^\bullet(Y,E)$ introduced in Definition \ref{Def-Equivariant-Floer} come equipped with filtrations and therefore give rise to spectral sequences. 
These will be called index spectral sequences
as the filtrations are induced by the index filtrations of $\wt{CI}(Y,E)$
and $DCI(Y,E)$. The purpose of this section is to recall
some theory on spectral sequences
and then to establish the basic properties of the index spectral sequences.
We will mainly follow Boardman's paper \cite{Boardman99}, but we will
stick to homological notation.  
All of this will be very useful for our calculations in the next section.   

Recall that an (unrolled) exact couple $(A,E,i,j,k)$ consists of two sequences
of graded modules $(A_s)_s$ and $(E_s)_s$, indexed over $s\in \Z$,  
and homomorphisms $i = (i_s:A_{s-1}\ra A_s)_s$, $j=(j_s:A_s\ra E_s)_s$
and $k=(k_s:E_s\ra A_{s-1})_s$ such that the sequence
\[ \begin{tikzcd} A_{s-1} \arrow{r}{i_s} & A_s \arrow{r}{j_s} & E_s
\arrow{r}{k_s} & A_{s-1} \arrow{r}{i_s} & A_s \end{tikzcd} \]
is exact for each $s$. We will generally omit the lower index on the maps,
letting the source and target specify the map in question. 
In our situation the degrees of $(i,j,k)$ will
be $(0,0,-1)$. For each $r\geq 1$ and $s$ set
\begin{align*} 
Z^r_s &\coloneqq k^{-1} \om{Im}[i^{(r-1)}:A_{s-r}\ra A_{s-1} ]  \\
B^r_s &\coloneqq j \om{Ker}[i^{(r-1)}:A_s\ra A_{s+r-1} ]   \\
E^r_s &\coloneqq Z^r_s/B^r_s,
\end{align*}
where $i^{(r)}$ denotes the $r$-fold composition $i\circ \cdots \circ i$
for $r\geq 1$ and $i^{(0)}$ is the identity.    
Furthermore, put $Z^\infty_s \coloneqq \bigcap_r Z^r_s$, $B^\infty_s \coloneqq \bigcup_r B^r_s$
and $E^\infty_s \coloneqq Z^\infty_s/B^\infty_s$. For the later convergence theory it is also important to introduce the derived limit 
$RE^\infty_s \coloneqq \om{Rlim}_s Z^r_s$. There are differentials
$d^r\colon E^r_s\ra E^r_{s-r}$ for each $s$ and $r\geq 1$, defined in the following way. Let $x\in Z^r_s$ represent the class $[x]\in E^r_s$. Choose $y\in A_{s-r}$
such that $i^{(r-1)}(y) = k(x)$. Then $d^r([x])=[j(y)]\in E^r_{s-r}$.
The situation is illustrated in the following diagram
\begin{equation} \label{SS-Diff-Def}
\begin{tikzcd} {} & y\in A_{s-r} \arrow{rr}{i^{(r-1)}} \arrow{ld}{j} & {} & 
k(x)\in A_{s-1} & {} \\
j(y)\in E_{s-r} & {} & {} & {} & x\in E_s. \arrow{lu}{k} \end{tikzcd}
\end{equation}
One may then show that $d^r$ is well-defined, that
$d^r\circ d^r=0$ and that 
\[ E^{r+1}_s \cong \om{Ker}(d^r\colon E_s^r\ra E^r_{s-r})/\om{Im}(d^r\colon E^r_{s+r}\ra E^r_{s})  \]
for each $r\geq 1$ and $s$. In other words, $(E^r,d^r)_{r\geq 1}$
is a spectral sequence. We call it the spectral sequence associated with
the exact couple $(A,E,i,j,k)$. Every exact couple we will meet
is covered by the following example (see \cite[Section~9]{Boardman99}).    

\begin{example} \label{Exact-Couple-Filt-Complex} 
Let $C$ be a DG module equipped with an increasing filtration $\{F_sC\}_s$ by
DG submodules. For each $s\in \Z$ let 
$A_s = H(F_sC)$ and $E_s = H(F_sC/F_{s-1}C)$. Define $i\colon A_{s-1}\ra A_s$, $j\colon A_s\ra E_s$ and $k\colon E_s\ra A_{s-1}$ to be the maps induced
by the short exact sequence
\[ \begin{tikzcd} 0\arrow{r} & F_{s-1}C \arrow{r} & F_sC \arrow{r} & F_sC/F_{s-1}C \arrow{r} & 0 \end{tikzcd} \]
upon passage to homology. In particular, $k$ is the connecting homomorphism
of degree $-1$. Then $(A,E,i,j,k)$ is an exact couple.  
The spectral sequence associated with this exact couple is the spectral
sequence mentioned in \eqref{FilteredSpectralSequence}. If $B$ is a
DG algebra and $C$ is a $B$-module filtered by $B$-submodules, then
$A_s$, $E_s$ carry the structure of graded $H(B)$-modules, and the maps
$i,j,k$ are $H(B)$-linear. Moreover, this structure carries in a natural
way over to the whole spectral sequence. 
\end{example}  

Returning to the situation of an exact couple $(A,E,i,j,k)$ put
\begin{equation} \label{Limit-Def-Eq} 
A_\infty \coloneqq \om{colim}_s A_s ,\;\; A_{-\infty} \coloneqq \om{lim}_s A_s \;\; \mbox{ and } \;\; RA_{-\infty} \coloneqq \om{Rlim}_s A_s  .
\end{equation} 
For us the relevant target for the spectral sequence is the colimit
$A_\infty$, filtered by $F_sA_\infty \coloneqq \om{Im}(A_s\ra A_\infty)$.
For each $s\in \Z$ there is a natural inclusion $\phi_s\colon F_sA_\infty/F_{s-1}A_\infty\ra E^\infty_s$ given by the composition
\begin{equation} \label{Phi-Def-Equation}
\begin{tikzcd} F_sA_\infty/F_{s-1}A_\infty \arrow{r}{\cong} & \om{Im}j/B^\infty_s \arrow{r} & Z^\infty_s/B^\infty_s =E^\infty_s.  
\end{tikzcd}
\end{equation}
Here, the first isomorphism is given by sending $[x]\in F_sA_\infty/F_{s-1}A_\infty$, represented by $x\in F_sA_\infty$, to $[j(y)]\in \om{Im}j/B^\infty_s$, where $y\in A_s$ is a lift of $x$ along the surjection $A_s\ra F_sA_\infty$,
and the second map is induced by the inclusion $\om{Im}j=\om{Ker}k\subset Z^\infty_s$ (see \cite[Lemma~5.6]{Boardman99} and its proof). 

\begin{definition} \label{Def-Strong-Convergence}
The spectral sequence $(E^r,d^r)_{r\geq 1}$ associated
with the exact couple $(A,E,i,j,k)$ converges strongly to the
colimit $A_\infty$ if the filtration $F_sA_\infty = \om{Im}(A_s\ra A_\infty)$ of $A_\infty$ is exhaustive and complete Hausdorff and the map
$\phi_s\colon F_sA_\infty/F_{s-1}A_\infty\ra E^\infty_s$ is an isomorphisms for each $s\in \Z$. 
\end{definition}

Strong convergence is the ideal form of convergence. In theory, 
if we can can calculate
$E^\infty$ we would know the subquotients $F_sA_\infty/F_{s-1}A_\infty$
and provided we can solve the extension problems to determine
$F_sA_\infty/F_tA_\infty$ for each pair $s>t$, we can recover the target group
as
\[ A_\infty \cong \om{lim}_s \om{colim}_tF_tA_\infty/F_sA_\infty  .\]  

\begin{definition}\cite[Definition~5.10]{Boardman99} 
Let $(A,E,i,j,k)$ be an exact couple with associated
spectral sequence $(E^r,d^r)_{r\geq 1}$. If
$RA_{-\infty}=0$ and $A_{-\infty}=0$ (see \eqref{Limit-Def-Eq}), then the spectral sequence is said to converge conditionally to the colimit $A_\infty$  \end{definition}

In our situation of a filtered DG module and the corresponding exact couple described in Example \ref{Exact-Couple-Filt-Complex} 
we have the following result. 

\begin{theorem}\cite[Theorem~9.2]{Boardman99} \label{Cond-Convergence}
Let $C$ be a DG module
equipped with an increasing filtration $F_sC$ by DG submodules.
Let $(A_s,E_s,i,j,k)$ be the associated
exact couple with corresponding spectral sequence $(E^r,d^r)_{r\geq 1}$. 
If the filtration of $C$ is exhaustive and complete Hausdorff, then the
spectral sequence $(E^r,d^r)_{r\geq 1}$ converges conditionally to the
colimit $A_\infty = H(C)$. \end{theorem}

In the presence of conditional convergence there are two remaining obstructions
to achieving strong convergence; namely, the vanishing of the group
$RE^\infty$ and another group $W$. We refer to  \cite[Lemma~8.5]{Boardman99}
for the general definition of $W$, but we note that if the exact
couple is constructed from a filtered complex, then $W$ is isomorphic to the kernel of the canonical interchange map (see \cite{HR19})
\[ \kappa\colon \om{colim}_p\om{lim}_q H(F_pC/F_qC)\ra \om{lim}_q\om{colim}_p H(F_pC/F_qC). \]
We then have the following convergence theorem. 

\begin{theorem} \cite[Theorem~8.2]{Boardman99} \label{Convergence-Theorem}
Let $(A,E,i,j,k)$
be an exact couple and assume that the associated spectral sequence
$(E^r,d^r)_{r\geq 1}$ converges conditionally to the colimit $A_\infty$. Then
if $W=0$ and $RE^\infty=0$ the spectral converges strongly to the
colimit $A_\infty$. \end{theorem}

The following result gives a few standard criteria for the vanishing of
$RE^\infty$ and $W$ in the situation of a filtered complex. 

\begin{proposition} \label{W-REinfty-Conditions} 
Let $C$ be a filtered complex with associated
exact couple $(A,E,i,j,k)$ and spectral sequence $(E^r,d^r)_{r\geq 1}$. 
\begin{enumerate}[label=(\roman*),ref=(\roman*)]
\item If the filtration of $C$ is degreewise bounded below, then
$RE^\infty=0$ and $W=0$. 
\item If the filtration of $C$ is degreewise bounded above then $W=0$.
\item If the spectral sequence degenerates at some finite stage; that is,
$E^{r_0}=E^{r_0+1}=\cdots = E^\infty$ for some finite $r_0\geq 1$, then
$RE^\infty=0$ and $W=0$. 
\end{enumerate}
\end{proposition} 

We may now apply this theory to the filtered complexes
$CI^\bullet(Y,E)$ for $\bullet\in \{+,-,(+,tw),\infty\}$. This is
a good point to explain that the mod $8$ periodicity of $\wt{CI}(Y,E)$
carries over to the complexes $CI^\bullet(Y,E)$. Let $M=\wt{CI}(Y,E)$
and let $A=C^{gm}_*(\om{SO}(3))$ or $A=\Lambda_R[u]$. Then the periodicity
may be expressed as an isomorphism $M\cong M[8]$ of $A$-modules. Moreover,
this isomorphism is compatible with the filtration in the sense that
$F_{p+8}M \cong (F_pM)[8]$ for all $p\in \Z$. There is therefore an
induced isomorphism
\begin{align*}
 CI^\bullet(Y,E)&=\; \om{lim}_q \om{colim}_p C^\bullet_A(F_{p+8}M/F_{q+8}M) \\ 
  &\cong \;  \om{lim}_q\om{colim}_p C^\bullet_A(F_pM/F_qM)[8]  
  =CI^\bullet(Y,E)[8]. 
\end{align*} 
This isomorphism is seen to be filtered in the same sense:
$F_pCI^\bullet(Y,E)_n \cong F_{p+8}CI^\bullet(Y,E)_{n+8}$ for all $p$ and
$n$.

\begin{theorem} \label{Index-SS-Theorem}
Let $A=C_*^{gm}(\om{SO}(3))$. There is a conditionally convergent spectral sequence of $H^-_A(R)$-modules
\[ E^1_{s,t} = \bigoplus_{j(\alpha)\equiv s} H^\bullet_{\om{SO}(3)}(\alpha)_t
\implies I^\bullet(Y,E)_{s+t}  \]
for each $\bullet\in \{+,-,(+,tw),\infty\}$, where
\[ H^\bullet_{\om{SO}(3)}(\alpha) \coloneqq H^\bullet_A(C^{gm}_*(\alpha)) \]
for each critical orbit $\alpha\in \MC{C}$. Moreover, the spectral sequences
are periodic in the sense that there are isomorphisms
$E_{s,t}^r\cong E_{s+8,t}^r$ for all $s,t,r$. These isomorphisms commute
with the differentials and the $H^-_A(R)$-action and they are compatible
with the target in the sense that
\[ \begin{tikzcd} 
\overline{F}_sI^\bullet_{s+t} \arrow{r}{\cong} \arrow{d} &
\overline{F}_{s+8}I^\bullet_{s+t+8} \arrow{d}  \\
E^\infty_{s,t} \arrow{r}{\cong} & E^\infty_{s+8,t}    \end{tikzcd} \]
commutes for all $s,t$.

The spectral sequence for $I^-(Y,E)$ converges strongly, while the 
spectral sequence for $I^+(Y,E)$ converges strongly provided $RE^\infty=0$.  
\end{theorem}
\begin{proof} Let $M=\wt{CI}(Y,E)$ equipped with the index filtration.
For each $\bullet\in \{+,-,(+,tw),\infty\}$ the complex $CI^\bullet(Y,E)=\hat{C}^\bullet_A(M)$ carries
by construction an exhaustive and complete Hausdorff filtration with
minimal filtration quotients
\[ \overline{F}_p\hat{C}^\bullet_A(M)\cong C^\bullet_A(\overline{F}_pM). \]
There is therefore by Theorem
\ref{Cond-Convergence} a conditionally convergent spectral sequence
\[ E^1_{s,t} = H^\bullet_A(\overline{F}_sM)_{s+t} \implies \hat{H}^\bullet_A(M)_{s+t} = I^\bullet(Y,E)_{s+t}  .\]
From the definition of $\wt{CI}(Y,E)$ (Definition \ref{FloerComplex}) we may identify
\[ \overline{F}_sM_{s+t} = \bigoplus_{j(\alpha)\equiv s} C^{gm}_t(\alpha) \]
so that $E^1_{s,t} \cong \bigoplus_{j(\alpha)\equiv s} H_{\om{SO}(3)}(\alpha)_t$
in the notation of the statement. 

Next, the periodicity isomorphisms
$F_p\hat{C}^\bullet_A(M)_n \cong F_{p+8}\hat{C}^\bullet_A(M)_{n+8}$,
compatible with the differentials and the $C^-_A(R)$-structure, induce a morphism of the associated exact couple. That is, there are isomorphisms
$A_s \cong A_{s+8}$, $E_s\cong E_{s+8}$ each $s$, compatible with the structure maps. It is then a straightforward exercise to verify
that this gives rise to the periodicity in the spectral sequence, compatible
with the target, as stated. 

Finally, for the convergence we know by Lemma \ref{Bounded-FloerFiltration}
that the filtration of $\hat{C}^+_A(M)$ is degreewise bounded above and that the filtration of $\hat{C}^-_A(M)$ is degreewise bounded below. This
implies by Proposition \ref{W-REinfty-Conditions} and Theorem
\ref{Convergence-Theorem} that the spectral sequence associated
with $I^-(Y,E)$ converges strongly and that the spectral sequence
associated with $I^+(Y,E)$ converges strongly provided $RE^\infty=0$.
\end{proof} 

Provided $\frac12\in R$, these spectral sequences can equally well be constructed from $M=DCI(Y,E)$ equipped with the index filtration and $A=\Lambda_R[u]$. This is what we will assume when we calculate
these groups for binary polyhedral spaces in the next section. Furthermore,
in this setting the calculation of the groups
\[ H^\bullet_{\om{SO}(3)}(C^{gm}_*(\alpha)) \cong H^\bullet_{\Lambda_R[u]}(H_*(\alpha))  \]
for each $\bullet\in \{+,-,(+,tw),\infty\}$, are contained in Lemma
\ref{Orbit-Calc-Lemma}, Corollary \ref{Orbit-Calc-Corollary} and Corollary \ref{Orbit-Calc-Tate-Corollary}. From these calculations it becomes
clear that the spectral sequences for $I^+$ and $I^-$ are contained
in the upper and lower half-planes respectively, while the spectral
sequence for $I^\infty$ is a whole-plane spectral sequence.  

\section{Calculations for Binary Polyhedral Spaces}
We will from this point require that $2\in R$ is invertible. The universal
example to have in mind is $R=\Z[\frac12]$. Our main aim is to give
explicit calculations of the $H^-_{\Lambda_R[u]}(R)=R[U]$-modules
\[ I^+(\overline{Y}_\Gamma,E), \;\;  I^-(\overline{Y}_\Gamma,E) \;\; \mbox{ and } \;\; I^\infty(\overline{Y}_\Gamma,E)   \]
for each finite subgroup $\Gamma\subset \om{SU}(2)$, where $E\ra \overline{Y}_\Gamma$ is the trivial $\om{SU}(2)$-bundle. 
To simplify the notation slightly we will omit the reference to the bundle in the notation from this point. In the final part we also include
the calculations of $I^\bullet(Y_\Gamma)$ for $\bullet\in \{+,-,\infty\}$,
omitting some details. 

The two results Theorem \ref{Structure-DCI} and Proposition \ref{Differential-Graphs} give a concrete description of the Donaldson model $\om{DCI}(\overline{Y}_\Gamma)$ in all cases. The key tools needed in the calculations are the index spectral sequence in combination with the explicit models for $DCI^\pm$ of Theorem \ref{Explicit-Cpm-Models}. In the $I^+$ case, we will see that the index spectral sequence immediately
degenerates, and we are left with solving an extension problem in the
category of $R[U]$-modules. In the $I^-$ case there are a number of
differentials in the index spectral sequence, but we will in each
case verify that the spectral sequence degenerates on some finite
page $E^{r_0}$. The extension problems in this case are neatly solved
by observing that $E^\infty_s$ is a free $R[U]$-module for all $s$. 

In \cite[Corollary~8.7]{Miller19} Miller calculates $I^\infty(Y)$ for all rational homology spheres. This, of course, includes $\overline{Y}_\Gamma$,
so our calculation of this group should only be regarded as a verification
of his result. In addition, we observe that the norm map vanishes in
homology so that the exact triangle relating the three groups splits up
into short exact sequences.       

Throughout this section we let $A=\Lambda_R[u]$ with $|u|=3$. Given
a finite subgroup $\Gamma\subset \om{SU}(2)$ we will always write
$\MC{C}$ for the set of critical orbits, or equivalently, the set of $1$-dimensional quaternionic representations of $\Gamma$. Recall that
for each $\alpha\in \MC{C}$ we have fixed generators
$b_\alpha \in H_0(\alpha)$, $t_\alpha = b_\alpha\cdot u\in H_3(\alpha)$
if $\alpha$ is irreducible and $t_\alpha\in H_2(\alpha)$ if $\alpha$
is reducible. 

\subsection{The Case \texorpdfstring{$I^+$}{Ipluss}}
In this section we will calculate $I^+(\overline{Y}_\Gamma)$ as an
$R[U]$-module for all finite subgroups $\Gamma\subset \om{SU}(2)$.
Even though our arguments do not explicitly rely on the index spectral
sequence, we note that the motivation for the upcoming arguments
is the following observation. 

\begin{lemma} \label{Spectral-Degeneracy-Lemma}
Let $M=DCI(\overline{Y}_\Gamma)$ for a
finite subgroup $\Gamma\subset \om{SU}(2)$. Then the index spectral sequence
\[ E^1_{s,t}=\bigoplus_{j(\alpha)\equiv s}H_A^+(\alpha)_t \implies I^+(\overline{Y}_\Gamma)_{s+t} \]
degenerates at the $E^1$ page, that is, all differentials are trivial and
$E^1=E^\infty$. \end{lemma} 
\begin{proof} From the calculations in Lemma \ref{Orbit-Calc-Lemma} and   Corollary \ref{Orbit-Calc-Corollary}
we see that $H^+_A(\alpha)$ vanishes in odd degrees for all types of
orbits $\alpha$. Then, as $j(\alpha)\equiv 0\Mod{4}$ for all $\alpha\in \MC{C}$, it follows that $E^1_{s,t}=0$ if $s$ or $t$ is odd. Since the differential
$d^r$ has bidegree $(-r,r-1)$ we conclude that $d^r=0$ for all $r\geq 1$
as required. 
\end{proof} 

From this one may determine $I^+(\overline{Y}_\Gamma)$ as an $R$-module. 
However, the structure as an $R[U]$-module is more subtle and this is
the reason we spell out our argument more directly in terms of the
filtration of the chain complex $DCI^+(\overline{Y}_\Gamma)$.

Let $\Gamma\subset \om{SU}(2)$ be a finite subgroup and
let $M=DCI(\overline{Y}_\Gamma)$. Let $\MC{C}^{irr}\subset \MC{C}$ denote the
subset of free critical orbits or equivalently irreducible flat connections.
Define $M^{irr}\subset M$ to be the submodule generated by $\MC{C}^{irr}$, i.e.,
\[ M^{irr}_{s,t} = \bigoplus_{\substack{\alpha\in \MC{C}^{irr} \\ j(\alpha)\equiv s}} 
H_{t}(\alpha) \;\; \mbox{ and } \;\; M^{irr}_n = \bigoplus_{s+t=n} M^{irr}_{s,t} .\] 
The following lemma is the key tool in resolving our extension problems in
the category of $R[U]$-modules. 

\begin{lemma} \label{Lambda-Lemma}
There is an $R$-linear map $\psi\colon M\ra M$ of degree $-4$ satisfying
\begin{enumerate}[label=(\roman*),ref=(\roman*)]
\item  $\psi(x)\cdot u = (-1)^{|x|}\dd_M x$ for all $x\in M$,
\item  $\om{Im}(\psi)\subset M^{irr}$,
\item  $\psi(\om{Im}(u))=0$ and
\item  $\psi( \om{Im}\dd_M)=0$. 
\end{enumerate}
\end{lemma}
\begin{proof} By Theorem \ref{Structure-DCI} the only possibly nontrivial
differentials in $M$ are $\dd_M\colon M_{4s}\ra M_{4s-1}$ for $s\in \Z$. The same result also shows that $u\colon M_{4s-4}^{irr}\ra M_{4s-1}^{irr}=M_{4s-1}$ is an isomorphism for each $s$.
We may therefore define $\psi\colon M_{4s}\ra M_{4s-4}^{irr}\subset M_{4s-4}$ by $u^{-1}\circ \dd_M$ and set $\psi=0\colon M_{t}\ra M_{t-4}$ in all other degrees. The relation $\psi(x)\cdot u = (-1)^{|x|}\dd_Mx$ is then satisfied for
all $x\in M$ and by construction $\om{Im}(\psi)\subset M^{irr}$.
For (iii) and (iv) it is then sufficient to note that $\om{Im}(u)$ and 
$\om{Im}(\dd_M)\subset M^{irr}$ are supported in the degrees $M_{4s-1}$ for $s\in \Z$.
\end{proof}

\begin{remark} It is not difficult to show that the conditions
(i) and (ii) determine the map $\psi$ uniquely, but we will not use
this explicitly. \end{remark} 

Recall from Corollary \ref{Explicit-DCI-Models} that for $M=DCI(\ovl{Y}_\Gamma)$ there is an identification $DCI^+(\ovl{Y}_\Gamma)_n \cong  \prod_{p\geq 0} M_{n-4p}$ for each $n$. 
Moreover, the differential is given by
\[ \dd (m_{n-4p})_{p\geq 0} = (\dd_M m_{n-4p}-(-1)^nm_{n-4(p+1)}u)_{p\geq 0}. \]
The point of the map $\psi$ should now be clear; given $m\in M_n$
the element $(\psi^p(x))_{p\geq 0} \in \prod_{p\geq 0}M_{n-4p}$
gives an explicit extension of $m$ to a cycle in $DCI^+(Y)$.

Using the description of the differentials in Proposition
\ref{Structure-DCI} we note that $\psi$ is given on a generator
$b_\alpha\in H_0(\alpha)\subset M_{4p}$ by
\begin{equation} \label{Psi-Formula} 
 \psi(b_\alpha) = u^{-1}\left( \sum_{\beta\in \MC{C}^{irr}} n_{\beta\alpha} t_\beta \right) = \sum_{\beta\in \MC{C}^{irr}} n_{\beta\alpha} b_\beta 
\end{equation}
(the integers $n_{\beta\alpha}$ were defined in Definition \ref{n-alpha-beta-Def}).

Recall from Lemma \ref{Bounded-FloerFiltration} that
\[ DCI^+(\overline{Y}_\Gamma) = \hat{C}_A^+(M)=\om{lim}_p C^+_A(M/F_pM), \]
filtered by $F_p\hat{C}_A^+(M)=\hat{C}_A^+(F_pM) = \om{lim}_{q< p} C^+_A(F_pM/F_qM)$. The concrete model given in Corollary \ref{Explicit-DCI-Models} applies equally well to these subcomplexes. Recall that we use the notation
$\overline{F}_p M = F_pM/F_{p-1}M$. 

\begin{lemma} \label{Technical-Extension-Lemma}
Let $M=DCI(\overline{Y}_\Gamma)$ be equipped with the index
filtration. Then $F_{4p}M=F_{4p+t}M$ for all $p\in \Z$ and $0\leq t\leq 3$. 
Moreover, there are identifications for $s\in \Z$ and $0\leq t\leq 3$
\begin{align} \label{Degree-Description}
C^+_A(\overline{F}_{4p}M)_{4s+t} &  \cong
\left\{ \begin{array}{cc} M_{4p+t} & \mbox{ for } s\geq p \\
                          0 & \mbox{ for } s<p \end{array} \right.  \\ 
\hat{C}^+_A(F_{4p}M)_{4s+t} & \cong \left\{ \begin{array}{cc} \prod_{q\geq 0} M_{4(p-q)+t} &
\mbox{ for } s\geq p \\
\prod_{q\geq 0} M_{4(s-q)+t} & \mbox{ for } s<p. \end{array} \right. \nonumber
\end{align}
Define $\zeta\colon C^+_A(\overline{F}_{4p}M)\ra \hat{C}^+_A(F_{4p}M)$
for $x\in M_{4p+t}\cong C^+(\overline{F}_{4p}M)_{4s+t}$, $s\geq p$, by
\[ \zeta(x) = (\psi^q(x))_{q\geq 0} \in \prod_{q\geq 0} M_{4(p-q)+t}
= \hat{C}^+_A(F_{4p}M), \]
where $\psi$ is the map of Lemma \ref{Lambda-Lemma}, and otherwise to be zero. Then $\zeta$ is a chain map splitting the
exact sequence of DG $R$-modules (i.e., $\pi\circ \zeta=1$)
\[ \begin{tikzcd} 0 \arrow{r} & \hat{C}^+_A(F_{4p-1}M) \arrow{r}{i} &
\hat{C}^+_A(F_{4p}M) \arrow{r}{\pi} & C^+_A(\overline{F}_{4p}M)
\arrow{r} \arrow[bend right]{l}[above]{\zeta} & 0. \end{tikzcd} \]
 \end{lemma}
\begin{proof} The fact that $F_{4p+t}M=F_{4p}M$ for $0\leq t\leq 3$
follows from the structure theorem \ref{Structure-DCI}. From the same
result one concludes that $F_{4p}M_n = M_n$ for $n\leq 4p+3$ and
$F_{4p}M_n=0$ for $n\geq 4(p+1)$. Therefore,
$\overline{F}_{4p}M_n$ coincides with $M_n$ for $4p\leq n\leq 4p+3$ and
vanishes otherwise. Using this, the given degreewise formulas for $C^+_A(\overline{F}_{4p}M)$ and $\hat{C}^+_A(F_{4p}M)$ follow from Proposition \ref{Explicit-Cpm-Models} and Corollary \ref{Explicit-DCI-Models}, respectively.

Next we have to verify that $\zeta$ is a chain map. The differential
in $C^+_A(\overline{F}_{4p}M)$ is given by  $u\colon M_{4p}\ra M_{4p+3}$ in
degree $4s$ for $s>p$ and is otherwise $0$. The differential in 
$\hat{C}^+_A(F_{4p}M)$ is given in degree $4s+t$ by
\[ \dd( (m_{4(s-q)+t})_{q\geq 0}) =(\dd_M m_{4(s-q)+t}-(-1)^tm_{4(s-q-1)+t}u)_{q\geq 0}, \]
where if $s>p$ it is assumed that $m_{4(s-q)+t}=0$ for $q\leq s-p$. In view
of this formula and the fact that $\psi(x)\cdot u = (-1)^{|x|}\dd_Mx$
for all $x\in M$ by Lemma \ref{Lambda-Lemma}, we see that for
$x\in M_{4p+t}=C^+_A(\overline{F}_{4p}M)_{4s+t}$, $s\geq p$,
\[ \dd\circ \zeta(x) = \dd(\psi^qx)_{q\geq 0} = \left\{ \begin{array}{cl}
0 & \mbox{ if } (s,t)=(p,0) \mbox{ or } 1\leq t\leq 3 \\
(x\cdot u,0,0,\cdots ) & \mbox{ if } s>p, t=0 \end{array} \right. \]
In the first case we also have $\dd x=0$ and hence $\zeta\circ \dd(x)=0$, 
while in the second case we have
$\dd x = x\cdot u$ so that $\zeta\circ \dd(x)=(\psi^q(x\cdot u))_{q\geq 0}
= (x\cdot u,0,0,\cdots)$ as $\psi$ vanishes on $\om{Im}u$ by
Lemma \ref{Lambda-Lemma}. In all other degrees $\zeta$ vanishes, so we
may conclude from the above that $\zeta$ is a chain map. 

Finally, the projection $\pi\colon \hat{C}^+_A(F_{4p}M)\ra C^+_A(\overline{F}_{4p}M)$ is in terms of the formulas of \eqref{Degree-Description} given by the
projection $\prod_{q\geq 0}M_{4(p-q)+t}\ra M_{4p+t}$ in degree
$4s+t$ for $s\geq p$ and otherwise $0$. From this description it is
clear that $\pi\circ \zeta=1$ and the proof is complete.  
\end{proof}

We need one final lemma describing the interaction between the map
$\zeta$ and the action of $R[U]$.

\begin{lemma} \label{Technical-U-Lemma}
In the situation of the above lemma the map
$\zeta\colon C^+_A(\overline{F}_{4p}M)\ra \hat{C}^+_A(F_{4p}M)$ satisfies
$\zeta(Ux)=U\zeta(x)$ for each $x$ with $|x|\neq 4p$. Moreover, in degree
$4p$ we have the following commutative diagram
\[ \begin{tikzcd} C^+_A(\overline{F}_{4p}M)_{4p} \arrow{r}{\zeta} \arrow{d}{\psi} & \hat{C}^+_A(F_{4p}M)_{4p} \arrow{r}{U} & 
\hat{C}^+_A(F_{4p}M)_{4(p-1)} \arrow{d}{=} \\
C^+_A(\overline{F}_{4(p-1)}M)_{4(p-1)} \arrow{rr}{\zeta} & {} & 
\hat{C}^+_A(F_{4(p-1)}M)_{4(p-1)}. \end{tikzcd} \]
\end{lemma}
\begin{proof} In degree $n$ for $n<4p$ the map $\zeta$ vanishes, so the first
assertion is clear in this case. In degree $4s+t$ with $s>p$ and $0\leq t\leq 3$ we have, using the formulas of Lemma \ref{Technical-Extension-Lemma},
\[ C^+_A(\overline{F}_{4p}M)_{4s+t}=M_{4p+t} \;\; \mbox{ and } \;\;
\hat{C}^+_A(F_{4p}M)_{4s+t}=\prod_{q\geq 0} M_{4(p-q)+t}.  \]
It then follows from Proposition \ref{Explicit-Cpm-Models} and
Corollary \ref{Explicit-DCI-Models} that the $U$ action is given by
the identity in these degrees. Therefore, $U\circ \zeta=\zeta\circ U$ in this case as well. 

For $x\in M_{4p+t}=C^+_A(\overline{F}_{4p}M)_{4p+t}$, where $0\leq t\leq 3$,
we have $Ux=0$, while
\[ U(\zeta(x))=U((\psi^qx)_{q\geq 0})=(\psi^{q+1}x)_{q\geq 0}
= (\psi^q(\psi x))_{q\geq 0} =\zeta(\psi(x)),  \]
where we regard $\psi x \in M_{4(p-1)+t}=C^+_A(\ovl{F}_{4(p-1)}M)_{4(p-1)+t}$.
This shows that the given diagram commutes in degree $4p+t$ for
$0\leq t\leq 3$. However, in the proof of Lemma \ref{Lambda-Lemma}
it was shown that $\psi\colon M_{4p+t}\ra M_{4(p-1)+t}$ vanishes for
$1\leq t\leq 3$. Therefore $\zeta\circ U=U\circ \zeta$ in degree
$4p+t$ for $1\leq t\leq 3$ as well. This completes the proof. 
\end{proof} 

To describe $I^+(\overline{Y}_\Gamma)$ and later $I^-(\overline{Y}_\Gamma)$, it is convenient to introduce the following definition.

\begin{definition} For an $R[U]$-module $X$ define the mod $8$ periodic
$R[U]$-modules $X^{\Pi,8}$ and $X^{\oplus,8}$ degreewise by
\[ X^{\Pi,8}_n = \prod_{s\in \Z} X_{n+8s} \;\; \mbox{ and } \;\;
  X^{\oplus,8}_n = \bigoplus_{s\in \Z} X_{n+8s}  .\]
The maps $U\colon X^{\Pi,8}_n \ra X^{\Pi,8}_{n-4}$ and $U\colon X^{\oplus,8}_n\ra X^{\oplus,8}_{n-4}$ are defined to be the product and direct sum over
the maps $U:X_{n+8s}\ra X_{n-4+8s}$ for $s\in \Z$, respectively.
\end{definition}

Let $\Gamma\subset \om{SU}(2)$ be a finite subgroup. Given $\alpha\in \MC{C}$
we will use the shorthand notation 
$H^+_A(\alpha) \coloneqq H^+_A(H_*(\alpha))$.  
Write $\MC{C}=\MC{C}^{irr}\cup \MC{C}^{red}\cup \MC{C}^{f.red}$ for the decomposition of the critical orbits into irreducible, reducibles and fully reducibles.
For each $\eta\in \MC{C}^{f.red}$, $\lambda\in \MC{C}^{red}$ and $\alpha\in \MC{C}^{irr}$ introduce variables $V_\eta$, $W_\lambda$ and a generator $g_\alpha$ so that
\[ H^+_A(\eta) = R[V_\eta], \;\;\;\; H^+_A(\lambda)= R[W_\lambda] \;\; \mbox{ and } \;\; H^+_A(\alpha) = R\cdot g_\alpha, \]
where $|V_\eta|=4$, $|W_\lambda|=2$ and $|g_\alpha|=0$. This is justified by the calculations of Lemma \ref{Orbit-Calc-Lemma} and Corollary 
\ref{Orbit-Calc-Corollary}. In the following we regard the relative grading 
$j:\MC{C}\ra \Z/8$ as taking values in $\{0,4\}$. 
  
\begin{theorem} \label{IPluss-1}
The positive equivariant instanton Floer homology $I^+(\ovl{Y}_\Gamma)$ associated with the trivial $\om{SU}(2)$-bundle over $\overline{Y}_\Gamma$ is given by
\[ \left[ 
\left( \bigoplus_{\alpha\in \MC{C}^{irr}}R\cdot g_\alpha[j(\alpha)] \right) \oplus
\left(\bigoplus_{\lambda\in \MC{C}^{red}}R[W_\lambda][j(\lambda)] \right) \oplus
\left( \bigoplus_{\eta\in \MC{C}^{f.red}}R[V_\eta][j(\eta)] \right) \right]^{\Pi,8}. \]
The $R[U]$-module structure is determined by 
\begin{align*}
 U\cdot V^p_\eta &= \left\{ \begin{array}{cc} V^{p-1}_\eta & \mbox{ for } p>0 \\ \sum_{\rho\in C^{irr}} n_{\rho\eta} g_\rho & \mbox{ for } p=0 
\end{array} \right.  \\
U\cdot W_\lambda^p &= \left\{ \begin{array}{cc} W^{p-2}_\lambda & \mbox{ for } p>0 \\ \sum_{\rho\in C^{irr}} n_{\rho\lambda} g_\rho & \mbox{ for } p=0 
\end{array} \right.   \\
U\cdot g_\alpha &= \sum_{\rho\in C^{irr}} n_{\rho\alpha} g_\rho,
\end{align*}
where we interpret $W_\beta^{-1}=0$ and the integers $n_{\rho_1\rho_2}$
are given in Proposition \ref{Differential-Graphs} for each pair $(\rho_1,\rho_2) \in \MC{C}$.    
\end{theorem}
\begin{proof} Let as usual $M=DCI(\overline{Y}_\Gamma)$ be equipped with
the index filtration. By Lemma \ref{Technical-Extension-Lemma} we have $F_{4p}M=F_{4p+t}M$ for $0\leq t\leq3 $, and the map 
$\zeta\colon C^+_A(\overline{F}_{4p}M)\ra \hat{C}^+_A(F_{4p}M)$
along with the inclusion $i\colon \hat{C}^+_A(F_{4(p-1)}M)\ra \hat{C}^+_A(F_{4p}M)$ define a chain isomorphism 
\[ i+\zeta:\hat{C}^+_A(F_{4(p-1)}M)\oplus C^+_A(\overline{F}_{4p}M) \cong \hat{C}^+_A(F_{4p}M)  \]
of $DG$ $R$-modules for each $p$. By induction we obtain isomorphisms
\[ \hat{C}^+_A(F_{4p}M) \cong \hat{C}^+_A(F_{4(p-r-1)}M)\oplus \left(
\bigoplus_{q=0}^r C^+_A(\overline{F}_{4(p-q)}M) \right) \]
for each $p$ and $r\geq 1$. Since the filtration $F_p\hat{C}^+_A(M)=\hat{C}^+_A(F_pM)$ is degreewise bounded above and complete Hausdorff, it follows
by first passing to the limit over $r$ and then to the colimit over $p$
that we obtain an isomorphism of DG $R$-modules
\[ \hat{C}^+_A(M) \cong \prod_{p\in \Z} C^+_A(\overline{F}_{4p}M) . \]
Hence, as products commutes with homology,
\begin{equation} \label{Product-Iso-I+}
I^+(\overline{Y}_\Gamma) = \hat{H}^+_A(M)\cong \prod_{p\in \Z} H^+_A(\overline{F}_{4p}M)  .
\end{equation} 
Write $\MC{C}=\MC{C}^0\cup \MC{C}^1$ where
$\MC{C}^i = \{\alpha\in \MC{C}: j(\alpha)\equiv 4i \Mod{8}\}$ for $i=0,1$.
Then
\[ H^+_A(\overline{F}_{4p}M)= \left\{ \begin{array}{cc}
 \bigoplus_{\alpha\in \MC{C}^0}H^+_A(\alpha)[4p] & \mbox{ if } p\equiv 0 \Mod{2}
  \\ \bigoplus_{\alpha\in \MC{C}^1}H^+_A(\alpha)[4p] & \mbox{ if }
  p\equiv 1 \Mod{2}, \end{array} \right.  \]
and we may simplify the expression in \eqref{Product-Iso-I+} 
\begin{equation} \label{Product-Iso-I+2}
\prod_{p\in \Z} H^+_A(\overline{F}_{4p}M)
= \prod_{p\in \Z} \bigoplus_{\alpha\in \MC{C}} H^+_A(\alpha)[8p+j(\alpha)]
= \left( \bigoplus_{\alpha\in \MC{C}} H^+_A(\alpha)[j(\alpha)] \right)^{\Pi,8}. 
\end{equation} 
By replacing each $H^+_A(\alpha)$ with $R[V_\alpha]$, $R[W_\alpha]$
or $R\cdot g_\alpha$ according to whether $\alpha$ is fully reducible,
reducible or irreducible we obtain the additive statement of the theorem.

To finish the proof we need to determine the action of $U$. Returning
to the notation of equation \eqref{Product-Iso-I+}, for a fixed degree
$n$, we have $I^+(\overline{Y}_\Gamma)_n = \hat{H}^+_A(F_{4p}M)_n$
for each $p$ with $4p>n$. This is a consequence of the formula \eqref{Degree-Description} given in Lemma \ref{Technical-Extension-Lemma}. Fix the minimal $p$ with $4p>n$ and consider the commutative diagram
\[ \begin{tikzcd} 
I^+(\overline{Y}_\Gamma)_n \arrow{d}{U} & \hat{H}^+_A(F_{4p}M)_n \arrow{d}{U} \arrow{l}{\cong} & \prod_{q:4q\leq n} H^+_A(\overline{F}_{4q}M)_n
\arrow{d}{U'} \arrow{l}{\phi} \\
I^+(\overline{Y}_\Gamma)_{n-4} & \hat{H}^+_A(F_{4p}M)_{n-4} \arrow{l}{\cong}
& \prod_{q:4q\leq n} H^+_A(\overline{F}_{4(q-1)}M)_{n-4} \arrow{l}{\phi}.
\end{tikzcd} \] 
Here, $\phi$ is the isomorphism obtained by taking the limit, as 
$j\to \infty$, over the maps
\[ ([x_q])_{j\leq q\leq n/4}\mapsto \sum_{j\leq q\leq n/4} [\zeta(x_q)] , \]
where $x_q\in C^+_A(\overline{F}_{4q}M)$ represents $[x_q]\in H^+_A(\overline{F}_{4q}M)$ and $U'$ is the map that forces the right
rectangle to commute. Our task is to determine $U'$. By Lemma \ref{Technical-U-Lemma}
we have for $y\in H^+_A(\overline{F}_{4q}M)_n$ that
$H(\zeta)\circ U(y)=U\circ H(\zeta)(y)$ provided $n>4q$, while if
$n=4q$, and $y=[x]$ then $U\circ H(\zeta)(y) = [\psi x]$. It follows
that $U'=\prod U$ if $n$ is not divisible by $4$. Otherwise, if
$4|n$ so that $n=4(p-1)$, then $U'=\tau+\prod U$, where
$\tau=([x_q])_{q\leq p-1}=[\psi(x_{p-1})]$. Therefore, in the description
\[ I^+(\overline{Y}_\Gamma)= \prod_{p\in \Z} \bigoplus_{\alpha\in \MC{C}}
H^+_A(\alpha)[8p+j(\alpha)]  \]
of \eqref{Product-Iso-I+2}, the action of  $U$ is given by taking the product over the internal $R[U]$-module structure of each factor $H^+_A(\alpha)$ and adding the correction
term $\tau$. For each $p$, $\tau$ only affects the terms
$H^+_A(\alpha)[8p+j(\alpha)]_{8p+j(\alpha)}=H^+_A(\alpha)_0$. In terms of
our generators this correspond to $V^0_\alpha$ if $\alpha$ is fully reducible,
$W_\alpha^0$ if $\alpha$ is reducible and $g_\alpha$ if $\alpha$ is
irreducible. Now, we may require that these generators
correspond to $b_\alpha \in H_0(\alpha)\subset M_{4q} \cong C^+_A(\overline{F}_{4q}M)=H^+_A(\overline{F}_{4q}M)_{4q}$ under the identifications of
Lemma \ref{Technical-Extension-Lemma}. Therefore, using the explicit
formula \eqref{Psi-Formula} for $\psi$, we obtain
\[ U\cdot V^0_\eta = \sum_{\rho\in \MC{C}^{irr}} n_{\rho\eta}g_\rho \]
and similarly for the other types of generators. This completes the
proof. 
\end{proof} 

\begin{example} For $Y =\overline{Y}_{O^*}$ there are two irreducibles
$\alpha$, $\beta$, the trivial connection $\theta$ and a fully
reducible $\eta$. The grading is given by $j(\theta)=j(\beta)=0$
and $j(\alpha)=j(\eta)=4$. Furthermore, $n_{\alpha\beta}=n_{\beta\alpha}=3$
and $n_{\alpha\theta}=n_{\beta\eta}=1$. This is obtained from
Proposition \ref{Differential-Graphs}. The above theorem gives
\[ I^+(Y)= \left( R[V_\theta]\oplus R\cdot g_\beta \oplus R[V_\eta][4]
\oplus R\cdot g_\alpha[4] \right)^{\Pi,8}  .\]
This means that
\[ I^+(Y)_n = \left\{ \begin{array}{cc} 
R\{g_\beta\}\oplus \prod_{i\geq 0} R\{V_\theta^{2i},V_{\eta}^{2i-1}\} & n\equiv 0\Mod{8} \\
R\{g_\alpha\} \oplus \prod_{i\geq 0} R\{V_\theta^{2i+1},V_\eta^{2i}\} & n\equiv 4\Mod{8} \\
0 & \mbox{ otherwise. } \end{array} \right. \]
For $x=(r_{2i}V^{2i}_\eta)_{i\geq 0} \in I^+(Y)_4$, $r_{2i}\in R$, we have 
$U\cdot x = r_0g_\alpha + (r_{2i}V^{2i-1}_\eta)_{i\geq 1}\in I^+(Y)_0$.
For $x=rg_\alpha$, $r\in R$, we have $U\cdot x=3rg_\beta$. 
\end{example}  

\subsection{The Case \texorpdfstring{$I^-$}{Iminus}}
In this section we will calculate $I^-(\ovl{Y}_\Gamma)$ for all
finite subgroups $\Gamma\subset \om{SU}(2)$. In contrast to the
calculations in the previous section, there will be a number of nontrivial
differentials in the index spectral sequence. Nevertheless, the spectral sequence will still stabilize after a finite number of steps in each
case.

\begin{lemma} \label{Non-trivial-Differential}
Let $\Gamma\subset \om{SU}(2)$ be a finite subgroup. Then
the only possibly nontrivial differentials in the index spectral sequence
\[ E^1_{s,t}=\bigoplus_{j(\alpha)\equiv s}H^-_A(\alpha)_t \implies I^-(\overline{Y}_\Gamma)_{s+t} \]
are $d^{4r}\colon E^{4r}_{4(s+r),-4(r-1)}\ra E^{4r}_{4s,3}$ for $s$ and $r\geq 1$. 
\end{lemma}
\begin{proof} As $j(\alpha)\equiv 0\mod{4}$ for each $\alpha\in \MC{C}$,
we have $E^1_{s,*}=0$ unless $4|s$. By Lemma \ref{Orbit-Calc-Lemma} and
Corollary \ref{Orbit-Calc-Corollary} the $R[U]$-module $H^-_A(\alpha)$
is given by $R[U]$ with $|U|=-4$, $R[Z][2]$ with $|Z|=-2$ or $R[3]$ depending on whether $\alpha$ is fully reducible, reducible or irreducible, respectively. This means that $E^1_{4s,t}=0$ for $t\geq 4$ and
odd $t\leq 2$. Since the differential $d^r$ has bidegree $(-r,r-1)$
we deduce that it can only be nonzero if it lands in the bidegree
$(4s,3)$ for some $s$. The only such differentials that also begin
in a possibly nonzero group are $d^{4r}\colon E^{4r}_{4(s-r),-4(r-1)}\ra E^{4r}_{4s,3}$. \end{proof}

Our first task will be to give an explicit formula for the nontrivial
differentials. 
Fix a finite subgroup $\Gamma\subset \om{SU}(2)$ and let
$M=DCI(\overline{Y}_\Gamma)$ be equipped with the index filtration.
Recall that by Lemma \ref{Bounded-FloerFiltration}
\[ DCI^-(\overline{Y}_\Gamma) =\om{colim}_p C^-_A(F_pM) =\bigcup_p C^-_A(F_pM).\]
In the following we will use the fact from Lemma \ref{Technical-Extension-Lemma} that $F_{4p+t}M=F_{4p}M$ for all $p\in \Z$ and $0\leq t\leq 3$
without further mention.   

\begin{lemma} \label{Technical-Degree-Lemma}
For all $p,s\in \Z$ and $0\leq t\leq 3$ we have
\begin{align*}
C^-_A(F_{4p}M)_{4s+t} & \cong \left\{ \begin{array}{cc}
\bigoplus_{0\leq q\leq p-s} M_{4(s+q)+t} & \mbox{ if } s\leq p \\
0 & \mbox{ if } s>p \end{array} \right.  \\
C^-_A(\overline{F}_{4p}M)_{4s+t} & \cong 
\left\{ \begin{array}{cc} M_{4p+t} & \mbox{ if } s\leq p \\
                           0 & \mbox{ if } s>p \end{array} \right.
\end{align*} 
The differential in $C^-_A(\overline{F}_{4p}M)$ is given by
$u\colon M_{4p}\ra M_{4p+3}$ in degree $4s$ with $s\leq p$, and vanishes otherwise.
The differential in $C^-_A(F_{4p}M)$ is given by
\[ \dd((x_q)_{0\leq q\leq p-s}) =
(\dd_Mx_q-(-1)^tx_{q-1}u)_{0\leq q\leq p-s+1}   \]
for $(x_q)_{0\leq q\leq p-s}\in \bigoplus_{0\leq q\leq p-s} M_{4(s+q)+t}
=C^-_A(F_{4p}M)_{4s+t}$ where $s\leq p$, and we interpret $x_{-1}=0=x_{p-s+1}$.
\end{lemma}
\begin{proof} This is a simple consequence of Proposition \ref{Explicit-Cpm-Models} using the fact that $F_{4p}M_n=M_n$ for $n\leq 4p+3$ and
$F_{4p}M_n=0$ otherwise. \end{proof}

For each $\eta\in \MC{C}^{f.red}$, $\lambda\in \MC{C}^{red}$ and $\alpha\in \MC{C}^{irr}$ we introduce variables $U_\eta, Z_{\lambda}$ and a generator
$h_\alpha$. Then, according to Lemma \ref{Orbit-Calc-Lemma} and
Corollary \ref{Orbit-Calc-Corollary},
\[ H^-_A(\eta) = R[U_\eta], \;\;\;\; H^-_A(\lambda)=R[Z_\lambda][2] \;\;
\mbox{ and } \;\; H^-_A(\alpha) = R\cdot h_\alpha \]
where $|U_\eta|=-4$, $|Z_\lambda|=-2$ and $|h_\alpha|=3$.

\begin{lemma} \label{Description-Differentials} 
Let $\{ E^r,d^r\}_{r\geq 1}$ be the index spectral sequence with
\[ E^1_{s,t}=\bigoplus_{j(\alpha)\equiv s}H^-_A(\alpha)_t \implies I^-_{s+t}(\overline{Y}_\Gamma)  .\]
Then, for each $r\geq 1$ and $s$, $E^{4r}_{4s,-4(r-1)}=E^1_{4s,-4(r-1)}$
is a free $R$-module with generators
\[ \{U^{r-1}_\eta : \eta\in \MC{C}^{f.red} \ni j(\eta)\equiv 4s\} \cup 
\{Z^{2r-1}_\lambda : \lambda\in \MC{C}^{red} \ni j(\lambda)\equiv 4s \} \]
and $E^1_{4s,3}$ is a free $R$-module with generators
$\{ h_\alpha : \alpha\in \MC{C}^{irr} \ni j(\alpha)\equiv 4s\}$.
Let $p_r:E^1_{4s,3}\ra E^{4r}_{4s,3}$ denote the natural surjection. Then in terms of these generators the differential $d^r\colon E^{4r}_{4(s+r),-4(r-1)}\ra E^{4r}_{4s,3}$ is determined by 
\begin{align*} 
d^{4r}(U^{r-1}_\eta) &= p_r\left(\sum_{(\alpha_1,\cdots,\alpha_r)} n_{\alpha_1\eta}\left(\prod_{i=1}^{r-1}n_{\alpha_{i+1}\alpha_{i}}\right)\cdot h_{\alpha_r}\right)  \\
d^{4r}(W^{2r-1}_\lambda)&= p_r\left(\sum_{(\alpha_1,\cdots,\alpha_r)} n_{\alpha_1\lambda}\left(\prod_{i=1}^{r-1}n_{\alpha_{i+1}\alpha_{i}}\right)\cdot h_{\alpha_r}\right),
\end{align*}
where the sums are taken over tuples $(\alpha_1,\cdots,\alpha_r)\in (\MC{C}^{irr})^r$ for which $(\eta,\alpha_1,\cdots,\alpha_r)$, respectively
$(\lambda,\alpha_1,\cdots,\alpha_r)$, form a path in the
graph $\MC{S}_\Gamma$ (see Definition \ref{Def-Adjacency}).
\end{lemma} 
\begin{proof} Before we introduce the explicit generators we will work
out a formula for the differentials in terms of the complex $M$.
By Lemma \ref{Non-trivial-Differential} we have
$E^{4r}_{4(s+r),-4(r-1)}=E^1_{4(s+r),-4(r-1)}=
H^-_A(\overline{F}_{4(s+r)}M)_{4(s+1)}$ and by Lemma \ref{Technical-Degree-Lemma} the latter group may be identified with
$\om{Ker}(u)\subset M_{4(s+r)}$.
Recall from diagram \eqref{SS-Diff-Def} that the differential $d^{4r}$
is defined by first applying the connecting homomorphism
$\delta\colon H^-_A(\overline{F}_{4(s+r)}M)\ra H^-_A(F_{4(s+r)-1}M)$, then lifting along the map
$H^-_A(F_{4s}M)\ra H^-_A(F_{4(s+r)-1}M)$ and then pushing down along
$H^-_A(F_{4s}M)\ra H^-_A(\overline{F}_{4s}M)$. On the chain level
we have the following diagram
\[ \begin{tikzcd} C^-_A(\overline{F}_{4(s+r)}M)_{4(s+1)} \arrow{d}{\cong}
\arrow[dashrightarrow]{r}{\delta} & C^-_A(F_{4(s+r)-1}M)_{4s+3} \arrow{d}{\cong} &
C^-_A(F_{4s}M)_{4s+3} \arrow{l}[swap]{\iota} \arrow{d}{\cong} \\
M_{4(s+r)} \arrow[dashrightarrow]{r}{\delta} & 
\bigoplus_{q=0}^{r-1}M_{4(s+q)+3}
& M_{4s+3} \arrow{l}[swap]{\iota}. \end{tikzcd} \] 
The arrows correspond to the connecting homomorphism $\delta$ are
dashed as they are not well-defined before passing to homology. 
Given $x\in \om{Ker}u\subset M_{4(s+r)}$ the
element $\delta(x)$ is represented on the chain level by
$(0,0,\cdots,0,\dd_Mx)$. As $\iota$ is the inclusion of the first summand,
this element does not lift along $\iota$. Let $\psi:M\ra M$ be the map
of Lemma \ref{Lambda-Lemma} and let
\[ y\coloneqq (\psi^{r-q}x)_{q=1}^{r-1}\in \bigoplus_{q=1}^{r-1}M_{4(s+q)} \cong C^-_A(F_{4(s+r-1)}M)_{4(s+1)}. \]
Then $\dd(y) = (\dd_M \psi^{r-1}x,0,0,\cdots,0,-\dd_M x)$, which implies
that $\iota[\dd_M \psi^{r-1}x]=[(0,\cdots,0,\dd_M x)]$ in homology. By
Lemma \ref{Technical-Degree-Lemma} we have the following sequence of
isomorphisms
\[ M_{4s+3} \cong C^-_A(F_{4s}M)_{4s+3} \cong C^-_A(\ovl{F}_{4s}M)_{4s+3}
\cong H^-_A(\ovl{F}_{4s}M)_{4s+3}\cong E^1_{4s,3} .\]
Moreover, according to Lemma \ref{Non-trivial-Differential}, there
are no outgoing differentials from $E^r_{4s,3}$ for $r\geq 1$. We conclude that
$d^{4r}(x)$ is given by the image of $\dd_M\psi^{r-1}x$ along
the surjection $E^1_{4s,3}\ra E^{4r}_{4s,3}$ for each
$x\in E^{4r}_{4(s+r),-4(r-1)}\subset M_{4(s+r)}$. In other words, we have the
following commutative diagram
\begin{equation} \label{Differential-Diagram}
\begin{tikzcd} E^{4r}_{4(s+r),-4(r-1)} \arrow{d}{d^{4r}} \arrow{r}{=} & 
E^1_{4(s+r),-4(r-1)} \arrow{r}{\cong} & \om{Ker}(u) \subset M_{4(s+r)}
\arrow{d}{\dd_M\circ \psi^{(r-1)}} \\
E^{4r}_{4s,3} & E^1_{4s,3} \arrow{l}[swap]{p_r}  \arrow{r}{\cong } 
& M_{4s+3}. \end{tikzcd} 
\end{equation}

Now, $\om{Ker}u\subset M_{4(s+r)}$ is freely generated by
$b_\eta\in H_0(\eta)$, $b_\lambda\in H_0(\lambda)$ for $\eta\in \MC{C}^{f.red}$ and $\lambda\in \MC{C}^{red}$ with $j(\eta)\equiv j(\lambda)\equiv 4(s+r)$. It is important to note that under the top horizontal identification in 
\eqref{Differential-Diagram} we may, and will, require that these
correspond to the generators $U^{r-1}_\eta$ and $W^{2r-1}_\lambda$ in $E^{1}_{4(s+r),-4(r-1)}$. Similarly, $M_{4s+3}$ is freely generated by $t_\alpha=b_\alpha\cdot u\in H_3(\alpha)$
for $\alpha\in \MC{C}^{irr}$ with $j(\alpha)\equiv 4s$,
and under the lower horizontal isomorphism we require that these 
correspond to the generators $h_\alpha\in E^1_{4s,3}$.
Recall from equation \eqref{Psi-Formula} that
\[ \psi(b_\eta)=u^{-1}\dd_M(b_\eta) = \sum_{\alpha\in\MC{C}^{irr}} u^{-1}(n_{\alpha\eta}t_\alpha)=\sum_{\alpha\in \MC{C}^{irr}} n_{\alpha\eta}b_\alpha. \]
and similarly for $\psi(b_\lambda)$. Taking into account the
fact that $n_{\alpha\beta}=0$ whenever $\alpha$ and $\beta$ are not adjacent
in the graph $\MC{S}_\Gamma$, we obtain
\[ \dd_M \psi^{r-1}(b_\eta) = \sum_{(\alpha_1,\cdots,\alpha_r)}
(n_{\alpha_1\eta}n_{\alpha_2\alpha_2}\cdots n_{\alpha_r\alpha_{r-1}})\cdot
t_{\alpha_r}, \]
where the sum is taken over $(\alpha_1,\cdots,\alpha_r)\in (\MC{C}^{irr})^r$
for which $(\eta,\alpha_1,\cdots,\alpha_r)$ forms an edge path in
$\om{S}_\Gamma$. An analogous formula holds for $b_\lambda$, so in view
of diagram \eqref{Differential-Diagram} the proof is complete.    
\end{proof} 

By combining the above lemma with Theorem \ref{Differential-Graphs} we obtain
complete control over all the differentials in the spectral sequences.
The following result takes care of all the extension problems we will meet.

\begin{lemma} \label{Extension-Lemma}
Let $\Gamma\subset \om{SU}(2)$ be a finite subgroup and
let $(E^r,d^r)_{r\geq 1}$ be the index spectral sequence with
\[ E^1_{s,t} = \bigoplus_{j(\alpha)\equiv s} H^-_A(\alpha)_t \implies
I^-(\overline{Y}_\Gamma)_{s+t}  .\]
Assume that $E^\infty_s$ is a free $R[U]$-module for 
$s=0,4$. Then there is an isomorphism of $R[U]$-modules
\[ I^-(\overline{Y}_\Gamma) \cong (E^\infty_0\oplus E^\infty_4)^{\oplus,8}  .\]
\end{lemma}
\begin{proof} Recall that the index spectral sequence is periodic
in the sense that $E^r_{s,t}\cong E^r_{s+8,t}$, or equivalently $E^r_s[8]\cong E^r_{s+8}$, for all $s,t\in \Z$ and $r\in \N\cup \{\infty\}$. Therefore,
the assumption that $E_0^\infty$ and $E_4^{\infty}$ are free $R[U]$-modules
implies that $E^\infty_s$ is a free $R[U]$-module for each $s\in \Z$, as $E^1_s=0$ for all $s$ not divisible by $4$.  

Write $I=I^-(\overline{Y}_\Gamma)$. By Theorem \ref{Index-SS-Theorem}
$(E^r,d^r)_{r\geq 1}$ is a spectral sequence
of $R[U]$-modules that converges strongly to $I$. In other words,
$I$ carries an exhaustive and complete Hausdorff filtration
$\{F_sI\}_s$ of $R[U]$-submodules and there is an isomorphism $E^\infty_s\cong \overline{F}_sI$ of $R[U]$-modules for each $s\in \Z$. In view
of the fact that $E^\infty_s$ is free over $R[U]$, the short exact
sequence
\[ \begin{tikzcd} 0\arrow{r} & F_{s-1}I \arrow{r}{\iota_s} & F_sI 
\arrow{r}{\pi_s} & \overline{F}_sI\cong E^\infty_s \arrow{r} & 0 
\end{tikzcd} \]
splits in the category of $R[U]$-modules for each $s\in \Z$. Hence, $F_{s-1}I\oplus \overline{F}_sI\cong F_sI$ and inductively
\[ F_{s-1}I\oplus \left( \bigoplus_{t=0}^{r} \overline{F}_{s+t}I \right) \cong F_{s+r}I  \]
for $s\in \Z$ and $r\geq 1$. These are all isomorphisms of $R[U]$-modules.
By Lemma \ref{Technical-Degree-Lemma} we have $E^1_{s,t}=0$ for all
$s$ and $t\geq 4$. Therefore, for fixed $n\in \Z$, $(E^1_s)_n = E^1_{s,n-s}=0$
for all $s\leq n-4=:s_0$. Hence, $(\overline{F}_sI)_n=0$ for all $s\leq s_0$.
Since the filtration of $I$ is complete Hausdorff it follows that
$(F_sI)_n=0$ for all $s\leq s_0$. Consequently, using the above isomorphisms,
we obtain
\[ I_n = \om{colim}_{s\geq s_0} (F_sI)_n \cong \om{colim}_{s\geq s_0}
\bigoplus_{i=s_0}^s (\overline{F}_sI)_n \cong \bigoplus_s (E^\infty_s)_n \]
for each $n\in \Z$. These isomorphisms piece together to an isomorphism 
$I\cong \bigoplus_s E^\infty_s$ of $R[U]$-modules. Finally, we may exploit the periodicity of $E^\infty$
and the fact that $E^\infty_s=0$ unless $4|s$ to simplify:
\[ I \cong \bigoplus_s E^\infty_s \cong \bigoplus_{s} (E^\infty_{8s}\oplus E^{\infty}_{8s+4}) \cong \bigoplus_s (E^\infty_0[8s]\oplus E^\infty_4[8s])
= (E^\infty_0\oplus E^\infty_4)^{\oplus,8} \]
and the proof is complete. 
\end{proof}    

It is not possible to give a uniform result for the calculation of $I^-(\ovl{Y}_\Gamma)$ for all the subgroups $\Gamma\subset \om{SU}(2)$ simultaneously.
We will therefore treat the simplest cases $\Gamma=I^*,O^*,T^*,C_m$ first and treat the more involved case of binary dihedral groups afterwards.

We have the following table over the flat connections extracted from 
Appendix $A$. The notation is compatible with Proposition \ref{Differential-Graphs}. 
\[ \def\arraystretch{1.1} 
\begin{array}{|c|c|c|c|}
\hline
\Gamma & \MC{C}^{f.red} & \MC{C}^{irr} & \MC{C}^{red} \\ \hline
C_{2m} & \theta,\eta & {} & \lambda_1,\cdots, \lambda_{m-1} \\ \hline
C_{2m+1} & \theta & {} & \lambda_1,\cdots,\lambda_m \\ \hline
I^* & \theta & \alpha, \beta & {} \\ \hline
O^* & \theta,\eta & \alpha,\beta  & {} \\ \hline
T^* & \theta & \alpha & \lambda \\ \hline  \end{array} \]

\begin{theorem} \label{IMinus-1}
The negative equivariant instanton Floer homology
associated with the trivial $\om{SU}(2)$-bundle over $\overline{Y}_\Gamma$
for $\Gamma=C_{2m},C_{2m+1}$ is given by
\begin{align*}
I^-(\overline{Y}_{C_{2m}}) &= \left( (R[U_\theta]\oplus R[U_\eta][4m]\oplus
\bigoplus_{i=1}^{m-1} R[Z_{\lambda_i}][4i+2] \right)^{\oplus,8}  \\
I^-(\overline{Y}_{C_{2m+1}}) &= \left( R[U_\theta]\oplus \bigoplus_{i=1}^m R[Z_{\lambda_i}][4i+2] \right)^{\oplus,8},
\end{align*}
while for $\Gamma=I^*,O^*,T^*$ we have $I^-(\overline{Y}_\Gamma)=(X_\Gamma)^{\oplus, 8}$ where $X_\Gamma$ is the $R[U]$-submodule
of $P_\Gamma$ generated by $G_\Gamma$ specified in the following table.     
\[ \def\arraystretch{1.3} 
\begin{array}{|c|c|c|}
\hline
\Gamma & P_\Gamma & G_\Gamma \\ \hline
I^* & R[U_\theta] & \{U_\theta^2\} \\ \hline
O^* & R[U_\theta]\oplus R[U_\eta][4] & \{U_\theta^1,U_\eta^1 \} \\ \hline
T^* & R[U_\theta]\oplus R[Z_\lambda][2] & \{U_\theta^1,Z_\lambda^0,3U_\theta^0-Z^1_\lambda\} \\ \hline \end{array} \]
\end{theorem}
\begin{proof} For $\Gamma=I^*,O^*,T^*,C_n$ we will write $\MC{C}=\MC{C}_\Gamma$
for the set of flat connections and $(E^r,d^r)_{r\geq 1}$ will denote the corresponding index spectral sequence with
\[ E^1_{s,t} = \bigoplus_{j(\alpha)\equiv s}H^-_A(\alpha)_t\implies I^-(\overline{Y}_\Gamma)_{s+t} .\]
In each case we will calculate $E^\infty$ and observe that $E^\infty_0$
and $E^\infty_4$ are free $R[U]$-modules. By Lemma \ref{Extension-Lemma}
this is enough to conclude that $I^-(\overline{Y}_\Gamma)=(E_0^\infty \oplus E_4^\infty)^{\oplus,8}$. In the following we will make consistent use
of Lemma \ref{Non-trivial-Differential}, Theorem \ref{Differential-Graphs} and Lemma \ref{Description-Differentials} that together determine all the
differentials explicitly.  

$\mathbf{C_n}$. As $\MC{C}^{irr}=\emptyset$ there are no nontrivial
differentials in the spectral sequence so that $E^1=E^\infty$. For $n=2m$ we have $\MC{C}=\{\theta,\eta,\lambda_1,\cdots,\lambda_{m-1}\}$ with $\theta$, $\eta$ fully reducible and the $\lambda_i$ reducible. The grading is given by $j(\lambda_i)\equiv 4i$ and $j(\eta)\equiv 4m$. This gives
\[ E^\infty_0\oplus E^\infty_4 = R[U_\theta]\oplus R[U_\eta][j(\eta)]
\oplus \left( \bigoplus_{i=1}^{m-1} R[Z_{\lambda_i}][j(\lambda_i)+2] \right).\]
For $n=2m+1$ we have $\MC{C}=\{\theta,\lambda_1,\cdots,\lambda_m\}$ with
$\theta$ fully reducible and the $\lambda_i$ reducible. The grading is
given by $j(\lambda_i) \equiv 4i$. Hence, 
\[ E^\infty_0\oplus E^\infty_4 = R[U_\theta] \oplus \left( \bigoplus_{i=1}^m R[Z_{\lambda_i}][j(\lambda_i)+2] \right) .\]  
In both cases these are free $R[U]$-modules ($R[Z_\lambda]$ is freely generated by $\{Z_\lambda^0,Z_\lambda^1\}$). The stated results are obtained by applying
$X\mapsto X^{\oplus,8}$ to the above formulas for $E_0^\infty\oplus E^\infty_4$.

$\mathbf{I^*}$. In this case we have $\MC{C}=\{\theta,\alpha,\beta\}$ with
$\alpha$ and $\beta$ irreducible. The grading is given by $j(\theta)=j(\beta)=0$ and $j(\alpha)=4$. The first nontrivial differential in the index spectral sequence is
$d^4\colon E^{4}_{8s,0}=R\{ U^0_\theta\} \ra R\{ h_\alpha\}=E^4_{8s-4,3}$
and is given by $d^4(U^0_\theta) = n_{\alpha\eta}h_\alpha = h_\alpha$. Therefore, $E^\infty_{8s+4,3}=0=E^\infty_{8s,0}$. The next differential is
$d^8\colon E^8_{8s,-4}=R\{ U_\theta^1\} \ra R\{ h_\beta\}= E^8_{8(s-1),3}$
and is given by $d^8(U^1_\theta)=n_{\beta\alpha}n_{\alpha\eta}h_\beta
=4h_\beta$. As $2\in R$ is invertible, this is an isomorphism and we conclude that $E^\infty_{8s,3}=0$ and $E^\infty_{8s,-4}=0$. There are no more nontrivial 
differentials in the spectral sequence. Therefore, $E^\infty_4=0$ and $E^\infty_0 \subset R[U_\theta]$ is the free $R[U]$-submodule generated by $U_\theta^2$ as required. 

$\mathbf{O^*}$. Here we have $\MC{C}=\{\theta,\eta,\alpha,\beta\}$ with
$\theta$, $\eta$ fully reducible and $\alpha,\beta$ irreducible. The grading
is $j(\theta)=j(\beta)=0$ and $j(\alpha)=j(\eta)=4$. The first nontrivial
differentials are
\begin{align*}
d^4\colon & E^4_{8s+4,0}=R\{ U^0_\eta\} \ra R\{ h_\beta\} =E^4_{8s,3}  \\
d^4\colon & E^4_{8s,0}=R\{ U^0_\theta\} \ra R\{ h_\alpha\}= E^4_{8s-4,3} .
\end{align*}
These are both isomorphisms as $n_{\beta\eta}=1=n_{\alpha\theta}$. This implies
that $E^\infty_{8s,3}=E^5_{8s,3}=0$ and $E^\infty_{8s-4,3}=E^5_{8s-4,3}=0$. There can therefore be no more nontrivial differentials. 
Hence, $E^\infty_0\subset R[U_\theta]$ and $E^\infty_4\subset R[U_\eta][4]$
are the free $R[U]$-submodules generated by $U_\theta^1$ and $U_\eta^1$,
respectively. 

$\mathbf{T^*}$. In this case $\MC{C}=\{\theta,\alpha,\lambda\}$ with
$\theta$ fully reducible, $\alpha$ irreducible and $\lambda$ reducible.
The grading is $j(\theta)=j(\lambda)=0$ and $j(\alpha)=4$. The first
nontrivial differential is $d^4\colon E^4_{8s,0}=R\{Z_\lambda^1,U_\theta^0\}\ra R\{h_\alpha\}=E^4_{8s-4,3}$ and is determined by
$d^4(U^0_\theta)=n_{\alpha\theta}h_\alpha = h_\alpha$ and
$d^4(Z^1_\lambda) = n_{\alpha\lambda}h_\alpha = 3h_\alpha$. This map
is surjective with kernel $R\{3U^0_\theta-Z^1_\lambda\}$. Therefore,
$E^\infty_{8s-4,3}=E^4_{8s-4,3}=0$ and there are no more nontrivial differentials. We find $E^\infty_4=0$, $(E^\infty_{0})_0= R\{3U^0_\theta-Z^1_\lambda\}$
and $(E^\infty_0)_n=(R[U_\theta]\oplus R[Z_\lambda][2])_n$ in all other
degrees $n$. It now suffices to observe that $E^\infty_0$ is indeed the free
submodule of $R[U_\theta]\oplus R[Z_\lambda][2]$ generated by
$\{Z_\lambda^0,3U^0_\theta-Z^1_\lambda,U_\theta^1\}$. This completes
the final case and the proof. \end{proof}  

We will now consider the binary dihedral groups $D_m^*$. It is necessary to
partition the calculations into cases depending on the residue of $m$
mod $4$. We have the following table over the flat connections 
\[ \def\arraystretch{1.1} 
\begin{array}{|c|c|c|c|}
\hline
{} & \MC{C}^{f.red} & \MC{C}^{irr} & \MC{C}^{red} \\ \hline
D_{4n}^* & \theta,\eta_1,\eta_2,\eta_3 & \alpha_1,\cdots,\alpha_{2n} &
{} \\ \hline
D_{4n+1}^* & \theta,\eta & \alpha_1,\cdots,\alpha_{2n} & \lambda \\ \hline
D_{4n+2}^* & \theta,\eta_1,\eta_2,\eta_3 & \alpha_1,\cdots,\alpha_{2n+1} &
{} \\ \hline
D_{4n+3}^* & \theta,\eta & \alpha_1,\cdots,\alpha_{2n+1} & \lambda \\
\hline \end{array} \]
in agreement with Appendix $A$ and Proposition \ref{Differential-Graphs}. 
In all cases the grading of the irreducibles are given by
$j(\alpha_i)\equiv 4i \mod{8}$. For $D^*_{4n+2}$ and $D^*_{4n+3}$ all
the fully reducibles and reducibles satisfy $j(\rho)=0$. For
$D^*_{4n}$ we have $j(\theta)=j(\eta_1)=0$, $j(\eta_2)=j(\eta_3)=4$, while
for $D^*_{4n+1}$ we have $j(\theta)=j(\eta)=0$ and $j(\lambda)=4$. For
the convenience of the reader we include the relevant diagrams for
$S_{D^*_m}$ from Proposition \ref{Differential-Graphs}. 

\begin{center}
\begin{tikzpicture}
\node (27) at (2,0.6) {$\MC{S}_{D^*_{2m}}$};

\node (1) at (0,0) {$\alpha_1$};
\node (2) at (1,0) {$\alpha_2$};
\node (3) at (2.6,0) {$\alpha_{m-1}$};
\node (4) at (4,0) {$\alpha_m$};
\node (5) at (-1,0.4) {$\theta$};
\node (6) at (-1,-0.4) {$\eta_1$};
\node (7) at (5,0.4) {$\eta_2$};
\node (8) at (5,-0.4) {$\eta_3$};

\node (9) at (0.5,-0.4) {$(2|2)$};
\node (10) at (3.4,-0.4) {$(2|2)$};
\node (11) at (-0.4,0.4) {$1$};
\node (12) at (-0.4,-0.4) {$1$};
\node (13) at (4.5,0.4) {$1$};
\node (14) at (4.5,-0.4) {$1$};

\draw (5)--(1)--(2); \draw (6)--(1);
\draw (3)--(4)--(7); \draw (4)--(8);
\draw[dotted,thick] (2)--(3);
\end{tikzpicture}
\end{center}
\begin{center}
\begin{tikzpicture} 
\node (28) at (10,0.6) {$\MC{S}_{D^*_{2m+1}}$};

\node (15) at (8,0) {$\alpha_1$};
\node (16) at (9,0) {$\alpha_2$};
\node (17) at (10.6,0) {$\alpha_{m-1}$};
\node (18) at (12,0) {$\alpha_m$};
\node (19) at (13,0) {$\lambda$};
\node (20) at (7,0.4) {$\theta$};
\node (21) at (7,-0.4) {$\eta$};

\node (22) at (8.5,-0.4) {$(2|2)$};
\node (23) at (11.4,-0.4) {$(2|2)$};
\node (24) at (12.5,-0.4) {$2$};
\node (25) at (7.6,0.4) {$1$};
\node (26) at (7.6,-0.4) {$1$};

\draw (16)--(15)--(21); \draw (15)--(20);
\draw (17)--(18)--(19);
\draw[dotted, thick] (16)--(17);
\end{tikzpicture}
\end{center}

\begin{theorem} \label{IMinus-2}
The negative instanton Floer homology associated with
the trivial $\om{SU}(2)$-bundle over $\overline{Y}_\Gamma$ for
$\Gamma=D^*_m$ is given by $I^-(\overline{Y}_{D_m^*})=(X_m)^{\oplus,8}$,
where $X_m$ is the $R[U]$-submodule of $P_m$ generated by $G_m$
specified in the following tables.
\[ \def\arraystretch{1.3}
\begin{array}{|c|c|}
\hline m & P_m \\ \hline
4n & R[U_\theta]\oplus R[U_{\eta_1}]\oplus R[U_{\eta_2}][4]\oplus
R[U_{\eta_3}][4] \\ \hline
4n+1 & R[U_\theta] \oplus R[U_\eta] \oplus R[Z_\lambda][6] \\ \hline
4n+2 & R[U_\theta]\oplus R[U_{\eta_1}]\oplus R[U_{\eta_2}]\oplus R[U_{\eta_3}]
\\ \hline
4n+3 &  R[U_\theta]\oplus R[U_\eta]\oplus R[Z_\lambda][2] \\ \hline
\end{array} \] 

\[ \def\arraystretch{1.3}
\begin{array}{|c|c|}
\hline m & G_m \\ \hline
4n & \{ U^0_{\theta}-U_{\eta_1}^0, U^0_{\eta_2}-U^0_{\eta_3},U_\theta^n,U_{\eta_2}^n\} \\ \hline 
4n+1 & \{U_\theta^0-U_\eta^0,Z^0_\lambda,U_\theta^n,Z^{2n+1}_\lambda \}
\\ \hline
4n+2 & \{U^0_\theta-U_{\eta_1}^0,U^0_{\eta_2}-U^0_{\eta_3},
U_\theta^n-U_{\eta_2}^n, U_\theta^{n+1}\} \\ \hline
4n+3 & \{U^0_\theta-U^0_\eta,Z^0_\lambda,2U^n_{\theta}-Z_\lambda^{2n+1},U_\theta^{n+1} \}  \\ \hline \end{array} \]

\end{theorem}
\begin{proof} 
Observe first that the modules $P_m$ in the above table
coincides with the $R[U]$-submodule of $E^1_0\oplus E^1_4$ in the index
spectral sequence generated by the reducibles and fully reducibles.
Furthermore, the submodule generated by $G_m$ is free in each case. 
Therefore, following the same procedure as in the proof of Theorem \ref{IMinus-1}, it will be sufficient to show that $E^\infty_0\oplus E^\infty_4$ is the $R[U]$-submodule of $P_m$ generated by $G_m$ in each
case.  

For $m=4n,4n+1$ we have
\[ E^1_{8s,3}=R\{h_{\alpha_2},h_{\alpha_4},\cdots ,h_{\alpha_{2n}}\}
\mbox{ and } E^1_{8s+4,3}=R\{h_{\alpha_1},h_{\alpha_3},\cdots ,h_{\alpha_{2n-1}} \}  \]
for each $s\in \Z$. For $m=4n+2,4n+3$ we have the same formulas except that
one generator $h_{\alpha_{2n+1}}$ is adjoined to the latter group. To simplify the notation put $h_i\coloneqq h_{\alpha_i}$
for $1\leq i\leq 2n+1$. For each $r\geq 1$ let $p_r\colon E^1_{4s,3}=E^4_{4s,3}\ra E^{4r}_{4s,3}$ denote the natural surjection. 

Consider first the cases $m=4n,4n+1$. To unify the notation slightly
we write $(U_1,U_2,U_3,U_4)=(U_\theta,U_{\eta_1},U_{\eta_2},U_{\eta_3})$ for $m=4n$ and $(U_1,U_2,Z)=(U_\theta,U_\eta,Z_\lambda)$ for $m=4n+1$. We claim that for $1\leq r\leq n$ the differentials
\begin{align*}
 f^r\coloneqq d^{4r}\colon & E^{4r}_{8s,-4(r-1)}\ra E^{4r}_{8s-4r,3} \\
 g^r\coloneqq d^{4r}\colon E^{4r}_{8s+4,-4(r-1)}\ra E^{4r}_{8s-4(r-1),3}
\end{align*}
are given by the formulas  
\begin{align*}
f^r(U_1^{r-1})&=f^r(U_2^{r-1})=2^{r-1}p_r(h_r) \\
g^r(U_3^{r-1})&=g^r(U_4^{r-1})=2^{r-1}p_r(h_{2n-r+1}) \\
g^r(Z^{2r-1})&=2^rp_r(h_{2n-r+1}).
\end{align*} 
We verify this by induction on $r$. Let $1\leq r\leq n-1$ and assume that
the statement is true for all $i$ with $1\leq i<r$. This implies in particular
that $p_r(h_i)=0$ for $1\leq i<r$ and $2n-r+1<i\leq 2n$. By Lemma
\ref{Description-Differentials} we have
\[ f^r(U_1^{r-1}) = p_r\left(\sum_{(\beta_1,\cdots,\beta_r)} (n_{\beta_1\theta}n_{\beta_2\beta_1}\cdots n_{\beta_r\beta_{r-1}})\cdot h_{\beta_r}\right), \]
where the sum runs over all $(\beta_1,\cdots,\beta_r)\in (\MC{C}^{irr})^r$ for which 
$(\theta,\beta_1,\cdots \beta_r)$ forms an edge path in $\MC{S}_\Gamma$.
As such a path has length $r$, it must, in view of the graphs $\MC{S}_{D_m^*}$ shown above, terminate at some vertex $\alpha_i$ with $i\leq r$. By the inductive hypothesis $p_r(h_{\alpha_i})=0$ for $i<r$, so the only nonzero term in the formula corresponds to the path $(\alpha_1,\alpha_2,\cdots,\alpha_r)$.
As $n_{\alpha_{i+1}\alpha_i}=2$ for each $i$ and $n_{\alpha_1\theta}=1$
we conclude that $f^r(U_1^{r-1})=2^{r-1}p_r(h_r)$. The other formulas
follow by essentially identical arguments, we only note that the additional
factor $2$ picked up in the last formula follows from the fact that
$n_{\alpha_{2n}\lambda}=2$. This completes the inductive step and the
claim is verified.

The above formulas for the differentials imply, as $2\in R$ is invertible, that the generators
$h_i$, $1\leq i\leq 2n$, of $E^1_{8s,3}$ and $E^1_{8s+4,3}$ are killed
off two by two until we reach $E^{4n+1}_{8s,3}=0$, $E^{4n+1}_{8s+4,3}=0$.
Therefore, by Lemma \ref{Non-trivial-Differential}, $E^{4n+1}=E^\infty$. Moreover, we find $E^\infty_{8s,-4(r-1)}=\om{Ker}(f^r)=R\{(U_1^{r-1}-U_2^{r-1})\}$ for $1\leq r\leq n$, and
similarly
\[ E^\infty_{8s+4,-4(r-1)}= \left\{ \begin{array}{cl} R\{(U_3^{r-1}-U_4^{r-1})\} & \mbox{ for } m=4n \\
0 & \mbox{ for } m=4n+1 \end{array} \right. \]
for $1\leq r\leq n$. 
In all other degrees we have $E^1_{s,t}=E^\infty_{s,t}$.
We therefore see that $E^\infty_0$ is freely generated as an $R[U]$-module
by $\{(U_1^0-U_2^0),U_1^n\}$ for both $m=4n$ and $m=4n+1$, while
$E^0_4$ is freely generated by $\{(U_3^0-U_4^0),U_3^n\}$ for $m=4n$ and
by $\{Z_\lambda^0,Z_\lambda^{2n+1}\}$ for $m=4n+1$. We have thus completed
the cases $m=4n,4n+1$. 

Consider now the cases $m=4n+2,4n+3$. As earlier we write 
$(U_1,U_2,U_3,U_4)=(U_\theta,U_{\eta_1},U_{\eta_2},U_{\eta_3})$ for $m=4n+2$
and $(U_1,U_2,Z)=(U_\theta,U_\eta,Z_\lambda)$.
Here $U^{r-1}_{i},Z^{2r-1}\in E^1_{8s,-4(r-1)}$ for each $i$ and $r\geq 1$.
For $1\leq r\leq n+1$ the differentials
\[ d^{4r}\colon E^{4r}_{8s,-4(r-1)}\ra E^{4r}_{8s-4r,3} \]
are given by the formulas
\begin{align*}
d^{4r}(U_1^{r-1})&=d^{4r}(U_2^{r-1})=2^{r-1}p_r(h_r) \\
d^{4r}(U_3^{r-1})&=d^{4r}(U_4^{r-1})=2^{r-1}p_r(h_{2n+2-r}) \\
d^{4r}(Z^{2r-1})&=2^rp_r(h_{2n+2-r}). 
\end{align*} 
This may be verified by induction in exactly the same way as in the cases $n=4m,4m+1$.
We conclude that $E^{4n+5}_{8s,3}=0=E^{4n+5}_{8s+4,3}$ and hence
$E^{4n+5}=E^\infty$. In addition, for $0\leq r\leq n-1$ we obtain
\[ E^\infty_{8s,-4r}=\om{Ker}(d^{4(r+1)})= \left\{ \begin{array}{ll}
R\{(U_1^{r}-U_2^{r}),(U_3^{r}-U_4^{r}) \} & \mbox{if } m=4n+2 \\
R\{ (U_1^{r}-U_2^{r}) \} & \mbox{if } m=4n+3 \end{array} \right. \]
Furthermore, for $m=4n+2$ we have $d^{4(n+1)}(U_i^n)=2^np_{n+1}(h_{n+1})$
for each $1\leq i\leq 4$, and for $m=4n+3$ we have
\[ 2d^{4(n+1)}(U_i^{n})=d^{4(n+1)}(Z^{2n+1})=2^{n+1}p_{n+1}(h_{2n+1}) \]
for $i=1,2$. From this we deduce that
\[ E^\infty_{8s,-4n} = \left\{ \begin{array}{ll}
R\{(U_1^n-U_2^n),(U_3^n-U_4^n),(U_1^n-U_3^n) \} & \mbox{if } m=4n+2 \\
R\{ (U_1^n-U_2^n),(2U_1^{n+1}-Z^{2n+1}) \} & \mbox{if } m=4n+3 \end{array} \right. \]
In all other degrees we have $E^1_{s,t}=E^\infty_{s,t}$. For $m=4n+2$
we conclude that $E^\infty_0$ is freely generated by
$\{U_1-U_2,U_3-U_4,U_1^n-U_3^n,U_1^{n+1}\}$ and
for $m=4n+3$ we conclude that $E^\infty_0$ is freely generated by $((U_1^0-U_2^0),Z_\lambda^0,2U_1^n-Z^{2n+1},U_1^{n+1}\}$. In both cases
$E^\infty_4=0$. This completes the final two cases and hence the proof.  
\end{proof}  

\subsection{The Case \texorpdfstring{$I^\infty$}{Iinfty}} 
This is the simplest calculation. Let $\Gamma\subset \om{SU}(2)$ be a finite
subgroup and write as usual $M=DCI(\overline{Y}_\Gamma)$ and
$\MC{C}$ for the set of flat connections. Introduce a variable
$T_\eta$ for each $\eta\in \MC{C}^{f.red}$ and a variable $S_\lambda$
for each $\lambda\in \MC{C}^{red}$, where $|T_\eta|=-4$, $|S_\lambda|=-2$.
Then according to Corollary \ref{Orbit-Calc-Tate-Corollary} we have
$H^\infty_A(\alpha)=0$ for $\alpha$ irreducible, $H^\infty_A(\lambda)=R[S_\lambda,S_\lambda ^{-1}]$ for $\lambda$ reducible
and $H^\infty_A(\eta) = R[T_\eta,T_\eta ^{-1}]$ for $\eta$ fully reducible.
The $R[U]$-module structure is given by $U\cdot T^i_\eta=T^{i+1}_\eta$
and $U\cdot S^i_\lambda = S^{i-2}_\lambda$ for each $i\in \Z$. 

\begin{definition} For an $R[U]$-module $X$ we define the mod $8$ periodic
$R[U]$-module $X^{\Pi_\infty,8}$ degreewise by
\[ X^{\Pi_\infty,8}_n = \prod_{s\to \infty} X_{n+8s}  \]
and define $U\colon X^{\Pi_\infty,8}_n\ra X^{\Pi_\infty,8}_{n-4}$
to be the product over $U\colon X_{n+8s}\ra X_{n-4+8s}$ for $s\in \Z$. 
\end{definition}

\begin{theorem} \label{Iinfty-1}
The Tate equivariant instanton Floer homology
associated with the trivial $\om{SU}(2)$-bundle over
$\overline{Y}_\Gamma$ is given by
\[ I^\infty(\overline{Y}_\Gamma) =  \left[ \left( \bigoplus_{\eta\in \MC{C}^{f.red}} R[T_\eta,T_\eta^{-1}][j(\eta)]\right) \oplus \left( \bigoplus_{\lambda\in \MC{C}^{red}} R[S_\lambda,S_\lambda^{-1}][j(\lambda)] \right) \right]^{\Pi_\infty,8}  .\]
\end{theorem} 
\begin{proof} Let $(E^r,d^r)_{r\geq 1}$ be the index spectral sequence
with
\[ E^1_{s,t} = \bigoplus_{j(\alpha)\equiv s}H^\infty_A(\alpha)_t \implies
I^\infty(\overline{Y}_\Gamma)_{s+t}  .\]
From the fact that $j(\alpha)\equiv 0 \Mod{4}$ for each $\alpha\in \MC{C}$
and that $H^\infty_A(\alpha)$ is concentrated in even degrees for all
types of orbits, we conclude that $E^1_{s,t}=0$ for all $(s,t)$ with
$s$ odd or $t$ odd. It follows that there are no nontrivial differentials in the spectral sequence so that $E^1=E^\infty$. This implies by Proposition \ref{W-REinfty-Conditions} part (iii) and Theorem \ref{Cond-Convergence} that the spectral sequence converges strongly.

Write $I=I^\infty(\overline{Y}_\Gamma)$. To resolve the extension problems observe first that $E^1_{4s}$ is a (shifted) direct sum of $R[U]$-modules
$R[T_\eta,T_\eta^{-1}]$ and $R[S_\lambda,S_\lambda^{-1}]$.
These contain free $R[U]$-submodules $R[T_\eta]$ and
$R[S_\lambda]$ (freely generated by $\{S^0_\lambda,S^1_\lambda\}$), respectively, and the whole module is the localization of these
free submodules in the multiplicatively closed subset $\{U^i\}_{i\geq 0} \subset R[U]$. By Corollary \ref{U-action-Corollary}
$U\colon I\ra I$ is an isomorphism, and the same is valid for the $R[U]$-submodule $F_sI$ for each $s\in \Z$. We may therefore for each $s$ construct a 
splitting of the short exact sequence
\[ \begin{tikzcd} 0 \arrow{r} & F_{s-1}I \arrow{r} & F_sI \arrow{r} &
\overline{F}_sI \cong E^\infty_s \arrow{r} \arrow[bend right]{l}{\xi_s} & 0 
\end{tikzcd} \]
by defining $\xi$ to be a section on the free submodule and then extending
to the whole module using the universal property of localization.
By induction we obtain an isomorphism
\[ F_sI/F_tI \cong \bigoplus_{q = t+1}^{s} E^\infty_q  \]
of $R[U]$-modules for each pair $s>t$. Since the filtration of $I$ is exhaustive and complete Hausdorff we obtain, exploiting the periodicity, 
\[ I \cong \om{lim}_t \om{colim}_s F_sI/F_tI \cong \prod_{s\to-\infty} E^\infty_s \cong \prod_{s\to -\infty} (E^\infty_0[-8s]\oplus E^\infty_4[-8s])
\]
and the final term is precisely 
$(E^\infty_0\oplus E^\infty_4)^{\Pi_\infty,8}=(E^1_0\oplus E^1_4)^{\Pi_\infty,8}$. The proof is completed by observing that
\[ E^1_0\oplus E^1_4 = \left( \bigoplus_{\eta\in \MC{C}^{f.red}} R[T_\eta,T_\eta^{-1}][j(\eta)] \right) \oplus
\left( \bigoplus_{\lambda\in \MC{C}^{red}} R[S_\lambda,S_\lambda^{-1}][j(\lambda)] \right). \]    
\end{proof} 

\begin{proposition} For each finite subgroup $\Gamma\subset \om{SU}(2)$
the homology norm map 
\[ H(N)\colon I^+(\overline{Y}_\Gamma)_n\ra I^-(\overline{Y}_\Gamma)_{n+3} \] vanishes. In particular, the exact triangle of $R[U]$-modules of Corollary \ref{Floer-Exact-Triangle-Corollary} splits into a short exact sequence
\[ \begin{tikzcd} 0 \arrow{r} & I^-(\overline{Y}_\Gamma) \arrow{r} &
I^\infty(\overline{Y}_\Gamma) \arrow{r} & I^+(\overline{Y}_\Gamma)[4]
\arrow{r} & 0 \end{tikzcd}  .\]
\end{proposition}
\begin{proof} The calculations of Theorem \ref{IPluss-1}, Theorem \ref{IMinus-1} and Theorem \ref{IMinus-2} show that
$I^+(\ovl{Y}_\Gamma)$ and $I^-(\ovl{Y}_\Gamma)$ are concentrated
in even degrees. Therefore, as the norm map has degree $3$, it
must vanish. \end{proof}   

\subsection{Calculations for \texorpdfstring{$Y_\Gamma$}{YGamma}}
In this section we explain, omitting some details, the necessary
modifications needed to calculate the equivariant instanton Floer
groups for $Y_\Gamma$, that is, $S^3/\Gamma$ equipped with the standard
orientation inherited from $S^3$. The tools needed to handle this
are contained in \cite[Theorem~7.10]{Miller19}. First, the given result
states that there is an isomorphism $I^\infty(Y_\Gamma)\cong I^\infty(\overline{Y}_\Gamma)$. Second, in the proof it is verified
that $DCI(Y_\Gamma)\cong DCI(\overline{Y}_\Gamma)^{\vee}$ and from this it 
is not hard to express $DCI(Y_\Gamma)$ as the totalization of a multicomplex.

\begin{lemma} Let $i\colon \MC{C}\ra \Z/8$ denote the grading for $Y_\Gamma$
and let $j\colon \MC{C}\ra \Z/8$ denote the grading for $\overline{Y}_\Gamma$. Then
\[ i(\alpha) = \left\{ \begin{array}{ll} j(\alpha)-3 & \mbox{ if } \alpha\in \MC{C}^{irr} \\ j(\alpha)-2 & \mbox{ if } \alpha\in \MC{C}^{red} \\
j(\alpha) & \mbox{ if } \alpha\in \MC{C}^{f.red}  \end{array} \right. \]
Moreover, $DCI(Y_\Gamma)= \om{Tot}^{oplus}(DCI(Y_\Gamma)_{*,*},\{\dd^r\}_{r=1}^{4})$ where
\[ DCI(Y_\Gamma)_{s,t} =\bigoplus_{i(\alpha)\equiv s} H_t(\alpha),  \]
$\dd^2=0$ and $\dd^1,\dd^3,\dd^4$ are given for
$b_\alpha\in H_0(\alpha) \subset DCI(Y_\Gamma)_{4s+1,0}$, $\alpha\in \MC{C}^{irr}$, by 
\begin{align*}  
\dd^1(b_\alpha)&= \sum_{\eta\in \MC{C}^{f.red}} n_{\alpha\eta} b_\eta  \\
\dd^3(b_\alpha)&= \sum_{\lambda \in \MC{C}^{red}}n_{\alpha\lambda} t_\lambda \\ 
\dd^4(b_\alpha)&= \sum_{\beta\in \MC{C}^{irr}} n_{\alpha\beta} t_\beta  . 
\end{align*}
and vanish otherwise.   
\end{lemma} 

The patterns seen in the calculations of $I^{\pm}(\overline{Y}_\Gamma)$
are reversed for $I^{\pm}(Y_\Gamma)$. Compare the following with
Lemma \ref{Spectral-Degeneracy-Lemma} and Lemma \ref{Non-trivial-Differential}.

\begin{lemma} Let $(E^r,d^r)_{r\geq 1}$ be the index spectral sequence with
\[ E^1_{s,t} = \bigoplus_{i(\alpha)\equiv s} H^-_A(\alpha)_t \implies I^-(Y_\Gamma)_{s+t}  .\]
Then the spectral sequence immediately degenerates, i.e., $E^1=E^\infty$.
  
Let $(E^r,d^r)_{r\geq 1} $ be the index spectral sequence with
\[ E^1_{s,t} =\bigoplus_{i(\alpha)\equiv s} H^+_A(\alpha)_t \implies
I^+(Y_\Gamma)_{s+t} .\]
Then the only possibly nontrivial differentials are $d^{2r+1}\colon E^{2r+1}_{4s+1,0} \ra E^{2r+1}_{4s-2r,2r}$ for $r\geq 0$ and $s\in \Z$. \end{lemma} 
\begin{proof} In the first case recall that $H^-_A(\alpha) = R\cdot h_\alpha$
with $|h_\alpha|=3$ for $\alpha\in \MC{C}^{irr}$, $H^-_A(\lambda) = R[Z_\lambda][2]$, $|Z_\lambda|=-2$ for $\lambda\in \MC{C}^{red}$ and
$H^-_A(\eta) = R[U_\eta]$, $|U_\eta|=-4$ for $\eta\in \MC{C}^{f.red}$.
By the above lemma $i(\rho)\equiv 0 \Mod{2}$ for each $\rho\in \MC{C}^{ref}\cup \MC{C}^{f.red}$ and $i(\alpha)\equiv 1\Mod{4}$ for each $\alpha\in \MC{C}^{irr}$. Therefore $E^1_{s,t}$ can only be nonzero if the total degree $s+t$ is even. As each differential $d^r$ reduces the total degree by one, they must all
vanish and the first assertion follows. 

In the second case, $H^+_A(\eta) = R[V_\eta]$, $|V_\eta|=4$, $H^+_A(\lambda)=R[W_\lambda]$, $|W_\lambda|=2$ and
$H^+_A(\alpha)=R\cdot g_\alpha$, $|g_\alpha|=0$, where
$\eta\in \MC{C}^{f.red}$, $\lambda\in \MC{C}^{red}$ and $\alpha\in \MC{C}^{irr}$. Due to the indexing mentioned above the only possibly nontrivial module
$E^1_{s,t}$ with $s+t$ odd are $E^1_{4s+1,0}$. This module is freely generated by $g_\alpha$ for $\alpha\in \MC{C}^{irr}$ satisfying $i(\alpha)\equiv 4s+1 \Mod{8}$. Therefore, a nontrivial differential must start in $E^r_{4s+1,0}$ for
some $r$ and $s$. As $E^r_{s,t}=0$ for each $t>0$ when $s$ is odd, we
conclude that the only possibly nontrivial differentials are
$d^{2r+1}\colon E^{2r+1} _{4s+1,0} \ra E^{2r+1}_{4s-2r,2r}$.
\end{proof}   

For the calculation of $I^-(Y_\Gamma)$ one is therefore left with solving
an extension problem. In this case the extension problem is slightly simpler. We only state the conclusion. 
Note that in the statement below we are using the grading function
$j\colon \MC{C}\ra \Z/8$ for $\overline{Y}_\Gamma$, regarded as taking
values in $\{0,4\}$.  

\begin{theorem} \label{Imin-StandardOrientation}
The negative equivariant instanton Floer homology $I^-(Y_\Gamma)$
associated with the trivial $\om{SU}(2)$-bundle over $Y_\Gamma$ is
given by
\[ \left[ \left( \bigoplus_{\alpha\in \MC{C}^{irr}} R\cdot h_\alpha'[j(\alpha)] \right) \oplus \left( \bigoplus_{\lambda\in \MC{C}^{red}} R[Z_\lambda][j(\lambda)] \right) \oplus \left( \bigoplus_{\eta\in \MC{C}^{f.red}}R[U_\eta][j(\eta)] \right) \right]^{\oplus,8} , \]
where $|h_\alpha'|=0$, $|Z_\lambda|=-2$ and $|U_\eta|=-4$. The $R[U]$-module
structure is determined by $U\cdot U_\eta^p = U_\eta^{p+1}$, $U\cdot Z_\lambda^p = Z_\lambda^{p+2}$ for $p\geq 0$ and
\[ U\cdot h_\alpha' = \left( \sum_{\beta\in \MC{C}^{irr}} n_{\alpha\beta}h_\beta' \right) +\left( \sum_{\lambda\in \MC{C}^{red}} n_{\alpha\lambda}Z^0_\lambda
\right) + \left( \sum_{\eta\in \MC{C}^{f.red}} n_{\alpha\eta} V^0_\eta \right).
\]
\end{theorem}

To calculate $I^+(Y_\Gamma)$ one may proceed as in Lemma \ref{Description-Differentials} to determine the differentials in the spectral sequence
explicitly and then resolve the resulting extension problem. However,
one can avoid this rather technical endeavor by using the following result.
If $M$ is a DG $R[U]$-module we write $M^\vee=\om{Hom}(M,R)$ for
the dual complex. Our convention is that
$(\dd_{M^{\vee}}f)(x) = (-1)^{|f|+1}f(\dd_M x)$ and $(U\cdot f)(x) = f(Ux)$
for $f\in M^\vee =\om{Hom}(M_n,R)$. The content of the following lemma,
which applies more generally, is stated in the proof of \cite[Theorem~7.10]{Miller19}.

\begin{lemma} There is an isomorphism of DG $R[U]$-modules
\[ DCI^+(Y_\Gamma) \cong DCI^-(\overline{Y}_\Gamma)^\vee  .\]
\end{lemma}
\begin{proof} 
Let $M=DCI(\overline{Y}_\Gamma)$ and recall that $M^\vee\cong DCI(Y_\Gamma)$.
Using the explicit models for $DCI^\pm$ given in Theorem 
\ref{Explicit-DCI-Models} we find
\begin{align*}
DCI^-(\overline{Y}_\Gamma)^\vee_n = \om{Hom}\left(\bigoplus_{p\geq 0} M_{-n+4p},R \right) \cong & \prod_{p\geq 0} \om{Hom}(M_{-n+4p},R)   \\
= & \prod_{p\geq 0} M^\vee_{n-4p}\cong DCI^+(Y_\Gamma)_n  .
\end{align*} 
It is a straightforward calculation to check that these isomorphisms piece together to give an isomorphism of DG $R[U]$-modules taking into account that the $\Lambda_R[u]$-module structure on $M^\vee$ is given by 
$(f\cdot u)(x) = (-1)^{|u||x|}f(xu) = (-1)^{|x|}f(xu)$.
\end{proof} 

\begin{theorem} There is an isomorphisms of $R[U]$-modules
\[ I^+(Y_\Gamma) \cong I^-(\overline{Y}_\Gamma)^\vee  \]
for each finite subgroup $\Gamma\subset \om{SU}(2)$. 
\end{theorem}
\begin{proof} Since $R$ is a principal ideal domain and the complex
$DCI^-(\overline{Y}_\Gamma)$ is degreewise free over $R$ (use the
explicit model and the fact that $DCI(\overline{Y}_\Gamma)$ is degreewise free), there is a universal coefficent theorem. In view of the
above lemma, this takes the form of a short exact sequence
\[ \begin{tikzcd}  \om{Ext}^1_R(I_{n-1}^-(\overline{Y}_\Gamma),R)
\arrow{r} & I^+_{-n}(Y_\Gamma) \arrow{r} & \om{Hom}(I^-(\overline{Y}_\Gamma)_n,R) \end{tikzcd} \] 
for each $n\in \Z$. By Theorem \ref{IMinus-1} and Theorem \ref{IMinus-2}
we know that $I^-(\overline{Y}_\Gamma)$ is degreewise free over $R$.
Consequently, the first term in the above sequence vanishes and we
are lift with an isomorphism 
$I^+(Y_\Gamma)\cong (I^-(\overline{Y})_\Gamma)^\vee$ of $R[U]$-modules
as required.   
\end{proof}  

Note that the calculations of Theorem \ref{Imin-StandardOrientation}
and Theorem \ref{IPluss-1} are compatible with the duality
$I^+(\overline{Y}_\Gamma)\cong (I^-(Y_\Gamma))^\vee$ as well.

\appendix

\section{Character Tables, Dynkin Graphs and Flat Connections}
This appendix contains various useful facts concerning the binary polyhedral
groups. Of particular importance are their complex representation theory,
the corresponding McKay graphs and complete lists of the $1$-dimensional
quaternionic representations.  

\subsection{Cyclic Groups}
The complex representation theory of cyclic groups are of course very well
known, so we only include a brief summary for completeness. Let $C_l = \langle g:g^l=1\rangle$ be the cyclic group of order $l$. The irreducible
complex representations are $\rho_k\colon C_l\ra U(1)$ for $0\leq k<l$, where
$\rho_k(g) = e^{\frac{2\pi i k}{l}}$. If we interpret the indices modulo
$l$ we have $\rho_k\otimes \rho_{k'}\cong \rho_{k+k'}$ and
$\rho_k^*\cong \rho_{-k}$.    
We embed $C_l\subset \om{SU}(2)$ through the representation $\rho_1\oplus \rho_{-1}$. 
 
Suppose that $l$ is odd. Then all of the $\rho_k$, $1\leq k<l$, are of complex
type, so the $1$-dimensional quaternionic representations are given by
$\theta = 2\rho_0$ and $\lambda_k = \rho_k\oplus \rho_{-k}$ for $1\leq k <l/2$. 
The McKay graph $\overline{\Delta}_{C_l}\cong \wt{A}_{l-1}$ takes the form

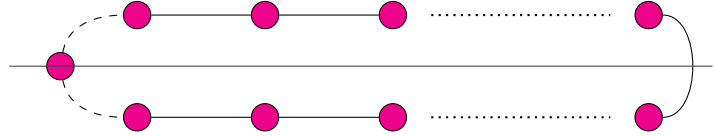
\begin{center}
\begin{tikzpicture}[scale=1.7]
\node (1) at (-0.6,0) [circle, draw, fill=magenta] {} ;
\node (2a) at (0,0.4) [circle, draw, fill = magenta] {};
\node (2b) at (0,-0.4) [circle, draw, fill=magenta] {};
\node (3a) at (1,0.4) [circle, draw, fill = magenta] {};
\node (3b) at (1,-0.4) [circle, draw, fill=magenta] {};
\node (4a) at (2,0.4) [circle, draw, fill = magenta] {};
\node (4b) at (2,-0.4) [circle, draw, fill=magenta] {};
\node (5a) at (4,0.4) [circle, draw, fill = magenta] {};
\node (5b) at (4,-0.4) [circle, draw, fill=magenta] {};

\path
(1) edge[dashed,out=80,in =180] (2a)
(1) edge[dashed, out = 280, in =180] (2b)
(2a) edge (3a)
(2b) edge (3b)
(3a) edge (4a)
(3b) edge (4b)
(5a) edge[out = 0, in = 0] (5b);
\draw[dotted, thick] (2.3,0.4)--(3.7,0.4);
\draw[dotted, thick] (2.3,-0.4)--(3.7,-0.4);
\draw[darkgray] (-1,0)--(4.5,0);
\end{tikzpicture}
\end{center}
where the line indicates the involution symmetry. The quotient graph $\ovl{\Delta}_{C_l}/(\iota)$ is given by 

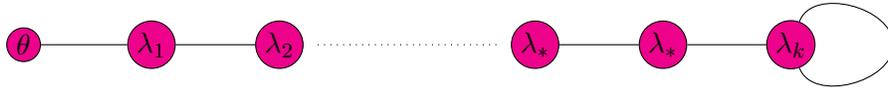
\begin{center}
\begin{tikzpicture}[inner sep=0.43mm, scale=1.7]
\node (1) at (0,0) [circle,draw,fill=magenta] {$\theta$} ;
\node (2) at (1,0) [circle,draw,fill=magenta] {$\lambda_1$} ;
\node (3) at (2,0) [circle,draw,fill=magenta] {$\lambda_2$} ;
\node (4) at (4,0) [circle,draw,fill=magenta] {$\lambda_*$} ;
\node (5) at (5,0) [circle,draw,fill=magenta] {$\lambda_*$} ;
\node (6) at (6,0) [circle,draw,fill=magenta] {$\lambda_k$} ;

\path
(1) edge (2)
(2) edge (3)
(4) edge (5)
(5) edge (6);
\draw[dotted] (2.3,0)--(3.7,0);
\draw (6) to [out=290, in = 270] (6.8,0) to [out=90,in = 70] (6);
\end{tikzpicture}
\end{center}  

Suppose that $l$ is even. Then $\rho_{l/2}$ is of real type, so the flat
connections are given by $\theta = 2\rho_0$, $\eta = 2\rho_{k/2}$ and
$\lambda_k = \rho_k\oplus \rho_{-k}$ for $1\leq k<l/2$. The McKay graph
takes the form

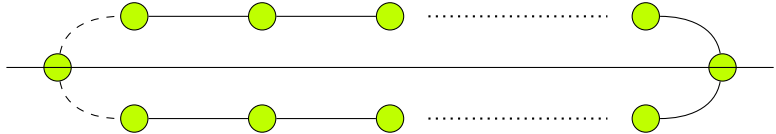
\begin{center}
\begin{tikzpicture}[scale=1.7]
\node (1) at (-0.6,0) [circle, draw, fill=lime] {} ;
\node (2a) at (0,0.4) [circle, draw, fill = lime] {};
\node (2b) at (0,-0.4) [circle, draw, fill=lime] {};
\node (3a) at (1,0.4) [circle, draw, fill = lime] {};
\node (3b) at (1,-0.4) [circle, draw, fill=lime] {};
\node (4a) at (2,0.4) [circle, draw, fill = lime] {};
\node (4b) at (2,-0.4) [circle, draw, fill=lime] {};
\node (5a) at (4,0.4) [circle, draw, fill = lime] {};
\node (5b) at (4,-0.4) [circle, draw, fill=lime] {};
\node (6) at (4.6,0) [circle, draw, fill=lime] {};

\path
(1) edge[dashed,out=80,in =180] (2a)
(1) edge[dashed, out = 280, in =180] (2b)
(2a) edge (3a)
(2b) edge (3b)
(3a) edge (4a)
(3b) edge (4b)
(5a) edge[out = 0, in = 100] (6)
(5b) edge[out = 0, in = -100] (6);
\draw[dotted, thick] (2.3,0.4)--(3.7,0.4);
\draw[dotted, thick] (2.3,-0.4)--(3.7,-0.4);
\draw (-1,0)--(5,0);
\end{tikzpicture}
\end{center}

with quotient graph

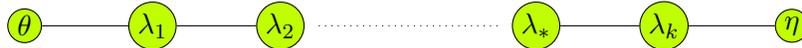
\begin{center}
\begin{tikzpicture}[inner sep=0.43mm, scale=1.7]
\node (1) at (0,0) [circle,draw,fill=lime] {$\theta$} ;
\node (2) at (1,0) [circle,draw,fill=lime] {$\lambda_1$} ;
\node (3) at (2,0) [circle,draw,fill=lime] {$\lambda_2$} ;
\node (4) at (4,0) [circle,draw,fill=lime] {$\lambda_*$} ;
\node (5) at (5,0) [circle,draw,fill=lime] {$\lambda_k$} ;
\node (6) at (6,0) [circle,draw,fill=lime] {$\eta$} ;

\path
(1) edge (2)
(2) edge (3)
(4) edge (5)
(5) edge (6);
\draw[dotted] (2.3,0)--(3.7,0);
\end{tikzpicture}
\end{center}

\subsection{Binary Dihedral Groups}
A presentation for the binary dihedral group $D_n^*$, $n\geq 2$, is given by
\[ D_n^* = \langle a,x \mid a^n=x^2, axa=x\rangle  .\]
Note that $a^{2n}=a^nx^2 = xa^{-n}x = xx^{-2}x=1$ follows from the two
relations. Each element may be expressed uniquely
in the form $a^kx^{\eps}$ for $0\leq k<2n$ and $0\leq \eps\leq 1$. The group may be realized within $\om{Sp}(1)\subset \HH$ by taking $a = e^{\pi i/n}$ and $x= j$. In this realization $a^n=x^2=-1$ so this element is central.

\begin{lemma} We have $[D_n^*,D_n^*]=\cyc{a^2}$ so that
\[ D_n^*/[D_n^*,D_n^*]\cong 
\left\{\begin{array}{cc}\Z/4 & \mbox{ if } n\equiv 1\Mod{2}\\
        \Z/2\times \Z/2 & \mbox{ if } n\equiv 0\Mod{2}
        \end{array} \right.   \]
\end{lemma}
\begin{proof} Using the fact that $xa^k=a^{-k}x$ one sees that
every commutator is of the form $a^k$. Moreover, $axa^{-1}x^{-1}=a^2$ from which it follows that $\cyc{a^2}\subset [D_n^*,D_n^*]$. As $\cyc{a^2}$ is normal with quotient of order $4$ and thus abelian, we must have equality.
If $n$ is odd $x$ has order $4$ modulo
$\cyc{a^2}$, so the quotient is cyclic If $n$ is even every element has order $2$ mod $\cyc{a^2}$, so the quotient is isomorphic to 
$\Z/2\times \Z/2$. \end{proof}

\begin{lemma} The conjugacy classes of $D_n^*$ are given by
$\{1\}, \{-1\}$, $\{a^j,a^{-j}\}$ for $1\leq j<n$,
\[ \{x,a^2x,\cdots,xa^{2n-2}x\} \;\;  \mbox{ and }  \;\; 
 \{ax,a^3x,\cdots,a^{2n-1}x\} .\]
 \end{lemma}
\begin{proof} The elements $1,-1$ are central so $\{1\}$, $\{-1\}$ are
conjugacy classes. From the calculation $(a^sx)a^j(a^sx)^{-1}=a^{-j}$, valid for any $s$ and $j$, it follows that $\{a^j,a^{-j}\}$, for $1\leq j<n$, are conjugacy classes.
Next, $a^jxa^{-j}=a^{2j}x$ from which we deduce that
that $\{x,a^2x,\cdots,a^{2n-2}x\}$ is a conjugacy class. Finally,
from the relation $a^j(ax)a^{-j}=a^{2j+1}x$ we deduce that $\{ax,a^3x,\cdots,a^{2n-1}x\}$
is at least a portion of a conjugacy class. However, since we have already
exhausted every element of the group, we conclude that this must in fact
be the whole conjugacy class and that we have found all the conjugacy classes. 
\end{proof}  

\begin{proposition} The irreducible representations of $D_n^*$
consist of the four $1$-dimensional representations associated with
the four distinct homomorphisms $D_n^*/[D_n^*,D_n^*]\ra \{\pm1,\pm i\}$
and the $2$-dimensional representations $\tau_k\colon D_n^*\ra \om{GL}_2(\C)$, for $1\leq k<n$, given by
\[ \tau_k(a)=\left( \begin{array}{cc} \xi^k & 0 \\
                        0 & \xi^{-k} \end{array} \right) \;\;
\mbox{ and } \;\; \tau_k(x)=\left( \begin{array}{cc} 0 & (-1)^k \\
                         1 & 0 \end{array} \right) \]
where $\xi = e^{\pi i /n}$.
\end{proposition}
\begin{proof} The statement about the $1$-dimensional representations is clear. Let $\chi_k$ denote the character
of $\tau_k$. Then $\chi_k(a^j)=2\cos(\pi jk/n)$,
$\chi_k(1)=2$, $\chi_k(-1)=2(-1)^k$ and $\chi_k$ vanishes
on the remaining conjugacy classes. Hence, 
\begin{align*}
||\chi_k||^2 = \frac1{4n}(8+2\sum_{j=1}^{n-1}4\cos^2(\pi kj/n) )) &= \frac1{4n}(8+4\sum_{j=1}^{n-1}(\cos(2\pi jk/n)+1)) \\
& =\frac1{4n}(8+4(n-2))=1,
\end{align*} 
which shows that each $\tau_k$ is irreducible. The representations 
$\tau_k$ are inequivalent for $1\leq k<n$ because the values 
$\chi_k(a)=2\cos(\pi k/n)$ are distinct for $1\leq k<n$. The sum of the 
squares of the dimensions of the irreducible representations found is $4+(n-1)4=4n=|D_n^*|$, so conclude that we have found all of them.
\end{proof}

The representation corresponding to the inclusion $D_n^*\subset \om{Sp}(1)=\om{SU}(2)$ is $\tau_1$. We denote the four representations
corresponding to the homomorphisms $(D_n^*)^{ab}\ra \{\pm1,\pm i\}$ by
$\rho_j$ for $0\leq j\leq 3$, where we take $\rho_0$ to be the trivial representation
and fix $\rho_1$ by the requirement $\tau_1\otimes \tau_1 = \tau_2\oplus \rho_1\oplus \rho_0$. Then $\rho_0,\rho_1$ are of real type, while if $n$ is odd $\rho_2\cong \rho_3^*,\rho_3$ are of
complex type and if $n$ is even $\rho_2$ and $\rho_3$ are of real type. 

The McKay graph $\overline{\Delta}_{D_n^*}\cong \wt{D}_{n+2}$ takes the form

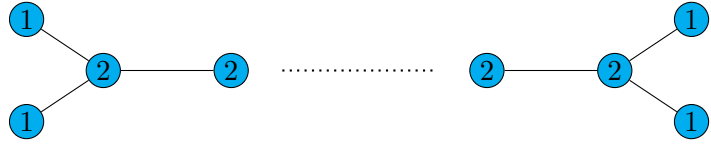
\begin{center} \label{Dn-graph}
\begin{tikzpicture}[inner sep=0.5mm, scale=1.7]
\node (1a) at (-0.6,0.4) [circle, draw, fill=cyan] {$1$};
\node (1b) at (-0.6,-0.4) [circle, draw, fill=cyan] {$1$};
\node (2) at (0,0) [circle, draw, fill=cyan] {$2$};
\node (3) at (1,0) [circle, draw, fill=cyan] {$2$};
\node (4) at (3,0) [circle, draw, fill=cyan] {$2$};
\node (5) at (4,0) [circle, draw, fill=cyan] {$2$};
\node (6a) at (4.6,0.4) [circle, draw, fill=cyan] {$1$};
\node (6b) at (4.6,-0.4) [circle, draw, fill=cyan] {$1$};

\path
(1a) edge (2)
(1b) edge (2)
(2) edge (3)
(4) edge (5)
(5) edge (6a)
(5) edge (6b);
\draw[dotted, thick] (1.4,0)--(2.6,0);
\end{tikzpicture}
\end{center}
where we have included the dimensions of the representations in the graph.
For $D_n^*$ there are $n-1$ vertices in the middle corresponding to
$\tau_1,\cdots,\tau_{n-1}$ from left to right. On the left hand side 
we have $\rho_0,\rho_1$ and on the right hand side we have $\rho_2,\rho_3$.  

\begin{lemma} The representation $\tau_k$ is quaternionic
for $k$ odd and real for $k$ even. \end{lemma}
\begin{proof}
We use the criterion of Lemma \ref{Type-Criterion}.
Since $\sum_{j=0}^{2n-1}\chi_k(a^{2j})=0$
and $(xa^j)^2=x^2=-1$ for each $j$, we obtain
$\sum_{g\in D_n^*} \chi_k(g^2) =4n(-1)^k$ and the result follows. \end{proof}

Let $\theta = 2\rho_0$ and $\alpha_k = \tau_{2k-1}$ for $1\leq k\leq n/2$.
For $n$ even the $1$-dimensional quaternionic representations are given by
\[ \theta, \eta_1 = 2\rho_1, \eta_2 = 2\rho_2,\eta_3=2\rho_3, \alpha_1,\cdots,\alpha_{n/2}  \]
and the quotient graph $\ovl{\Delta}_{D^*_n}/(\iota)$ is given by
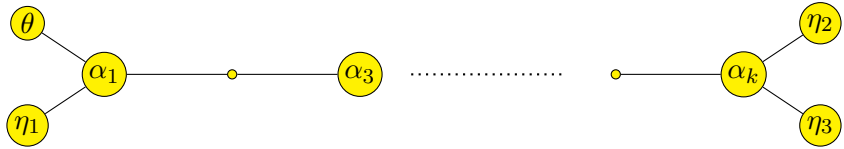
\begin{center}
\begin{tikzpicture}[inner sep=0.43mm, scale=1.7]
\node (1a) at (-0.6,0.4) [circle, draw, fill=yellow] {$\theta$};
\node (1b) at (-0.6,-0.4) [circle, draw, fill=yellow] {$\eta_1$};
\node (2) at (0,0) [circle, draw, fill=yellow] {$\alpha_1$};
\node (3) at (1,0) [circle, draw, fill=yellow] {};
\node (4) at (2,0) [circle,draw, fill=yellow] {$\alpha_3$};
\node (5) at (4,0) [circle, draw, fill=yellow] {};
\node (6) at (5,0) [circle, draw, fill=yellow] {$\alpha_k$};
\node (7a) at (5.6,0.4) [circle, draw, fill=yellow] {$\eta_2$};
\node (7b) at (5.6,-0.4) [circle, draw, fill=yellow] {$\eta_3$};

\path
(1a) edge (2)
(1b) edge (2)
(2) edge (3)
(3) edge (4)
(5) edge (6)
(6) edge (7a)
(6) edge (7b);
\draw[dotted, thick] (2.4,0)--(3.6,0);
\end{tikzpicture}
\end{center}
For $n$ odd the $1$-dimensional quaternionic representations are given by
\[ \theta, \eta = 2\rho_1, \alpha_1,\cdots,\alpha_{(n-1)/2},\lambda = \rho_2\oplus \rho_3  \]
and the corresponding quotient graph is given by

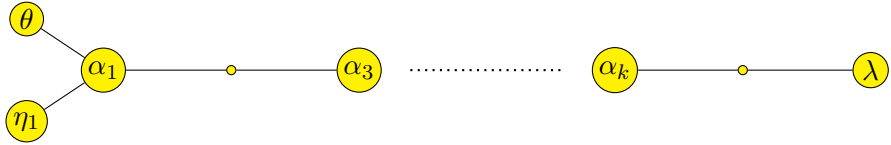
\begin{center}
\begin{tikzpicture}[inner sep=0.43mm, scale=1.7]
\node (1a) at (-0.6,0.4) [circle, draw, fill=yellow] {$\theta$};
\node (1b) at (-0.6,-0.4) [circle, draw, fill=yellow] {$\eta_1$};
\node (2) at (0,0) [circle, draw, fill=yellow] {$\alpha_1$};
\node (3) at (1,0) [circle, draw, fill=yellow] {};
\node (4) at (2,0) [circle,draw, fill=yellow] {$\alpha_3$};
\node (5) at (4,0) [circle, draw, fill=yellow] {$\alpha_k$};
\node (6) at (5,0) [circle, draw, fill=yellow] {};
\node (7) at (6,0) [circle, draw, fill=yellow] {$\lambda$};

\path
(1a) edge (2)
(1b) edge (2)
(2) edge (3)
(3) edge (4)
(5) edge (6)
(6) edge (7);
\draw[dotted, thick] (2.4,0)--(3.6,0);
\end{tikzpicture}
\end{center}

\subsection{The Binary Tetrahedral Group}
Let $Q=\{\pm1,\pm i,\pm j, \pm k\}\subset \om{Sp}(1)$ be the quaternion
group (note that $Q\cong D_2^*$).   
Then $T^*$ may be realized in $\om{Sp}(1)\subset \HH$ as
\[ T^* =Q \cup \{\frac12 (\eps_0 1+\eps_1 i+\eps_2 j+\eps_3 k) : \eps_r=\pm 1,
 \; 0\leq r\leq 3 \}  .\]
We have $[T^*,T^*]=Q$ so as $|T^*|=24$, it follows 
that $(T^*)^{ab}=T^*/[T^*,T^*]\cong \Z/3$. 

The character table of $T^*$ is given below. Here $\xi=e^{2\pi i/3}$
is a primite third root of unity. 

\[ \begin{array}{|c|c|c|c|c|c|c|c||c|}
\hline
  {}    &  1  &  2  &  3a &  3b &  4  &  6a &  6b & \mbox{type}  \\ \hline
 \rho_1 &  1  &  1  &  1  &  1  &  1  &  1  &  1 & \R   \\ \hline
 \rho_2 & 1 & 1 & \xi & \xi^2 & 1 & \xi & \xi^2  & \C   \\ \hline
 \rho_2^* & 1 & 1 & \xi^2 & \xi & 1 & \xi^2 & \xi & \C \\   \hline
 \rho_3 & 2 & -2 & -\xi & -\xi^2 & 0 & \xi & \xi^2 & \C  \\ \hline
 \rho_3^* & 2 & -2 & -\xi^2 & -\xi & 0 & \xi^2 & \xi & \C \\ \hline
 \rho_4 & 2 & -2 & -1 & -1 & 0 & 1 & 1   & \HH \\ \hline 
 \rho_5 & 3 & 3 & 0 & 0 & -1 & 0 & 0 & \R  \\ \hline
\end{array} \]
The type of the representation in the right hand column is calculated
using Lemma \ref{Type-Criterion} and the following table. The first row
gives a representative for each conjugacy class and the second row
gives the conjugacy class in which the square of this element belongs.
Here $x=(1+i+j+k)/2\in T^*$.

\[ \begin{array}{|c|c|c|c|c|c|c|c|}
\hline
{} & 1 & 2 & 3a & 3b & 4 & 6a & 6b  \\ \hline
\mbox{rep} & 1 & -1 & -x & -x^* & i & x & x^*  \\ \hline
\mbox{sm} & 1 & 1 & 3b & 3a & 2 & 3b & 3a  \\ \hline 
\end{array}  \]

The canonical representation given by $T^*\inj \om{SU}(2)$ corresponds
to the irreducible character $\rho_4$. From this one may calculate
the McKay graph $\overline{\Delta}_{T^*}\cong \wt{E}_6$ to be

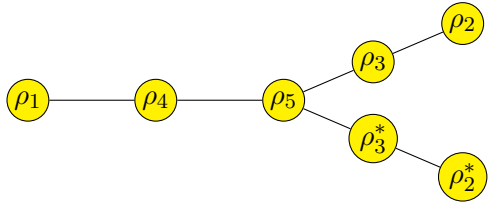
\begin{center}
\begin{tikzpicture}[inner sep=0.43mm, scale = 1.7]
\node (1) at (0,0) [circle,draw,fill=yellow] {$\rho_1$};
\node (2) at (1,0) [circle,draw,fill=yellow] {$\rho_4$};
\node (3) at (2,0) [circle,draw,fill=yellow] {$\rho_5$};
\node (4a) at (2.7,0.3) [circle,draw,fill=yellow] {$\rho_3$};
\node (4b) at (2.7,-0.3) [circle,draw,fill=yellow] {$\rho_3^*$};
\node (5a) at (3.4,0.6) [circle,draw,fill=yellow] {$\rho_2$};
\node (5b) at (3.4,-0.6) [circle,draw,fill=yellow] {$\rho_2^*$};
\path
(1) edge (2)
(2) edge (3)
(3) edge (4a)
(3) edge (4b)
(4a) edge (5a)
(4b) edge (5b);
\end{tikzpicture}
\end{center}
From the type calculation given in the character table and Proposition \ref{Irr-Decomp-H},
we see that the $1$-dimensional quaternionic representations are
$\theta = 2\rho_1$, $\alpha = \rho_4$ and $\lambda = \rho_2\oplus \rho_2^*$.
The quotient graph $\ovl{\Delta}_{T^*}/(\iota)$ then takes the form

\begin{center}
\begin{tikzpicture}[inner sep=0.8mm, scale = 1.7]
\node (1) at (0,0) [circle, draw,fill=yellow] {$\theta$};
\node (2) at (1,0) [circle,draw,fill=yellow] {$\alpha$};
\node (3) at (2,0) [circle,draw,fill=yellow] {};
\node (4) at (3,0) [circle,draw,fill=yellow] {};
\node (5) at (4,0) [circle,draw,fill=yellow] {$\lambda$};
\path
(1) edge (2)
(2) edge (3)
(3) edge (4)
(4) edge (5);
\end{tikzpicture}
\end{center}

\subsection{The Binary Octahedral Group}
Let $Q\subset T^*\subset \om{Sp}(1)$ be as in the previous section. 
Then $O^*$ may be realized as 
\[  O^* = T^*\cup \{ (u+v)/\sqrt2: u, v\in Q\;,\; u\neq \pm v  \}\subset \om{Sp}(1)    .\]
We have $[O^*,O^*]=T^*$, so as $|O^*|=48$ it follows that
$(O^*)^{ab}=O^*/[O^*,O^*]\cong \Z/2$.

The character table of $O^*$ is given below

\[ \begin{array}{|c|c|c|c|c|c|c|c|c||c|}
\hline
{} & 1 & 2 & 3 & 4a & 4b & 6 & 8a & 8b & \mbox{type} \\ \hline
\rho_1 & 1 & 1 & 1 & 1 & 1 & 1 & 1 & 1 & \R \\ \hline
\rho_2 & 1 & 1 & 1 & 1 & -1 & 1 & -1 & -1 & \R \\ \hline
\rho_3 & 2 & 2 & -1 & 2 & 0 & -1 & 0 & 0 & \R \\ \hline
\rho_4 & 2 & -2 & -1 & 0 & 0 & 1 & \sqrt2 & -\sqrt2 & \HH \\ \hline
\rho_5 & 2 & -2 & -1 & 0 & 0 & 1 & -\sqrt2 & \sqrt2 & \HH \\ \hline
\rho_6 & 3 & 3 & 0 & -1 & -1 & 0 & 1 & 1 & \R \\ \hline 
\rho_7 & 3 & 3 & 0 & -1 & 1 & 0 & -1 & -1 & \R \\ \hline
\rho_8 & 4 & -4 & 1 & 0 & 0 & -1 & 0 & 0  & \HH \\ \hline
\end{array} \]

Let $x=(1+i+j+k)/2$, as in the previous section, let $y=(1+i)/\sqrt2$ and let
$z = (i+j)/\sqrt2$. The following table contains the information needed
for the type calculation given in the above right hand column.  

\[ \begin{array}{|c|c|c|c|c|c|c|c|c|c|}
\hline 
{} & 1 & 2 & 3 & 4a & 4b & 6 & 8a & 8b  \\ \hline
\mbox{rep} & 1 & -1 & -x & i & z & x & y & -y  \\ \hline
\mbox{sm} & 1 & 1 & 3 & 2 & 2 & 3 & 6 & 6  \\ \hline
\end{array} \]

The canonical representation given by the inclusion
$O^*\inj \om{SU}(2)$ corresponds to $\rho_4$. From this one obtains
the McKay graph $\overline{\Delta}_{O^*}\cong \wt{E}_7$.

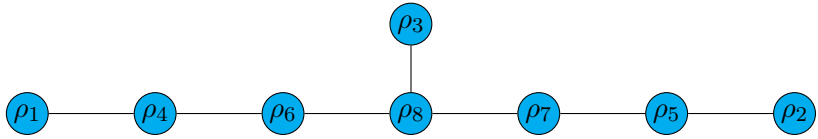
\begin{center}
\begin{tikzpicture}[inner sep = 0.43mm, scale = 1.7]
\node (1) at (0,0) [circle,draw,fill=cyan] {$\rho_1$};
\node (2) at (1,0) [circle,draw,fill=cyan] {$\rho_4$};
\node (3) at (2,0) [circle,draw,fill=cyan] {$\rho_6$};
\node (4) at (3,0) [circle,draw,fill=cyan] {$\rho_8$};
\node (5) at (4,0) [circle,draw,fill=cyan] {$\rho_7$};
\node (6) at (5,0) [circle,draw,fill=cyan] {$\rho_5$};
\node (7) at (6,0) [circle,draw,fill=cyan] {$\rho_2$};
\node (8) at (3,0.7) [circle,draw,fill=cyan] {$\rho_3$};
\path
(1) edge (2)
(2) edge (3)
(3) edge (4)
(4) edge (5)
(5) edge (6)
(6) edge (7)
(4) edge (8);
\end{tikzpicture}
\end{center}

Using the by now standard method one finds the $1$-dimensional quaternionic
representation to be $\theta = 2\rho_1$, $\alpha=\rho_4$, $\beta = \rho_5$
and $\eta = 2\rho_2$. In this case there are no irreducible representations
of complex type, so the quotient graph is simply

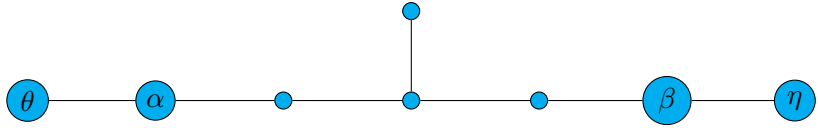
\begin{center}
\begin{tikzpicture}[inner sep = 0.8mm, scale = 1.7]
\node (1) at (0,0) [circle,draw,fill=cyan] {$\theta$};
\node (2) at (1,0) [circle,draw,fill=cyan] {$\alpha$};
\node (3) at (2,0) [circle,draw,fill=cyan] {};
\node (4) at (3,0) [circle,draw,fill=cyan] {};
\node (5) at (4,0) [circle,draw,fill=cyan] {};
\node (6) at (5,0) [circle,draw,fill=cyan] {$\beta$};
\node (7) at (6,0) [circle,draw,fill=cyan] {$\eta$};
\node (8) at (3,0.7) [circle,draw,fill=cyan] {};
\path
(1) edge (2)
(2) edge (3)
(3) edge (4)
(4) edge (5)
(5) edge (6)
(6) edge (7)
(4) edge (8);
\end{tikzpicture}
\end{center}

\subsection{The Binary Icosahedral Group}
Let $\phi=(1+\sqrt5)/2=2\cos(\pi/5)$ be the golden ratio (thus $\phi^2=\phi+1$) and let $S$ be the set of quaternions
obtained by even coordinate permutations of
\[ (\eps_0 i+ \eps_1 \phi^{-1}j+\eps_2 \phi k)/2  \]
where $\eps_r\in \{\pm 1\}$ for $0\leq r\leq 2$. Note that there are
$12=|A_4|$ even permutations and $8=2^3$ sign combinations so
$|S|=12\cdot 8=96$. The binary icosahedral group may be realized in
$\om{Sp}(1)$ as $I^*=T^*\cup S$. It is well-known that $I^*$ is a perfect
group, that is, $[I^*,I^*]=I^*$ or equivalently $(I^*)^{ab}=0$.
The binary polyhedral space $Y_{I^*}=S^3/I^*$ is also called the
Poincar\'{e} sphere and is an integral homology sphere. 

The character table of $I^*$ is given below.

\[ \begin{array}{|c|c|c|c|c|c|c|c|c|c||c|}
\hline
{} & 1 & 2 & 3 & 4 & 5a & 5b & 6 & 10a & 10b & \mbox{type} \\ \hline
\rho_1 & 1 & 1 & 1 & 1 & 1 & 1 & 1 & 1 & 1 & \R \\ \hline
\rho_2 & 2 & -2 & -1 & 0 & \phi^{-1} & -\phi & 1 & \phi & -\phi^{-1} & \HH \\ \hline
\rho_3 & 2 & -2 & -1 & 0 & -\phi & \phi^{-1} & 1 & -\phi^{-1} & \phi & \HH \\ \hline
\rho_4 & 3 & 3 & 0 & -1 & -\phi^{-1} & \phi & 0 & \phi & -\phi^{-1} & \R \\ \hline
\rho_5 & 3 & 3 & 0 & -1 & \phi & -\phi^{-1} & 0 & -\phi^{-1} & \phi & \R \\ \hline
\rho_6 & 4 & 4 & 1 & 0 & -1 & -1 & 1 & -1 & -1 &  \R \\ \hline 
\rho_7 & 4 & -4 & 1 & 0 & -1 & -1 & -1 & 1 & 1 & \HH \\ \hline
\rho_8 & 5 & 5 & -1 & 1 & 0 & 0 & -1 & 0 & 0  & \R \\ \hline
\rho_9 & 6 & -6 & 0 & 0 & 1 & 1 & 0 & -1 & -1 & \HH \\ \hline
\end{array} \]

Put
\[ u = (\phi\cdot 1+\phi^{-1}\cdot i+j)/2 \in S \]
and let $x\in T^*$ be defined as earlier. We then have the following
table for the type calculation in the above right hand column. 

\[ \begin{array}{|c|c|c|c|c|c|c|c|c|c|c|}
\hline 
{} & 1 & 2 & 3 & 4 & 5a & 5b & 6 & 10a & 10b  \\ \hline
\mbox{rep} & 1 & -1 & -x & i & u^2 & -u & x & u & -u^2  \\ \hline
\mbox{sm} & 1 & 1 & 3 & 2 & 5b & 5a & 3 & 5a & 5b  \\ \hline
\end{array} \]

The character corresponding to the canonical representation 
$I^*\inj \om{Sp}(1)=\om{SU}(2)$ is $\rho_2$. From this one may show
that the McKay graph $\overline{\Delta}_{I^*}\cong \wt{E}_8$ is given by

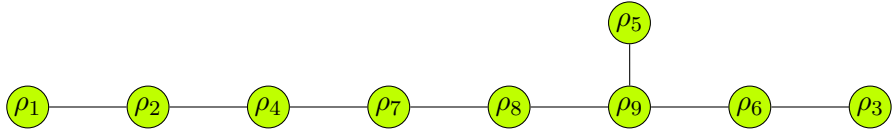
\begin{center}
\begin{tikzpicture}[inner sep=0.43mm, scale = 1.6]
\node (1) at (0,0) [circle,draw,fill=lime] {$\rho_1$};
\node (2) at (1,0) [circle,draw,fill=lime] {$\rho_2$};
\node (3) at (2,0) [circle,draw,fill=lime] {$\rho_4$};
\node (4) at (3,0) [circle,draw,fill=lime] {$\rho_7$};
\node (5) at (4,0) [circle,draw,fill=lime] {$\rho_8$};
\node (6) at (5,0) [circle,draw,fill=lime] {$\rho_9$};
\node (7) at (6,0) [circle,draw,fill=lime] {$\rho_6$};
\node (8) at (7,0) [circle,draw,fill=lime] {$\rho_3$};
\node (9) at (5,0.7) [circle,draw,fill=lime] {$\rho_5$};
\path
(1) edge (2)
(2) edge (3)
(3) edge (4)
(4) edge (5)
(5) edge (6)
(6) edge (7)
(7) edge (8)
(6) edge (9);
\end{tikzpicture}
\end{center} 

The $1$-dimensional quaternionic
representations are $\theta = 2\rho_1$, $\alpha = \rho_2$ and $\beta = \rho_3$.
Once again there are no irreducible representations of complex type, so
the quotient graph is simply

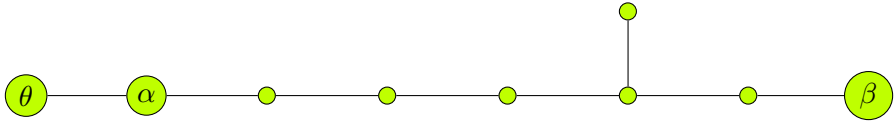
\begin{center}
\begin{tikzpicture}[inner sep=0.8mm, scale = 1.6]
\node (1) at (0,0) [circle,draw,fill=lime] {$\theta$};
\node (2) at (1,0) [circle,draw,fill=lime] {$\alpha$};
\node (3) at (2,0) [circle,draw,fill=lime] {};
\node (4) at (3,0) [circle,draw,fill=lime] {};
\node (5) at (4,0) [circle,draw,fill=lime] {};
\node (6) at (5,0) [circle,draw,fill=lime] {};
\node (7) at (6,0) [circle,draw,fill=lime] {};
\node (8) at (7,0) [circle,draw,fill=lime] {$\beta$};
\node (9) at (5,0.7) [circle,draw,fill=lime] {};
\path
(1) edge (2)
(2) edge (3)
(3) edge (4)
(4) edge (5)
(5) edge (6)
(6) edge (7)
(7) edge (8)
(6) edge (9);
\end{tikzpicture}
\end{center}

\section{Equivariant \texorpdfstring{$\om{SU}(2)$}{SU2}-Bundles over \texorpdfstring{$S^4$}{S4} and the
Chern-Simons Invariant}
In this appendix we give a proof of the classification of
$\Gamma$-equivariant $\om{SU}(2)$-bundles over $S^4$ and
apply the equivariant index theorem to the twisted Dirac
operators $D_A\colon \Gamma(S^+\otimes E)\ra \Gamma(S^-\otimes E)$
to obtain a proof of the equations in Lemma \ref{Austin-Equation}. 
As a byproduct of this we obtain a simple way to calculate the Chern-Simons
invariants of the flat connections in the trivial
$\om{SU}(2)$-bundle over the binary polyhedral spaces.
In the final part we also show that this invariant
can in a natural way be related to the algebraic Chern class of the holonomy representation associated with the flat connection.

\subsection{Conventions on Quaternions}
It will be convenient to employ quaternions in the the upcoming
theory, so in this section we fix our conventions. We will regard
$\HH$ as the $4$-dimensional real algebra with standard basis
$(1,i,j,k)$ and multiplication rules
\[ i^2=j^2=k^2=-1 \mbox{ and } ij=-ji = k  .\]
The involution determined by $1,i,j,k\mapsto 1,-i,-j,-k$ is written as
$q\mapsto q^*$. 
Our convention is that a quaternionic vector space is a left
$\HH$-module and that the standard quaternionic structure on $\HH^n$
is given by left multiplication. We write $M_n(F)$ for the algebra
of $n\times n$ matrices with entries in $F=\R,\C,\HH$. The standard identification
\[ M_n(\HH)\cong \om{End}_\HH(\HH^n)  \]
is then given by $A\mapsto (x\mapsto xA^*)$, that is, we multiply the
conjugate transpose of $A$ with $x$, regarded as a row vector, from the right.
To avoid confusion we will always write $q\cdot x\coloneqq xq^*$ for the left
action of $q\in \om{Sp}(1)$ on $x\in \HH$ and always omit the dot
when we mean actual multiplication. 

We define the standard inclusion by $\C=\om{Span}_\R(1,i)\subset \HH$
and the standard identification $\C^2\cong \HH$ by $(z,w)\mapsto z+wj$.
Note that the complex orientation on $\C^2$ corresponds to the
standard orientation on $\HH\cong \R^4$ in which $(1,i,j,k)$ is a positive
basis.

With these conventions in place we fix the isomorphism
$\om{Sp}(1)\cong \om{SU}(2)$ by requiring the following diagram to
commute
\[ \begin{tikzcd} \om{Sp}(1) \arrow{d}{\cong}  \arrow{r}{\subset} & \HH \arrow{r}{\cong} \arrow{d}
& \om{End}_\HH(\HH) \arrow{d}{\subset}  \\
\om{SU}(2) \arrow{r}{\subset} & M_2(\C) \arrow{r}{\cong} & \om{End}_\C(\C^2).
\end{tikzcd} \] 
Explicitly,
\[ q = z+wj\in \om{Sp}(1) \mapsto \left( \begin{array}{cc} z^* & w^* \\
 -w & z  \end{array} \right) .\]
These conventions ensure that the standard left action
$\om{SU}(2)\times \C^2\ra \C^2$ corresponds to the above defined
left action $\om{Sp}(1)\times \HH\ra \HH$. We define the orientation
of $\om{Sp}(1)=\om{SU}(2)$ by requiring that $(i,j,k)\in \om{Im}\HH=\mathfrak{sp}(1)$ is a positive basis. The orbit map
$\om{Sp}(1)\ra S^3\subset \HH$ given by $q\mapsto q\cdot 1=1q^*=q^*$
is orientation reversing. Since our work in this appendix is very
sensitive to orientations, we will for this reason distinguish between
the group $\om{Sp}(1)=\om{SU}(2)$ and the unit sphere in the
representation space $S^3\subset \HH=\C^2$. 

\subsection{Classification of \texorpdfstring{$\Gamma$}{Gamma}-equivariant 
\texorpdfstring{$\om{SU}(2)$}{SU2}-bundles over
\texorpdfstring{$S^4$}{S4}}
Let $\Gamma\subset \om{Sp}(1)=\om{SU}(2)$ be a finite subgroup.
Recall that $\Gamma$ acts on $S^3$ by restriction of the linear action
on $\HH=\C^2$. By regarding $S^4\subset \HH\oplus \R$ the suspended
action $\Gamma\times S^4\ra S^4$ takes the form $\gamma\cdot (x,t)=(\gamma\cdot x,1)$. The fixed points of this action are
$N \coloneqq (0,1)$ and $S \coloneqq (0,-1)$.

Let $\om{Rep}^1(\Gamma,\HH)$ denote the set of isomorphism classes
of $1$-dimensional quaternionic representations of $\Gamma$ and let
$\om{Vec}^1_\Gamma(S^4,\HH)$ denote the set of isomorphism classes
of $\Gamma$-equivariant quaternionic line bundles over $S^4$.
If $V$ is a $1$-dimensional quaternionic representation of $\Gamma$
we write $[V]$ for its isomorphism class in $\om{Rep}^1(\Gamma,\HH)$.
The same applies to $\Gamma$-equivariant bundles and $\om{Vec}^1_\Gamma(S^4,\HH)$. Define
\[ \phi\colon \om{Vec}^1_\Gamma(S^4,\HH)\ra \om{Rep}^1(\Gamma,\HH)\times \om{Rep}^1(\Gamma,\HH)\times \Z  \]
by $[E]\mapsto ([E_N],[E_S],c_2(E)[S^4])$. Our aim is to show that this map
is injective and identify the image. 
We will need the following two well-known results for our work.

\begin{lemma} \label{SP(1)-triv}
Let $Y$ be a closed $3$-manifold. Then every
$\om{Sp}(1)=\om{SU}(2)$-bundle over $Y$ is trivial.
\end{lemma}
\begin{proof} Every closed $3$-manifold admits the structure of a finite
$CW$ complex with cells of dimension $\leq 3$. We have $B\om{Sp}(1)=\HH P^\infty$ and this space has a $CW$ structure with a single cell
in each dimension $n = 4k$, $k\geq 0$. By cellular approximation
every map $f\colon Y\ra \HH P^\infty$ is homotopic to a map with image in the
$3$-skeleton $(\HH P^\infty)^3=*$. This implies that every $\om{Sp}(1)$-bundle must be trivial. \end{proof}

\begin{lemma} \label{Degree-Sum}
Let $Y$ be a closed oriented $3$-manifold and let
$f,g\colon Y\ra \om{Sp}(1)$ be a pair of continuous maps. Define
$h\colon Y\ra \om{Sp}(1)$ by $h(y)=f(y)g(y)$ for $y\in Y$. Then
\[ \om{deg}(h)=\om{deg}(f)+\om{deg}(g)   .\]
\end{lemma} 
\begin{proof} The map $h$ is given by the
following composition
\[ \begin{tikzcd} Y \arrow{r}{\Delta} & Y\times Y \arrow{r}{f\times g}
& \om{Sp}(1)\times \om{Sp}(1) \arrow{r}{\mu} & \om{Sp}(1), \end{tikzcd} \]
where $\Delta$ is the diagonal map and $\mu$ is the multiplication map.
Let $\pi_i\colon \om{Sp}(1)\times \om{Sp}(1)\ra \om{Sp}(1)$ for $i=1,2$
be the two projections and let $\iota_i\colon \om{Sp}(1)\ra \om{Sp}(1)\times \om{Sp}(1)$ for $i=1,2$ the two inclusions given by $\iota_1(a)=(a,1)$ and $\iota_2(a)=(1,a)$ for $a\in \om{Sp}(1)$. Then $\pi_i\circ \iota_i=\om{id}$ for $i=1,2$ and $\pi_i \circ \iota_j$
is constant for $i\neq j$. Clearly, $H_3(\om{Sp}(1)\times \om{Sp}(1))\cong \Z\oplus \Z$ and from this we deduce that
\[ q\coloneqq ((\pi_1)_*,(\pi_2)_*)\colon H_3(\om{Sp}(1)\times \om{Sp}(1))\ra H_3(\om{Sp}(1))\oplus H_3(\om{Sp}(1))  \]
is an isomorphism with inverse $(\iota_1)_*+(\iota_2)_*$. Furthermore,
\[ \mu_*\circ ((\iota_1)_*+(\iota_2)_*)=(\mu\circ \iota_1)_*+(\mu\circ \iota_2)_* =\om{id}_*+\om{id}_*. \]
These considerations imply that the following diagram commutes
\[ \begin{tikzcd} H_3(Y) \arrow{r}{(f,g)_*\circ \Delta_*} \arrow{rd}[swap]{(f_*,g_*)}
& H_3(\om{Sp}(1)\times \om{Sp}(1)) \arrow{d}{\cong}
\arrow{r}{\mu_*} & H_3(\om{Sp}(1)) \\
{} & H_3(\om{Sp}(1))\oplus H_3(\om{Sp}(1)) \arrow{ru}[swap]{+} & {}.
\end{tikzcd} \]
Since the upper composition is $h_*\colon H_3(Y)\ra H_3(\om{Sp}(1))$ by definition, it follows that 
\[ \om{deg}(h) [\om{Sp}(1)] =h_*([Y])=f_*([Y])+g_*([Y])
= (\om{deg}(f)+\om{deg}(g))[\om{Sp}(1)], \]
where $[Y]\in H_3(Y)$ and $[\om{Sp}(1)]\in H_3(\om{Sp}(1))$ are the fundamental
classes. Hence, $\om{deg}(h)=\om{deg}(f)+\om{deg}(g)$ as required.  
\end{proof}   

Let $\alpha\in \om{Rep}^1(\Gamma,\HH)$ and choose a homomorphism
$\rho_\alpha\colon \Gamma\ra \om{Sp}(1)$ representing $\alpha$. Write
$\om{Sp}(1)^\alpha$ for the group $\om{Sp}(1)$ equipped with the
$\Gamma$-action given by $\gamma\cdot q = \rho_\alpha(\gamma)q$. If $X$ and
$Y$ are $\Gamma$-manifolds we write $[X,Y]_\Gamma$ for the set
of $\Gamma$-equivariant homotopy classes of equivariant maps. 

\begin{lemma} \label{Classification-Lemma}
There exist $\Gamma$-equivariant maps
$S^3\ra \om{Sp}(1)^\alpha$. Moreover, for any pair $g,g'$ of such maps
it holds true that $\om{deg}g\equiv \om{deg}g' \Mod{|\Gamma |}$. \end{lemma}
\begin{proof} Since $\Gamma$ acts freely on $S^3$ the set of equivariant
maps $S^3\ra \om{Sp}(1)^\alpha$ is in natural bijection with the set of sections of the associated principal $\om{Sp}(1)$-bundle 
\[ S^3\times_\Gamma \om{Sp}(1)^\alpha\ra S^3/\Gamma. \]
According to Lemma \ref{SP(1)-triv} this bundle must be trivial, and hence admits global sections. This proves the first assertion. 

For the second assertion assume that $f,g\colon S^3\ra \om{Sp}(1)^\alpha$ is a pair of equivariant maps. We may then form the
map $h\colon S^3\ra \om{Sp}(1)$ given by $h(x)=f(x)^*g(x)$. This map
is equivariant when $\Gamma$ acts trivially on $\om{Sp}(1)$ and therefore
descends to a map $\bar{h}\colon S^3/\Gamma\ra \om{Sp}(1)$. Since the covering
map $S^3\ra S^3/\Gamma$ has degree $|\Gamma |$, we deduce that
$\om{deg}h = |\Gamma | \om{deg}(\bar{h}) \equiv 0 \Mod{|\Gamma |}$.
Finally, by Lemma \ref{Degree-Sum} and the fact that
$\om{deg}f^*=-\om{deg}f$ we obtain $\om{deg}h=\om{deg}(g)-\om{deg}f$ and the
proof is complete. \end{proof} 

Let $\alpha,\beta \in \om{Rep}^1(\Gamma,\HH)$ have representatives
$\rho_\alpha,\rho_\beta\colon \Gamma\ra \om{Sp}(1)$. Let $\HH(\alpha)$ and
$\HH(\beta)$ denote $\HH$ equipped with the $\Gamma$-action determined
by $\rho_\alpha$ and $\rho_\beta$, respectively. 
Furthermore, write $\om{Sp}(1)^{(\alpha;\beta)}$ for the group $\om{Sp}(1)$ equipped with the $\Gamma$-action $\gamma\cdot q = \rho_\alpha(\gamma)q\rho_\beta(\gamma)^*$.

\begin{proposition} \label{Clutching-Lemma}
The set of isomorphism classes of $\Gamma$-equivariant
$\om{Sp}(1)$-bundles $E\ra S^4$ with $[E_N]=\alpha$ and $[E_S]=\beta$
is in natural bijection with
\[ [S^3,\om{Sp}(1)^{(\alpha;\beta)}]_\Gamma  .\]
\end{proposition}
\begin{proof} This is well-known in the non-equivariant setting, so we
will only include the details necessary to adapt the usual proof
to the equivariant setting.
 
Let $E$ be a $\Gamma$-equivariant $\om{Sp}(1)$-bundle
over $S^4$ (regarded as a vector bundle) with $[E_N]=\alpha$ and $[E_S]=\beta$.
Let $U_N = S^4-\{S\}$ and $U_S=S^4-\{N\}$. Let $A$ be
a $\Gamma$-invariant $\om{Sp}(1)$-connection in $E$. Then, using parallel transport along radial geodesic from $S$ and $N$, one obtains equivariant trivializations
$U_S\times E_S\cong E|_{U_S}$ and $U_N\times E_N\cong E|_{U_N}$. By
making suitable choices of unit basis vectors in $E_S$ and $E_N$,
we may identify $E_S\cong \HH(\beta)$ and $E_N\cong \HH(\alpha)$. The transition function between these trivializations is an equivariant map $U_N\cap U_S\ra \om{Hom}_\HH(E_S,E_N)$ taking values in the subset of isometries. This subset may be identified with $\om{Sp}(1)^{(\alpha;\beta)}$ through
\[ \om{Sp}(1)^{(\alpha;\beta)}\inj \om{Hom}_\HH(\HH(\beta),\HH(\alpha))
\cong \om{Hom}_\HH(E_S,E_N),  \]
where we recall that by convention the first map is given by 
$q\mapsto r_{q^*}$. Let $t\colon S^3\ra S^3(\alpha,\beta)$ denote the restriction of this map to the middle sphere. One may verify that the
$\Gamma$-equivariant homotopy class of this map is independent of the trivializations chosen. We have therefore constructed one direction of
the equivalence.

The inverse may be defined as follows. Let 
$f\colon S^3\ra \om{Sp}(1)^{(\alpha;\beta)}$
represent a homotopy class in $[S^3,\om{Sp}(1)^{(\alpha;\beta)}]_\Gamma$.
Let $p\colon U_N\cap U_S\ra S^3$ be the equivariant projection onto the middle
sphere. Using $f$ we may therefore construct an equivariant bundle
$E$ by gluing $U_S\times \HH(\beta)$ and $U_N\times \HH(\alpha)$ along
$\psi\colon U_N\cap U_S\times \HH(\beta)\ra U_S\cap U_N\times \HH(\alpha)$ given by $\psi(x,v) = (x,f(p(x))\cdot v)=(x,vf(p(x))^*)$. The equivariance of $f$ and $p$ ensure that the actions of $\Gamma$ on $U_N\times \HH(\alpha)$ and $U_S\times \HH(\beta)$ match over the intersection $U_N\cap U_S$. The resulting bundle is therefore a $\Gamma$-equivariant quaternionic line bundle. One may verify that the isomorphism class of this bundle is independent of
the representative $f$ chosen using the equivariant
bundle homotopy theorem of \cite[Prop. 1.3]{Segal68}.

The verification of the fact that these two constructions are mutual inverses proceeds just as in the non-equivariant case. \end{proof} 

Recall that we regard $S^4\subset \HH\oplus \R$. Let $u\colon S^4-\{0\}\cong \HH$ be stereographic projection from the north pole. We give $\HH$ the
standard orientation and orient $S^4$ by requiring $u$ to be
orientation preserving. Moreover, we orient the middle sphere $S^3\subset S^4$
by requiring that the restriction $u|_{S^3}:S^3\cong S^3\subset \HH$ preserves
orientation. 

\begin{proposition} \label{Homotopy-Eq}
Let $\alpha,\beta\in \om{Rep}^1(\Gamma,\HH)$ have
representatives $\rho_\alpha,\rho_\beta\colon \Gamma\ra \om{Sp}(1)$. Then
the natural map
\[ [S^3,\om{Sp}(1)^{(\alpha;\beta)}]_\Gamma \ra [S^3,\om{Sp}(1)^{(\alpha;\beta)}] \]
is injective. In other words, two equivariant maps $f,g\colon S^3\ra \om{Sp}(1)^{(\alpha;\beta)}$ are homotopic if and only if they are equivariantly homotopic.
Moreover, if we identify $[S^3,\om{Sp}(1)^{(\alpha;\beta)}]\cong \Z$ through
$[f]\mapsto \om{deg}(f)$, there is a constant $c\in \Z$
so that the image is $\{ n\in \Z: n\equiv c\mod{|\Gamma |} \}$.   
\end{proposition}
\begin{proof} 
By Lemma \ref{Classification-Lemma} we can find equivariant maps
$g\colon S^3\ra \om{Sp}(1)^\alpha$ and $h\colon S^3\ra \om{Sp}(1)^\beta$. 
Given an equivariant map $f\colon S^3\ra \om{Sp}(1)^{(\alpha;\beta)}$ define 
$\tau(f)=\tau_{g,h}(f)\colon S^3\ra \om{Sp}(1)^\theta$ to be the composition
\[  \begin{tikzcd} S^3\arrow{r}{(h,f,g)} & \om{Sp}(1)^\alpha\times 
\om{Sp}(1)^{(\alpha;\beta)}\times \om{Sp}(1)^\beta \arrow{r}{\mu} & \om{Sp}(1)^\theta, \end{tikzcd} \]
where $\om{Sp}(1)^\theta$ denotes $\om{Sp}(1)$ equipped with the
trivial $\Gamma$-action and $\mu$ is the equivariant map $(x,y,z) \mapsto x^*yz$. Observe that if $f'\colon S^3\ra \om{Sp}(1)^{(\alpha;\beta)}$
is another equivariant map, then $f$ and $f'$ are equivariantly homotopic
if and only if $\tau(f)$ and $\tau(f')$ are equivariantly homotopic.
Indeed, if $f_t\colon S^3\ra \om{Sp}(1)^{(\alpha;\beta)}$ is an equivariant homotopy between $f$ and $f'$,
then $\tau(f_t)$ is an equivariant homotopy between $\tau(f)$ and $\tau(f')$
and the same argument applies for the converse using $\tau_{g^{-1},h^{-1}}$
in place of $\tau = \tau_{g,h}$.

Since the $\Gamma$-action on $\om{Sp}(1)^\theta$ is trivial, we obtain a bijection
\[ [S^3,\om{Sp}(1)^\theta]_\Gamma \cong [S^3/\Gamma,\om{Sp}(1)]\cong \Z,  \]
where we give $S^3/\Gamma$ the orientation induced from $S^3$ and the second isomorphism is given by $[u]\mapsto \om{deg}(u)$. Given an equivariant map $f\colon S^3\ra \om{Sp}(1)^{(\alpha;\beta)}$
let $\overline{\tau(f)}\colon S^3/\Gamma\ra \om{Sp}(1)$ denote the map
induced by $\tau(f)$. The above work then amounts to the fact that the map
$[S^3,\om{Sp}(1)^{(\alpha;\beta)}]_\Gamma \ra \Z$ given by 
$f\mapsto \om{deg}(\overline{\tau(f)})$
is a bijection. To complete the proof we will relate this degree
to the degree of $f$. First, since $S^3\ra S^3/\Gamma$ has degree $|\Gamma|$,
we deduce that $|\Gamma |\om{deg}(\ovl{\tau(f)})=\om{deg}\tau(f)$.
Using Lemma \ref{Degree-Sum} and the definition $\tau(f)=\mu\circ (g,f,h)$
one finds that
\[ \om{deg}(\tau(f)) = \om{deg}h-\om{deg}g+\om{deg}(f) . \]
Hence, $\om{deg}(f) = (\om{deg}g-\om{deg}h)+|\Gamma |\om{deg}(\overline{\tau(f)})$.
We therefore have the following commutative diagram 
\begin{equation} \label{Diagram-Constant-c} 
\begin{tikzcd} 
 {[} S^3,\om{Sp}(1)^{(\alpha;\beta)}]_\Gamma \arrow{r} \arrow{d}{\cong}
& {[} S^3,\om{Sp}(1)^{(\alpha;\beta)}] \arrow{d}{\cong} \\
\Z \arrow{r}{m} & \Z ,\end{tikzcd}
\end{equation} 
where the vertical maps are given by $[f]\mapsto \om{deg}(\overline{\tau(f)})
$ and $[f]\mapsto \om{deg}f$, from left to right, and 
$m(n) = |\Gamma |n+(\om{deg}g-\om{deg}h)$. 
Since $m$ is injective we conclude that the upper horizontal map is injective.
The final assertion follows from the fact that $\om{deg}g-\om{deg}h$
is independent of $g$ and $h$ by Lemma \ref{Classification-Lemma}. 
\end{proof} 

\begin{theorem} \label{Classification of Equivariant Bundles over $S^4$}
Let $\Gamma\subset \om{Sp}(1)$ be a finite subgroup acting
on $S^4\subset \HH\oplus R$ by $\gamma\cdot (x,t)=(\gamma\cdot x,t)=(x\gamma^*,t)$. Then the map 
\[ \phi\colon \om{Vec}^1_\Gamma(S^4,\HH)\ra \om{Rep}^1(\Gamma,\HH)\times \om{Rep}^1(\Gamma,\HH)\times \Z \]
given by $[E]\mapsto ([E_N],[E_S],c_2(E)[S^4])$ is injective.
Furthermore, for each pair $\alpha,\beta\in \om{Rep}^1(\Gamma,\HH)$
there is a constant $c\in \Z$ such that $(\alpha,\beta,k)\in \om{Im}\phi$
if and only if $k\equiv c \Mod{|\Gamma |}$.
\end{theorem} 
\begin{proof} The follows from Proposition \ref{Clutching-Lemma} and
Proposition \ref{Homotopy-Eq} in view of the fact that
$c_2(E)[S^4]$ coincides with the degree of the transition function
$t\colon S^3\ra \om{Sp}(1)\subset \om{Hom}_\HH(E_S,E_N)$.
\end{proof} 

The proof of Proposition \ref{Homotopy-Eq} gives a formula for the constant
$c$ that will be useful later. 

\begin{corollary} \label{c-Corollary} 
The constant $c$ in the above theorem may be taken to be
$c=\om{deg}g-\om{deg}h$ for any choices of equivariant maps 
$g\colon S^3\ra \om{Sp}(1)^\alpha$ and $h\colon S^3\ra \om{Sp}(1)^\beta$.
\end{corollary}

\subsection{Application of the Equivariant Index Theorem}
Let $E\ra S^4$ be a $\Gamma$-equivariant $\om{SU}(2)$ bundle and let
$S^+$ and $S^-$ be the complex spinor bundles associated with the unique
spin structure on $S^4$. It turns out that this spin structure
is naturally $\Gamma$-equivariant. This means that the spinor bundles
carry the structure of $\Gamma$-equivariant bundles and that
the Clifford multiplication map $TS^4\ra \om{Hom}_\C(S^+,S^-)$ is
a map of equivariant bundles. Therefore, if $A$ is a $\Gamma$-invariant connection in $E$ the twisted Dirac operator
\[ D_A\colon \Gamma(S^+\otimes E)\ra \Gamma(S^-\otimes E)  \]
is a $\Gamma$-equivariant elliptic operator. The non-equivariant index
of such an operator takes values in the $K$-theory of a point
$K(*)=\Z$. In the $\Gamma$-equivariant setting it takes values in
the $\Gamma$-equivariant $K$-theory of a point, namely,
$K_\Gamma(*)=R(\Gamma)$; the complex representation ring of $\Gamma$.
It is this index we will calculate in this section.

For this purpose it is more convenient to use the model $\HH P^1$ for
$S^4$. This is the identification space $\HH^2-\{0\}/\sim $ where
$(x,y)\sim (qx,qy)$ for $q\in \HH-\{0\}$. The equivalence class
of $(x,y)$ in $\HH P^1$ is denoted by $[x:y]$. There are two canonical charts
\[ \HH \cong U = \{[x:y]\in \HH P^1: x\neq 0\} \; \mbox{ and } \; \HH \cong V =\{[x:y]\in \HH P^1: y\neq 0\}\]
given by $z\mapsto [1:z]$ and $w\mapsto [w:1]$, respectively.
The transition function $\HH-\{0\}\ra \HH-\{0\}$ (in either direction)
is given by $q\mapsto q^{-1}$. This map is orientation preserving so
we define the orientation of $\HH P^1$ by requiring that both of the canonical
charts are positive, where $\HH=\C^2$ has the standard orientation.

The tautological quaternionic line bundle $\gamma\ra \HH P^1$ is defined
to be
\[ \gamma \coloneqq \{ ([x:y],(z,w))\in \HH P^1\times \HH^2 : \exists q\in \HH-\{0\}\ni
(z,w)=(qx,qy) \}  .\]
Let the trivial bundle $\underline{\HH}^2 =\HH P^1\times \HH^2$ carry the
standard symplectic inner product, i.e., $(z_1,z_2)\cdot (w_1,w_2)=z_1w_1^*+z_2w_2^*$. Then $\wt{\gamma} \coloneq \gamma^\perp$ is
another quaternionic line bundle. We give $\gamma$ and $\tilde{\gamma}$
the connections obtained from the product connection in $\underline{\HH}^2$
by orthogonal projection and equip them with the symplectic inner products
inherited from the trivial bundle $\underline{\HH}^2$.  

The standard action $\om{Sp}(2)\times \HH^2\ra \HH^2$ descends to a
transitive action $\om{Sp}(2)\times \HH P^1\ra \HH P^1$. The diagonal
action on $\HH P^1\times \HH^2$ preserves the subbundles $\gamma,\tilde{\gamma}\subset \underline{\HH}^2$ and therefore give these the structure of
$\om{Sp}(2)$-equivariant bundles over $\HH P^1$. The connections
are invariant under this action.

The unique spin structure on $\HH P^1$ may now be realized explicitly
with $S^+=\gamma$, $S^-=\wt{\gamma}$ and Clifford multiplication
$\chi\colon T\HH P^1 \cong \om{Hom}_\HH(\gamma,\wt{\gamma})$ defined in the
following way. Let $\iota\colon \gamma\ra \underline{\HH}^2$ denote the inclusion and $\pi\colon \underline{\HH}^2\ra \wt{\gamma}$ the projection.
Given $p\in \HH P^1$, $v\in T_p\HH P^1$ and
$w\in \gamma_p$ choose a local smooth section $s$ of $\gamma$ with
$s(p)=w$. Then $\chi(v)(w) = \pi(d(\iota(s)))_p$, where $d$ is the
product connection in $\underline{\HH}^2$. In suitable trivializations
of $\gamma$ and $\tilde{\gamma}$ over $\HH\cong U\subset \HH P^1$, the map
$\chi_q\colon T_q\HH=\HH \ra \om{Hom}_\HH(\HH,\HH)$ for $q\in \HH$ takes the
form  $v\mapsto r_v\colon \HH\ra \HH$, where $r_v(w)=wv$ is right multiplication. Finally, we note that this spin structure is $\om{Sp}(2)$-equivariant.

\begin{remark} With our orientation conventions it holds true
that $c_2(\gamma)[\HH P^1]=-1$. This can be verified, as suggested in \cite[p.~85]{UhlenbeckFreed91}, as follows. First check
that the transition function restricted to $S^3\subset \HH \cong U =\HH P^1-\{[0:1]\}$ is a diffeomorphism and hence has degree $\pm 1$. Then
explicitly calculate the connection and curvature forms in the chart
$\HH\cong U$ and conclude that the connection of $\gamma$ is self-dual.
From these two pieces of information the assertion follows.
As $\gamma\oplus \wt{\gamma}$ is trivial it follows that $c_2(\tilde{\gamma})[\HH P^1]=1$ as well.  
\end{remark}   
 
Let $\Gamma\subset \om{Sp}(1)$ be a finite subgroup. We define
a $\Gamma$-action on $\HH P^1$ by restricting the $\om{Sp}(2)$-action
along $\Gamma\subset \om{Sp}(1)\subset \om{Sp}(2)$, where the latter inclusion
is specified by
\begin{equation} \label{Action-Des}
 q\in \om{Sp}(1) \mapsto \left( \begin{array}{cc} 1& 0 \\ 0 & q
\end{array} \right) .
\end{equation} 
From the above discussion it follows that the spinor bundles $S^\pm$
obtain $\Gamma$-equivariant structures and that the spinor connections,
are $\Gamma$-invariant. By convention the action of $\om{Sp}(2)$ on $\HH P^1$
is given by
\[ \left( \begin{array}{cc} a & b \\ c & d \end{array} \right)
\cdot [x:y] = [x:y] \left( \begin{array}{cc} a^* & c^* \\ b^* & d^*
\end{array} \right) = [xa^*+yb^*:xc^*+yd^*] . \]
From this we see that in the chart 
$\HH\cong U\subset \HH P^1$, $z\mapsto [1:z]$,
this action corresponds to the standard linear action $\Gamma\times \HH\ra \HH$
given by $\gamma\cdot x = x\gamma^*$.

\begin{remark} By regarding $S^4\subset \HH\oplus \R$ as before the map
$\psi\colon \HH P^1\ra S^4$ given by
\[ \psi([x:y]) = \frac1{||x||^2+||y||^2}(2x^*y,||y||^2-||x||^2)  \]
is a $\Gamma$-equivariant, orientation preserving diffeomorphism, 
when $S^4$ is given the suspended action and orientation described earlier
(note that this is the opposite of the standard orientation on $S^4$). 
\end{remark} 

Let $G$ be a compact Lie group, $X$ a $G$-manifold and $P\colon \Gamma(V)\ra \Gamma(W)$ an elliptic $G$-operator, that is, $V$ and $W$ are $G$-equivariant
complex vector bundles and $P$ is equivariant with respect the induced actions. Then the $G$-index of $P$ is defined by
\[ \om{Ind}_G(P) = [\om{Ker}P]-[\om{Coker}(P)]\in R(G) . \]
The character of this virtual representation is denote by $\om{Ind}_g(P)=
\om{tr}(g|_{\om{Ker}(P)})-\om{tr}(g|_{\om{Coker}(P)})\in \C$ for
each $g\in G$. The fixed-point formula due to Atiyah, Segal and Singer expresses the fact that $\om{Ind}_g(P)$ only depends on information above the
fixed point set $F_g\coloneqq \{x\in X:gx=x\}$. In the case where $F_g$ is a finite set the formula (see \cite[Theorem~14.3]{LawsonMichelsohn89})
simplifies to
\begin{equation} \label{Index-Formula}
\om{Ind}_gP = \sum_{x\in F_g}\frac{\om{ch}_g(V_x)-\om{ch}_g(W_x)}
{\om{ch}_g(\lambda_{-1}(T_xX)_\C)}   .
\end{equation}
Here, $\om{ch}_g(V_x)=\om{tr}(g\colon V_x\ra V_x)$, $(T_xX)_\C$ is the complexification of $T_xX$ and $\lambda_{-1}$ is defined for a $G$
representation $U$ by 
\[ \lambda_{-1}(U) = \sum_{k=0}^{\om{dim}_\C U}(-1)^k \Lambda^k_\C(U)\in R(G)  .\]

Let $\Gamma\subset \om{SU}(2)=\om{Sp}(1)$ be a finite subgroup acting
on $\HH P^1$ in the way described above. This action is free away
from the two fixed points $S=[1:0]$ and $N=[0:1]$. 
Let $E\ra \HH P^1\cong S^4$ be a $\Gamma$-equivariant $\om{Sp}(1)$-bundle
with $[E_S]=\alpha$, $[E_N]=\beta$ and $c_2(E)[\HH P^1]=k$. Let $A$ be a $\Gamma$-invariant connection in $E$. Then the twisted Dirac operator
$D_A\colon \Gamma(S^+\otimes E)\ra \Gamma(S^-\otimes E)$
is an elliptic $\Gamma$-operator as described above. For our purpose
the formal adjoint $D_A^*\colon \Gamma(S^-\otimes E)\ra \Gamma(S^+\otimes E)$
is more relevant.

\begin{remark} Here we regard $S^\pm$ and $E$ as complex vector bundles
by restriction along $\C\subset \HH$. In particular, the tensor products
$S^\pm\otimes E=S^\pm\otimes_\C E$ are formed over the complex numbers.
\end{remark}  

\begin{proposition} \label{B-IndexCalc}
Let $E\ra \HH P^1\cong S^4$ be a $\Gamma$-equivariant
$\om{SU}(2)$-bundle with $[E_N],[E_S]=\alpha,\beta\in R(\Gamma)$
and $c_2(E)[\HH P^1]=k\in \Z$. Then $\Har\coloneqq \om{Ind}_\Gamma(D_A^*)\in R(\Gamma)$ satisfies
\[ (2-Q)\Har = \alpha-\beta,  \]
where $Q$ is the $2$-dimensional complex representation
associated with the inclusion $\Gamma\subset \om{SU}(2)$. Furthermore,
if $\eps:R(\Gamma)\ra \Z$ denotes the augmentation, then $\eps(\Har) = k$.  
\end{proposition}
\begin{proof} To apply the formula (\ref{Index-Formula}) we have to
determine the action of $\Gamma$ on the fibers of $S^\pm\otimes E$
over the fixed points $S=[1:0]$ and $N=[0:1]$ as well as the
action on the tangent spaces $T_N\HH P^1, T_S\HH P^1$. From the
definition of $S^+=\gamma$ and $S^-=\wt{\gamma}$ we observe
that in terms of the inclusions $\gamma,\wt{\gamma}\subset \underline{\HH}^2$
\[ \begin{array}{cc} \gamma_S =\HH\oplus 0 & \gamma_N = 0\oplus \HH \\
                  \wt{\gamma}_S = 0\oplus \HH & \wt{\gamma}_N = \HH\oplus 0 \end{array} . \]
From the description of the action in (\ref{Action-Des}) it follows that
$\gamma_S \cong \wt{\gamma}_N \cong \C^2$, the trivial representation,
and $\gamma_N\cong \wt{\gamma}_S\cong Q$. 

Next we need to investigate the action on the tangent spaces. We have
already seen that in the chart $\HH\cong U\subset \HH P^1$ the action
corresponds to the standard linear action $\gamma\cdot x=x\gamma^*$.
Here $0\in \HH$ corresponds to $S\in \HH P^1$. 
It follows that $T_S\HH P^1\cong Q_\R$, that is, the canonical representation
regarded as a real representation. In the other chart
$\HH\cong V \subset \HH P^1$ given by $z\mapsto [z:1]$ one may verify
that the action is given by $\gamma\cdot z=\gamma z$. The differentiated
action on $\HH\cong T_0\HH$ is given by the same formula. 
The conjugation map $*\colon \HH\ra \HH$ is real linear and satisfies
$(\gamma x)^*=x^*\gamma^*$. Hence, $T\HH P^1_N \cong Q_\R$ as well. 
We need to calculate the term
$\lambda_{-1}((Q_\R)_\C)$. First, note that $(Q_\R)_\C\cong Q\oplus Q^*\cong 2Q$, since $Q$ is an $\om{SU}(2)$-representation and therefore self-dual. Thus,
\[ \lambda_{-1}((Q_\R)_\C) = \lambda_{-1}(2Q)=(\lambda_{-1}Q)^2
=(\C-Q+\Lambda^2Q)^2 = (2-Q)^2  ,\]
where we have used that $\lambda_{-1}(V\oplus W)=\lambda_{-1}(V)\lambda_{-1}(W)$ and the fact that $\Lambda^2 Q\cong \C$, the trivial representation.

Given an element $V\in R(\Gamma)$ we let $\chi_V\colon\Gamma\ra \C$ denote
the associated character. For any $g\neq 1\in \Gamma$ the index formula \eqref{Index-Formula} now gives $\chi_\Har(g)$ as
\[ \left( \frac{2\chi_\alpha(g) -\chi_Q(g)\chi_\alpha(g)}{(2-\chi_Q(g))^2} \right) + \left(\frac{\chi_Q(g)\chi_\beta(g)-2\chi_\beta(g)}{(2-\chi_Q(g))^2}\right) =\frac{\chi_\alpha(g)-\chi_\beta(g)}{2-\chi_Q(g)}, \]
where the first term corresponds to the fixed point $S=[1:0]$ and
the second term to $N=[0:1]$. 
This yields the equality of characters $(2-\chi_Q)\chi_\Har =\chi_\alpha-\chi_\beta$, since by the above it holds for all $g\neq 0$, while for
$g=1$ it is trivially satisfied as $\chi_Q(1)=\chi_\alpha(1)=\chi_\beta(1)=2$.
We therefore obtain the relation $(2-Q)\Har = \alpha-\beta$ in $R(\Gamma)$.
The final assertion follows from the regular index theorem
\[ \eps(\Har) = \om{Ind}(D_A^*) = c_2(E)[S^4]  .\]
\end{proof}

\subsection{The Chern-Simons Invariant}
Let $Y$ be a closed oriented $3$-manifold and let $E\ra Y$ be the
necessarily trivial $\om{Sp}(1)=\om{SU}(2)$-bundle over $Y$. The Chern-Simons
invariant of a flat connection $A$ in $E$ is the value $\om{cs}(A)\in \R/\Z$,
where $\om{cs}$ is the Chern-Simons functional defined in equation
(\ref{CS-Functional}). In this section we will explain that the theory
of section $(4)$ lead to a simple procedure for the calculation of
this invariant when $Y=Y_\Gamma$ for some finite subgroup $\Gamma\subset \om{SU}(2)$. 

Equip $\mathfrak{sp}(1)=\om{Im}\HH$ with the standard invariant inner product
for which $i,j,k$ is an orthonormal basis and write $g(x,y)=-(x,y)$
for its negative. The identification $\mathfrak{sp}(1)\cong \mathfrak{su}(2)$
matches $g$ with the symmetric bilinear form $(x,y)\mapsto \frac12 \om{tr}(xy)$. In this section we will prefer to work with the group $\om{Sp}(1)$
so we make the following definition (compare equation (\ref{CS-Functional})).

\begin{definition} Let $X$ be a manifold. For any $a\in \Omega^1(X,\mathfrak{sp}(1))$ define the Chern-Simons form $\zeta(a)\in \Omega^3(X,\R)$ by
\[ \zeta(a) = \frac1{4\pi^2}[ a\wedge_g da +\frac13 a\wedge_g (a\wedge_{ad}a)] .\] 
\end{definition} 
Here $\wedge_g$ and $\wedge_{ad}$ denote the combination of the wedge product
and the bilinear maps $g$ and $\om{ad} = [\cdot,\cdot]$, respectively. 
Note that $f^*(\zeta(a))=\zeta(f^*(a))$ for any smooth map $f\colon X\ra Y$ and $a\in \Omega^1(Y,\mathfrak{sp}(1))$. 

Let $Y$ be a closed oriented
$3$-manifold and let $E\ra Y$ be an $\om{Sp}(1)$-bundle. Fix a global
trivialization $E\cong Y\times \om{Sp}(1)$ so that the space of connections
$\MC{A}_E$ is identified with $\Omega^1(Y,\mathfrak{sp}(1))$. Then for
$a\in \Omega^1(Y,\mathfrak{sp}(1))$ we have $\om{cs}(a)=\int_Y \zeta(a) \mod{\Z}$. Observe that if $a$ is flat, that is, $da+\frac12 a\wedge_{ad}a=0$,
then 
\begin{equation}\label{CS-Flat-Formula}
\zeta(a) = \frac{-1}{24\pi^2} a\wedge_g (a\wedge_{ad} a) .
\end{equation} 

\begin{lemma} \label{VolS3}
Let $\theta\in \Omega^1(\om{Sp}(1),\mathfrak{sp}(1))$ be the
left Maurer-Cartan form, that is, $\theta_g(v)=dl_{g^{-1}}(v)$ for all $g\in \om{Sp}(1)$ and $v\in T_g\om{Sp}(1)$. Then
\[ \int_{\om{Sp}(1)} \zeta(\theta) =  1,\]
when $\om{Sp}(1)$ is oriented by requiring that $(i,j,k)$ is a positive
basis for $T_1\om{Sp}(1)=\mathfrak{sp}(1)=\om{Im}\HH$. 
\end{lemma}
\begin{proof} As $\theta$ is left invariant and $l_g^*\zeta(\theta)=\zeta(l_g^*(\theta))=\zeta(\theta)$ for each $g\in \om{Sp}(1)$, it follows that
$\zeta(\theta)\in \Omega^3(\om{Sp}(1),\R)$ is left invariant. Every
left invariant form of top dimension in a compact, connected Lie group is necessarily right invariant as well. Therefore, if we regard
$\om{Sp}(1)=S^3\subset \HH=\R^4$, we see that $\zeta(\theta)$ is invariant
under the action $\om{Sp}(1)\times \om{Sp}(1) \times S^3\ra S^3$
given by $(a,b)\cdot x = axb^{-1}$. This action factors through the
double covering $\om{Sp}(1)\times \om{Sp}(1)\ra \om{SO}(4)$ where
$\om{SO}(4)$ acts in the standard way on $S^3\subset \R^4$.   
We conclude that $\zeta(\theta)$ is $\om{SO}(4)$-invariant. 

The Riemannian volume form of $S^n$ equipped with the standard round
metric is the restriction of
\[ \omega_n =  \sum_{i=1}^{n+1} (-1)^{i+1}x_i dx_1\wedge \cdots \wedge
\hat{dx}_i \wedge \cdots \wedge dx_n \in \Omega^n(\R^{n+1},\R), \]
where the hat denotes omission. This form is invariant under the
standard transitive action of $\om{SO}(n+1)$. Specializing to
the case $n=3$, it follows that there is a constant $c\in \R$ such   
such that $\zeta(\theta)=c\omega_3$.  

We determine the constant $c$ by evaluating at $1=(1,0,0,0)\in \HH=\R^4$.
By the formula above we have $(\omega_3)_1 = dx_2\wedge dx_3\wedge dx_4$,
so that $c=\zeta(\theta)(i,j,k)$. As $\theta_1\colon \mathfrak{sp}(1)\ra \mathfrak{sp}(1)$ is the identity, we have for $a,b,c\in \mathfrak{sp}(1)$ that
\[ \theta\wedge_g (\theta\wedge_{ad}\theta) (a,b,c) =
   g(a,2[b,c])-g(b,2[a,c])+g(c,2[a,b]) =6g(a,[b,c]),  \]
since $g$ is symmetric and $g(b,[a,c])=g([b,a],c)=-g([a,b],c)=-g(a,[b,c])$.
Inserting $(a,b,c)=(i,j,k)$ we find
\[ \theta\wedge_g(\theta\wedge_{ad} \theta)(i,j,k)
= 6g(i,[j,k])=6g(i,2i) = -12(i,i)=-12  .\]
Consequently, using formula (\ref{CS-Flat-Formula}) we obtain $\zeta(\theta)_1(i,j,k)=\frac{1}{2\pi^2}$ and therefore
$\zeta(\theta) = \frac{1}{2\pi^2}\omega_3$. The volume of $S^3$ is given by
$\om{Vol}(S^3) = \int_{S^3} \omega = 2\pi^2$, so it follows that
$\int_{\om{Sp}(1)}\zeta(\theta) = 1$ as required. \end{proof}  

In the following we give $Y_\Gamma=S^3/\Gamma$ the orientation induced
from the standard orientation on $S^3\subset \HH$. 

\begin{lemma} \label{CS-degree-formula} 
Let $\alpha\in \om{Rep}^1(\Gamma,\HH)$ correspond to
a flat connection in the $\om{Sp}(1)$-bundle over $Y_\Gamma$. Choose a
representative $\rho_\alpha\colon \Gamma\ra \om{Sp}(1)$
for $\alpha$ and let $\om{Sp}(1)^\alpha$ denote the group $\om{Sp}(1)$
equipped with the $\Gamma$-action $\gamma\cdot q = \rho_\alpha(\gamma)q$.
Then for any $\Gamma$-equivariant map $g\colon S^3\ra \om{Sp}(1)^\alpha$ it
holds true that
\[ \om{cs}(\alpha) = \om{deg}(g)/|\Gamma | \mod{\Z}  .\]
\end{lemma}
\begin{proof} Choose a representative $a\in \Omega^1(Y_\Gamma,\mathfrak{sp}(1))$ for the flat connection in a global trivialization of the bundle. 
Then by definition
\[ \om{cs}(\alpha) = \int_{Y_\Gamma} \zeta(a) \mod{\Z}  .\]
Let $p\colon S^3\ra Y_\Gamma$ be the quotient map. Since $Y_\Gamma$ carries
the orientation induced from $S^3$ this map has degree $|\Gamma |$.
Let $b=p^*(a)\in \Omega^1(S^3,\mathfrak{sp}(1))$. Since $a$ and
$b$ are flat and $S^3$ is simply connected, there exists a smooth map
$g\colon S^3\ra \om{Sp}(1)$ such that $b = g^*\theta$, where
$\theta\in \Omega^1(\om{Sp}(1),\mathfrak{sp}(1))$ is the left Maurer-Cartan
form. Moreover, the map $g$ is unique up to left translation in the sense
that if $h\colon S^3\ra \om{Sp}(1)$ is another map with $h^*\theta=b$, then
$h=l_u\circ g$ for some $u\in \om{Sp}(1)$, where $l_u(v)=uv$ denotes left
translation in $\om{Sp}(1)$. For each $\gamma\in \Gamma$ we have
\[ \gamma^*(b)=\gamma^* p^*(a)=(p\circ \gamma)^*a = p^*(a)=b . \]
Then as $(g\circ \gamma)^*(\theta)=\gamma^*b=b$ we deduce from the
uniqueness of $g$ that there is a unique $\rho(\gamma)\in \om{Sp}(1)$
such that $g\circ \gamma = l_{\rho(\gamma)}\circ g$. One may easily verify
that $\rho\colon \Gamma\ra \om{Sp}(1)$ is a homomorphism, in fact, it encodes
the holonomy representation of the flat connection $a\in \Omega^1(Y_\Gamma,\mathfrak{sp}(1))$ we started with. If we replace $g\colon S^3\ra \om{Sp}(1)$
with $l_u\circ g$, then $\rho:\Gamma\ra \om{Sp}(1)$ is replaced by
$c_u\circ \rho$, where $c_u$ is conjugation by $u$ in $\om{Sp}(1)$.
Therefore, for a suitable choice of $u\in \om{Sp}(1)$ we can arrange
that $\rho=\rho_\alpha$. From the defining formula
$g\circ \gamma = l_{\rho_\alpha(\gamma)}\circ g$ it follows that
$g\colon S^3\ra \om{Sp}(1)^\alpha$ is equivariant. 

Our setup is summarized in the following diagram
\[ \begin{tikzcd} \theta  \arrow{r}{g^*} & b &
a \arrow{l}[swap]{p^*} \\
\om{Sp}(1) & S^3 \arrow{l}[swap]{g} \arrow{r}{p} & Y_\Gamma. \end{tikzcd} \]
We may now calculate
\[ \int_{Y_\Gamma} \zeta(a) = \frac{1}{|\Gamma |} \int_{S^3} p^*(\zeta(a))
=\frac{1}{|\Gamma |} \int_{S^3} g^*(\zeta(\theta)) =\frac{\om{deg}(g)}{|\Gamma |}
\int_{\om{Sp}(1)} \zeta(\theta)  .\]
By the previous lemma $\int_{\om{Sp}(1)}\zeta(\theta)=1$, so by reducing
this modulo integers we obtain $\om{cs}(\alpha)=\om{cs}(a)=\frac{\om{deg}g}{|\Gamma |} \mod{\Z}$. To complete the proof we must show that
$\om{deg}(h)/|\Gamma | \mod{\Z}$ is independent of the equivariant map
$h\colon S^3\ra \om{Sp}(1)^\alpha$. This is a consequence of Lemma
\ref{Classification-Lemma}, which shows that $\om{deg}h \mod{|\Gamma |}$
is independent of $h$. 
\end{proof} 
  
\begin{example} \label{CS-Calc-Example1}
Let $\Gamma\subset \om{Sp}(1)=\om{SU}(2)$ be any finite
subgroup. Let $\alpha =Q$ denote the the canonical representation
associated with the inclusion above. Then the map
$g\colon S^3\ra \om{Sp}(1)^\alpha$ given by $g(x)=x^*$ is equivariant. Indeed,
\[ g(\gamma\cdot x)=g(x\gamma^*)=(x\gamma^*)^*=\gamma x^*=\gamma\cdot g(x)  \]
for each $\gamma\in \Gamma$ and $x\in S^3$. Since $S^3\subset \HH$ and $\om{Sp}(1)\subset \HH$ both carry the standard orientation for which $(i,j,k)\in T_1S^3 =T_1\om{Sp}(1)$ is a positive basis, it follows that $\om{deg}g = -1$ and therefore
\[ \om{cs}(\alpha) = \frac{-1}{|\Gamma |} \mod{\Z} . \]
If we reverse the orientation of $Y_\Gamma$ then the value of $\om{cs}$
changes by a sign. 
\end{example} 

\begin{proposition} \label{CS-Calc-Tool}
Let $\Gamma\subset \om{SU}(2)$ be a finite subgroup
and let $\alpha,\beta$ be a pair of $\om{SU}(2)$-representations of
$\Gamma =\pi_1(Y_\Gamma)$ corresponding to a pair of flat connections
in the $\om{SU}(2)$-bundle over $Y_\Gamma$. Let $\Har$ be a solution
of the equation
\[ (2-Q)\Har = \alpha-\beta \]
in $R(\Gamma)$. Then $\om{cs}(\alpha)-\om{cs}(\beta)=\eps(\Har)/|\Gamma | \in \R/\Z$. 
\end{proposition}
\begin{proof} Let $E\ra S^4$ be a $\Gamma$-equivariant
$\om{SU}(2)$-bundle with $[E_N]=\alpha$ and $[E_S]=\beta$. Then by Proposition \ref{B-IndexCalc} the element
$\Har \coloneqq \om{Ind}(D_A^*)\in R(\Gamma)$, where $A$ is any $\Gamma$-invariant
connection in $E$, satisfies $(2-Q)\Har = \alpha-\beta$ and
$\eps(\Har) = c_2(E)[S^4]$. Moreover, by the classification theorem
\ref{Classification of Equivariant Bundles over $S^4$} and
Corollary \ref{c-Corollary} the integer $c_2(E)[S^4]$ satisfies a congruence 
\[ c_2(E)[S^4]\equiv \om{deg}g-\om{deg}h \mod{|\Gamma |},  \]
where $g\colon S^3\ra \om{Sp}(1)^\alpha$ and $h\colon S^3\ra \om{Sp}(1)^\beta$ are any choices of equivariant maps. From Lemma \ref{CS-degree-formula} we obtain
\[ \om{cs}(\alpha)-\om{cs}(\beta) \equiv (\om{deg}g-\om{deg}h)/|\Gamma |
\equiv c_2(E)[S^4]/|\Gamma | \equiv \eps(\Har)/|\Gamma |\mod{\Z}  .\]
This shows that the result is valid for a specific solution $\Har$.
However, by Lemma \ref{R(G)-equation} any other solution is given 
by $\Har+mR$, where $m\in \Z$ and $R=\C[\Gamma]$ is the regular representation.
Since $\om{dim}_\C R=|\Gamma |$ we conclude that $\eps(\Har)/|\Gamma |$
is independent of the solution $\Har$ modulo $\Z$ and the proof is
complete. 
\end{proof} 

This theorem reduces the calculation of Chern-Simons invariants to solving an equation in $R(\Gamma)$. Employing the
graphical solution procedure of Proposition \ref{R(G)-Adjacent-Solution}
one may easily solve for all the flat connections recursively
starting at the trivial connection. We will give an example in the
next section. 

\subsection{Connections to Group Cohomology}
Given a finite group $G$ its group cohomology may be defined
by $H^*(G;\Z)\coloneq H^*(BG;\Z)$, where $BG$ is its classifying space
(see for instance \cite[Chapter~II]{AdemMilgram04}). A representation
$\rho\colon G\ra U(n)$ gives rise to characteristic classes
$c_i(\rho)\in H^{2i}(\Gamma;\Z)$ by pulling back the universal
Chern classes $c_i \in H^{2i}(BU(n);\Z)$ along the induced map
$B\rho\colon B\Gamma\ra BU(n)$. The purpose of this section is to show that
for a finite group $\Gamma\subset \om{SU}(2)$ there is a natural
way to identify the Chern-Simons invariant $\om{cs}(\alpha)\in \R/\Z$
with the second Chern class $c_2(\rho_\alpha)\in H^4(\Gamma;\Z)$,
where $\rho_\alpha:\Gamma\ra \om{SU}(2)$ is the holonomy representation associated with $\alpha$. 

The following calculation is certainly known. We include a quick proof
for completeness.

\begin{lemma} Let $\Gamma\subset \om{SU}(2)$ be a finite subgroup.
Then $H^0(\Gamma;\Z)=\Z$, while for $i>0$
\[ H^i(\Gamma;\Z) \cong \left\{  \begin{array}{cl}
\Gamma^{ab} & \mbox{ for } i\equiv 2 \Mod{4} \\
\Z/|\Gamma | & \mbox{ for } i\equiv 0\Mod{4} \\
0 & \mbox{ for } i\equiv 1,3 \Mod{4} \end{array} \right. \]
Moreover, for any generator $e\in H^4(\Gamma;\Z)$ the cup product
$e\cup \colon H^i(\Gamma;\Z)\cong  H^{i+4}(\Gamma;\Z)$ is an isomorphism
for each $i>0$. \end{lemma}
\begin{proof} Let $\Gamma$ act freely on $S^3$ in the standard way.
The action induces a fibration $S^3\ra S^3/\Gamma \ra B\Gamma$ (see for instance \cite[Lemma~6.2]{AdemMilgram04}).   
The associated Gysin sequence breaks up into the following pieces of
information:
\begin{enumerate}[label=(\arabic*),ref=(\arabic*)]
\item There are isomorphisms $H^i(\Gamma)\cong H^i(S^3/\Gamma)$ for $0\leq i\leq 2$.
\item There is an exact sequence 
\[ \begin{tikzcd} 0 \arrow{r} & H^3(\Gamma)\arrow{r} & H^3(S^3/\Gamma) \arrow{r}{p^*} & H^3(S^3) \arrow{r} & H^4(\Gamma) \arrow{r} & 0,
\end{tikzcd} \]
where $p:S^3\ra S^3/\Gamma$ is the quotient map.
\item The map $e\cup \colon H^i(\Gamma)\ra H^{i+4}(\Gamma)$ is surjective for 
$i=0$ and an isomorphism for $i\geq 1$, where $e\in H^4(\Gamma)$ is the Euler class of the fibration. 
\end{enumerate}
By Lemma \ref{BP-homology} we have $H^1(S^3/\Gamma)=0$ and $H^2(S^3/\Gamma)=\Gamma^{ab}$ so that
$H^1(\Gamma)=0$ and $H^2(\Gamma)\cong \Gamma^{ab}$ by the first point. Since
$p\colon S^3\ra S^3/\Gamma$ is a covering map of degree $|\Gamma|$, it follows
from the second point that $H^3(\Gamma)=0$ and $H^4(\Gamma)\cong \Z/|\Gamma|$.
The rest of the statement now follows from the third point.  
\end{proof} 

The above lemma does not give complete information about the product
structure in $H^*(\Gamma)$. The underlying Serre spectral sequence
giving rise to the Gysin sequence simply does not contain information
about the cup product $H^2(\Gamma;\Z)\times H^2(\Gamma;\Z)\ra H^4(\Gamma;\Z)$.
It is of course known that $H^*(\Z/m,\Z)\cong \Z[\beta]/(m\beta)$ with
$|\beta|=2$ for cyclic groups \cite[Exercise~6.7.4]{Weibel94}.
For the binary icosahedral group $I^*$ one has $(I^*)^{ab}=0$, so the cup product is necessarily trivial. However, for the remaining groups more refined techniques seem to be required.

\begin{definition} Define $e_\Gamma\in H^4(\Gamma;\Z)$ to be
$c_2(\iota)$ where $\iota\colon \Gamma\subset \om{SU}(2)$ denotes the
inclusion. \end{definition} 

\begin{proposition} Let $\Gamma\subset \om{SU}(2)$ be a finite subgroup.
Then $e_\Gamma\in H^4(\Gamma;\Z)$ is a generator. Define
$\tau\colon H^4(\Gamma;\Z)\ra \R/\Z$ to be the group homomorphism given by
$\tau(ke_\Gamma) = k/|\Gamma | \mod{\Z}$. Let $\alpha$ be a flat
connection in the trivial $\om{SU}(2)$-bundle over $Y_\Gamma$ with
associated representation $\rho_\alpha\colon \Gamma\ra \om{SU}(2)$. Then
\[ \tau(c_2(\rho_\alpha)) = -\om{cs}(\alpha) .\]
\end{proposition} 
\begin{proof} According to Lemma \ref{Classification-Lemma} there
exists an equivariant map $g\colon S^3\ra \om{Sp}(1)\cong \om{SU}(2)$ in the
sense that $g(\gamma\cdot q)= \rho_\alpha(\gamma)\cdot g(q)$ for each $\gamma\in \Gamma$ and $q\in S^3$. Composing this map
with the orientation reversing orbit map $\om{SU}(2)\ra S^3$ given by
$x\mapsto x\cdot 1$, we obtain a map $h:S^3\ra S^3$, which is equivariant
along $\rho_\alpha:\Gamma\ra \om{SU}(2)$. Here $\om{SU}(2)$ acts
on $S^3\subset \C^2$ in the standard way. Note that $\om{deg}h = -\om{deg}g$. There is an induced morphism of fibrations
\[ \begin{tikzcd} S^3 \arrow{d}{h} \arrow{r} & E\Gamma \times_\Gamma S^3 \arrow{d} \arrow{r} & B\Gamma \arrow{d}{B\rho_\alpha} \\
 S^3 \arrow{r} & E\om{SU}(2)\times_{\om{SU}(2)} S^3 \arrow{r} & B\om{SU}(2),
\end{tikzcd}  \]
where $E\Gamma \ra B\Gamma$ and $E\om{SU}(2)\ra B\om{SU}(2)$ are the
universal bundles. As both $\Gamma$ and $\om{SU}(2)$ act freely on
$S^3$ there are homotopy equivalences
\[ E\Gamma \times_\Gamma S^3 \simeq S^3/\Gamma \; \mbox{ and } \;
 E\om{SU}(2) \times_{\om{SU}(2)} S^3 \simeq S^3/\om{SU}(2)=* .\]
The above morphism of fibrations induces a morphism of the associated (cohomology) Serre spectral sequences. As both fibrations have $S^3$ as fiber
the only nontrivial differentials are $d_4\colon E^{p,3}_4\ra E^{p+4,0}_4$ for $p\geq 0$. In particular, we obtain a commutative diagram
\[ \begin{tikzcd} H^3(S^3;\Z)= E^{0,3}_4 \arrow{r}{d_4} \arrow{d}{h^*} &
E^{4,0}_4=H^4(B\om{SU}(2);\Z)  \arrow{d}{B\rho_\alpha ^*} \\
H^3(S^3;\Z)=E^{0,3}_4 \arrow{r}{d_4} & E^{4,0}_4=H^4(B\Gamma;\Z)
\end{tikzcd} \] 
The upper differential must be an isomorphism since the spectral sequence
converges to $H^*(*)=\Z$. Let $u\in H^3(S^3;\Z)$ be the unique generator
with $d_4(u) = c_2 \in H^4(B\om{SU}(2);\Z)$. Then from the commutativity
of the diagram we obtain
\[ c_2(\alpha)=(B\rho_\alpha)^*(c_2) = d_4(h^*u)=(\om{deg}h) d_4(u)=(\om{deg}h)e_\Gamma  .\]
For the final equality, consider the same commutative diagram with
$\rho_\alpha$ replaced by the inclusion $\iota\colon \Gamma\subset \om{SU}(2)$.
Then we can take $h\colon S^3\ra S^3$ to be the identity, and it follows from
the same calculation as above that $e_\Gamma\coloneqq B\iota^*(c_2)=d^4(u)$.
This also shows that $e_\Gamma$ is a generator for $H^4(B\Gamma;\Z)$
as the lower differential must be surjective because $H^4(S^3/\Gamma;\Z)=0$. 

To conclude, recall from Lemma \ref{CS-degree-formula} that for the
equivariant map $g\colon S^3\ra \om{SU}(2)$ we started with, it holds true
that $\om{cs}(\alpha) = \om{deg}(g)/|\Gamma | \mod{\Z}$. Hence,
\[ \tau(c_2(\alpha)) = \tau(\om{deg}(h)e_\Gamma)
= -\tau(\om{deg}(g) e_\Gamma) = -\om{deg}(g)/|\Gamma | \mod{\Z}  \]
and the proof is complete. \end{proof} 

\begin{example} Let $\Gamma = T^*$. There are three flat connections
$\theta$, $\alpha$ and $\lambda$. Here $\theta$ corresponds to the
trivial representation, $\alpha$ corresponds to the canonical representation
$Q$, while $\lambda$ is reducible and corresponds to a representation
$\rho\oplus \rho^*$ where $\rho\colon T^*\ra U(1)$ is one of the $1$-dimensional representations (see the character table in $A.3$). By example
\ref{CS-Calc-Example1} we know that $\om{cs}(\alpha)=-1/|T^*|=-1/24$,
and, of course, $\om{cs}(\theta)=0$. To determine $\om{cs}(\lambda)$
we solve the equation $(2-Q)\Har = \alpha-\lambda$ in $R(\Gamma)$ using the technique of Proposition \ref{R(G)-Adjacent-Solution}. This yields a solution
$\Har$ with $\eps(\Har)=\om{dim}_\C \Har=15$. Then by Proposition
\ref{CS-Calc-Tool} we obtain
\[ \om{cs}(\alpha)-\om{cs}(\lambda) = \eps(\Har)/|T^*| = 15/24 \mod{\Z} .\]
Hence, $\om{cs}(\lambda) = -2/3 = 1/3 \mod{\Z}$.

This calculation can be used to determine the cup product 
$H^2(T^*;\Z)\times H^2(T^*;\Z)\ra H^4(T^*;\Z)$ using the
above proposition. As $c_1\colon \om{Hom}(T^*,U(1))\ra H^2(T^*;\Z)$ is
an isomorphism, we can take $c_1(\rho)\in H^2(T^*;\Z)\cong \Z/3$
as a generator. Then
\[ -c_1(\rho)^2 = c_1(\rho)\cup c_1(\rho^*) = c_2(\rho\oplus \rho^*)
= c_2(\lambda) = 8e_\Gamma .\]
In particular, the cup product is an isomorphism from 
$\Z/3\otimes \Z/3 \cong \Z/3$ onto the unique cyclic subgroup $8\Z/(24)$
of order $3$.   
\end{example}

\printbibliography

\end{document}